\documentclass[11pt]{report}
\usepackage[dvips]{graphicx} 
\usepackage{color} 

\newcommand\DMO[2]{\DeclareMathOperator{#1}{#2}}

\usepackage{amsmath,amsthm,amsfonts,amscd,amssymb}
\usepackage{color}

\usepackage[all]{xy} 

\usepackage{makeidx}

\newcounter{enumitemp}
\newenvironment{enumeratecontinue}{
  \setcounter{enumitemp}{\value{enumi}}
  \begin{enumerate}
  \setcounter{enumi}{\value{enumitemp}}
}
{
  \end{enumerate}
}

\numberwithin{equation}{section}

\newcommand\pref[1]{(\ref{#1})}

\newcommand\marginparLee[1]{\marginpar{\tiny #1}}

\newcommand\ds\displaystyle

\newcommand\vp{{\vphantom\prime}}

\theoremstyle{plain}

\newtheorem*{theorem*}{Theorem}
\newtheorem{theorem}{Theorem}[section]
\newtheorem{proposition}[theorem]{Proposition}
\newtheorem{lemma}[theorem]{Lemma}

\newtheorem{corollary}[theorem]{Corollary}

\theoremstyle{definition}
\newtheorem{definition}[theorem]{Definition}

\theoremstyle{definition}
\newtheorem{exercise}{Exercise}[subsection]

\DeclareMathOperator{\Out}{Out}
\DeclareMathOperator{\Aut}{Aut}
\DeclareMathOperator{\Inn}{Inn}
\DeclareMathOperator{\rank}{rank}
\DeclareMathOperator{\Length}{Len}

\DeclareMathOperator{\Stab}{Stab}

\DeclareMathOperator\image{image}
\DeclareMathOperator\closure{cl}

\DeclareMathOperator\Vertices{\mathcal{V}}
\DeclareMathOperator\Edges{\mathcal{E}}

\DeclareMathOperator\Spec{Spec}
\DeclareMathOperator\sys{sys}

\newcommand\reals{{\mathbf R}}
\newcommand\euc{{\mathbf E}}
\newcommand\hyp{{\mathbf H}}
\newcommand\N{{\mathbf N}}

\newcommand\Z{{\mathbf Z}}
\newcommand\complex{{\mathbf C}}

\newcommand\inject{\hookrightarrow}
\newcommand\infinity{\infty}
\newcommand{\bdy}{\partial}
\newcommand{\from}{\colon}
\newcommand\composed{\circ}
\newcommand\suchthat{\bigm|}
\newcommand\inv{{-1}}
\newcommand\union{\cup}

\newcommand\abs[1]{\left| #1 \right|}

\newcommand\intersect{\cap}

\newcommand\restrict{\bigm|}
\newcommand\subgroup{<}
\newcommand\cross{\times}

\newcommand\C{\mathcal C}
\newcommand\D{\mathcal D}
\newcommand\E{\mathcal E}

\newcommand\K{{\mathcal K}}
\renewcommand\L{\mathcal L}

\newcommand\<\langle
\renewcommand\>\rangle
\newcommand\wh\widehat
\newcommand\disjunion\sqcup

\DeclareMathOperator\interior{int}

\DeclareMathOperator\ab{ab}

\newcommand\act\curvearrowright

\newcommand\X{\mathcal{X}}
\newcommand\CV\X

\newcommand\BookOne{\cite{BFH:TitsOne}}
\newcommand\TitsOne\BookOne

\newcommand\ti {\tilde}

\DeclareMathOperator\Ends{Ends}
\DMO\Core{Core}
\DMO\ACore{{\hat{\mathcal{C}}}}
\DMO\truss{truss}
\newcommand\wt\widetilde

\newcommand\Ad{\text{Ad}}

\newcommand\collapsesto\succ
\newcommand\collapse\collapsesto
\newcommand\collapses\collapsesto
\newcommand\expandsto\prec
\newcommand\expand\expandsto
\newcommand\expands\expandsto
\newcommand\tensor\otimes

\DeclareMathOperator\GL{\mathrm{GL}}
\DeclareMathOperator\SL{\mathrm{SL}}
\DeclareMathOperator\MCG{\text{MCG}}
\DeclareMathOperator\HMCG{HMCG}
\DeclareMathOperator\HEnd{\text{HEnd}}

\newcommand\Homeo{Homeo}

\newcommand\EG{\mathcal EG}  
\newcommand\EH{\mathcal EH}
\newcommand\LG{\mathcal LG}

\newcommand\turn\widehat
\newcommand\turngraph\Lambda

\newcommand\Poincare{Poincar\`e}
\newcommand\Teichmuller{Teichm\"uller}

\newcommand\indexemph[1]{\index{#1}\emph{#1}}

\title{The topology, geometry, and dynamics of free groups \\ \quad \\ 
\small{(containing most of):} \\ \quad \\
\large{Part I: Outer space, fold paths, and the Nielsen/Whitehead problems}}

\author{Lee Mosher}

 \makeindex
 
\begin{document}

\maketitle

\setcounter{tocdepth}{3}
\tableofcontents

\vfill\break

\section*{VERSION NOTES}

\bigskip

Chapters, sections, and subsections of this work will be released in serial version. Here is a brief account of what's here and what to expect ``soon''.

\begin{itemize}
\item Part I Chapter 1: complete
\item Part I Chapter 2: partially complete
\begin{itemize}
\item Sections 2.1--2.5: complete.
\item Section 2.6:  The first subsection is present in draft form but needs polish. Later subsections exist in rough drafts, but are represented in this version only as stubs. \emph{Top priority} for subsequent versions.
\item Section 2.7: exists in rough draft. In this version it is mostly a stub, represented solely as a statement of one theorem. \emph{Second priority} for subsequent versions.
\end{itemize}
\item Part II Chapter 3: exists in draft form, some sections being somewhat complete. In this version, a few sections are represented by stubs with titles, to whet the readers appetite for later versions. But, even those section divisions and their titles are subject to change.\emph{Third priority} for subsequent versions.
\item Part II Chapter 4, and Part III Chapter 5, both exist in very rough draft forms. In this version, each is a stub represented solely by chapter titles.
%
\end{itemize}

\vfill\break

\section*{Introduction.}

\hfill{\emph{Moses supposes his toeses are roses.}}

\hfill{\emph{But Moses supposes erroneously.}}

\smallskip

\hfill  --- Adolph Green and Betty Comden

\hfill from their musical \emph{Singin' in the Rain}

\bigskip

Geometric group theory studies a group $\Gamma$ using topology and geometry. For instance, a topological model of $\Gamma$ might be a space $X$ whose fundamental group $\pi_1 X$ is isomorphic to $\Gamma$. A geometric model might be obtained by choosing some geometric structure on $X$, and then lifting that structure to the universal cover $\wt X$ to obtain a deck action $\Gamma \act \wt X$ whose elements are isometries of the geometry on~$\wt X$. 

When a group $\Gamma$ has an interesting outer automorphism group $\Out(\Gamma)$, the topology, geometry, and dynamics of $\Out(\Gamma)$ can be enriched by considering how to ``vary'' or ``deform'' topological and geometric models of $\Gamma$. If one can package these deformations into a ``deformation space'' on which $\Out(\Gamma)$ itself acts properly and cocompactly with an invariant geometry, then the tools of geometric group theory can then be applied to this deformation space in order to study the group $\Out(\Gamma)$ itself. Also, one can learn a lot about $\Gamma$ and $\Out(\Gamma)$ by studying the dynamical behavior of individual elements of $\Out(\Gamma)$ acting on the deformation space. 

Before explaining our main theme about using deformation spaces to understand outer automorphism groups of free groups, we first describe some more classical examples of deformation spaces with which the reader may be familiar.

\paragraph{Deformation spaces from classical geometry.} Look at a torus $T^2$. Its fundamental group is $\Z^2$. There are many Euclidean metrics on $T^2$, and for each of them, the universal cover $\wt T^2$ is isometric to the Euclidean plane $\euc^2$ equipped with an isometric deck transformation action $\Z^2 \act \euc^2$. This action has a parallelogram as its fundamental domain, whose sides are displacement vectors for the action be the standard basis elements of~$\Z^2$. By deforming the metric on $T^2$, or equivalently by deforming the deck action $\Z^2 \act \euc^2$, or equivalently by deforming the shape of the fundamental parallelogram, one produces a nice deformation space on which the group $\Out(\Z^2) \approx \GL_2(\Z)$ acts. To be precise, after normalizing the action $\Z^2 \act \euc^2$, conjugating it appropriately by a rotation and similarity, the fundamental parallelogram in $\euc^2 \approx \complex$ has vertices $0,1,z,z+1$ for a unique $z \in \complex$ having positive imaginary part. The index 2 subgroup $\Out_+(\Z^2) \approx \SL_2(\Z)$ acts on the upper plane of $\complex$ by fractional linear transformations, preserving the \Poincare\ metric. 

Thus we obtain one of the marvelous facts of mathematics: the deformation space of euclidean structures on a torus is the hyperbolic plane $\hyp^2$, equipped with the fractional linear action of $SL_2(\Z)$.

As will be explained in Chapter~3, the growth behavior of an individual nontrivial element $M \in \SL_2(\Z) \approx \Out(Z^2)$ is closely tied with the geometry of the action of $M$ on $\hyp^2$: $\abs{tr(M)} < 2$ if and only $M^k(x)$ is periodic for each $x \in \Z^2$, if and only if $M$ has a fixed point in $\hyp^2$; $\abs{tr(M)} > 2$ if and only if $M^k(x)$ grows exponentially for each nonzero $x \in \Z^2$ if and only if $M$ is a loxodromic isometry of $\hyp^2$, translating along an invariant geodesic; and $\abs{tr(M)} = 2$ if and only if $M^k(x)$ has linear growth for some elements $x \in \Z^2$ if and only if $M$ is a parabolic isometry of $\hyp^2$, fixing a unique point on the circle at infinity.

The example of $T^2$ generalizes to the $n$-dimensional torus $T^n$, whose fundamental group is $\Z^n$. By lifting Euclidean structures on $T^n$ one obtains isometric actions $\Z^n \act \euc^n$. By normalizing the action, one can produce a nice deformation space. In the language of Lie groups, this deformation space is the symmetric space often denoted $SL(n,\reals)/SO(n,\reals)$, which can be thought of concretely as the space of normalized quadratic forms on $\reals^n$. The group $\Out(\Z^n) \approx GL_n(\Z)$ acts naturally on this symmetric space. The dynamical behavior of an individual element $M \in \GL_n(Z)$ under this action is closely related to the real Jordan form of $M$.

The example of $T^2$ also generalizes to $2$-dimensional closed, oriented surfaces $S_g$ of genus $g \ge 2$, with fundamental group $\pi_1(S_g)$. There are many hyperbolic structures on $S_g$, and the lift of each such structure to the universal cover $\wt S_g$ is isometric to the hyperbolic plane $\hyp^2$, leading to an isometric deck action $\pi_1(S_g) \act \hyp^2$. There is a canonical isomorphism $\Out(\pi_1(S_g)) \approx \MCG(S_g)$, the mapping class group of the surface $S_g$. The corresponding deformation space of hyperbolic structures is known as the \Teichmuller\ space of $S_g$, on which the group $\Out(\pi_1(S_g)) \approx \MCG(S_g)$ acts. The dynamics of an individual element of this action, after initial investigations by Nielsen in the 1940's, came to full fruition  with the studies of Thurston in the 1970's.

\paragraph{The outer space of $F_n$ and its action by $\Out(F_n)$.} The finite rank free groups $F_n$ and their outer automorphism groups $\Out(F_n)$ have a rich theory that has been studied by developing analogies with the theory of surfaces and their mapping class groups, and of lattices in classical Lie groups such as $GL_n \Z$. These analogies occur at the ``top level'', however, and often do not extend down to the lower, technical levels. Instead of employing the tools of classical geometry --- analysis, complex structures, hyperbolic geometry, Riemannian metrics, quadratic forms etc.\ etc. --- the topology, geometry, and dynamics of free groups are studied using\ldots\ldots\, \emph{graphs}. This lack of classical structure can make the study of free groups seem, oddly enough, rather technical---one must ask a lot from graphs, in order to prove the deepest theorems about $F_n$ and $\Out(F_n)$.

In the examples of $\Z^n$ and $\pi_1(S_g)$ mentioned above, it was unnecessary to vary the topological model to get a good deformation space. Those topological models $T^n$ and $S_g$, and their appropriate geometries, have a ``canonical'' nature to them. For example, all of the Euclidean structures on $T^n$ have universal cover isometric to $\euc^n$, all of the hyperbolic structures on $S_g$ have universal cover isometric to $\hyp^2$. Varying these structures, equivalently varying the actions on the universal cover, suffices for producing a good deformation space. One way to capture this canonical nature is using theorems saying that each element of the outer automorphism group of $\Z^n$ or of $\pi_1(S_g)$ is represented by a homeomorphism of $T^n$ or of $S_g$, respectively.

To study the rank~$n$ free group $F_n$, one might choose the rank~$n$ rose $R_n$ as a topological model for $F_n$. But if Moses supposes that roses are the only topological models needed for the study of $F_n$, then Moses supposes erroneously. The rose does not suffice for modelling $F_n$, in the same way that the topological models described above suffice for $\Z^n$ and $\pi_1(S_g)$. The group $F_n$ is isomorphic to the fundamental group of any connected graph of Euler characteristic $1-n$. There are multiple choices for such graphs, even with reasonable restrictions such as finiteness of the graph and no valence~$1$ vertices. The three such ``core graphs'' in rank~$2$ are depicted in Figure~\ref{FigureRankTwoCoreGraphs}. No matter what core graph one chooses as a model for~$F_n$, the homeomorphisms of that graph represent only a finite subgroup of the infinite group $\Out(F_n)$. In order to study $F_n$ and $\Out(F_n)$, one needs to vary not only the geometries of core graphs --- by varying the lengths of its edges --- one needs also to vary the choice of the graph itself. 

The idea of using graphs to prove theorems about free groups goes back to work of Dehn and Reidemeister who gave a topological proof of the Nielsen--Schreier theorem \cite{Nielsen:Noncommutative} saying that every subgroup of a free group is free; for the history of this topological proof see Section 2.1.1 of Stillwell's book \cite{Stillwell:TopologyAndGroupTheory}. It is arguable that a solid intuition for the relation between graphs and free groups was behind the earliest advances of Nielsen and Whitehead in the 1920s and '30s. Starting in the late 1980's and early 1990's, particularly with the work of Culler and Vogtmann on outer space \cite{CullerVogtmann:moduli} and the work of Bestvina and Handel on relative train tracks \cite{BestvinaHandel:tt}, the concept of core graphs plainly emerged as the \emph{correct} basis for the study of the topology, geometry, and dynamics of free groups. 

The object of this work is to try to unveil the topological/geometric intuition behind the theory of free groups. The method we follow is to focus on a series of problems in the study of free groups, and use the solutions of those problems to motivate topological/geometric tools. We do not aim to write down proofs which minimize the number of alphanumeric characters. We instead strive to write down proofs which maximize the development of broadly applicable geometric tools.

In Part I we study problems solved by Nielsen and Whitehead in the 1930's, but we approach these problems from a modern topological/geometric viewpoint and we formulate their solutions so as to motivate marked graphs and fold paths and, eventually, the outer space of $F_n$. 

In Part II we study growth properties of outer automorphisms of $F_n$, and we use that study to motivate relative train track maps as topological/dynamical representatives of outer automorphisms.

In Part III we study dynamical properties of the topological representatives that were introduced in Part II. We focus in particular on studying how to count fixed points and periodic points of relative train track maps. We shall use that study to motivate various asymptotic properties of relative train track maps and of the outer automorphisms that they represent.

\paragraph{Exercises.} Exercises are chosen with many goals. Sometimes they are simply examples, other times the first hint of concepts which will be developed in full at a later stage. They may be true exercises for the student to build their mathematical muscles, or they may be bits of proof which the author is either too lazy or exhausted to write. They range from easy to unsolved, and from tedious to fascinating. 

We often use exercises to immediately explore and understand a new concept, and occasionally there will be a subsection consisting almost exclusively of exercises with new definitions interspersed; the first occurrence of this is Section~\ref{SectionBasedHMCG}.

\paragraph{Acknowledgements.} This work grew out of a semester course taught twice at Rutgers--Newark, and minicourse versions taught at the University of Chicago and at a conference at Princeton University.

\paragraph{Disclaimers.} Please feel free to inform me via e-mail of any typographical errors, mis-citations, mathematical errors, or other similar errors, which I will try to swiftly correct.
\begin{itemize}
\item My e-mail: \texttt{mosher@newark.rutgers.edu}
\end{itemize}

\part{Outer space, fold paths, and the Nielsen/Whitehead problems}

\section*{Introduction to Part I.}

A rank $n$ free group $F_n$ has a rich automorphism group $\Aut(F_n)$ and a similarly rich outer automorphism group $\Out(F_n)$, as exhibited already by Nielsen and Whitehead in their works \cite{Nielsen:Noncommutative,Whitehead:CertainSets,Whitehead:EquivalentSets}. In those papers they pursued and solved various algebraic and computational questions about free groups. Consider, for example free basis $s_1,\ldots,s_n$ and an $n$-tuple of elements $w_1,\ldots,w_n \in F_n$. How do you tell whether there exists an automorphism $\Phi \in \Aut(F_n)$ such that $\Phi(s_i)=w_i$ for all $i=1,\ldots,n$? How do you tell whether $\Phi$ exists so that $\Phi(s_i)$ is conjugate to $w_i$ for $i=1,\ldots,n$? Whitehead gave a marvelous 
algorithmic solution to these problems in his paper \cite{Whitehead:CertainSets}.  

Over the course of part I we shall slowly explore these problems of Nielsen and Whitehead, using these explorations to develop modern tools used today for studying $\Aut(F_n)$ and $\Out(F_n)$, and building up to a modern geometric version of Whitehead's solutions.

In Chapter 1, using just the formulations of the Nielsen/Whitehead problems as motivation and justification, we shall develop the following tools:

\smallskip $\bullet$ Geometric structures for free groups known as \emph{marked graphs} with \emph{length structures}.

\smallskip $\bullet$  \emph{Outer space}, a deformation space of marked graphs and length structures.

\smallskip\noindent
The outer space of $F_n$, denoted $\X_n$, was introduced by Culler and Vogtmann in the paper \cite{CullerVogtmann:moduli}, where they proved $\X_n$ is contractible, which they immediately applied to deduce important results about the group $\Out(F_n)$. Here in Chapter 1, even without contractiblity of $\X_n$, we will get some immediate applications of marked graphs, such as: the classification of finite subgroups of $\Out(F_n)$, due independently to Culler \cite{Culler:FiniteSubgroups}, Zimmerman \cite{Zimmerman:FiniteSubgroups}, and Khramtsov~\cite{Khramtsov:FiniteSubgroups}; and the existence of a torsion free finite index subgroup of $\Out(F_n)$ due to Baumslag and Taylor \cite{BaumslagTaylor:Center}.

In Chapter 2, we will continue our explorations of Whitehead's solutions, using them to formulate further tools, including:

\smallskip$\bullet$ \emph{Fold paths} in outer space

\smallskip\noindent
Fold sequences were introduced by Stallings \cite{Stallings:folding}, as a geometric method for getting modernized solutions of old problems about free groups. When fold sequences qre applied in the context of the outer space $\X_n$, they may be reconfigured as \emph{fold paths} in outer space, and used them to navigate $\X_n$ in a manner not unlike how geodesics and other nice paths are used in Riemannian geometry. For example, Skora used fold paths to give a new proof of contractibility of $\X_n$ \cite{Skora:deformations}. We will show how the exploration of Whitehead's problems leads naturally to the concept of fold paths, and to a modernized description of Whitehead's algorithms which are used to solve those problems. We will also give a modern derivation of Nielsen's generators, and the closely related Whitehead generators, for $\Aut(F_n)$ and $\Out(F_n)$.

\chapter{Marked graphs and outer space: \\ Topological and geometric structures for free groups}
\label{ChapterMarkedGraphs}

\section{Free groups and free bases} 
\label{SectionFreeGroupsAndBases}

\centerline{\emph{What is the sound of one hand clapping?}}

\smallskip

\hfill  --- Zen koan

\bigskip

Consider a group $F$, a subset $S \subset F$, and a set of ``formal inverses'' $\overline S = \{\bar s \suchthat s \in S\}$. The details of the definition of $\overline S$ are unimportant, all that matters are that $S,\overline S$ are disjoint and that they are equipped with a specific bijection $s \leftrightarrow \bar s$; but one could, for example, define $\overline S = S \times \{0\}$ and $\bar s = (s,0)$. A \emph{word over $S \union \overline S$} means simply a sequence $w = (s_1,\ldots,s_K)$ where $s_k \in S \disjunion \overline S$ for each $k=1,\ldots,K$, and the integer $K$ is called the \emph{length} of the word. The \emph{product} or \emph{evaluation} of the word $w$ is simply the product defined using the group operation on $F$, and is written as a concatenation $s_1 \cdots s_K$; the associative law in $F$ guarantees that parentheses may be ignored in this concatenation. In combinatorial group theory one almost universally abuses notation by confusing the word $(s_1,\ldots,s_K)$ with its concatenation $s_1 \cdots s_K$; hopefully the reader will get used to this. Any equation of the form $g = s_1 \cdots s_K$ can be simplified by canceling consecutive letter pairs of the form $s s^\inv$ or $s^\inv s$. A word is \emph{reduced} if no cancellation is possible, meaning that no such consecutive letter pairs exist. Since cancellation shortens the word, repeated cancellation must always stop at a reduced word. Even the word of length $K=0$ is covered by these definitions, as long as the mathematical koan ``What is the product of the empty word?'' is answered by ``The identity element''.

We say that $S \subset F$ is a \emph{free basis} for $F$ if every element $g \in F$ is equal to the product of a \emph{unique} reduced word over $S \disjunion \overline S$. A group $F$ is said to be a \emph{free group} if it has a free basis $S$. If this is so then any two free bases of $F$ have the same cardinality, a number called the \emph{rank} of $F$, equal to the rank over $\Z$ of the abelianization of~$F$ (see Exercise~\ref{ExerciseFreeToAbelianBasis} below). In general, equality in $F$ is determined simply as equality of reduced words. Any finite word over $S \union \overline S$ can therefore be evaluated in the group $F$ by repeated cancellation until reaching a reduced word. For example the product of two reduced words is evaluated by concatenation and repeated cancellation.

One might wonder: Do free groups exist? The discussion above leads directly to a constructive definition. Given a set $S$, the \emph{free group with free basis $S$}, denoted $F(S)$~or~$F\<S\>$ or~$\<S\>$, is defined as follows. Again let $\overline S = \{\overline s \suchthat s \in S\}$ denote a set of formal inverses of the elements of $S$. Define the elements of $F\<S\>$ to be all reduced words over $S \disjunion \overline S$. Define a binary operation on $F\<S\>$ called ``concatenate and cancel'', as follows. Given two reduced words $(s_1, \ldots, s_K)$ and $(t_1, \ldots, t_L)$, first form the concatenation $(s_1, \ldots, s_K, t_1, \ldots, t_L)$. Then cancel repeatedly until one obtains a reduced word. The resulting reduced word is well-defined because at first the only possible cancelling pair is $(s_K, t_1)$: if they do not cancel then the concatenation is already reduced; if they do cancel, the only possible cancelling pair of the resultant word $(s_1, \ldots, s_{K-1}, t_2, \ldots, t_L)$ is $(s_{K-1},t_2)$; now continue by induction. The identity element of $F\<S\>$ is the empty word, and the inverse of a reduced word $s_1 \cdot \ldots \cdot s_K$ is $s_K^\inv \cdot \ldots \cdot s_1^\inv$, where we use the rule $(s^\inv)^\inv = s$. But there is one slightly tricky question about this definition:
\begin{itemize}
\item Is the ``concatenate and cancel'' operation on $F\<S\>$ associative?
\end{itemize}
For the moment we suspend resolution of this question. The reader may wish to work out an algebraic proof of associativity, perhaps by inducting on length and using a case analysis. In Section~\ref{SectionGraphsAndPaths} we will give a topological proof of associativity, arising from our geometric intuition of graphs and trees. Understanding this topological proof leads usefully to an understanding of other topological connections between algebra and free groups which will arise later.

\subparagraph{Notation for homomorphisms defined on free groups.} Exercise~\ref{ExerciseUniversal} invites the reader to prove the universality property for free groups, which tells us that a group homomorphism $F\<a,b,\ldots,c\> \mapsto \Gamma$ with domain a free group is well-defined \emph{and} uniquely defined by simply listing the images of the free basis elements in the following format: 
$$\begin{matrix} F \\ a \\ b \\ \vdots \\ c \end{matrix} \, \mapsto \, 
\begin{matrix} \Gamma \\ \alpha \\ \beta \\ \vdots \\ \gamma \end{matrix}
$$
Having done that we are then free to add into the mix the images of some other elements of $F$, for example
$$\begin{matrix} b a^\inv c \\ \vdots \\ a^{42} b
\end{matrix}
\, \mapsto \,
\begin{matrix} \beta \alpha^\inv \gamma \\ \vdots \\ \alpha^{42} \beta
\end{matrix}
$$

\paragraph{Exercises for Section \ref{SectionFreeGroupsAndBases}}

Here are a few exercises that have quick geometric proofs, when properly visualized.
\begin{exercise} 
\label{ExerciseFreeIsTorsionFree}
Prove that a free group is torsion free.
\end{exercise}

\begin{exercise} 
\label{ExerciseInversesNotConjugate}
Prove that the only element of a free group that is conjugate to its own inverse is the identity (we will revisit this in Exercise~\ref{ExerciseInversesNotConjugateRevisited}).
\end{exercise}

A reduced word $h  = (s_1,\ldots,s_K) \in F\<S\>$ is \emph{cyclically reduced} if $s_1 \ne \bar s_K$. Assuming $h$ is cyclically reduced, we say that $h$ is \emph{root free} if there does not exist a reduced word $v \in F\<S\>$ and an integer $p \ge 2$ such that $h = v^p$ (without cancellation).

\begin{exercise}\label{ExerciseMaximalCyclic}
Prove that every nontrivial element of a free group is contained in a unique maximal infinite cyclic subgroup. Prove furthermore that a reduced word $w \in F\<S\>$ generates a maximal cyclic subgroup if and only if there exist reduced words $g,h \in F\<S\>$ such that $w=ghg^\inv$ (without cancellation), and $h$ is cyclically reduced and root free.
\end{exercise}

The next exercise gives a universality property of free groups. It is something that one applies all the time without even thinking about it; the exercise that follows it is an example of such application.

\begin{exercise}
\label{ExerciseUniversal}
Prove that a free group $F$ with free basis $S$ satisfies the following universality property: for any group $\Gamma$, any function $S \mapsto \Gamma$ extends uniquely to a homomorphism $F \mapsto \Gamma$. 
\end{exercise}

\begin{exercise}
\label{ExerciseOrderedFreeBasis}
Consider a finite rank free group $F_n = F(s_1,\ldots,s_n)$. Consider an $n$-tuple of reduced words $(w_1,\ldots,w_n)$, not necessarily distinct. Prove that the homomorphism defined by $s_i \mapsto w_i$ is an automorphism of $F_n$ if and only if the $w_i$ are indeed pairwise distinct, and the set $W=\{w_1,\ldots,w_n\}$ is a free basis of $F_n$. 
\end{exercise}

\begin{exercise}
\label{ExerciseFreeBasesAndAuts}
Following up Exercise~\ref{ExerciseOrderedFreeBasis}, define an $k$-tuple of reduced words $(w_1,\ldots,w_k)$ to be an \emph{ordered free basis} of $F_n$ if the $w$'s are pairwise distinct and the subset $W=\{w_1,\ldots,w_k\}$ is a free basis of $F_n$. Prove that $\Aut(F_n)$ acts freely and transitively on the set of ordered free bases of cardinality~$n$. Hence, the map from $\Aut(F_n)$ to $n$-tuples of reduced words, given by $\Phi \mapsto (\Phi(s_1),\ldots,\Phi(s_n))$, is a bijection between $\Aut(F_n)$ and the set of ordered free bases of cardinality~$n$.
\end{exercise}

Exercise~\ref{ExerciseFreeBasesAndAuts} (along with other considerations) motivates the following question: if $F$ is a free group, do all of its free bases have the same cardinality? The following exercise answers this question affirmatively, allowing $\rank(F)$ to be well-defined as the cardinality of any free basis, which equals the dimension over~$\Z$ of the abelianization of~$F$.

\begin{exercise}
\label{ExerciseFreeToAbelianBasis}
Given a free group $F$, consider the abelianization homomorphism $F \mapsto \ab(F) = F / [F,F]$, where $[F,F]$ denotes the commutator subgroup of $F$. Prove that a free basis of $F$ maps one-to-one onto a basis of $\ab(F)$ in the sense of abelian groups. 
\end{exercise}

Here is a hint/outline for Exercise~\ref{ExerciseFreeToAbelianBasis}. Given a free basis $S$, let $\oplus_S \Z$ denote a direct sum of one copy of $\Z$ for each element of $S$. Define a map $m \from F \mapsto \oplus_S \Z$ in which for any word $w$ in $S$ and for any generator $s \in S$, the $s$-coordinate of $f(w)$ is equal to the sum of the exponents of the occurrences of $s$ in $w$. For example, in $F_2=F\<s_1,s_2\>$ we have $m(s_1)=\<1,0\>$, $m(s_2)=\<0,1\>$, $m(s_1^2 s_2 s_1^\inv s_2^\inv) = (1,0)$, $m(s_1s_2s_1s_2) = (2,2)$, etc. Prove that $m$ is surjective and that its kernel is $[F,F]$.

\section{Graphs} 
\label{SectionGraphs}

\hfill{\emph{[Sam:] ``But where have you been to?''}}

\hfill{\emph{``Sneaking,'' said Gollum...}}

\smallskip

\hfill  --- J.\ R.\ R.\ Tolkien, \emph{The Two Towers}

\subsection{Graphs, trees, and their paths.} 
\label{SectionGraphsAndPaths}
We recall some definitions and set some notation for paths. Consider a continuous path $\gamma \from [0,1] \to X$ in a space $X$. The \emph{path homotopy class} of $\gamma$ is denoted $[\gamma]$; when the endpoints of $\gamma$ are both at $p \in X$ then $[\gamma]$ is an element of $\pi_1(X,p)$. The \emph{reversal} of $\gamma$ is the continuous path $\bar\gamma \from [0,1] \to X$ defined by $\bar\gamma(t)=\gamma(1-t)$; when $\gamma$ has endpoints at $p \in X$ then $\bar\gamma$ represents $[\gamma]^\inv$. A \emph{reparameterization} of $\gamma$ means a continuous path of the form $\gamma\composed r$ where $r \from [0,1] \to [0,1]$ is an orientation preserving homeomorphism. Reparameterization is an equivalence relation on continuous paths which in general is much stronger than homotopy. Usually we regard two continuous paths as identical when one is a reparameterization of the other; for a short while we will remind the reader of this by speaking about a property of paths being true ``up to reparameterization'', but soon will will drop this reminder. 

A \indexemph{graph} $G$ is a 1-dimensional CW complex, its $0$-cells called \index{vertex}\emph{vertices}, its $1$-cells called \index{edge}\emph{edges}. We let $\Vertices(G) \subset G$ denote the set of vertices. Formally each edge $e$ has an \emph{interior} which is a component of $G - \Vertices(G)$, and $e$ itself is the closure of its interior; we let $\Edges(G)$ denote the set of edges. An edge $e$ may be \emph{parameterized} by a continuous function $\gamma \from [0,1] \to G$ so that $\gamma(0)$, $\gamma(1)$ are vertices and $\gamma$ maps $(0,1)$ homeomorphically onto the interior of~$e$; in the language of CW complexes, such a parameterization $\gamma$ is called a ``characteristic function'' of~$e$. Two parameterizations of $e$ are equivalent up to reparameterization if and only if they induce the same orientation on $e$. When an orientation of $e$ is specified, we use the notation $\bar e$ to denote the edge with its opposite orientation. And when an orientation preserving parameterization $\gamma$ is chosen for an oriented edge $e$, then $\gamma(0)$ is designated as the \emph{initial vertex} of $e$ and $\gamma(1)$ is the \emph{terminal vertex}; these designations are independent of the choice of $\gamma$. The \emph{boundary} or \emph{endpoint set} of $e$ is $\bdy e = \{\gamma(0),\gamma(1)\}$, and this set may consists of a single point $\gamma(0)=\gamma(1)$ or two distinct points. We say that $e$ is a \emph{loop edge} if $\gamma(0)=\gamma(1)$, otherwise $\gamma$ is a \emph{non-loop edge}.

To say that a graph $G$ is finite means that its vertex and edge sets $\Vertices(G)$, $\Edges(G)$ are finite, in which case we have the following result from algebraic topology:
\begin{theorem}[The Euler Characteristic Theorem]
$$\abs{\Vertices(G)} - \abs{\Edges(G)}  = \rank(H_0(G)) - \rank(H_1(G))
$$
and this number is defined to be the \emph{Euler characteristic} $\chi(G)$. If $G$ is connected we therefore have
$$\rank(H_1(G)) = \abs{\Edges(G)} - \abs{\Vertices(G)} + 1
$$
\qed
\end{theorem}

Define an \indexemph{edge path} in a graph $G$ from a vertex $v$ to a vertex $w$ to be a path $\gamma$ which (up to reparameterization) is a concatenation of oriented edges $\gamma = e_1 * e_2 * \cdots *e_K$, so that the initial vertex of $e_1$ is $v$, the terminal vertex of $e_{i-1}$ equals the initial vertex of $e_{i}$ for $i=2,\ldots,K$, and the terminal vertex of $e_K$ is $w$. We allow the possibility that $K=0$ and that $\gamma$ degenerates to the trivial path based at some vertex. An edge path is said to be \emph{reduced}, or \emph{tight}, or to have \emph{no cancellation}, if $e_{i-1} \ne \bar e_{i}$ for each $i=2,\ldots,K$; for a nondegenerate edge path this happens if and only if it is locally injective. 

More generally, a path in a graph is \emph{tight} if it is either constant or locally injective. Graphs have the following strong geometric property: every continuous path \hbox{$\gamma \from [a,b] \to G$} with endpoints at vertices may be uniquely \emph{tightened}, meaning that it may be path homotoped to a tight path, and the result is unique up to  reparameterization (by default, reparameterizations must preserve orientation). This induces a bijection between the set of tight edge paths (modulo reparameterization) and the set of path homotopy classes of continuous paths. Furthermore, if the endpoints $\gamma(a),\gamma(b)$ are vertices of $G$ then the tightened path which is path homotopic to $\gamma$ is either constant or is a nondegenerate tight edge path in~$G$. As a special case, for each vertex $v \in G$, we obtain a bijection between closed, tight edge paths based at~$v$ and nonidentity elements of $\pi_1(G,v)$; the identity element is, of course, represented by the constant path. The concept of connectivity of a graph has two equivalent formulations: the ordinary concept of a connected topological space; and a graph theoretic formulation expressed solely in terms of the incidence relation amongst vertices and edges (see Exercise~\ref{ExerciseGraphConnectivity}). 

These properties of graphs and tight edge paths are analogous to properties of a complete Riemannian manifold $M$ of nonpositive sectional curvature: any path homotopy class in $M$ may be uniquely tightened, representing it by a unique geodesic; in particular for each $p \in M$ each element of $\pi_1(M,p)$ is represented by a unique geodesic. 

\paragraph{Terminology conventions:}  \quad
\begin{itemize}
\item When working in a graph, the bare terminology ``\emph{path}'' will usually mean a \emph{tight edge path}. When we want more other kinds of path objects we shall usually append qualifiers, for example ``continuous paths''.
\item Given a graph $G$ and a subgraph $H \subset G$ we define two ``difference operations'':
\begin{itemize}
\item The \emph{set complement} or \emph{set difference} is $G-H = \{x \in G \suchthat x \notin H\}$. This is never a subgraph unless $H$ is a union of components of $G$, which in a context where $G$ is connected means $H=\emptyset$ or $G$.
\item The \indexemph{graph complement} or \emph{graph difference} is $G \setminus H = \overline{G-H}$, which as a set is the closure of $G-H$, and which is a subgraph of $G$ whose vertex set $\Vertices(G \setminus H) = (G \setminus H) \intersect \Vertices(G)$ consists of those vertices of $G$ that are not in the interior of $H$, and whose edge set $\Edges(G \setminus H)$ consists of those edges of $G$ not contained in $H$.
\end{itemize}
\end{itemize}

\paragraph{The fundamental group of a graph is free.} In topology we learn that for any connected graph $G$ and any vertex $v$ the fundamental group $\pi_1(G,v)$ is a free group. To quickly review the proof, start with the fact that $G$ has a maximal tree $T$---for a finite graph, the existence of $T$ follows by induction; for a general graph, use the Hausdorff maximal principal. Each maximal tree contains $\Vertices(G)$. If $T \subset G$ is a maximal tree, and if each edge of $G \setminus T$ is assigned an orientation, then there is a free basis of $\pi_1(G,v)$ in one-to-one correspondence with the edges of $G \setminus T$: the free basis element corresponding to an oriented edge $e$ with initial vertex $u_-$ and terminal vertex $u_+$ is the path homotopy class $[\delta_- \, e \, \bar \delta_+]$ where $\delta_-$ and $\delta_+$ are the unique paths in $T$ from $v$ to $u_-$ and $u_+$, respectively. Verifying that this subset is a free basis is an application of the Van Kampen theorem. 

\medskip

In the remaining portions of this section we will study trees, and we will give a topological proof (independent of Van Kampen's theorem) that the fundamental group of a graph is free, based on properties of trees and on covering space theory.

\paragraph{The definition of a trees.} 
A \emph{tree} is a connected graph $T$ which satisfies any of a long list of equivalent properties: 
\begin{enumerate}
\item\label{TreeContractible}
$T$ is contractible; 
\item $T$ is simply connected;
\item $H_1(T;\Z)=0$;
\item $T$ has no subgraph homeomorphic to a circle;
\item\label{TreeNoCircleSubspace}
$T$ has no subspace homeomorphic to a circle;
\item\label{TreeEdgePath} For any vertices $x,y \in T$ there is a unique oriented tight edge path with initial vertex $x$ and terminal vertex $y$;
\item\label{TreeArcUniqueness}
For any points $x \ne y \in T$ there is a unique subset of $T$ homeomorphic to $[0,1]$ having endpoints $x,y$;
\item\label{TreeListContr}
Every connected subgraph of $T$ is contractible;
\item\label{TreeListSimplConn}
Every connected subgraph of $T$ is simply connected;
\item versions of \pref{TreeListContr}--\pref{TreeListSimplConn} quantifying over connected \emph{finite} subgraphs of $T$;
\item\label{TreeEuler}
Every connected finite subgraph of $T$ has Euler characteristic $1$.\end{enumerate}
Perhaps item~\pref{TreeArcUniqueness} is the truly characteristic, defining property of a tree, particularly because of its strong, purely topological nature. In Exercise~\ref{ExerciseTreeEquivalence} below we invite the reader to work out a proof of equivalence of most of the above, but in Section~\ref{SectionTreesAreContractible} we shall prove the key implication  \pref{TreeArcUniqueness}$\implies$\pref{TreeContractible} which says ``trees are contractible''.

To actually construct a tree when you need one, induction is useful. Start with a vertex. Attach a bunch of edges each having one endpoint at the vertex. Repeat inductively, at each stage attaching a bunch of edges each having one endpoint at a vertex of the previous stage. Finally, take the union of all the stages of the induction, and you'll get a tree. See Exercises~\ref{ExerciseTreeInduction} and~\ref{ExerciseTreeUnion} for careful statements of the inductive step and the union step.

\subparagraph{Applications of trees to free groups.} In Section~\ref{SectionFreeGroupsAndBases}, we left open the question of whether the binary operation ``concatenate and cancel'' on the set $F\<S\>$ is associative. We now describe a geometric method, based on trees, that proves associativity, produces the Cayley graph for $F\<S\>$, and has other useful features. The central idea is that the entire set of reduced words can be visualized as a tree.

We start with some general definitions, which the reader may recognize in relation to Cayley graphs. Define an \emph{$S$-labelling} of a graph $G$ to be an assignment, to each edge $E \subset T$, of two pieces of data, such that certain properties hold. The data assigned to $E$ are: a label consisting of an element of $S$; and an orientation of $E$. The defining properties are that for each vertex $V \in G$ and each $s \in S$ there exists at most one edge labelled by $s$ having initial vertex $V$, and there exists at most one edge labelled by $s$ having terminal vertex~$V$. An $S$-labelled graph is said to be \emph{complete} if, in the defining properties, the two occurrences of the phrase ``at most one'' can be replaced by ``exactly one''. 

In any $S$-labelled graph $G$, each edge path $\gamma = E_1 \ldots E_L$ is labelled by a word in $S \disjunion \overline S$ denoted $W(\gamma) = w_1 \ldots w_L$, where $w_i = s^{\pm 1}$ if and only if $E_i$ is labelled by $s$, and the exponent is $+1$ (resp.\ $-1$) if and only if $\gamma$ passes over $E_i$ in the direction that agrees (resp.\ disagrees) with the orientation on $E_i$. If the $S$-labelling on $G$ is complete then for any vertex $V$ the word labelling function $W(\cdot)$ restricts to a bijection, denoted $W_V(\cdot)$, between edge paths having initial vertex $V$ and words in $S \disjunion \overline S$.

There exists a \emph{complete $S$-labelled tree}, denoted $T\<S\>$, 
which may be constructed by an inductive process following Exercises~\ref{ExerciseTreeInduction} and~\ref{ExerciseTreeUnion}: we construct a nested sequence of $S$-labelled trees $T_0 \subset T_1 \subset T_2 \subset \cdots$ and then take the union $T\<S\> = T_0 \union T_1 \union T_2 \union \cdots$. First take $T_0$ to be a single vertex. In the induction step, $T_{i+1}$ is obtained from $T_i$ by attaching new edges as required in Exercise~\ref{ExerciseTreeInduction} and described as follows: for each vertex $V \in \Vertices(T_i)$, each $s \in S$, and each $\epsilon \in \{-,+\}$, one checks to see whether $T_i$ already has an edge labelled $s$ with $\bdy_\epsilon E = V$; if not, attach a new edge of $T_{i+1}$ having those properties. Applying Exercise~\ref{ExerciseTreeUnion}, the union $T\<S\> = T_0 \union T_1 \union T_2 \union \cdots$ is a tree, and by induction one sees that $T\<S\>$ is a complete $S$-labelled tree. 

See Exercise~\ref{ExerciseSLabelledTreeUnique} for a uniqueness result regarding complete $S$-labelled trees.

Here are some key facts relating $T\<S\>$ and $F\<S\>$. Given a finite edge path $\gamma$ in $T\<S\>$ let its initial and terminal vertices be denoted $i(\gamma)$, $\tau(\gamma)$. For any complete $S$-labelled tree $T\<S\>$ and any $V \in \Vertices(T\<S\>)$ we have bijections
$$\xymatrix{
\{\text{edge paths $\gamma$ in $T\<S\>$ with $V=i(\gamma)$}\} \ar[r]^{W_V} & \{\text{words over $S \cup \overline S$}\}  \\
\{\text{tight edge paths $\gamma$ in $T\<S\>$ with $V = i(\gamma)$}\} \ar[r]^{W_V}  \ar[u]^{\subset} \ar[d]_{\tau_V} & \{\text{reduced words over $S \cup \overline S$}\}=F\<S\> \ar[u]_{\subset} 
\\
\Vertices(T\<S\>) \ar[ur]_{R_V}
}$$
The $W_V$ in the first line of the above diagram is the word labeling bijection explained earlier for any complete $S$-labelled graph. Furthermore, it is evident $\gamma$ is a tight edge path if and only if $W_V(\gamma)$ is a reduced word, and so the restricted $W_V$ in the second line of the diagram is also a bijection. Also, the fact that the terminal vertex function $\tau$ becomes the bijection $\tau_V$ when it is restricted to the set of tight edge paths with initial vertex $V$ is an immediate consequence of item~\pref{TreeEdgePath} in the list of equivalent conditions defining a tree. We also define the bijection $R_V = W_V \circ \tau_V^\inv$, making the triangle commute.

Another fact following immediately from the definitions is that concatenation of paths in $T\<S\>$ and words in $F\<S\>$ correspond precisely under $W_V$. In more detail, consider $V \in \Vertices(T\<S\>)$ and words $w,v$ over $S \disjunion \overline S$. Let $\gamma_w = W_V^\inv(w)$ denote the edge path in $T\<S\>$ corresponding to $w$ with initial vertex $V$. Let $\gamma_v = (W_{\tau(\gamma_w)})^\inv(v)$ denote the edge path corresponding to $v$ with $\tau(\gamma_w)=i(\gamma_v)$. Then $W_V(\gamma_w \gamma_v) = wv$.

Our first application is the following algebraic fact:
\begin{description}
\item[Cancellation uniqueness lemma:] For any word $w$ over $S \disjunion \overline S$ there exists a unique reduced word $r(w)$ such that if $w$ is inductively reduced by eliminating cancelling pairs until there are no such pairs, then the result is $r(w)$ --- \emph{independent of the order in which one eliminates cancelling pairs}. 
\end{description}
To prove this, consider $\gamma = W_V^\inv(w)$, the edge path in $T\<S\>$ with initial vertex $V$ corresponding to $w$. The correspondence $W_V$ evidently preserves cancellation, in the sense that eliminating from $w$ a cancelling pair $s s^\inv$ or $s^\inv s$ corresponds, under $W_V$, to removing from $\gamma$ a backtracking subpath of the form $E E^\inv$ or $E^\inv E$. Note that the terminal endpoint $\tau(\gamma)$ is unchanged by this removal. This backtrack elimination process, when carried out inductively on $\gamma$, must therefore end with the unique tight edge path $\tau_V^\inv(\tau(\gamma))$ having initial vertex $V$ and terminal vertex $\tau(\gamma)$. It follows that the cancellation process carried out on $w$ must end with the reduced word $r(w) = R_V^\inv(\tau(\gamma))$. 

Note that the above fact is an \emph{a posteriori} consequence of $F\<S\>$ being a group under the operation ``concatenate and cancel'', but the proof of the latter is not complete until we establish the next fact:

\begin{description}
\item[Associativity lemma for $F\<S\>$:] The binary operation ``concatenate and cancel'' on the set $F\<S\>$ is associative --- and hence is a group operation.
\end{description}
To prove this, consider any reduced words $u,v,w \in F\<S\>$, consider the concatenated word $uvw$, and let $\gamma = W_V^\inv(uvw)$ be the path in $T\<S\>$ with initial vertex $V$ that corresponds to $uvw$. The results of ``concatenate and cancel'', implemented as either of the two associations $(uv)w$ or $u(vw)$, both result in the reduced word $r(uvw)$.

Next we have:

\begin{description}
\item[Cayley lemma for $F\<S\>$:] The $S$-labelled tree $T\<S\>$ is a Cayley graph for the group $F\<S\>$ with respect to the generating set $S$. 
\end{description}
To see why, the first requirement for a Cayley graph is that $T\<S\>$ be a complete $S$-labelled graph, which it is by construction. The next requirement is a bijection between the vertex set and the group, which provided by the map $R_V \from \Vertices(T\<S\>) \to F\<S\>$. The final requirement is that the bijection $R_V$ satisfies the following: for any edge $E$ labelled by a generator $s \in S$, if its initial vertex $i(E)$ corresponds to $w \in F\<S\>$ then its terminal vertex $\tau(E)$ corresponds to $ws$. This follows from the correspondence shown earlier between concatenation of paths in $T\<S\>$ and concatenation of words in $F\<S\>$. 

The following application is simply a restatement in our current context of a general theorem about Cayley graphs. 
\begin{itemize}
\item The group $\Aut_S(T\<S\>)$ of simplicial automorphisms of $T\<S\>$ that preserves the $S$-labelling acts freely and transitively on the vertex set of $T\<S\>$ (see Exercise~\ref{ExerciseSLabelledTreeUnique}). The group $\Aut_S(T\<S\>)$ is isomorphic to $F\<S\>$. A formula $\alpha \from F\<S\> \to \Aut_S(T\<S\>)$ for this isomorphism, depending on a choice of $V \in \Vertices(T\<S\>)$, is as follows: the automorphism $\alpha_w \from T\<S\> \to T\<S\>$ corresponding to $w \in F\<S\>$ is given on vertices by
$$\alpha_w(V') = R_V^\inv(w w') \quad\text{for $V' \in \Vertices(T\<S\>)$ where $w' = R_V(V')$}
$$
\end{itemize}

\marginparLee{Put in this ``Exercises for \ldots'' in any subsection with have exercises. Put all exercises at the end of the subsection. Add pointers wherever I currently have an exercise that has to be moved to the end, to fix any continuity issues.}
\subsection*{Exercises for Section \ref{SectionGraphsAndPaths}}

\begin{exercise}\label{ExerciseGraphConnectivity}
Prove that a graph is connected in the topological sense if and only if for any vertices $v \ne w$ there exists an edge path from $v$ to $w$.
\end{exercise}

\begin{exercise}
\label{ExerciseEulerPoincare}
In any finite graph $G$, prove the graph theoretic analogue of the Euler-Poincare index formula for vector fields: 
$$\chi(G) = \sum_{v \in \Vertices(G)} \text{Index}(v) \qquad\text{where}\qquad \text{Index}(v) = 1 -  \frac{\text{Valence}(v)}{2}
$$
\end{exercise}

\begin{exercise}
\label{ExerciseTreeEquivalence}
Prove equivalence of items~\pref{TreeContractible}--\pref{TreeEuler} in the definition of a tree. Add your favorites to the list and prove their equivalence to the others.
\end{exercise}

\begin{exercise}
\label{ExerciseTreeInduction}
Let $T$ be a tree. Let $\{E_i\}_{i \in I}$ be a pairwise disjoint collection of spaces homeomorphic to $[0,1]$, and for each $i$ let $p_i \in E_i$ be one of the endpoints. Let $f \from \{p_i\}_{i \in I} \to \Vertices(T)$ be a function. Let $T'$ be the quotient space obtained from the disjoint union of the $T$'s and the $E_i$'s by identifying each $p_i$ with $f(p_i)$. Prove that $T'$ is a tree, in which $T$ and each $E_i$ are naturally embedded as subcomplexes.
\end{exercise}

\begin{exercise}
\label{ExerciseTreeUnion}
Let $T$ be a graph, and suppose that there exist nested subgraphs 
$$T_1 \subset T_2 \subset T_3 \subset \cdots \subset T
$$
each of which is a tree, and suppose that $T = \union_{i=1}^\infinity T_i$. Prove that $T$ is a tree.
\end{exercise}

\begin{exercise}\label{ExerciseSLabelledTreeUnique}
Prove that complete $S$-labelled trees are unique in the following sense: for any two complete $S$-labelled trees $T,T'$ and for any vertices $v \in T$, $v' \in T'$ there exists a unique $S$-label preserving graph isomorphism $f \from T \to T'$ satisfying $f(v)=v'$.
\end{exercise}

As we saw earlier, associativity for $F\<S\>$ follows from cancellation uniqueness in the tree $T\<S\>$. You can back-engineer this proof to give a purely algebraic proof of the associative law $(uv)w=u(vw)$ for $F\<S\>$, using pictures in $T\<S\>$ to formulate a finite case analysis. No one would suspect that your proof has topological origins!

\begin{exercise}
Construct a sneaky algebraic proof of associativity of $F\<S\>$ by using the topological proof of associativity as a guide for explicitly writing out all of the various cases needed for an algebraic proof.
\end{exercise}

One might ponder that since every tree is contractible, every tree is homotopy equivalent to a point. So any two trees are homotopy equivalent to each other. Not only that, but any continuous function between two trees is a homotopy equivalence. Not only that, but\ldots

\begin{exercise}
\label{ExerciseTreeHE}
Let $S,T$ be trees, and let $f \from S \to T$ and $g \from T \to S$ be any continuous functions. 
\begin{enumerate}
\item Prove that $f$ and $g$ are homotopy inverses. 
\item Prove, more generally, that for any vertex subsets $A \subset \Vertices(S)$ and $B \subset \Vertices(B)$, if $f$ restricts to a bijection between $A$ and $B$, then then the maps of pairs $f \from (S,A) \to (T,B)$ and $g \from (T,B) \to (S,A)$ are homotopy inverses in the category of topological pairs.
\end{enumerate}
\end{exercise}

\begin{exercise}
\label{ExerciseTreeCount}
Given an integer $t \ge 1$, define a \emph{$t$-labelled tree} to be a pair $(T,\ell)$ consisting of a finite tree $T$ together with a bijection $\ell$ from the set $\{1,\ldots,t\}$ to the set of valence~$1$ vertices of $T$. An \emph{isomorphism} between two $t$-labelled trees $(T,\ell)$, $(T',\ell')$ is a homeomorphism $f \from T \to T'$ such that $\ell'(f(i))=f'(i)$ for $i=1,\ldots,t$. Let $N(t)$ be the number of isomorphism classes of $t$-labelled trees. Given another integer $k \ge 0$ let $N(t;k)$ be the number of isomorphism classes of $t$-labelled trees having $k$ vertices of valence~$\ge 3$. Prove the following:
\begin{enumerate}
\item For all $t \ge 1$, $N(t)$ is finite. In addition, for all $k \ge 0$, $N(t;k)$ is finite.
\item $N(1)=0$; $N(2)=N(3)=1$; and $N(4)=4$. \\ Also, $N(2,0)=1$; $N(3,0)=N(4,0)=0$; $N(3,1)=N(4,1)=1$; and $N(4,2)=3$.
\item If $t \ge 3$ then: $N(t,0)=0$; and $N(t,1)=1$; and $N(t,k)=0$ if $k \ge t-1$; and therefore 
$$N(t) = \sum_{k=1}^{t-2} N(t;k)
$$
\item If $t \ge 4$ and $k \ge 2$ then 
$$N(t;k) = k \cdot N(t-1;k) \,\, + \,\, (t+k-3) \cdot N(t-1;k-1)
$$
\end{enumerate}
\end{exercise}

\subsection{Trees are contractible}
\label{SectionTreesAreContractible}
In this section we prove the implication \pref{TreeArcUniqueness}$\implies$\pref{TreeContractible} of Section~\ref{SectionGraphsAndPaths}. If you wish to take \pref{TreeArcUniqueness} as the definition for a graph to be a tree, then this implication simply says ``trees are contractible''.

Let $T$ be a graph such that for any two points $x \ne y$ there is a unique subset of $T$ homeomorphic to $[0,1]$ having endpoints $x,y$, which we shall denote $[x,y]$ and will call the \emph{arc} with endpoints $x,y$. We'll make use of the easy implication \pref{TreeArcUniqueness}$\implies$\pref{TreeNoCircleSubspace}, and so $T$ has no subspace homeomorphic to the circle. In particular no edge of $T$ is a loop edge. 

Pick a vertex $P \in T$ called the \emph{root}. To prove contractibility of $T$, we shall construct a deformation retraction from $T$ to $P$, namely a continuous function $h \from T \times [0,1] \to T$ such that $h(x,0)=P$ and $h(x,1)=x$ for all $x \in T$, and $h(P,t)=P$ for all $t \in [0,1]$. The idea is to mimic the standard proof that star convex subsets of $\mathbb{R}^n$ are contractible: each point moves $x \in T$ along the unique path connecting it to the base point. Most of the work is just setting up the notation to do this, and checking continuity of various functions using the CW complex topology on $T$ (as described in Section~\ref{SectionGraphsAndPaths}).

By induction define a sequence of subgraphs
$$T_0 \subset T_1 \subset \cdots \subset T
$$
where $T_0 = \{P\}$ and for each $i$ the subgraph $T_i$ is the union of $T_{i-1}$ with all edges $e$ such that $e \intersect T_{i-1} = \bdy e \intersect T_{i-1} \ne \emptyset$; this set of edges $e$ is denoted $\E_i$, and so we can formally write 
$$T_i = T_{i-1} \union \biggl( \,\bigcup_{e \in \E_i} e\biggr)
$$
One shows by induction that each $T_i$ satisfies \pref{TreeArcUniqueness}. From \pref{TreeArcUniqueness} it follows that each $e \in \E_i$ has distinct endpoints, exactly one of which is contained in $T_{i-1}$; we denote $\bdy e = \{\bdy_- e,\bdy_+ e\}$ where $\bdy e \intersect T_{i-1} = \{\bdy_- e\}$. Evidently we have $T = \union_{i=0}^\infty T_i$. For each $e$ we choose a parameterization $\gamma_e \from [0,1] \to \overline e$ such that $\gamma(0)=\bdy_- e$ and $\gamma(1)=\bdy_+ e$ (if the chosen parameterization of $e$ has that property then fine, otherwise reverse the order of the parameter).

For each $x \ne P$ in $T$, consider the arc $[P,x]$. By stitching together the parameterizations of the edges that occur along $[P,x]$ we obtain a parameterization of $[P,x]$, as follows. There exists a unique sequence of edges $e_1,\ldots,e_n$ and a unique nontrivial initial subsegment $\eta \subset e_n$ such that such that $e_i \in \E_i$ for $i=1,\ldots,n$, such that $\bdy_+e_{i-1} = \bdy_- e_i $ for $i=2,\ldots,n$, and such that $\bdy_- e_n = \bdy_- \eta$ and $\bdy_+ \eta = x$. Thus we can write $[P,x]$ as a concatenation
$$[P,x] = e_1 \cdots e_{n-1} \eta
$$
Define the \emph{radius} of $x$ to be the number 
$$\rho(x) = n-1+\gamma^\inv_{e_n}(x)
$$
and so $n-1 < \rho(x) \le n$, and $\gamma_{e_n}$ restricts to an orientation preserving homeomorphism from the interval $[0,\rho(x)-n+1]$ to $\eta$. We obtain a unique parameterization
$$\gamma_x \from [0,\rho(x)] \to [P,x]
$$
having the property that for each $1 \le i \le n$ and each $t \in [0,1]$ we have
$$\gamma_x(t) = \begin{cases}
\gamma_{e_i}(t-i-1) & \quad\text{if $1 \le i \le n-1$ and $i-1 \le t \le i$} \\
\gamma_{e_n}(t-n-1)   &\quad\text{if $i=n$ and $n-1 \le t \le n-1+\rho(x)$}
\end{cases}
$$

The radius function $\rho$ extends to the root $P$ by setting $\rho(P)=0$.

Define the subspace $V \subset T \times [0,\infty)$ by
$$V = \{(x,s) \suchthat x \in T, \, 0 \le s \le \rho(x)\}
$$
Define the function $H \from V \to T$ by the formula
$$H(x,s)=\gamma_x(s), \quad 
$$
and scale $H$ to define the function $h \from T \cross [0,1] \to T$ by the formula
$$h(x,t) =  \gamma_x(t \cdot \rho(x)) = H(x,t \cdot \rho(x))
$$
Clearly $h(x,0)=P$ and $h(x,1)=x$ for all $x$, and so once continuity of $h$ is established it follows that $h$ is a homotopy between the constant map $T \mapsto \{P\}$ and the identity map on $T$. To prove continuity of $h$ it suffices to prove continuity of $H(x,s)$ and of $\rho(x)$. 

Continuity of $\rho$ is proved locally at each $x \in T$ as follows. When $x$ is not a vertex, and so $x$ is an interior point of $e_n$, the formula $\rho(y) = n-1+\gamma_{e_n}^\inv(y)$ is valid for interior points of $e_n$; this formula is continuous on the interior of $e_n$ because $\gamma_{e_n} \from [0,1] \to e_n$ is a homeomorphism, and since the interior of $e_n$ is an open neighborhood of $x$ in $T$, continuity of $\rho$ at $x$ follows. Suppose now that $x$ is a vertex. Assuming for the moment that $x \ne P$, we have $x = \bdy_+ e_n$. Let $\{e_{x,j}\}_{j \in J} \subset \E_{n+1}$ be an indexing of those edges of $\E_{n+1}$ whose initial endpoint equals $x$. The union of $x$ with the interiors of $e_n$ and the interiors of the edges $e_{x,j}$ is an open neighborhood $U_x$ of $x$ in $T$. The half-open intervals obtained by taking the union of $x$ with each of those edge interiors are closed subsets $\hat e_n$, $\hat e_{x,j}$ of $U_x$ that cover $U_x$, and so by the pasting lemma it suffices to show that $\rho$ is continuous on each of $\hat e_n$, $\hat e_{x,k}$. Continuity of $\rho$ on $\hat e_n$ comes from the formula $\rho(y) = n-1+\gamma_{e_n}^\inv(y)$, and continuity on $e_{x,k}$ comes from the formula $\rho(y) = n + \gamma_{e_{x,k}}^\inv(y)$. 

Continuity of $h$ at $(x,s) \in V$ is proved as follows. In the case $0 \le s < \rho(x)$, if $y$ is sufficiently close to $x$ then the following are true: $s < \rho(y)$; $y$ and $x$ are contained in some common edge in which case $\gamma_x$ and $\gamma_y$ agree on their common domain $[0,\min\{\rho(x),\rho(y)\}]$; and $0 \le s < \min\{\rho(x),\rho(y)\}$. If in addition $(y,t) \in V$ and $t$ is sufficiently close to $s$ then $0 \le t < \min\{\rho(x),\rho(y)\}$ and so we have the continuous formula
$$h(y,t) = \gamma_y(t,\rho(y)) = \gamma_x(t \cdot \rho(y))
$$
In the case where $s=\rho(x)$ and $x$ is not a vertex, it follows that $x$ is an interior point of~$e_n$; letting $z$ be the terminal point of $e_n$, if $y$ is sufficiently close to $x$ then $y$ is also an interior point of~$e_n$ then again $\gamma_z,\gamma_y$ agree on their common domain $[0,\min\{\rho(z),\rho(y)\}] = [0,\rho(y)]$ and we obtain $h(y,t) = \gamma_z(t,\rho(y))$. The remaining case where $s=\rho(x)$ and $x$ is a vertex is a bit more complicated. In this case $h(x,s)=x$. Let $\{e_j\}_{j \in J}$ be an indexing of all the edges with initial vertex $x$; if $x = P$ then that's all of the edges incident to $x$; whereas if $x \ne P$ then there is one more edge incident to $x$ which $e_n$ having terminal vertex $x$. The point $x$ has a neighborhood basis consisting of sets of the form $U = (z_n,x] \union \bigl(\bigcup_{j \in J} [x,z_j) \bigr) $ where $z_j$ is an arbitrary point in the interior of $e_j$ and, if $x \ne P$, $z_n$ is an arbitrary point in the interior of $e_n$ (if $x=P$ then we ignore $z_n$ and $(z_n,x]$ in what follows). Consider the open subset of $T \cross [0,\infty)$ consisting of all $(y,s) \in V$ satisfying the following constraints: $y \in U$, $\rho(z_n) < s$, and if $y \in [x,z_j)$ then $s < (\rho(x) + \rho(z_j))/2$. Let $W$ be the intersection of this open subset with $V$, which is defined by imposing the additional constraint that $s \le \rho(y)$. It remains to note that if $(y,s) \in W$ then $h(y,s) \in \gamma_y(\rho(z_n),(\rho(x)+\rho(z_j))/2) \subset U$.

\subsection{Roses.} \label{SectionRoses}
A \emph{rose} is a graph with one vertex. Given a set $S$, the \emph{rose on $S$}, denoted $R\<S\>$, is the graph in which each edge is assigned an orientation and a label from the set $S$, such that we have a bijection $s \leftrightarrow e_s$ between $S$ and the oriented edges. A clean way to formalize this is to let $R\<S\>$ be the quotient of $S \times [0,1]$, using the discrete topology on $S$, and taking the quotient by identifyin the set of endpoints $S \times \{0,1\}$ to a single vertex $v$. 

As with any connected graph, the fundamental group $\pi_1(R\<S\>,v)$ is a free group. This is often proved by application of Van-Kampen's Theorem. This can also be proved using the results on the complete $S$-labelled tree $T\<S\>$ constructed in Section~\ref{SectionGraphsAndPaths}, as follows. There is a unique simplicial map $p \from T\<S\> \to R\<S\>$ that preserves the $S$-labelling: each vertex of $T\<S\>$ maps to the unique vertex $v \in R\<S\>$, and each oriented edge of $T\<S\>$ labelled by $s \in S$ maps to the unique oriented edge of $R\<S\>$ labelled by $s$. It follows immediately from the definition of a complete $S$-labelling that $p$ is a covering map. Also, $p$ is a universal covering map because $T\<S\>$ is simply connected. It follows from covering space theory that $\pi_1(R\<S\>)$ acts on $T\<S\>$ as the group of deck transformations of~$p$. Clearly the deck transformations of $p$ are precisely the automorphisms of $T\<S\>$ that preserve the $S$-labelling. As we showed in Section~\ref{SectionGraphsAndPaths}, this group is isomorphic to $F\<S\>$. 

The isomorphism $F\<S\> \leftrightarrow \pi_1(R\<S\>,v)$ described in the previous paragraph can be made explicit: a reduced word $w$ corresponds to the element of $\pi_1(R\<S\>,v)$ represented by the path $\gamma=E_1\ldots E_L$ if and only if \hbox{$W(\gamma)=w$,} where $W$ is the word labeling function defined earlier for $S$-labelled graphs. In particular, $\pi_1(R\<S\>,v)$ is a free group with free basis corresponding to the oriented edges of the $S$-labelling of $R\<S\>$. 

\begin{exercise} Verify the isomorphism of the previous paragraph in two ways: by examining the form of the isomorphism given by Van Kampen's Theorem; and by using the universal covering map $T\<S\> \mapsto R\<S\>$ defined on the complete $S$-labelled tree $T\<S\>$. 
\end{exercise}

\subsection{Notational conventions for the free group $F_n$.} 
\label{SectionFreeGroupNotation}
From now on we use $F_n = F\<S_n\>$ as a shorthand for a rank~$n$ free group equipped with a given free basis $S_n = \{s_1,\ldots,s_n\}$. As a base topological model for $F_n$ we use the \emph{base rose of rank~$n$} denoted $R_n = R\<S_n\>$, which has one vertex $v$ and $n$-oriented edges $e_1,\ldots,e_n$. We fix once and for all the isomorphism $F_n \approx \pi_1(R_n,v)$ given by $s_i \leftrightarrow [e_i]$, and we use this isomorphism to identify the groups $F_n = \pi_1(R_n,v)$.

\section{The Nielsen/Whitehead problems: free bases.}

\hfill{\emph{If you want to win her hand, let the maiden understand}}

\hfill{\emph{That she's not the only pebble on the beach.}}

\smallskip

\hfill  --- an inappropriate song by Harry Braisted and Stanley Carter

\bigskip

In their papers \cite{Nielsen:Noncommutative}, \cite{Whitehead:CertainSets}, and  \cite{Whitehead:EquivalentSets}, J.\ Nielsen and J.\ H.\ C.\ Whitehead considered several problems about the rank~$n$ free group $F_n$, regarding subgroups and free bases, giving solutions to these problems with an algorithmic flavor.

Our immediate goal, to be pursued over the remainder of Chapter~\ref{ChapterMarkedGraphs}, is to present specific problems of Nielsen and Whitehead, to translate these problems into the topological language of homotopy equivalences between graphs, and then to use these translations to motivate some of the important topological concepts in the modern-day study of $\Out(F_n)$: marked graphs and outer space, and applications thereof.  Having laid all of this groundwork, in Chapter~\ref{ChapterFoldPaths} we will then turn to a modern solution of these problems.

\subsection{Statements of the central problems.} 
\label{SectionCentralProblems}
Using the notation from Section~\ref{SectionFreeGroupNotation}, consider the free group $F_n = \<s_1,\ldots,s_n\>$. Our first batch of problems were considered and solved by Nielsen in \cite{Nielsen:Noncommutative} with refined solutions by Whitehead in \cite{Whitehead:CertainSets}.
\begin{description}
\item[Free basis problem:] Given a set of $n$ reduced words $\{w_1,\ldots,w_n\} \subset F_n$, how do you tell if it generates $F_n$? More specifically, how do you tell if it is a free basis for $F_n$?
\item[Automorphism version (reduced $n$-tuple version):] Given an ordered $n$-tuple of reduced words $W = (w_1,\ldots,w_n)$, how do you tell if there is an automorphism $\Phi \in \Aut(F_n)$ such that $\Phi(s_i)=w_i$ for $i=1,\ldots,n$?
\end{description}
As seen in Exercise~\ref{ExerciseFreeBasesAndAuts} the previous two problems may be regarded as equivalent. Furthermore, by Exercise~\ref{ExerciseUniversal} they are each equivalent to:
\begin{description}
\item[Automorphism problem (endomorphism version):] Given an endomorphism $F_n \mapsto F_n$ defined by a map $s_i \mapsto w_i$, how do you tell if it is an automorphism? 
\end{description}
This problem can be broken into two problems, each interesting on its own:
\begin{description}
\item[Injective/surjective problems (endomorphism versions):] Given an endomorphism $F_n \mapsto F_n$ defined by a map $s_i \mapsto w_i$,
\begin{description}
\item[Injectivity:] How do you tell if the map is a monomorphism (injective)?
\item[Surjectivity:] How do you tell if the map is an epimorphism (surjective)?
\end{description}
\end{description}

\bigskip

In their respective papers \cite{Nielsen:Noncommutative} and \cite{Whitehead:EquivalentSets}, Nielsen and Whitehead had already considered the ``free basis problem'' in a more general context: given two lists of elements $\{v_1,\ldots,v_I\}$ and $\{w_1,\ldots,w_I\}$, how do you tell if there exists $\Phi \in \Aut(F_n)$ such that $\Phi(v_i)=w_i)$ for all $i=1,\ldots,I$? The ``free basis problem'' is the specialization of this general problem to the case $w_i=s_i$. Although we shall not consider this general problem in this work, we shall consider a somewhat broader specialization which was already proposed by Whitehead in his paper \cite{Whitehead:CertainSets}:
\begin{description}
\item[Partial free basis problem:] Given a set of pairwise distinct reduced words $V = \{v_1,\ldots,v_k\} \in F_n$, how do you tell if $V$ forms a partial free basis? 
\end{description}
To say that $V$ is a \emph{partial free basis} means simply that $V$ is a subset of a free basis.  Whenever we list the elements the elements of a partial free basis in the form $V = \{v_1,\ldots,v_k\}$ we will assume that there are no repetitions in this list, that is, $v_i \ne v_j$ for $i \ne j$; with that assumption, we will abuse notation by just listing the elements in order, something like ``$v_1,\ldots,v_k$ is a partial free basis''. As a special case of the partial free basis problem, defining a \emph{free basis element} to be a single reduced word which is an element of some free basis, we can ask
\begin{description}
\item[Free basis element problem:] How do you tell if a given reduced word $w$ is a free basis element?
\end{description}

Nielsen and Whitehead, in their various papers cited above, gave complete algorithmic solution to the above questions. The most well known of these is ``Whitehead's Algorithm'' which solves the ``Partial free basis problem''. 

In his papers \cite{Whitehead:CertainSets,Whitehead:EquivalentSets}, Whitehead had already considered conjugacy class versions of the Nielsen/Whitehead problems. We shall formulate those conjugacy versions in Section~\ref{SectionWhiteheadConjugacy}, and they will play a prominent role for us in Chapter 2. But for now, the versions already stated give us enough grist to work out important concepts of $\Aut(F_n)$ and $\Out(F_n)$.

Eventually, in Chapter~\ref{ChapterFoldPaths}, we will give algorithmic answers to all of the above problems and their conjugacy versions stated in Section~\ref{SectionWhiteheadConjugacy}. While our answers will follow in the steps of Whitehead, they are designed to illuminate the modern viewpoint of the topology and geometry of free groups and their automorphism and outer automorphism groups. In particular, our answers are are couched in the language of marked graphs and Stallings fold sequences. Another, briefer account of this modern viewpoint can be found in \cite{Stallings:Whitehead}.

\subsection{Negative tests, using abelianization.} 

Exercise~\ref{ExerciseFreeToAbelianBasis}, regarding the abelianization of $F_n$, can be used as the basis of some simple negative tests for the Nielsen/Whitehead problems.

In the free group $F_n = F\<s_1,\ldots,s_n\>$, consider the abelianization map 
$$F_n \mapsto \ab(F_n) = F_n / [F_n,F_n]
$$
We saw in Exercise~\ref{ExerciseFreeToAbelianBasis}, that the abelianization $\ab(F_n)$ is isomorphic to $\Z^n$, and we saw in the outline that followed how to compute the image of a word as a vector in $\Z^n$. This abelianization computation is the basis of a negative test for Whitehead's problems, using that one can decide whether a subset of $\Z^n$ is a basis (or a partial basis, or a basis element) using elementary linear algebra and number theory. 

For example, a nonzero vector $(m_1,\ldots,m_n) \in \Z^n$ is a basis element if and only if $gcf(m_1,\ldots,m_n)=1$. It follows that neither $abab$ nor $a^2b^2$ is a free basis element of $F\<a,b\>$, because each has abelianized image $(2,2) \in \Z^2$ which is not a basis element of $\Z^2$. 

However, $a^2 b a b$ is not ruled out as a free basis element of $F\<a,b\>$, because its image $(3,2)$ \emph{is} a basis element of $\Z^2$ (see Exercise~\ref{ExerciseATwoBAB}). Since abelianization gives only a negative test, we cannot yet determine whether $a^2 b a b$ is a free basis element.

\subsection{Positive tests, using Nielsen transformations.} 
\label{SectionPositiveTests}
Exercise~\ref{ExerciseFreeBasesAndAuts} can be used as a positive test, as follows. Given a set of reduced words $\{w_1,\ldots,w_n\}$ in $F_n = F\<s_1,\ldots,s_n\>$, if one can construct an automorphism $\Phi \in \Aut(F_n)$ taking $s_i$ to $w_i$ for each $i$, then the set $\{w_1,\ldots,w_n\}$ is a free basis of $F_n$. For this purpose it is useful to have a few automorphisms to start with, the \emph{elementary automorphisms} of $F_n$ that were described by Nielsen in \cite{Nielsen:IsomorphismGroup} and were proved by him to generate the group $\Aut(F_n)$ (see Section~\ref{SectionWhiteheadNielsenGenerators}):
\begin{description}
\item[Transvections:] There are four transvections for every ordered pair $i \ne j \in \{1,\ldots,n\}$, namely
$$s_i \mapsto s_i s_j, \quad s_i \mapsto s_i \bar s_j, \quad s_i \mapsto s_j s_i, \quad s_i \mapsto \bar s_j s_i
$$
Note that the first pair of these are inverses to each other, as is the last pair (that's how we know they are all automorphisms).
\end{description}
Each of the next two kinds of elementary automorphism are clearly self-inverse:
\begin{description}
\item[Transpositions:] There is one transposition for every unordered pair $i \ne j \in \{1,\ldots,n\}$, namely 
$$s_i \leftrightarrow s_j
$$ 
\item[Letter inversion:] There is one letter inversion for every $i \in \{1,\ldots,n\}$, namely 
$$s_i \mapsto \bar s_i
$$ 
\end{description}
In these formulas we specify only the image of one free basis element $s_i$; by implicit assumption, every other free basis element is fixed. For example, in the free group $F\<a,b,c\>$ the transvection $a \mapsto ab$ is defined more fully as
$$\begin{matrix} a \\ b \\ c \end{matrix} \, \mapsto \, \begin{matrix} ab \\ b \\ c \end{matrix}
$$
It follows that $\{ab,b,c\}$ is a free basis of $F\<a,b,c\>$, being the image of $\{a,b,c\}$ under that transvection. Once we have delved into the topology and geometry of $F_n$, in Section~\ref{SectionWhiteheadNielsenGenerators} we will prove Nielsen's Theorem that $\Aut(F_n)$ is generated by the elementary automorphisms: the tranvections, transpositions, and letter inversions. For now we will not need this fact, all we do is to put the elementary automorphisms to work to construct interesting free bases, and to set up examples for testing the Nielsen/Whitehead problems.

Since $\Aut(F_n)$ is a group, we may successively apply any sequence of elementary automorphisms to a free basis element to get another free basis element, and we can get some quite complicated free basis elements by this manner. For example,
$$a \xrightarrow{a \mapsto ab} ab \xrightarrow{b \mapsto ba} aba \xrightarrow{a \mapsto ab} abbab=ab^2ab \xrightarrow{b \mapsto b^\inv} a b^{-2}ab^\inv
$$
and so $ab^{-2}ab^\inv$ is a free basis element.

\paragraph{Exercises for Section \ref{SectionPositiveTests}}

\begin{exercise} \cite[Section 1]{Nielsen:IsomorphismGroup}
\label{ExerciseSignedPermSubgroup}
Prove that the transpositions and letter inversions together generate a finite subgroup of $\Aut(F_n)$ having cardinality $n! \, 2^n$. Prove that this group is isomorphic to the group of $n \times n$ invertible matrices whose entries are all from the set $\{-1,0,+1\}$ and such that there is one nonzero entry in each row and in each column. This group is known as the ``signed permutation group on $n$ symbols''.
\end{exercise}

\begin{exercise}\label{ExerciseATwoBAB}
Prove that $a^2 b a b$ is a free basis element of $F\<a,b\>$. 
\end{exercise}

\begin{exercise} \label{ExerciseNonFreeBasis}
The pair $a^2bab$, $b^3a^3b^{-2} a^\inv$ in $F\<a,b\>$ maps to the basis $(3,2)$, $(2,1)$ in~$\Z^2$. Is that pair a free basis of $F\<a,b\>$?
\end{exercise}

\subsection{Topological versions of the central problems, using roses.}
\label{SectionBasedHMCG}
We have stated the Nielsen/Whitehead problems in the language of ``free bases'', we have translated them into the language of automorphisms, and we have made some progress towards solving those problems, presenting some positive and negative tests. However, it's pretty evident that there is a large gap between these tests, leaving us very far from a complete solution to those problems.

To explore avenues for making further progress, we shall translate those problems into a topological language involving graphs, with an emphasis on roses. As this section progresses, we will alternate between definition and discussion of new concepts, and exercises on those concepts.

\medskip

For any topological space $X$ and any base point $p \in X$, define the \emph{pointed homotopy endomorphism semigroup} $\HEnd(X,p)$ as follows. As a set, $\HEnd(X,p)$ consists of continuous self-maps $(X,p) \mapsto (X,p)$ modulo the equivalence relation of homotopy rel~$p$ (meaning homotopy that keeps $p$ stationary). The operation of composition descends to a well-defined associative binary operation that makes $\HEnd(X,p)$ into a semigroup with identity element represented by the identity map of $X$. 

Now define the \emph{pointed homotopy mapping class group} $\HMCG(X,p)$ to be the subgroup of $\HEnd(X,p)$ represented by all \emph{pointed homotopy equivalences} $f \from (X,p) \to (X,p)$, meaning that $f$ has a \emph{pointed homotopy inverse} $\bar f \from (X,p) \to (X,p)$ satisfying the property that $f \composed \bar f$ and $\bar f \composed f$ are both homotopic to the identity rel~$p$. 

\begin{exercise} 
Check that the binary operation on $\HMCG(X,p)$ that is induced by composition satisfies the group axioms.
\end{exercise}

\begin{exercise}
\label{ExerciseHEAutHomomorphism}
Prove that there is a well-defined homomorphism 
$$\HMCG(X,p) \mapsto \Aut(\pi_1(X,p))
$$
defined by the formula $[f] \mapsto f_*$ where $[f]$ is the homotopy class rel~$p$ of a pointed homotopy equivalence $f \from (X,p) \to (X,p)$, and $f_* \from \pi_1(X,p) \to \pi_1(X,p)$ is the induced homomorphism of $f$.
\end{exercise}

In Exercise~\ref{ExerciseHEAutHomomorphism}, in general one cannot say that the homomorphism $\HMCG(X,p) \mapsto \Aut(\pi_1(X,p))$ is an isomorphism. Here is a simple example.
\begin{exercise}\label{ExerciseHEAutCounterexample} For the 2-sphere $X=S^2$, 
\begin{enumerate}
\item Show that $\HMCG(S^2,p)$ is a nontrivial group and hence is not isomorphic to $\Aut(\pi_1(S^2,p))$.
\item Compute $\HMCG(S^2,p)$.
\end{enumerate}
\end{exercise}

The reason for the existence of counterexamples as in Exercise~\ref{ExerciseHEAutCounterexample} is that the fundamental group is not the only homotopy invariant on the beach, however inappropriate that may seem. However, for an Eilenberg-Maclane space of type $K(G,1)$ --- meaning a connected CW complex with fundamental group $G$ and with contractible universal cover --- the fundamental group $G$ \emph{is} the only homotopy invariant. This gives a hint to the proof of the following exercises:

\begin{exercise}
\label{ExerciseHEAutIsomorphismF_n}
Prove that if $G$ is a connected graph and $p \in G$ then the homomorphism $\HMCG(G,p) \mapsto \Aut(\pi_1(G,p))$ is an isomorphism. As a special case  obtain an isomorphism
$$(*) \qquad \HMCG(R_n,v) \mapsto \Aut(\pi_1(R_n,v)) = \Aut(F_n)
$$
in which the equation was specified by the notational conventions of Section~\ref{SectionFreeGroupNotation}.
\end{exercise}

\smallskip

You can either do Exercise~\ref{ExerciseHEAutIsomorphismF_n} using what you know about graphs, or by proving a more general version: 

\begin{exercise}
\label{ExerciseEilenbergMaclane}
Prove that if $X$ is a connected CW complex and an Eilenberg--Maclane space of type $K(G,1)$, and if $p \in X$ is a $0$-cell, then the homomorphism $\HMCG(X,p) \mapsto \Aut(\pi_1(X,p)) \approx \Aut(G)$ is an isomorphism.
\end{exercise}

Exercise~\ref{ExerciseEilenbergMaclane} lets us translate the study of $\Aut(G)$ for \emph{any} group $G$ into a study of homotopy mapping class groups of pointed $K(G,1)$ spaces, which is particularly useful for groups that have simple and well understood $K(G,1)$ spaces, in particular for free groups.

\begin{exercise}
\label{ExerciseAutToHMCG}
Prove that the inverse $\Aut(F_n) \mapsto \HMCG(R_n,v)$ of the isomorphism $(*)$ given in Exercise~\ref{ExerciseHEAutIsomorphismF_n} has the following effect: it takes each $\Phi \in \Aut(F_n)$ to the pointed homotopy class of the map $f_\Phi \from (R_n,v) \to (R_n,v)$ defined by $f_\Phi(e_i)=\gamma_i$ $(i=1,\ldots,n)$, where $\gamma_i$ is the tight edge path obtained from the reduced word $w_i = \Phi(s_i)$ by replacing each occurrence of the generator $s_j \in S_n$ by the corresponding edge $e_j$ and each occurence of the inverse generator $s_j^\inv \in \overline S_n$ by the reversed edge $\bar e_j$. 
\end{exercise}

\bigskip

By applying Exercise~\ref{ExerciseHEAutIsomorphismF_n}, we can translate Nielsen's and Whitehead's problems from Section~\ref{SectionCentralProblems} into topological language. For example, the \emph{Free Basis Problem}, and the two versions of the \emph{Automorphism Problem}, are translated as follows:
\begin{description}
\item[Automorphism problem (rose version):] \quad \\ Given a self-map $f \from (R_n,v) \mapsto (R_n,v)$, taking each $e_i$ to some tight edge path $w_i$, how do you tell if $f$ is a homotopy equivalence? 
\end{description}

\begin{description}
\item[Injective/surjective problem (rose version):] Given $f$ as above, how do you tell if $f$ is a $\pi_1$-injection? or a $\pi_1$-surjection?
\end{description}
Our eventual approach to the solutions of Whitehead's problems will be to solve these topological versions.

\begin{exercise} Consider an endomorphism $\Phi$ of a free group $F\<s_1,\ldots,s_n\>$ defined by $\Phi(s_i) = w_i$ for each $i=1,\ldots,n$. Describe an algorithm which takes each $\Phi$ as input and decides whether $\Phi$ is an inner automorphism (recall that an inner automorphism of a group $\Gamma$ is an automorphism of the form $i_h(g)=hgh^\inv$, defined for each $h \in \Gamma$).
\end{exercise}

\begin{exercise} Consider a finite connected graph $\Gamma$ and a self-map $h \from \Gamma \mapsto \Gamma$, defined to take vertices to vertices and to take each edge $e$ to either a vertex or an edge path. Describe an algorithm which takes $\Gamma$ and $h$ as input and decides whether $h$ is homotopic to the identity map.
\end{exercise}

\section{Marked graphs}

\centerline{\emph{A \emph{four}-footed lion's not much of a beast.}}

\centerline{\emph{The one in my zoo will have \emph{ten} feet at least.}}

\smallskip

\hfill  --- Dr. Seuss, \emph{If I Ran the Zoo}

\subsection{Core graphs and their ranks.} 
\label{SectionCoreGraphs}
We have introduced the base rose $R_n$ as a topological model for the rank~$n$ free group~$F_n$. But we shall need other finite graphs as topological models for $F_n$ as well, namely the collection of ``rank-$n$ core graphs''.

Given a finite connected graph $G$, its \emph{rank} is the non-negative integer given by the following equations:
\begin{align*}
\rank(G) &= \text{the rank of the free group $\pi_1(G,v)$ for any vertex $v \in G$} \\
              &= \text{the rank of the free abelian group $H_1(G;\Z)$} \\
              &= 1 - \chi(G) \\ &= 1 - \abs{\Vertices(G)} + \abs{\Edges(G)}\\
              &= \abs{\Edges(G \setminus T)} \,\, \text{where $T \subset G$ is any maximal tree}
\end{align*}
In these formulas, $\abs{A}$ denotes the cardinality of a set~$A$. Given $n \ge 0$, a \emph{rank~$n$ graph~$G$} is a finite connected graph satisfying $\rank(G)=n$. When $n$ is fixed, all rank~$n$ graphs are homotopy equivalent to each other: every rank~$0$ graph is a tree, hence contractible; and every rank~$n$ graph with $n \ge 1$ is homotopy equivalent to a rank~$n$ rose by collapsing a maximal tree to a point.

For any $n$, one can construct rank~$n$ graphs with infinitely many homeomorphism types. Start with one rank~$n$ graph $G$. Then choose $T$ to be one of the infinitely many homeomorphism types of finite trees --- it could have four valence 1 vertices, or 10, or any number. Finally, identify a vertex in $G$ with a vertex in $T$. 

To somewhat tame the zoo of rank~$n$ graphs, we define a \emph{core graph}\index{core graph} to be a finite graph having no vertex of valence~$1$. The rank~$n$ rose is a core graph, and there are only finitely many rank~$n$ core graphs up to homeomorphism; a proof is outlined in the exercises. The circle is the only rank~$1$ core graph up to homeomorphism. There are three rank~$2$ core graphs up to homeomorphism, depicted in Figure~\ref{FigureRankTwoCoreGraphs}. 
\begin{figure}
\centerline{
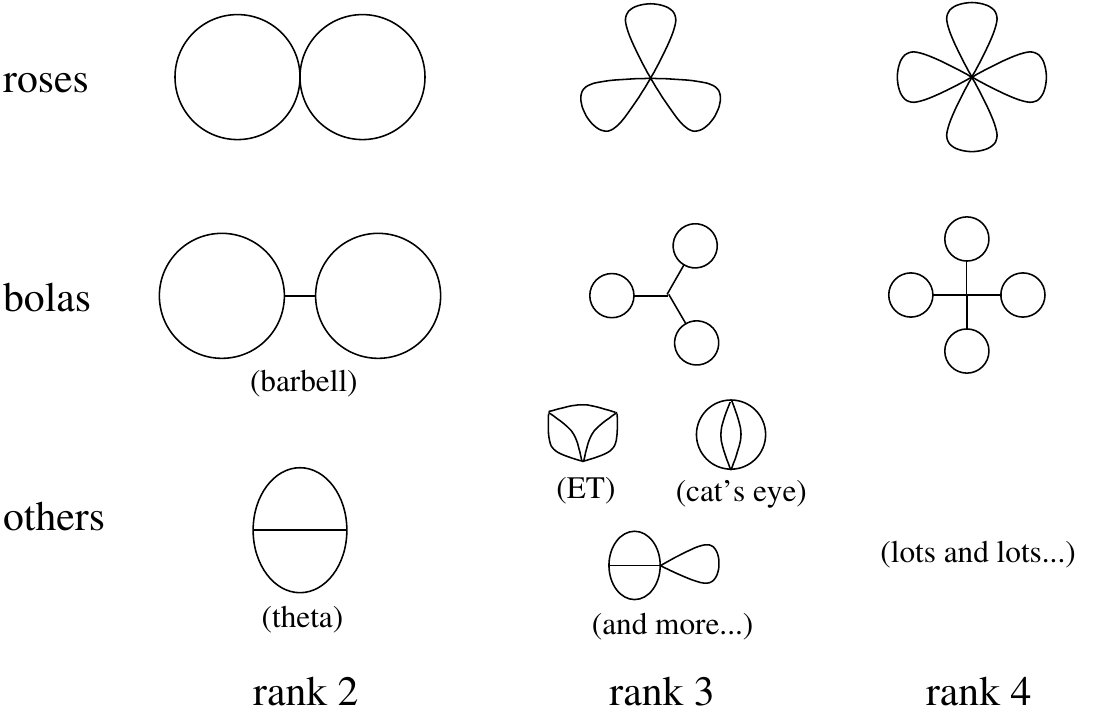 
}
\caption{A zoo of core graphs.}
\label{FigureRankTwoCoreGraphs}
\end{figure}

\paragraph{Exercises for Section~\ref{SectionCoreGraphs}} 

\begin{exercise}\label{ExerciseCoreCharacter}
Prove that if $G$ is a connected graph of finite rank~$n \ge 1$ then $G$ has a unique rank~$n$ core subgraph, which is denoted $\Core(G)$. Prove furthermore that $\Core(G)$ is the unique core subgraph which is a deformation retract of~$G$.
\end{exercise}

\begin{exercise}
Prove that a finite connected graph is a core graph if and only if it can be written as the union of subgraphs each of which is homeomorphic to a circle or a barbell (see Figure~\ref{FigureRankTwoCoreGraphs} for the barbell).
\end{exercise}

In any graph define a \indexemph{natural vertex} to be a vertex of valence~$\ne 2$. Note that a graph has no natural vertices if and only if it is a 1-manifold; of course, all connected graphs with no natural vertices are homeomorphic to the circle or the line. 

\begin{exercise}
Prove that if $G$ is a core graph of rank~$\ge 2$ then the natural vertices are the $0$-cells of a unique CW structure on $G$. We call this the \indexemph{natural graph structure} on~$G$, and its edges are called the \indexemph{natural edges} of $G$. 
\end{exercise}

For the next exercise, in a core graph $G$ a natural edge $e \subset G$ is said to be \indexemph{separating} if the graph complement $G \setminus e$ is disconnected, otherwise $e$ is \indexemph{nonseparating}. A separating edge of $G$ is also called a \emph{bridge} of $G$. A core graph is \emph{bridgeless} if it has no bridges. A maximal bridgeless core subgraph of $G$ is called an \emph{island} of $G$.

\begin{exercise}\label{ExerciseTreeOfCoreGraphs}
Prove that every finite core graph $G$ has a unique \emph{bridge decomposition} of the form
$$G = (e_1 \union\cdots\union e_K)  \union (H_1 \union\cdots\union H_L)
$$
where $\{e_1,\ldots,e_K\}$ is a set of all bridges of $G$, $\{H_1,\ldots,H_L\}$ is the set of all islands (uniqueness, of course, means unique up to reindexing those two sets), and each island is a core graph. Prove also that the quotient graph of $G$, obtained by collapsing the islands $H_1,\ldots,H_L$ to vertices $v_1,\ldots,v_L$, is a tree.
\end{exercise}

\begin{exercise}
Prove that for each $n \ge 2$, the numbers of natural vertices and natural edges of a rank~$n$ core graph are bounded above by constants depending only on~$n$. Find the optimal values for these constants.
\end{exercise}

\begin{exercise}\label{ExerciseFinManyCoreGraphs}
Prove that for each integer $n \ge 1$ there are only finitely many rank~$n$ core graphs up to homeomorphism. Find an explicit upper bound to this number (the exact value of this number is difficult to compute).
\end{exercise}

Next we turn to two exercises concerned with the concept of a ``relative core graph'' (one could formulate versions of each of Exercises~\ref{ExerciseCoreCharacter}--\ref{ExerciseFinManyCoreGraphs} for this concept). Given a finite graph $G$ and a finite subset of vertices $Q \subset G$, we say that $G$ is a \emph{core graph relative to $Q$} if every vertex of valence $\le 1$ in $G$ is an element of $Q$.

\begin{exercise}\label{ExerciseRelativeCoreGraph}
Let $G$ be a connected graph and $Q \subset G$ a finite set of vertices; if $\rank(G)=0$ assume that $Q \ne \emptyset$. Prove that $G$ has a unique finite subgraph of rank $n$ which is a core graph relative to $Q$, denoted $\Core(G;Q)$. Prove furthermore that $\Core(G;Q)$ is the unique deformation retract of $G$ which is a core graph relative to $Q$.
\end{exercise}

\begin{exercise} Given integers $n \ge 0$ and $k \ge 0$, not both equal to zero, prove that there exist up to homeomorphism only finitely many pairs $(G,Q)$ such that $G$ is a rank $n$ core graph relative to $Q$ and $\abs{Q} = k$.
\end{exercise}

\subsection{Based marked graphs.} 
\label{SectionBasedMarkedGraphs}
Given a connected graph of finite rank $n \ge 2$, its fundamental group with respect to any base point is isomorphic to the rank~$n$ free group $F_n = \<s_1,\ldots,s_n\>$. But there are many choices for the base point of the graph, and even once that point is fixed there are many possible choices of the isomorphism. Fixing those choices gives that graph the structure of a ``based marked graph'', although the precise definition is couched in topological terms.

Consider a rank~$n$ core graph $G$ and a base vertex $p$; we refer to the pair $(G,p)$ as a \emph{based core graph} of rank~$n$.
Define a \emph{based marking of $(G,p)$} to be
\begin{enumerate}
\item\label{ItemBasedMarkingHE}
a homotopy equivalence of pairs $\rho \from (R_n,v) \to (G,p)$.
\end{enumerate}
From an algebraic point of view we define a \emph{based algebraic marking} of $(G,p)$ to be
\begin{enumeratecontinue}
\item\label{ItemBasedMarkingIso}
an isomorphism $h \from F_n \mapsto \pi_1(G,p)$. 
\end{enumeratecontinue}
When $(G,p)$ is equipped with a based marking then we say that it is a \emph{based marked graph} of rank~$n$. Thus, a based marked graph can formally be regarded as a triple $(G,p,\rho)$; when $\rho$ is implicit we may abuse terminology by referring to $(G,p)$ alone as a based marked graph. We could also say that $(G,p)$ is an ``based marked algebraic graph'' when it is equipped with an algebraic marking; however, in light of the next exercise, such additional terminology is superfluous.

The proof that the fundamental group $\pi_1(G,p)$ is free, which was reviewed earlier in Section~\ref{SectionGraphs}, gives a method for constructing a particular class of based markings of $(G,p)$. First choose a maximal subtree $T \subset G$, and recall that every vertex, including $p$, is in $T$. For the moment we do \emph{not} assume that $G$ is a core graph, \emph{nor} that the cell structure on $G$ is the natural one, but we \emph{do} require that $T$ is a subtree with respect to whatever graph structure is given on $G$. Next, choose an orientation of each of the $n$ edges of the difference graph $G \setminus T$, and choose an ordering of each of those $n$ edges, hence we may write $G \setminus T =  \eta_1 \union\cdots\union \eta_n$. Let $\delta^-_i$, $\delta^+_i$ be the unique paths in $T$ from $p$ to the initial and terminal vertices of $\eta_i$, respectively. The based marking $(R_n,v) \mapsto (G,p)$ that corresponds to these choices is defined by the formula 
$$e_i \mapsto \delta^-_i \, \eta_i \bar\delta^+_i, \quad 1 \le i \le n
$$
The corresponding algebraic based marking is given by the formula $$s_i \mapsto [\delta^-_{i} \, \eta_i \, \bar \delta^+_{i}], \quad 1 \le i \le n
$$
The based markings of $(G,p)$ which arise from this construction will be called the \emph{visible based markings}, and we will similarly refer to the \emph{visible algebraic based markings} of $(G,p)$.

\paragraph{Exercises for Section \ref{SectionBasedMarkedGraphs}}

\begin{exercise}\label{ExerciseAlgMarking}
Given a based core graph $(G,p)$ of rank $n$, prove that the fundamental group functor induces a bijection $[\rho] \mapsto h_\rho$ between the set of based homotopy classes of based markings $\rho \from (R_n,v) \to (G,p)$ and the set of based algebraic markings $h_\rho \from F_n = \pi_1(R_n,v) \to \pi_1(G,p)$.  
%
\end{exercise}

\begin{exercise} 
\label{ExerciseRankTwoMixerBasicMarkings}
In a rank~$2$ theta graph $G$ with its natural cell structure, and fixing a natural vertex $p \in G$, what is the total number of visible based algebraic markings, relative to all choices of maximal tree?
\end{exercise}

\begin{exercise}
\label{ExerciseBasicMarkingCount} Let $G$ be a finite connected graph of rank~$n$ with base point~$p$. Fixing the choice of maximal tree $T$, how many different visible algebraic based markings does $(G,p)$ have relative to $T$? 
\end{exercise}

\begin{exercise} 
\label{ExerciseDiffTreesSameBasicMarking}
Does there exist a finite connected graph $G$ of rank~$n \ge 2$ with base vertex $p \in G$, an algebraic based marking $h \from F_n \to \pi_1(G,p)$, and two distinct maximal trees $T_1 \ne T_2 \subset G$, such that $h$ is induced by some visible based marking relative to $T_1$ \emph{and} another relative to $T_2$? 
\end{exercise}

\begin{exercise}
Does Exercise~\ref{ExerciseDiffTreesSameBasicMarking} have a different outcome if $G$ is equipped with its natural cell structure?
\end{exercise}

\subsection{Based marked graphs and $\Aut(F_n)$.}
\label{SectionTopInterpAut}
By using based marked graphs we obtain some further topological interpretations of the automorphism group $\Aut(F_n)$. 

Given a based algebraic marking $h_0 \from F_n \mapsto \pi_1(G,p)$, by precomposing $h_0$ with arbitrary automorphism of $F_n$ we get a formula for a bijection which associates to the automorphism $\Phi \in \Aut(F_n)$ the based algebraic marking $h_0 \composed \Phi \from F_n \to \pi_1(G,p)$ (see Exercise~\ref{ExerciseAutBijectBAM}).  Putting this together with Exercises~\ref{ExerciseFreeBasesAndAuts} and~\ref{ExerciseAlgMarking} we now have a chain of bijections
\begin{align*}
\{\text{Ordered free bases of $F_n$}\} \leftrightarrow \Aut(F_n) & \leftrightarrow \{\text{based algebraic markings of $(G,p)$\}}  \\
& \leftrightarrow \{\text{based markings of $(G,p)$\}}
\end{align*}

For each rank $n$ based marked graph $\rho \from (R_n,v) \to (G,p)$ there is a chain of isomorphisms
$$\Aut(F_n) \approx \HMCG(R_n,v) \xrightarrow{\Ad_\rho} \HMCG(G,p) \approx \Aut(\pi_1(G,p))
$$
which we may use to \emph{canonically identify} $\Aut(F_n)$ with $\HMCG(G,p)$ and with $\Aut(\pi_1(G,p))$ as long as the based marking $\rho$ is specified. The first and third of this string of automorphisms come from exercises in Section~\ref{SectionBasedHMCG}, and the second is verified in Exercise~\ref{ExerciseTopAutFn} below.

\paragraph{Exercises for Section \ref{SectionTopInterpAut}}

\begin{exercise}\label{ExerciseAutBijectBAM}
Prove that for a based algebraic marking $h_0 \from F_n \mapsto \pi_1(G,p)$ does indeed induce a bijection $\Phi \mapsto h_0 \composed \Phi$ betwen $\Aut(F_n)$ and based algebraic markings of $(G,p)$.
\end{exercise}

\begin{exercise}
\label{ExerciseTopAutFn}
For any pair of rank $n$ core graphs $G,H$ with base points $p \in G$, $q \in H$, and for any homotopy equivalence of pairs $h \from (H,q) \to (G,p)$, the \emph{adjoint map}
$$\Ad_h \from \HMCG(H,q) \to \HMCG(G,p)
$$
is defined as follows: choosing a homotopy inverse $\bar h \from (G,p) \to (H,q)$ relative to the base points, each $[f] \in \HMCG(H,q)$ is taken to $\Ad_h(f) = [h f \bar h] \in \HMCG(G,p)$. Prove that $\Ad_h$ is a group isomorphism, well-defined independent of the choice of $\bar h$. Prove that $\Ad$ defines a functor from the category of based core graphs and homotopy equivalences rel base point to the category of groups.
\end{exercise}

\subsection{The Nielsen-Whitehead problems: Based marked graphs.} Consider a rank~$n$ graph $G$ with base point~$q$, and let us assume that each vertex not equal to~$q$ has valence~$\ge 2$. Whitehead's problem on free bases can be restated in the language of based marked graphs as follows:
\begin{itemize}
\item Given a $k$-tuple of tight edge paths $w_1,\ldots,w_k$ in $G$ based at~$q$, how do you tell if $w_1,\ldots,w_k$ is a partial basis for $\pi_1(G,p)$? Equivalently, how do you tell if there is a based marking $\rho \from (R_n,v) \to (G,p)$ such that for each $i=1,\ldots,k$ we have $\rho(e_i)=w_i$?
\end{itemize}

The Nielsen-Whitehead problems on $\pi_1$-automorphisms, $\pi_1$-surjections, and $\pi_1$-injections can also be given based marked graph versions. Consider any graph $(G,q)$ with the restriction that each vertex $\ne q$ has valence~$\ge 2$. Consider also a map $f \from (R_n,v) \mapsto (G,q)$; we may assume that $f$ is a tight map, meaning that each of the restrictions of $f$ to an edge of $R_n$ is either constant or a tight edge path in $G$. Listing those restrictions gives a complete description of $f$. Given that description:
\begin{enumerate}
\item How do you tell if $f$ induces a $\pi_1$-injection?
\item How do you tell if $f$ induces a $\pi_1$-surjection?
\item How do you tell if $f$ induces a $\pi_1$-isomorphism, i.e.\ whether $f$ is a homotopy equivalence, i.e.\ whether $f$ is a based marking?
\end{enumerate}

\subsection{Marked graphs.}
\label{SectionMarkedGraphDef} 
We have seen that a good topological setting for the Nielsen-Whitehead problems on free bases are based marked graphs and their fundamental groups. 

Our next step regarding these problems, carried out in Section~\ref{SectionWhiteheadConjugacy}, is to formulate conjugacy class versions of those problems and to give a topological interpretation thereof. It turns out that the topological translations of the conjugacy class versions are easier to solve that the topological translations of the original free basis versions, and so we shall solve the conjugacy versions first. The correct topological setting for these new conjugacy class versions will be marked graphs, which are similar to based marked graphs except that we discard the base point. Doing this has the effect of somehow simplifying the topological problems. 

\medskip

Dropping the base point of a based marked graph causes certain problems, for instance you can't define the fundamental group without a base point. To deal with these problems and to formulate a good definition of marked graphs, we choose the base point only once, in only one graph, namely in the base rose $R_n$, whose valence~$2n$ vertex $v$ is the chosen base point. We get a standard isomorphism $F_n \leftrightarrow \pi_1(R_n,v)$ identifying each free basis element $s_i \in F_n$ with the corresponding path homotopy class $[e_i] \in \pi_1(R_n,v)$. 

\subparagraph{Definition.} Given a rank~$n$ core graph $G$, a \indexemph{marking} of $G$ is a homotopy equivalence $\rho \from R_n \to G$. The pair $(G,\rho)$ is called a \indexemph{marked graph}, or an \emph{$F_n$-marked graph} for emphasis. 

\medskip

Note that we do not pick a base point of $G$, we do not refer to the fundamental group of~$G$, and we do not ``mark''~$G$ by choosing an isomorphism to its fundamental group. We instead mark $G$ topologically.

\subparagraph{Caution:} You will often see abuses of notation for the concept of a marked graph. In fact, such abuses abound in this very document. The most formal expression ``$(G,\rho)$ is a marked graph'' will sometimes be written as ``$\rho \from R_n \to G$ is a marked graph''. We~very often simply write ``$G$ is a marked graph'', in which case one must assume that the marking map $\rho \from R_n \to G$ is given implicitly. For an example of this, see Exercise~\ref{ExerciseInducedMarking}. However, also implicit in the terminology of a marked graph are all the notational conventions listed in Section~\ref{SectionFreeGroupNotation} regarding the free group $F_n = F\<s_1,\ldots,s_n\>$ and its base rose $R_n$. 

\paragraph{Exercises for Section \ref{SectionMarkedGraphDef}}

The question arises: how are ``based'' markings related to the unbased ones? Of course, every based marking $f \from (R_n,v) \to (G,q)$ determines a marking $f \from R_n \to G$ by forgetting the base point. But what information is preserved and what information is lost when the base point is forgotten? The next two exercises explore this question; the first is a special case of the second.

\begin{exercise} Consider a rank~$n$ core graph $G$ with base vertex $q$ and two based markings $f,f' \from (R_n,v) \to (G,q)$. Prove that the (unbased) markings $f,f' \from R_n \to G$ are homotopic if and only if the isomorphism 
$$(f')^\inv_* \composed f_* \from F_n = \pi_1(R_n,v) \to \pi_1(R_n,v) = F_n
$$
is an inner automorphism.
\end{exercise}

\begin{exercise} Consider a rank~$n$ core graph $G$, two vertices $q,q' \in G$, and two based markings $f \from (R_n,v) \to (G,q)$, $f' \from (R_n,v) \to (G,q')$. Choose any path $\delta$ in $G$ from $q$ to $q'$ and let $i_\delta \from \pi_1(G,q) \to \pi_1(G,q')$ be the ``change of base point'' isomorphism defined by $i_\delta[\gamma]=[\bar\delta \gamma \delta]$ for each closed path $\gamma$ based at $q$. Prove that the unbased markings $f, f' \from R_n \to G$ are homotopic if and only if the isomorphism $(f')^\inv_* \composed i_\delta \composed f_* \from F_n = \pi_1(R_n,v) \to \pi_1(R_n,v) = F_n$ is an inner automorphism.
\end{exercise}

\subsection{Marked graphs and $\Out(F_n)$.} 
\label{SectionMarkedGraphsOutFn}
Using marked graphs we can give a topological interpretation of the group $\Out(F_n)$ itself, similar to the interpretation of $\Aut(F_n)$ in terms of pointed marked graphs given in Section~\ref{SectionTopInterpAut}. It is tempting to work in the language of a general group $G$ and its $K(G,1)$ spaces as was done in a few places earlier, such as Exercise~\ref{ExerciseEilenbergMaclane}. Instead from this point we will focus solely on the special group $G=F_n$ and rank~$n$ core graphs, formulating our results only in that special language; the interested reader should be able to reformulate the results in general language. 

There are several steps in this topological interpretation, which one can compare to earlier steps in the topological interpretation of $\Aut(F_n)$: define the homotopy mapping class group $\HMCG(G)$ of a core graph $G$ (compare the pointed version of Section~\ref{SectionBasedHMCG}) and the isomorphism between $\HMCG(G)$ and $\Out(\pi_1(G))$ (compare the ``$\Aut$'' version of Exercise~\ref{ExerciseHEAutIsomorphismF_n}); define the adjoint isomorphisms between homotopy mapping class groups induced by homotopy equivalences of core graphs (compare Exercise~\ref{ExerciseTopAutFn}); and finally put the pieces together to define the canonical isomorphism between $\Out(F_n)$ and $\HMCG(G) \approx \Out(\pi_1(G))$ that is induced by any marking of $G$ (compare the ``$\Aut$'' version at the end of Section~\ref{SectionTopInterpAut}). As with the ``$\Aut$'' versions, we shall leave a lot to the reader in exercises.

%

Given a core graph $G$, its \emph{homotopy mapping class group} $\HMCG(G)$ is the group of self-homotopy equivalences modulo homotopy, with the operation induced by composition. We let $[g] \in \HMCG(G)$ denote the homotopy class of a self-homotopy equivalence $g \from G \to G$. Given a homotopy equivalence of core graphs $f \from G \to H$, its \emph{adjoint map} is the group isomorphism
$$\Ad_f \from \HMCG(G) \to \HMCG(H)
$$
defined as follows: choosing $\bar f \from H \to G$ to be a homotopy inverse for $f$, for each self-homotopy equivalence $g \from G \to G$ representing $[g] \in \HMCG(G)$ define
$$\Ad_f[g] = [f \circ g \circ \bar f] \in \HMCG(H)
$$

\paragraph{Exercises for Section~\ref{SectionMarkedGraphsOutFn}}

\begin{exercise}\label{ExerciseHMCGFunctorNBP}
Prove the following properties of $\HMCG$ and $\Ad$:
\begin{enumerate}
\item The element $\Ad_f(g) \in \HMCG(H)$ is well-defined independent of the choices of $f$ and $g$ within their homotopy classes and the choice of homotopy inverse $\bar f$ of $f$.
\item $\Ad_f$ is an isomorphism of groups.
\item The object and morphism assignments
\begin{align*}
G &\mapsto \HMCG(G) \\
(f \from G \to H) &\mapsto (\Ad_f \from \HMCG(G) \to \HMCG(H))
\end{align*}
define a functor from the groupoid of core graphs and homotopy equivalences to the groupoid of groups and isomorphisms.
\end{enumerate}
\end{exercise}

Given a base point $p \in G$, define a homomorphism $\HMCG(G) \mapsto \Out(\pi_1(G,p))$ which associates to each $f \in \HMCG(G)$ some $\phi \in\Out(\pi_1(G,p))$ as follows. Choose a homotopy equivalence $F \from G \to G$ representing~$f$, choose a path $\delta$ in $G$ from $p$ to $F(p)$, and using these choices define $\Phi \in \Aut(\pi_1(G,p))$ as follows: for each continuous closed path $\gamma$ in $G$ based at $p$ representing $[\gamma] \in \pi_1(G,p)$, let $\Phi[\gamma] = [\delta * (f \composed \gamma) * \bar \delta] \in \pi_1(G,p)$. Finally, let $\phi$ be the outer automorphism class of~$\Phi$.

\begin{exercise}\label{ExerciseHMCGIsomorphicOut} (c.f.\ Exercise~\ref{ExerciseHEAutIsomorphismF_n})
Prove that the above formula $f \mapsto \phi$ gives a well-defined isomorphism $\HMCG(G) \mapsto \Out(\pi_1(G,p))$, independent of the choice of~$\delta$. In the special case that $G = R_n$, deduce that the above formula gives a well-defined isomorphism
$$(**) \qquad \HMCG(R_n) \mapsto \Out(F_n)
$$
\end{exercise}

\begin{exercise} Combine Exercises~\ref{ExerciseHMCGFunctorNBP} and~\ref{ExerciseHMCGIsomorphicOut} to obtain a chain of canonical isomorphisms for any marked graph $\rho \from R_n \to G$, and any base point $p \in G$: 
$$\Out(F_n) \approx \Out(\pi_1(R_n,v)) \approx \HMCG(R_n) \xrightarrow{\Ad_\rho} \HMCG(G) \approx \Out(\pi_1(G,p))
$$
\end{exercise}

\medskip

For the following exercise, recall our standard notations from Section~\ref{SectionFreeGroupNotation}: $S_n = \{s_1,\ldots,s_n\}$ for the standard free basis of $F_n$; and $e_1,\ldots,e_n$ for the corresponding oriented edges of the base rose $R_n$. The following exercise can be done, for example, by combining Exercises~\ref{ExerciseHEAutIsomorphismF_n}, \ref{ExerciseAutToHMCG} and~\ref{ExerciseHMCGIsomorphicOut}.

\begin{exercise}\label{ExerciseOutIsomorphicHMCG}
Prove that the inverse isomorphism $\Out(F_n) \mapsto \HMCG(R_n)$ of $(**)$ from Exercise~\ref{ExerciseHMCGIsomorphicOut} has the following description. Given $\phi \in \Out(F_n)$, choose a representative $\Phi \in \Aut(F_n)$, and let $f_\phi \from R_n \to R_n$ be the map described in Exercise~\ref{ExerciseAutToHMCG}: for each $i=1,\ldots,n$ the tight edge path $f_\phi(e_i)$ is obtained from the reduced word $\Phi(s_i)$ by replacing each occurence of $s_j$ with $e_j$ and each occurence of $s_j^\inv$ with $\bar e_j$. Prove that $f_\phi$ is a homotopy equivalence well-defined up to homotopy independent of the choice of~$\Phi$, and that the image of $\phi$ under the isomorphism $\Out(F_n) \mapsto \HMCG(R_n)$ is the free homotopy class of the homotopy equivalence $f_\phi \from R_n \to R_n$.
\end{exercise}

\subsection{Equivalence of marked graphs} 
\label{SectionMarkedGraphEquiv}
Consider two rank~$n$ marked graphs $\rho \from R_n \to G$ and $\rho' \from R_n \to G'$. Using the markings we obtain a natural homotopy class of maps between $G \mapsto G'$ represented by the map  $\rho' \composed \bar \rho \from G \to G'$ where $\bar \rho \from G \to R_n$ is any homotopy inverse to $\rho$. This map $\rho' \composed \bar\rho$, and anything homotopic to it, is said to \emph{preserve marking}. Equivalently, a homotopy equivalence $f \from G \to G'$ \emph{preserves marking}\index{preserves marking} if the maps $f \composed \rho, \rho' \from R_n \to G'$ are homotopic, which we express by saying that the following diagram is \emph{homotopy commutative}:
$$\xymatrix{
       & R_n \ar[dl]_{\rho} \ar[dr]^{\rho'} \\
G \ar[rr]_{f} & & G'
}$$
For example, letting the base rose $R_n$ be marked by the identity map, for any marked graph $\rho \from R_n \to G$ the map $\rho$ itself preserves marking between $R_n$ and $G$.

Two marked graphs $\rho \from R_n \to G$, $\rho' \from R_n \to G'$ are \emph{equivalent} if there exists a homeomorphism $h \from G \to G'$ that preserves marking.

\paragraph{Exercises for Section \ref{SectionMarkedGraphEquiv}}

\begin{exercise} Prove that the equivalence relation on marked graphs defined above is, indeed, an equivalence relation.
\end{exercise}

\begin{exercise} Working on the rank~$2$ theta graph, and following up Exercise~\ref{ExerciseRankTwoMixerBasicMarkings}, how many different equivalence classes of marked graphs are represented by the visible markings that are constructed in that exercise?
\end{exercise}

\begin{exercise}\label{ExerciseInducedMarking}
Prove that for any two rank $n$ marked graphs $G$, $G'$ and any homotopy equivalence $f \from G \mapsto G'$, given a marking on one of $G,G'$ there is a unique (up to homotopy) marking on the other one such that $f$ preserves marking.
\end{exercise}

\smallskip

\textbf{Further Caution:} The language of Exercise~\ref{ExerciseInducedMarking} is rife with terminology abuses. The reader may wish to review the initial ``caution'' expressed in Section~\ref{SectionMarkedGraphDef}. 

\medskip

Exercise~\ref{ExerciseInducedMarking} is often silently applied when working with diagrams of homotopy equivalences. This can be seen in Exercise~\ref{ExerciseDiagramMarkingOne} to follow, which can be thought of as a souped up version of Exercise~\ref{ExerciseInducedMarking}. For a more explicit use of Exercise~\ref{ExerciseInducedMarking} in the construction of outer space, see Section~\ref{SectionIdealSimplicial}.

\begin{exercise}\label{ExerciseDiagramMarkingOne}
Suppose you are given a homotopy commutative diagram $\D$ of rank~$n$ core graphs and homotopy equivalences. In more detail: 
\begin{enumerate}
\item $\D$ is a connected graph; 
\item Associated to each vertex $v \in \D$ is a rank~$n$ core graph $G_v$; 
\item Associated to each oriented edge $e \subset \D$ with initial vertex $v$ and terminal vertex $w$ is a homotopy equivalence $f_e \from G_v \to G_w$; 
\item\label{ItemHomotopyCommutativeD}
For any closed edge path $e_1 e_2 \ldots e_K$ in $\D$ that starts and ends at a vertex $v$, the composition $f_K \circ \cdots\circ f_1 \from G_v \to G_v$ is homotopic to the identity. 
\begin{itemize}
\item As a special case of~\pref{ItemHomotopyCommutativeD}, for each oriented edge $e$ with initial vertex $v$ and terminal vertex $w$, and letting $\bar e$ be the oppositely oriented edge, the maps $f_e \from G_v \to G_w$ and $f_{\bar e} \from G_w \to G_v$ are homotopy inverses of each other.
\end{itemize}
\end{enumerate}
Suppose also that you are given a vertex $v_0 \in \D$ and a marking $\rho_{v_0} \from R_n \to G_{v_0}$. Prove that there exists a unique (up to homotopy) assignment of markings $\rho'_v \from R_n \to G_v$, one for each vertex $v \in \D$, such that $\rho'_{v_0}$ is homotopic to $\rho_{v_0}$, and such that for each edge $e \subset \D$ the map $f_e$ preserves $\rho'$-markings.
\end{exercise}

\begin{exercise}
\label{ExerciseDiagramMarkingTwo}
Suppose you are given $\D$ as in Exercise~\ref{ExerciseDiagramMarkingOne}. Does the conclusion of Exercise~\ref{ExerciseDiagramMarkingOne} hold if $v_0$ is replaced by an arbitrary nonempty subset of vertices $V \subset \D$? If not, what additional conditions guarantee that the conclusion holds? 

To be precise, the desired conclusion is worded as follows: Given an assignment of markings $\rho_v \from R_n \to G_v$, one for each $v \in V$, there exists a unique (up to homotopy) assignment of markings $\rho'_v \from R_n \to G_v$, one for each vertex $v \in \D$, such that $\rho'_v$ is homotopic to $\rho_v$ for each $v \in V$, and such that for each edge $e \subset \D$ the map $f_e$ preserves $\rho'$-markings.
\end{exercise}

\subsection{Applications: Conjugacy classes in free groups and circuits in marked graphs} 
\label{SectionConjugacyClasses}

\paragraph{Conjugacy classes.} 
In a group $\Gamma$, the conjugacy class of $g \in \Gamma$ is denoted $[g] = \{hgh^\inv \suchthat h \in \Gamma\}$. The identity element of $\Gamma$ is sole member of the \emph{trivial} conjugacy class. The set of all nontrivial conjugacy classes in $\Gamma$ is denoted $\C(\Gamma)$.

In a free group $F\<S\>$, given a word $s_1 \ldots s_K$ in $S \union \overline S$, a \emph{cyclic permutation} of this word is any word of the form $s_{j+1} \ldots s_K s_1 \ldots s_j$ for some $j=1,\ldots,K$, together with the given word itself. The word $s_1 \ldots s_K$ is \emph{cyclically reduced} if all of its cyclic permutations are reduced; equivalently $s_1 \ldots s_K$ is itself reduced and $s_1 \ne s_K^\inv$.

\begin{exercise}
\label{ExerciseConjCyclicWords}
Prove that the nontrivial conjugacy classes in $F\<S\>$ correspond bijectively to the cyclically reduced words in $S \union \overline S$, up to cyclic permutation: each nontrivial conjugacy class in $F\<S\>$ has a representative which is cyclically reduced; and this representative is unique up to cyclic permutation.
\end{exercise}

\begin{exercise}
Solve the conjugacy problem in $F\<S\>$: describe an algorithm which, given any two words $W,W'$ in $S \union \overline S$, decides whether or not the elements of $F\<S\>$ represented by $W,W'$ are conjugate. Estimate the running time of your algorithm, as a function of $\max\{\Length(W),\Length(W')\}$.
\end{exercise}

\paragraph{Circuits.} Although Exercise~\ref{ExerciseConjCyclicWords} is formulated (and can be solved) in the algebraic language of words, the bijection that is described in that exercise has a topological formulation:
\begin{itemize}
\item Conjugacy classes in $F_n$ correspond bijectively to immersed circles in the base rose $R_n$ up to orientation preserving change of parameter. 
\end{itemize}
We now generalize this formulation to any graph. For this we first define circuits in graphs, then we recall the topological meaning of conjugacy classes in the fundamental group.

In any graph $G$, define a \indexemph{circuit} to be a locally injective continuous map $\gamma \from S^1 \mapsto G$. Two circuits $\gamma,\gamma'$ are regarded as equivalent when the differ by precomposition with some orientation preserving homeomorphism $h \from S^1 \to S^1$, meaning $\gamma' = \gamma\composed h$. Every circuit may be represented as a ``cyclically reduced edge path'' in~$G$, and this representation induces a bijection between circuits up to equivalence and cyclically reduced edge paths up to cyclic permutation. We rarely remark on this bijection and these equivalence relations, thinking of equivalent circuits or their representing cyclically reduced edge paths as being ``the same''. 

Consider any path connected topological space $X$ and any base point $p \in X$. Associated to any closed path $\gamma \from [0,1] \to X$ based at $p$ there is a circle map $\sigma_\gamma \from S^1 \to X$ defined by the formula $\sigma_\gamma(e^{2 \pi i t}) = \gamma(t)$. Recall from basic algebraic topology that for any two closed paths $\gamma,\gamma' \from [0,1] \to X$ based at $p$, the corresponding fundamental group elements $[\gamma], [\gamma'] \in \pi_1(X,p)$ are conjugate in the group $\pi_1(X,p)$ if and only if the two circle maps $\sigma_\gamma$, $\sigma_{\gamma'}$ are homotopic; furthermore, this correspondence induces a bijection between conjugacy classes of $\pi_1(X,p)$ and homotopy classes of maps $S^1 \mapsto X$.

\begin{proposition}
\label{PropCircuitRep}
For any connected graph $G$, each homotopically nontrivial continuous map $\gamma \from S^1 \to G$ is homotopic to a unique circuit in~$G$. It follows that for each $p \in G$ each nontrivial conjugacy class in $\pi_1(G,p)$ is represented by a unique circuit.
\end{proposition}

\begin{proof} The map $\gamma$ induces an injection $\gamma_*$ of fundamental groups, giving an infinite cyclic subgroup of $\pi(G,p)$ with generator $[\gamma]$. Consider the connected covering space $\wt G \to G$ associated that subgroup. The graph $\wt G$, like all connected graphs with infinite cyclic fundamental group, is a circle with trees attached to its vertices. The image of $\ti\sigma$ down in $G$, with appropriate orientation, is a circuit representing the conjugacy class of $[\gamma]$. Any such circuit $\sigma$ lifts to a circuit in $\wt G$ representing the same generator as $\ti\sigma$ and so must equal $\ti\sigma$.
\end{proof}

\begin{exercise}
\label{ExerciseFreeBasisConjSep}
Prove that for any free basis $g_1,\ldots,g_n$ of $F_n$, if $i \ne j$ then $g_i$ is conjugate to neither $g_j$ nor $g^\inv_j$. 
\end{exercise}

\begin{exercise}\label{ExerciseInversesNotConjugateRevisited}
Redo Exercise~\ref{ExerciseInversesNotConjugate} in the following geometric fashion. Let $v$ be the base vertex of the base rose $R_n$. Consider a nontrivial $g \in F$ whose conjugacy class is represented by a circuit $\sigma \from S^1 \to R_n$, and its inverse $g^\inv$ which is represented by the circuit $\bar\sigma \from S^1 \xrightarrow{(a+bi) \mapsto (a-bi)} S^1 \xrightarrow{\sigma} R_n$. Using $\sigma$ and $\bar\sigma$, define labelled, oriented cell decompositions $\C, \bar\C$ of $S^1$: the vertices of $\Vertices(\C) = \sigma^\inv(v)$ and the edges of $\C$ are oriented and labelled by their image under $\sigma$; and similarly for $\bar\C$. Prove that there does not exist an orientation reversing homeomorphism $\rho \from S^1 \to S^1$ such that $\rho$ is a cellular isomorphism from $\C$ to $\bar\C$ preserving edge orientations and labels. Using this prove that $g,g^\inv$ are not conjugate.
\end{exercise}

\paragraph{Primitive conjugacy classes.} In a graph $G$, given a circuit $\rho \from S^1 \to G$ and an integer $k \ge 1$ we let $\rho^k$ be the circuit defined by the composition 
$$\rho^k \from S^1 \xrightarrow{z \mapsto z^k} S^1 \xrightarrow{\rho} G
$$
Equivalently, if $\rho$ is broken at some vertex to give the edge path $\gamma = e_1\cdots e_m$, then $\rho^k$ can be broken at a vertex to give the edge path
$$\gamma^k = \underbrace{(e_1\cdots e_m)(e_1\cdots e_m) \cdots (e_1\cdots e_m)}_{\text{$k$ times}}
$$
We say that $\sigma$ is a \emph{primitive} or \emph{root-free} circuit if there does not exist a circuit $\rho$ and $k \ge 2$ such that $\sigma = \rho^k$.

These adjectives also apply to individual elements of $F_n$, or elements of any group~$\Gamma$: an element $g \in \Gamma$ is \emph{primitive} or \emph{root-free} if there does not exist $g' \in \Gamma$ and $k \ge 2$ such that $g = (g')^k$.

\begin{exercise}
\label{ExercisePrimitiveCircuit}
Prove that for each marked graph $G$ and any non-identity element $g \in G$, the circuit in $G$ representing the conjugacy class of $g$ is primitive if and only if $g$ is primitive.
\end{exercise}

\begin{exercise}
\label{ExerciseBasisElementPrimitive}
Prove that every free basis element of $F_n$ is primitive.
\end{exercise}

\begin{exercise}
Redo Exercise~\ref{ExerciseFreeIsTorsionFree} in a geometric fashion.
\end{exercise}

\section{Finite subgroups of $\Out(F_n)$: Applying marked graphs}
\label{SectionFiniteSubgroups}

\hfill{\emph{Q: What did the pig say when the farmer caught him by the tail?}}

\hfill{\emph{A: ``This is the end of me!''}}

\smallskip

\hfill  --- An old chestnut, recorded in Bennet Cerf's ``Book of Riddles''

\bigskip

In this section we study the automorphism group of a finite core graph $G$, a finite group denoted $\Aut(G)$. We will prove Lemma~\ref{LemmaIvanov} which says the kernel of the action of $\Aut(G)$ on $H_1(G;\Z/3)$ is trivial. As a corollary we obtain Theorem~\ref{ThmAutGInjHMCG} which says that if $G$ is a \emph{marked} graph then $\Aut(G)$ embeds naturally into $\Out(F_n)$. This provides a wealth of finite subgroups of $\Out(F_n)$, and in Theorem~\ref{TheoremCZK} we will prove that every finite subgroup $H \subgroup \Out(F_n)$ is realized in this manner, a~result first proved independently by Culler \cite{Culler:FiniteSubgroups}, by Khramtsov \cite{Khramtsov:FiniteSubgroups}, and by Zimmerman \cite{Zimmerman:FiniteSubgroups}. As a corollary it will follow that $\Out(F_n)$ has a finite index torsion free subgroup, which was first proved as a quick corollary of theorem Baumslag and Taylor saying that the kernel of the action of $\Out(F_n)$ on $H_1(F_n;\Z)$ is torsion free (\cite{BaumslagTaylor:Center}).

It is interesting to compare this theory for $\Out(F_n)$ with the corresponding results for the mapping class group $\MCG(S)$ of a finite type surface~$S$. 
Nielsen conjectured in \cite{Nielsen:Abbildgungsklassen} that every finite subgroup $G \subgroup \MCG(S)$ is realized by a finite group of homeomorphisms of $S$, and he proved this conjecture when $G$ is cyclic. Serre proved in \cite{Serre:FiniteSubgroupsMCG} that every finite group of homeomorphisms of $S$ acts faithfully on $H_1(S;\Z/3)$, and combined with Nielsen's result it follows that $\MCG(S)$ has a finite index torsion free subgroup. As for the Nielsen realization conjecture itself, after a long further history (see \cite{Zieschang:FiniteGroups} for a full account), eventually Kerckhoff gave a complete proof \cite{Kerckhoff:NielsenRealization}.

\subsection{Automorphism groups of finite graphs}
\label{SectionAutFinGraph}
Starting very generally with any graph $G$, we will define automorphisms of $G$ in two different ways: a graph theoretic definition which produces a group denoted $\Aut(G)$; and a topological definition which produces the \emph{mapping class group} $\MCG(G)$. In the case where $G$ is a core graph of rank $\ge 2$ equipped with its natural cell decomposition, these two groups are naturally isomorphic (see Exercise~\ref{ExerciseGraphAutMCG}), and  one of the main theorems of this section says that the resulting group injects into the homotopy mapping class group $\HMCG(G)$ (see Lemma~\ref{LemmaIvanov}).

First we give the graph theoretic definition of automorphisms. Let $V(G)$ denote the vertex set of G. Let $E_\pm(G)$ denote the set of oriented edges of $G$. Given $e \in E_\pm(G)$, let $\bar e \in E_\pm(G)$ denote the same edge $e$ with the opposite orientation; also let $\bdy_-(e)$ and $\bdy_+(e)$ denote the initial and terminal endpoints of $e$. An automorphism of $G$ is a bijection
$$f \from V(G) \disjunion E_\pm(G) \to V(G) \disjunion E_\pm(G)
$$
which respects the structures of vertices, edges, orientation reversal, initial endpoints, and terminal endpoints; to be precise:
\begin{itemize}
\item $f$ takes $V(G)$ to $V(G)$
\item $f$ takes $E_\pm(G)$ to $E_\pm(G)$
\item For each $e \in E_\pm(G)$, letting $e'=f(e) \in E_\pm(G)$, we have:
\begin{itemize}
\item $f(\bar e) = \bar e'$
\item $f(\bdy_- e) = \bdy_- e'$
\item $f(\bdy_+ e) = \bdy_+ e'$
\end{itemize}
\end{itemize}

For the purely topological automorphism group of $G$, we simply use the mapping class group of $G$:
$$\MCG(G) = \Homeo(G) / \Homeo_0(G)
$$
where $\Homeo(X)$ is the group of homeomorphisms and $\Homeo_0(G)$ is the normal subgroup of homeomorphisms isotopic to the identity.

There is a natural homomorphism $\Aut(G) \mapsto \MCG(G)$ which for each $f \in \Aut(G)$ produces the isotopy class of a homeomorphism $F \from G \to G$ defined as follows. First, for each $V \in V(G)$ one defines $F(V)=f(V)$. Next, for each unoriented edge $e$ one chooses an orientation of $e$ thus determining $E \in E_\pm (G)$, and one chooses $F \restrict e \from E \to f(E)$ to be a homeomorphism that fixes the endpoints and preserves orientation. The resulting map $F$ is well-defined up to isotopy, independent of the choice of orientation of each $e$ and independent of the choice of homeomorphism $F \restrict e$.

There is also a natural homomorphism $\MCG(G) \to \HMCG(G)$, which maps the isotopy class of a homeomorphism of $G$ to the homotopy class of that homeomorphism. By composition we obtain a natural homomorphism $\Aut(G) \to \HMCG(G)$.

As an example, consider the rank~$n$ rose $R_n$, equipped with its natural cell decomposition having a single vertex of valence~$2n$. Its automorphism group $\Aut(R_n)$ is isomorphic to the signed permutation group which was first considered back in Exercise~\ref{ExerciseSignedPermSubgroup}. Here we may regard the signed permutation group as the group all permutations of the symbols $\{e_1, \bar e_1, e_2, \bar e_2, \ldots, e_n, \bar e_n\}$ which respect the partition into two element subsets $\{\{e_1,\bar e_1\},\{e_2,\bar e_2\},\ldots,\{e_n,\bar e_n\}\}$. This is a group of order $2^n \cdot n!$, and by Ivanov's Lemma stated below it is isomorphic to a subgroup of $\Out(F_n)$, giving a super-exponential lower bound to the maximum order of a finite subgroup of $\Out(F_n)$. This stands in sharp contrast to the fact that every finite subgroup of the mapping class group of a closed oriented surface of genus~$g$ has linearly bounded order $\le 84(g-1)$ \cite{FarbMargalit:primer}.

\paragraph{Exercises for Section \ref{SectionAutFinGraph}}


\begin{exercise} 
\label{ExerciseGraphAutMCG}
Let $G$ be a connected, finite graph of rank $\ge 2$. 
\begin{enumerate}
\item Prove that if $G$ has its natural cell decomposition then the natural homomorphism $\Aut(G) \to \MCG(G)$ is an isomorphism. 
\item\label{ItemAutInjects} Prove more generally (without assuming $G$ has the natural cell decomposition) that the natural homomorphism $\Aut(G) \to \MCG(G)$ is an injection.
\end{enumerate}
\end{exercise}

\begin{exercise}
Under what conditions on a connected graph $G$ can the natural homomorphism $\Aut(G) \to \MCG(G)$ fail to be an injection?
\end{exercise}

\subsection{Automorphisms act faithfully on homology}
\label{SectionAutsActFaithfully}

We start with Ivanov's Lemma, from his book \cite{Ivanov:subgroups}. This lemma is a graph theoretic analogue of Serre's Theorem mentioned early in Section \ref{SectionFiniteSubgroups} (and which Ivanov used to give a proof of Serre's Theorem). 

Given a graph $G$ consider its first homology $H_1(G;\Z/3)$ with coefficients in the group of integers modulo~$3$. For any homotopy equivalence \hbox{$f \from G \to G$,} using functorial properties of homology we obtain a natural induced isomorphism $f_* \from H_1(G;\Z/3) \to H_1(G;\Z/3)$, and the function $f \mapsto f_*$ induces a well-defined natural group homomorphism
$$\HMCG(G) \to \Aut(H_1(G;\Z/3)) \approx \Aut((\Z/3)^n)
$$


\begin{lemma}[Ivanov \cite{Ivanov:subgroups}]
\label{LemmaIvanov}
If $G$ is a finite core graph of rank~$n \ge 2$, then the following composed homomorphism is injective:
$$\Aut(G) \mapsto \HMCG(G) \mapsto \Aut((\Z/3)^n)
$$
\end{lemma}

Before proving this lemma, we first discuss its statement and a few of its applications.

\medskip

It will be clear from the proof of Lemma~\ref{LemmaIvanov} that the same result holds when $\Z/3$ is replaced by $\Z/n$ for $n \ge 3$, or indeed by any abelian group $A$ for which there exists a nontrivial element that does not have order~$2$. The advantage of $\Z/3$ is that it gives the best upper bounds on the cardinality of $\Aut(G)$: the order of the group $\Aut((\Z/3)^n) = \GL_n(\Z/3)$ is bounded above $3^{n^2}$ which is smaller than upper bounds that come from using other abelian groups $A$ as above.

\medskip

Lemma~\ref{LemmaIvanov} has several useful applications. We will consider later its applications in conjunction with Theorem~\ref{TheoremCZK}, but here are some immediate applications to start with:

\begin{corollary}\label{ThmAutGInjHMCG}
For each core graph $G$ of finite rank $n \ge 2$, the natural homomorphism $\Aut(G) \mapsto \HMCG(G)$ is injective. \qed
\end{corollary}

As a consequence of Lemma~\ref{LemmaIvanov} we get lots of finite subgroups of $\Out(F_n)$, one for each marked graph $G$. To prove this, consider the composition $\Aut(G) \mapsto \HMCG(G) \approx \Out(F_n)$, where the latter isomorphism is induced by the marking on $G$ (see the end of Section~\ref{SectionMarkedGraphsOutFn}). By further composition, we obtain a homomorphism
\begin{align*}
\Aut(G) &\to \Out(F_n) \to \Aut(ab(F_n)) \\ & \to \Aut(ab(F_n)) \tensor \Z/3 = \Aut(H_1(F_n;\Z/3)) \approx \Aut(H_1(G;\Z/3))
\end{align*}
and a diagram chase shows that this composition is the same map as the induced homomorphism $\Aut(G) \mapsto \Aut(H_1(G;\Z/3))$. The latter homomorphism is injective, by Lemma~\ref{LemmaIvanov}, and therefore the homomorphism $\Aut(G) \mapsto \Out(F_n)$ is injective. We record this as:

\begin{corollary} 
\label{CorollaryAutGInjection} \qquad
For each rank~$n$ marked graph $G$ the natural homomorphism \hbox{$\Aut(G) \to \Out(F_n)$} is injective. \qed
\end{corollary}

\medskip

\begin{proof} [Proof of Lemma~\ref{LemmaIvanov}] Applying Exercise~\ref{ExerciseGraphAutMCG} reduces the proof to the case that $G$ is equipped with its natural cell structure, in which case $\Aut(G) \approx \MCG(G)$, and so from here onwards we work in that case.

After choosing an orientation on each edge of $G$, we obtain the cellular homology groups of $G$ with coefficients in~$\Z/3$. Consider the quotient map from the group of $1$-cycles to the first homology group
$$Z_1(G;\Z/3) \mapsto H_1(G;\Z/3)
$$
Since there are no 2-cells, this quotient map is an isomorphism, using which we identify the abelian groups $Z_1(G;\Z/3) \approx H_1(G;\Z/3)$, and in particular for any graph automorphism $f \from G \to G$, the induced abelian group automorphism  
$$f_* \from Z_1(G;\Z/3) \to Z_1(G;\Z/3) 
$$ 
is equal to the identity if and only if $f$ induces the identity on $H_1(G;\Z/3)$. Our task is therefore to assume that $f_*$ fixes each 1-cycle with $\Z/3$ coefficients, and to prove for each natural edge $e$ that $f(e)=e$ preserving orientation. It easily follows that $f$ fixes each natural vertex and is isotopic to the identity relative to the vertex set.

\begin{description}
\item[Step 1: Circles are preserved:] For each embedded circle $c \subset G$,
\begin{itemize}
\item[(a)] $f(c)=c$ and $f$ cyclically permutes the edges of $c$, preserving orientation of $c$.
\item[(b)] If $f$ fixes a vertex of $c$ then $f$ fixes each edge of $c$ preserving orientation.
\end{itemize}
\end{description}
To prove this, choose an orientation on $c$ which we shall call the ``positive'' orientation, the opposite being called the ``negative'' orientation. Denoting $\Z/3 = \{-1,0,+1\}$, associated to $c$ there is a 1-cycle $[c] \in Z_1(G;\Z/3)$ as follows: given a natural edge $e$ of $G$, the 1-cycle $[c]$ assigns a nonzero coefficient to $e$ if and only if $e \not \subset c$, and if so then $[c]$ assigns coefficient $+1$ or $-1$ depending on whether or not the given orientation on $e$ agrees with the restriction of the positive orientation on~$c$. Corresponding to the negative orientation on $c$ is the additive inverse 1-cycle $-[c]$. Since $f_*$ is the identity we have $[f(c)]=f_*[c]=[c]$. It follows that $f(c)=c$ as a subgraph, because $f(c)$ is an oriented circle which assigns the same coefficients to edges that $c$ assigns; furthermore, since $-1 \ne +1$ in $\Z/3$ it follows that $[c] \ne -[c]$, and so $f$ preserves orientation on $c$. It also follows, by induction going around the edges of $c$, that $f$ cyclically permutes the edges; this proves item~(a), and item~(b) immediately follows.

For subsequent steps, recall the concepts of bridges, islands, and the bridge decomposition of a core graph, from Exercise~\ref{ExerciseTreeOfCoreGraphs} and the preceding material. Denote the bridge decomposition of $G$ as
$$G = (e_1 \union\cdots\union e_K) \union (H_1 \union \cdots \union H_L)
$$ 

\begin{description}
\item[Step 2: Bridgeless subgraphs are preserved:] For each bridgeless subgraph $H \subset G$ we have $f(H)=H$.
\end{description}
For the proof, choose an edge $e \subset H$. Since $e$ is not a bridge of $H$, there exists an embedded circle $c \subset H$ such that $e \subset c$. Therefore $f(e) \subset f(c)=c \subset H$.

\medskip

Applying Step 2 we obtain $f(H_l)=H_l$ for each $l=1,\ldots,L$, and we also have:
\begin{description}
\item[Step 3: Bridges are preserved:] $f(e_k)=e_k$ preserving orientation, for each $k=1,\ldots,K$.
\end{description}
To prove this, each bridge $e_k$ partitions the set of islands $\{H_1,\ldots,H_L\}$ into two subsets corresponding to the two components of the graph complement $G \setminus e_k$. Furthermore, if $e_k \ne e_{k'}$ are distinct bridges then the two parts of the $e_k$ island partition are each distinct from the two parts of the $e_{k'}$ island partition. As $k$ varies from $1$ to $K$, the collection of island partitions of the bridges $e_k$ therefore determines $2K$ distinct subsets of $\{H_1,\ldots,H_L\}$. Since $f$ preserves each individual island $H_l$, it also preserves each of these $2K$ subsets. It follows $f(e_k)=e_k$ preserving orientation for each $k$.

\medskip

After Step~3, what remains to complete the proof of Lemma~\ref{LemmaIvanov} is:

\begin{description}
\item[Step 4: Nonseparating edges are preserved:] For each nonseparating edge $e \subset G$ we have $f(e)=e$ preserving orientation.
\end{description}
For the proof, let $H_l$ be the island of $G$ for which $e \subset H_l$ ($l=1,\ldots,L$). We break the proof into cases depending on whether $H_l$ is a circle, and if not there will be further subcases.


\textbf{Case 1: $H_l$ is a circle.} Since $\rank(G) \ge 2$ whereas $\rank(H_l)=1$, it follows that the bridge decomposition is nontrivial, and in particular there exists a bridge $e_k$ and an endpoint $v$ of $e_k$ such that $v \in H_l$. By Step 3 it follows that $f(v)=v$, and by Step 1(b) it follows that $f(e)=e$ preserving orientation. 

\textbf{Case 2: $H_l$ is not a circle.} Choose a circle subgraph $c \subset H_l$ such that $e \subset c$. Applying Step 1 we have $f(c)=c$ preserving orientation, and so $e' = f(e) \subset c$. If $e'=e$ then we are done, so suppose that $e' \ne e$. Applying Step 1(b) the map $f$ acts on $c$ by a nontrivial cyclic permutation of the edges and, similarly, a nontrivial cyclic permutation of the vertices.

Consider the graph complement $H_l \setminus c$, which is nonempty because $H_l$ is not a circle. Choose any component $K$ of $H_l \setminus c$. The intersection $K \intersect H_l$ is nonempty set of vertices. There are two subcases depending on the cardinality of $K \intersect c$, and in each case we derive a contradiction.

If $K \intersect c$ is a single vertex $v$, then $f(K) \intersect c = f(K) \intersect f(c) = f(K \intersect c) = f(v) \ne v$, and therefore $K \ne f(K)$, hence $K \intersect f(K) = \emptyset$. Also, since $H_l$ is a core graph, it follows that $K$ is not a tree. The graph $K$ contains a circle subgraph $c'$, but $f(c') \subset f(K)$ hence $c' \ne f(c')$, contradicting Step~1.

If on the other hand $K \intersect c$ contains at least two vertices, let $\delta \subset c$ be an arc whose two endpoints are in $K \intersect c$ and whose interior is disjoint from $K \intersect c$. Let $\gamma \subset K$ be an arc with endpoints $v,w$. Consider the circle subgraph $c' = \gamma \union \delta$, and note that $c' \intersect c$ is the union of the arc $\delta$ and a finite set of vertices (possible empty). It follows that $f(c') = f(\gamma) \union f(\delta)$ and that $f(c') \intersect c$ is the union of the arc $f(\delta)$ and a finite set of vertices. Since $f$ acts on $c$ as a nontrivial cyclic permutation, it follows that $f(\delta) \ne \delta$, implying that $f(c') \ne c'$, contradicting Step~1.
\end{proof}

\subsection{Realizing finite subgroups of $\Out(F_n)$}
\label{SectionRealizationTheorem}

The following ``realization theorem'' says that every finite subgroup of $\Out(F_n)$ is realized in the manner described in the Corollary \ref{CorollaryAutGInjection}.

\begin{theorem}[Culler \cite{Culler:FiniteSubgroups}; Khramtsov \cite{Khramtsov:FiniteSubgroups}; Zimmerman \cite{Zimmerman:FiniteSubgroups}]
\label{TheoremCZK}
For every finite subgroup $H \subgroup \Out(F_n)$ there exists a marked graph $G$ such that $H \subgroup \Aut(G)$ (using the natural embedding $\Aut(G) \subgroup \Out(F_n)$ from Corollary~\ref{CorollaryAutGInjection}).
\end{theorem}

Before turning to the proof, we list some further corollaries regarding torsion elements, finite subgroups, and finite index subgroups of $\Out(F_n)$.

\begin{corollary}
There are only finitely many conjugacy classes of finite subgroups of $\Out(F_n)$, and there are only finitely many conjugacy classes of finite order elements of $\Out(F_n)$.
\end{corollary}

\begin{proof}
There are only finitely many homeomorphism types of marked graphs (see Exercise~\ref{ExerciseFinManyCoreGraphs}). If $G,G'$ are marked graphs of the same homeomorphism type, and if $f \from G \to G'$ is a homeomorphism, then $f$ induces an outer automorphism $\phi$ such that $\phi \Aut(G) \phi^\inv = \Aut(G')$.
\end{proof}

\begin{corollary}[Baumslag-Taylor \cite{BaumslagTaylor:Center}]
\label{CorollaryIAFiniteIndex}
$\Out(F_n)$ has a torsion free subgroup of finite index, namely
$$\text{IA}_n(\Z/3) = \text{kernel}\biggl(\Out(F_n)) \mapsto \Aut(H_1(F_n;\Z/3)) = GL_n(\Z/3) \biggr)
$$
\end{corollary}

\begin{proof}
If $\phi \in \Out(F_n)$ has finite order then by Theorem~\ref{TheoremCZK} the finite cyclic group $\<\phi\>$ is a subgroup of $\Aut(G)$ for some marked graph $G$. By the corollary to Ivanov's Lemma \ref{LemmaIvanov}, the homomorphism $\Aut(G) \subgroup \Out(F_n) \mapsto H_1(F_n;\Z/3) = GL_n(\Z/3)$ is injective, and so $\phi$ is not in its kernel.
\end{proof}

The proof of the following corollary is left to the reader in Exercise~\ref{ExerciseIASurjection}.

\begin{corollary}[Nielsen \cite{Nielsen:IsomorphismGroup}]
The homomorphism $\Out(F_n) \mapsto \GL_n(\Z)$.
\end{corollary}

\begin{corollary} The homomorphism $\Out(F_n) \mapsto \GL_n(\Z/3)$ is surjective. The index of $IA_n(\Z/3)$ in $\Out(F_n)$ is therefore equal to the cardinality of $\GL_n(\Z/3)$, which is $\le 3^{n^2}$. \qed
\end{corollary}

The universe contains such monsters as infinite torsion groups, even finitely generated ones, as shown by Olshanskii \cite{Olshanskii:Monster}, but $\Out(F_n)$ is not one of them:

\begin{corollary}
Every torsion subgroup of $\Out(F_n)$ is finite. Every finite subgroup has cardinality $\le 3^{n^2}$.
\end{corollary}

\begin{proof}
If $H \subgroup \Out(F_n)$ is a torsion subgroup, then the restricted homomorphism $H \to GL_n(\Z/3)$ is injective because each element generates a finite cyclic subgroup. It follows that the cardinality of $H$ is less than the cardinality of $GL_n(\Z/3)$ which has order bounded above by $3^{n^2}$.
\end{proof}

\begin{proof}[Proof of Theorem \ref{TheoremCZK}] By definition of $\Aut(F_n)$ and $\Out(F_n)$, and the fact that $F_n$ is centerless, we have a short exact sequence
$$1 \to F_n = \Inn(F_n) \xrightarrow{i} \Aut(F_n) \xrightarrow{q} \Out(F_n) = 1
$$
Consider a finite subgroup $H \subgroup \Out(F_n)$, let $j \from H \inject \Out(F_n)$ denote the inclusion homomorphism, and extend the above diagram as follows:
$$\xymatrix{
 &  &  & H \ar[r] \ar[d]^{j} & 1 \\
1 \ar[r] & F_n \ar[r]^{i} & \Aut(F_n) \ar[r]^{q} & \Out(F_n) \ar[r] & 1
}$$
Define $\wh H = q^\inv(j(H))$, a subgroup of $\Aut(F_n)$ whose image in $\Out(F_n)$ equals $H$. Extend the diagram further using the projection homomorphisms $\wh H \to \Aut(F_n)$ and $\wh H \to H$ obtained by restricting the projections of $\Aut(F_n) \oplus H$ to its direct factors. We get the following commutative diagram of short exact sequences:
$$\xymatrix{
1 \ar[r] & F_n \ar[r] \ar@{=}[d] & \wh H \ar[r] \ar[d] & H \ar[r] \ar[d]^j & 1 \\
1 \ar[r] & F_n \ar[r] & \Aut(F_n) \ar[r]^q & \Out(F_n) \ar[r] & 1
}$$
We note that some aspects of this construction can be carried out for any group homomorphism $j \from H \mapsto \Out(F_n)$ whatsoever: the group $\wh H$ is defined to be the ``fiber product'' of the two homomorphisms $q$ and $j$, also known as the ``pullback''. For this reason the short exact sequence $1 \mapsto F_n \mapsto \Aut(F_n) \mapsto \Out(F_n) \mapsto 1$ is called the \emph{universal extension} of $F_n$.

Now we bring in Hopf's theory of ends of groups \cite{Hopf:Ends}, which has its roots in Freudenthal's theory of ends of topological spaces \cite{Freudenthal:Ends}. We outline this theory briefly; full details can be found, for instance, in \cite{ScottWall}. The reader may also want to review basic concepts of group actions which can be found in Section~\ref{SectionTopActions}.

Suppose $G$ is a finitely generated group and $\Gamma$ is a connected graph of uniformly bounded valence on which $G$ acts freely, properly, and cocompactly by simplicial isomorphisms. For example we can take $\Gamma$ to be the Cayley graph with respect to some finite generating set of~$G$, equipped with its natural left action by~$G$. Define the \emph{set of ends} of $\Gamma$ as follows. For any nested pair of compact sets $K_1 \subset K_2 \subset \Gamma$, each component of $\Gamma-K_2$ is contained in a unique component of $\Gamma-K_1$, inducing a well-defined function $\pi_0(\Gamma-K_2) \inject \pi_0(\Gamma-K_1)$. As $K$ varies, these inclusion induced functions on the sets $\pi_0(\Gamma-K)$ form an inverse system whose inverse limit is defined to be the set $\Ends(\Gamma)$. Formally this means that an end of $\Gamma$ is a function $\eta(K)$, defined for each compact $K \subset \Gamma$, such that $\eta(K) \in \pi_0(\Gamma-K)$, and such that if $K_1 \subset K_2$ then $\eta(K_2) \subset \eta(K_1)$.

If $\Gamma,\Gamma'$ are any two connected graphs with uniformly bounded valence, each equipped with a free, properly discontinuous, cocompact, simplicial actions by the group $G$, we wish to show that $\Ends(\Gamma)$ and $\Ends(\Gamma')$ have the same cardinality. The cardinality of $\Ends(\Gamma)$ is therefore well-defined independent of the choice of $\Gamma$ and its action by $G$, and we call this cardinality the \emph{number of ends} of the group~$G$. 

To prove that $\Ends(\Gamma)$, $\Ends(\Gamma')$ have the same cardinality, first we show, using that the actions are free, that there exists a $G$-equivariant proper, continuous function $f \from \Gamma \to \Gamma'$ which induces a function $f_E \from \Ends(\Gamma) \mapsto \Ends(\Gamma')$. For example $f_E$ is defined by choosing one vertex $v$ out of each vertex orbit of the action $F_n \act \Gamma$, then choosing the value $f(v) \in \Gamma'$, then extending equivariantly over the orbit of $v$; after doing this for each vertex orbit, one then chooses one edge $e$ out of each edge orbit, then one chooses the image path $f(e)$ whose endpoint values agree with the already defined values of $f$ on vertices, then one extends equivariantly over the orbit of $e$. Using that the actions of $F_n$ are equivariant, properly discontinuous, and cocompact, it follows that $f$ is a proper function. Therefore $f$ has an induced end function $f_E$ defined so that for each end $\eta \in \Ends(\Gamma)$ and each compact $K' \subset \Gamma'$, the value $f_E(\eta)(K)$ is defined to be the unique component of $\Gamma'-K$ that contains the connected subset $f(\eta(f^\inv(K')))$. 

Next, the same construction produces a $G$-equivariant proper, continuous function $f' \from \Gamma' \to \Gamma$ which induces a function $f'_E \from \Ends(\Gamma') \mapsto \Ends(\Gamma)$. The composed $G$-equivariant functions $f' \circ f \from \Gamma \to \Gamma$ and $f \circ f' \from \Gamma' \to \Gamma$ each have bounded distance from the identity, meaning that the quantities $d(f' \circ f(x),x)$ and $f(f \circ f(x'),x')$ have finite upper bounds independent of $x \in \Gamma$ and $x' \in \Gamma'$. Using these upper bounds, it follows that the compositions $f'_E \circ f^\vp_E \from \Ends(\Gamma) \to \Ends(\Gamma)$ and $f^\vp_E \circ f'_E \from \Gamma' \to \Gamma'$ are the respective identity functions on their domains. The maps $f^\vp_E$, $f'_E$ are therefore inverse bijections.

The theory of ends goes further, in fact the inverse limit process produces a natural ``end topology'' on the $\Ends(\Gamma)$, namely the inverse limit topology obtained using the discrete topologies on each of the sets $\pi_0(\Gamma-K)$. This ``end topology'' is also well-defined independent of the choice of $\Gamma$ in the sense that for any other $\Gamma$ the induced bijection $\Ends(\Gamma) \mapsto \Ends(\Gamma')$ is a homeomorphism. We may therefore speak of $\Ends(\Gamma)$ as the \emph{end space} of~$\Gamma$. For example, by using the Cayley tree of $F_n$ with respect to a free basis, we see that $F_n$ has uncountably many ends, and in fact using the inverse limit topology the end space is homeomorphic to a Cantor set.

The following theorem captures some of the essentials of the theory of ends of groups. 

\begin{theorem} 
\label{TheoremHopfStallingsDunwoody}
For every finitely generated group~$G$ the following hold:
\begin{enumerate}
\item \cite{Hopf:Ends} The space of ends of $G$ is either a set of cardinality $0$, $1$ or $2$, or it is homeomorphic to the Cantor set and hence has the cardinality of the real numbers. 
\item\label{ItemStallingsEnds} \cite{Stallings:ends,ScottWall} If the number of ends of $G$ is $\ge 2$ then there exists a simplicial tree $T$ and a cocompact simplicial action $G \act T$ so that the following property holds: for each edge $e \subset T$ its stabilizer subgroup $\Stab_G(e) \subgroup G$ is finite. 
\item \cite{Dunwoody:Accessible} If in item \pref{ItemStallingsEnds} the group $G$ is finitely presented, then we can choose the action $G \act T$ so that the following additional property holds: for each vertex $v \in T$ its stabilizer subgroup $\Stab_G(v)$ is finitely generated and has $\le 1$ end. 
\end{enumerate}
\end{theorem}

We return now to the normal subgroup $F_n \subgroup \wh H$ with finite quotient group $H = \wh H / F_n$, and so $F_n$ has finite index in $\wh H$. Since $F_n$ is finitely generated and finitely presented, so is $\wh H$. Let $\Gamma$ be a Cayley graph for $\wh H$ with respect to some finite generating set. The action of $\wh H$ on $\Gamma$, which is properly discontinuous and cocompact, restricts to an action of $F_n$ on $\Gamma$ which is also properly discontinuous and cocompact, since $F_n$ has finite index in $\wh H$. Knowing that $F_n$ has infinitely many ends, and it follows that $\Gamma$ has infinitely many ends, and therefore $\wh H$ has infinitely many ends (this is where we use well-definedness of the number of ends). 

Applying the Theorem~\ref{TheoremHopfStallingsDunwoody}, there exists a simplicial action on a tree $\wh H \act T$ such that $\Stab_H(e)$ is finite for each edge $e \subset T$ and $\Stab_H(v)$ is finitely generated and either finite or one-ended. We may also assume that this action is ``minimal'' which means that no smaller subtree of $T$ is invariant under $\wh H$; using cocompactness of the action it follows that no vertex of $T$ has valence~$1$. 

Consider the restricted action $F_n \act T$. Since $F_n$ is a free group, each of its nontrivial finitely generated subgroups is a free group and has $\ge 2$ ends. Since $\Stab_{F_n}(e)$ is finite for each edge $e$ of $T$, and since $\Stab_{F_n}(v)$ is finitely generated and has $\le 1$ end for each vertex $v$ of $T$, it follows that each $\Stab_{F_n}(e)$ and each $\Stab_{F_n}(v)$ is trivial. Thus the action $F_n \act T$ is a free action. We may assume that the action $\wh H \act T$ and the restricted action $F_n \act T$ are minimal, meaning that no proper subtree of $T$ is invariant under either of these actions: if this is not already true, we can replace $T$ by  the pre-image of the unique core of the quotient graph $G = T / F_n$, hence the action of $F_n$ on this subtree is minimal; and since $F_n$ is a normal subgroup of $\wh H$ it follows that this subtree is also invariant and minimal under the action of $\wh H$. One may choose an $F_n$-equivariant map $\wt R_n \mapsto T$, which descends to a homotopy equivalence $\rho \from R_n \to G$, making the core graph $G$ into a rank~$n$ marked graph. Any two choices of the map $\wt R_n \mapsto T$ are equivariantly homotopic, making the marking well-defined up to equivalence of markings. Since $F_n$ is normal in $\wh H$, the action $\wh H \act T$ descends to an action of the quotient group $\wh H / F_n = H \act G$. It remains to verify that the homomorphism $H \mapsto \Out(\pi_1 G) = \Out(F_n)$ determined by this action is identical to the original injection $H \xrightarrow{j} \Out(F_n)$, which is a diagram chase argument.
\end{proof}

\paragraph{Exercises for Section \ref{SectionRealizationTheorem}}

\begin{exercise}
Construct examples of finite connected graphs $G$ in all ranks such that $\Aut(G) \to \HMCG(G)$ is not an injection (Hint: Serre's Theorem gives a necessary condition on $G$ for this to happen).
\end{exercise}

\begin{exercise} 
\label{ExerciseMarkingChangeUnique}
Recall from Section~\ref{SectionMarkedGraphEquiv} that two marked graphs $(G,\rho)$, $(G',\rho')$ are equivalent if and only if there exists a homeomorphism $h \from G \to G'$ such that $h \circ \rho$ is homotopic to $\rho'$. Prove that the homeomorphism $h$ is unique up to isotopy; to put it another way, $h$ is unique in the sense that the maps that $h$ induces from vertices of $G$ to vertices of $G'$ and from oriented edges of $G$ to oriented edges of $G'$ are unique. (Hint: Use Serre's Lemma~\ref{LemmaIvanov}.)
\end{exercise}

\begin{exercise}
Prove that no finite subgroup of $\Out(F_n)$ is normal (thanks to Andres Meija for posing this question on math.stackexchange.com \cite{Mejia:OutFnQuestion}).
\end{exercise}

\begin{exercise}
\label{ExerciseIASurjection}
Prove that the natural homomorphism $\Out(F_n) \mapsto \Aut(H_1(F_n;\Z)) \approx \GL_n(\Z)$ is surjective (Hint: Consider the Nielsen transformations listed in Section~\ref{SectionPositiveTests}).
\end{exercise}

\subsection{Appendix: Group actions and their properties}
\label{SectionTopActions}

Here we collect some basic concepts of the theory of group actions on topological spaces. Two places where this material is used are: the theory of ends which is presented applied in Section~\ref{SectionRealizationTheorem}; and the study of the action of $\Out(F_n)$ on the outer space $\X_n$ starting in Section~\ref{SectionOuterSpaceAction}.

Suppose that $G$ is a group and $X$ is an object in some unspecified but concrete category, meaning a category equipped with a forgetful functor to the category of sets. An \emph{action} of $G$ on $X$ is a homomorphism from $G$ to the group of automorphisms of~$X$. There are two conventions for denoting automorphisms, and hence group actions: for a \emph{right action} of $G$ on $X$, the automorphism of $X$ associated to $g \in G$ is denoted in postfix notation as $x \mapsto x \cdot g$; and for a \emph{left action} it is denoted in prefix notation as $x \mapsto g \cdot x$. We shall formulate the definitions in this section using \emph{right actions} because of their natural use in the action of $\Out(F_n)$ on outer space (see Section~\ref{SectionOuterSpaceAction}). Nonetheless where appropriate in this work we will use left actions as well. 

Fix an action of $G$ on $X$. For each point $x \in X$, its \indexemph{stabilizer subgroup} is defined to be
$$\Stab(x) = \{g \in G \suchthat x \cdot g = x\}
$$
The action is \emph{free}\index{action!free} if the stabilizer subgroup of every point is trivial. More generally, the action is \emph{faithful}\index{action!faithful} if for every $g \in G$ there exists $x \in X$ such that $g \not\in \Stab(x)$; equivalently, $g$ acts as a nontrivial automorphism of~$X$.

For the rest of this section we fix $X$ to be an object in a Hausdorff topological category, meaning a category equipped with a forgetful functor to the category of Hausdorff spaces; for example, $X$ could be a simplicial complex. We also fix an action of a group $G$ on $X$. An open subset $U \subset X$ is called an \emph{open fundamental domain} if the set $U \cdot G = \{x \cdot g \suchthat x \in U, g \in G\}$ is equal to~$X$, and the set of group elements $\{g \in G \suchthat (U \cdot g) \intersect U \ne \emptyset\}$ is finite.  

%
%

A version of the following lemma is sometimes incorporated into the Milnor-Svarc Lemma of geometric group theory (see Lemma~\ref{LemmaMilnorSvarc}), although the idea may be somewhat older.


\begin{lemma}
\label{LemmaFDFinGen}
If $X$ is path connected, and if the action has an open fundamental domain $U \subset X$, then $G$ is finitely generated. To be precise, the finite set $S \in \{g \in G \suchthat (U \cdot g) \intersect U \ne \emptyset\}$ is a symmetric generating set for the group~$G$.
\end{lemma}

\begin{proof} Fix a base point $p \in U \subset X$. Given $g \in G$, let $\gamma \from [0,1] \to X$ be a path from $p$ to $p \cdot g$. Pull back the open cover $U \cdot G$ via the continuous map $\gamma$ to get an open cover of $[0,1]$. Choose a Lebesgue number $\lambda>0$ for the pullback cover, choose a natural number $k > \frac{1}{\lambda}$, and let
$$0=x_0 < x_1 < \cdots < x_{k-1} < x = 1
$$
be the partition of $[0,1]$ into subintervals of length $1/k$, where $x_i = i/k$. Since each subinterval of the partition has length $< \lambda$, we may  choose a sequence of group elements $h_1,\ldots,h_k \in G$ such that 
$$\gamma\left[x_{i-1},x_i\right] \subset U \cdot h_i \quad(i=1,\ldots,k)
$$
Since $p = \gamma(0) \in U \intersect (U \cdot h_1)$, it follows that $h_1 \in S$. For each $i=1,\ldots,k-1$, since $\gamma(x_i) \in (U \cdot h_i) \intersect (U \cdot h_{i+1})$, it follows that $h_{i+1} h_i^\inv \in S$. Since $g \cdot p \in (U \cdot h_k) \intersect (U \cdot g)$, it follows that $g h_k^\inv \in S$. Therefore
$$g = (gh_k^\inv)(h_k^{\vphantom{\inv}} h_{k-1}^\inv)\cdots(h_2^{\vphantom{\inv}} h_1^\inv) h_1^{\vphantom{\inv}}
$$
is a product of elements of $S$.
\end{proof}

To get some good examples of actions with fundamental domains, we introduce two important properties of an action of a group $G$ on a topological space $X$:
\begin{description}
\item[The action is \emph{cocompact}]\index{action!cocompact} if there exists a compact $D \subset X$ such that the set $D \cdot G = \{x \cdot g \suchthat x \in D, g \in G\}$ is equal to~$X$. 
\item[The action is \emph{proper}]\index{action!proper} if the following two equivalent statements hold: 
\begin{itemize}
\item the function $X \times G \mapsto X \times X$ given by $(x,g) \mapsto (x,x \cdot g)$ is a proper function, meaning that the inverse image of every compact subset of $X \times X$ is compact in $X \times G$ (with respect to the discrete topology on $G$); 
\item for any two compact sets $A,B \subset X$, the set of group elements $\{g \in G \suchthat (A \cdot g) \intersect B \ne \emptyset\}$ is finite.
\end{itemize}
\end{description}
Recall that $X$ is \emph{locally compact} if every point $x \in X$ has an open neighborhood $U \subset X$ with compact closure $\overline U \subset X$.

\begin{lemma}
\label{LemmaFDExists}
If $X$ is locally compact, and if $G$ acts properly and cocompactly on $X$, then $X$ has an open fundamental domain.
\end{lemma}

\begin{proof}
Choose a compact $D \subset X$ such that $D \cdot G=X$. For each $x \in D$ choose an open $U_x \subset X$ with compact closure $C_x = \overline U_x$ such that $x \in U_x$. Let $U_1,\ldots,U_I$ with $U_i = U_{x_i}$ be a finite collection of these sets which covers $D$. Let $U = \bigcup_1^I U_i$, so $D \subset U$ and hence $U \cdot G = X$. 

Consider $g \in G$ such that $(U \cdot g) \intersect U \ne \emptyset$. It follows that there exist $i,j \in \{1,\ldots,I\}$ such that $(U_i \cdot g) \intersect U_j \ne \emptyset$, and hence $(C_i \cdot g) \intersect C_j \ne \emptyset$. By properness, there are only finitely many such $g$ for each $i,j \in \{1,\ldots,I\}$, and hence there only finitely many such $g$ altogether. The set $U$ is therefore an open fundamental domain for the action.
\end{proof}

Combining Lemmas~\ref{LemmaFDExists} and~\ref{LemmaFDFinGen} we obtain:

\begin{corollary} 
\label{CorollaryGeneralFiniteGeneration}
If $X$ is locally compact, and if $G$ acts properly and cocompactly on $X$, the $G$ is finitely generated. \qed
\end{corollary}

\paragraph{Exercise for Section \ref{SectionTopActions}.} Continuing as above with a Hausdorff topological space $X$ and an action of $G$ on $X$, we also assume $X$ to be locally compact. 

We have defined an open fundamental domain above, with a strong finiteness property. It is somewhat more traditional to consider fundamental domains which are closed or even compact subsets. We adopt the following definition: a subset $A \subset X$ is a \emph{fundamental domain} if $A \cdot G$ covers $X$ and if there exists an open fundamental domain $U$ such that $A \subset U$. (For group actions on smooth manifolds even more strict conditions are usually adopted, namely that the fundamental domain be some kind of polygonal object, and that any intersection with any of its translates is either empty or a common face.)

\begin{exercise}\label{ExerciseCocompactChar}
Prove that the action is cocompact if and only if it has a compact fundamental domain.
\end{exercise}

\begin{exercise}\label{ExerciseCocompactAndNot}
Prove that the action cannot have both a compact fundamental domain and a closed, noncompact fundamental domain.
\end{exercise}

In the following exercises, let $X$ be a locally finite, connected simplicial 1-complex, and let the group $G$ act on $X$ by simplicial isomorphisms. 

\begin{exercise}\label{ExerciseGraphHypEquiv} Prove that $G$ acts properly on $X$ if and only if the stabilizer of each vertex is finite (hence the stabilizer of each edge is finite). Prove that $G$ acts cocompactly on $X$ if and only if there are finitely many orbits of edges (hence finitely many orbits of vertices).
\end{exercise}

\begin{exercise} Prove that if there are finitely many edge orbits, and if each vertex stabilizer is a finitely generated group, then $G$ is finitely generated (n.b.\ if some vertex stabilizer is infinite then the action is not proper).
\end{exercise}

\section{The Nielsen/Whitehead problems: Conjugacy versions.} 
\label{SectionWhiteheadConjugacy}

\subsection{$\Out(F_n)$ and its action on conjugacy classes.} 
\label{SectionOutActionOnConjugacy}
Consider for a moment a general group $\Gamma$. We fix some notation regarding automorphisms and outer automorphisms of $\Gamma$. Given $h \in \Gamma$ let $i_h(g) = hgh^\inv$ denote the associated inner automorphism, the set of which forms the normal subgroup $\Inn(\Gamma) \subgroup \Aut(\Gamma)$ with quotient $\Out(\Gamma)=\Aut(\Gamma)/\Inn(\Gamma)$. We use capital Greek letters like $\Phi$ to denote an automorphism, and small Greek letters like $\phi$ to denote the outer automorphism class of $\Phi$, which means the left or right coset 
$$\phi = \Phi \cdot \Inn(\Gamma) = \{\Phi \, i_g \suchthat g \in \Gamma\} = \Inn(\Gamma) \cdot \Phi = \{i_g \, \Phi \suchthat g \in \Gamma\}
$$
To put it another way, $\phi$ denotes the image of $\Phi$ under the quotient homomorphism $\Aut(\Gamma) \mapsto \Out(\Gamma)$. 

The group $\Aut(\Gamma)$ acts on the set $\Gamma$, of course: $\Phi(g) \in \Gamma$ is well-defined for each $\Phi \in \Aut(\Gamma)$ and each $g \in \Gamma$. However $\Out(\Gamma)$ does not act on the set $\Gamma$: given $\phi \in \Out(\Gamma)$ and $g \in \Gamma$ there is no well-defined $\phi(g) \in \Gamma$. Nonetheless $\Out(\Gamma)$ acts on the set $\C(\Gamma)$ of conjugacy classes of the group $\Gamma$: for each $\phi \in \Out(\Gamma)$ and each $c \in \C(\Gamma)$, choose $\Phi \in \Aut(\Gamma)$ representing $\phi$, choose $g \in \Gamma$ representing $c$, and define $\phi(c) = [\Phi(g)]$. 

\begin{exercise} 
Prove that $\phi(c)$ is well-defined independent of the choice of the representatives $\Phi$ of $\phi$ and $g$ of~$c$. Prove that this defines a left action of $\Out(\Gamma)$ on $\C(\Gamma)$, meaning:
\begin{itemize}
\item $\phi (\psi (c)) = (\phi \psi)(c)$, for all $\phi,\psi \in \Out(\Gamma)$ and $c \in \C(\Gamma)$.
\item The identity outer automorphism fixes each conjugacy class.
\end{itemize}
\end{exercise}

\medskip

The following exercises explain how to use the action of $\Out(F_n)$ on conjugacy classes to obtain useful information about individual outer automorphisms.

\begin{exercise}
Find an example of an infinite order element $\phi \in \Out(F_n)$. (Hint: find $\phi$ and $c \in \C(F_n)$ so that $\phi^i(c) \ne c$ for any integer $i \ge 1$.
\end{exercise}

\begin{exercise} Find an example of a non-identity element $\phi \in \Out(F_n)$ that fixes the conjugacy classes of each of the basis elements $[s_1],\ldots,[s_n]$. 
\end{exercise}


\begin{exercise}
\label{ExerciseMovedEdge}
Let $G$ be a core graph and let $e, e' \subset G$ be natural edges, and let $f \from G \to G$ be a continuous map.
\begin{enumerate}
\item\label{ItemCrossesThisNotThat}
Prove that if $e \ne e'$ then there exists a circuit in $G$ whose image contains $e'$ but not $e$.
\item\label{ItemMovesEdgeNotIdentity}
Prove that if $e \ne e'$, if $f$ maps $e$ homeomorphically onto $e'$, and if $f(G \setminus  e) \subset G \setminus  e'$, then $f$ is not homotopic to the identity. (Hint: Apply part \pref{ItemCrossesThisNotThat} together with the results of Sections~\ref{SectionConjugacyClasses} 
and~\ref{SectionOutActionOnConjugacy}).
\item\label{ItemMovesEdgeArcsNotIdentity}
In Section~\ref{SectionFaceMaps} we will need this slightly stronger version which, unlike \pref{ItemMovesEdgeNotIdentity}, can also be applied when one of $e$ is a loop edge and the other edge isn't. Let $\check e \subset \interior(e)$ and $\check e' \subset \interior(e')$ be compact arcs. Prove that if $f(\check e)=\check e'$ and if $f(G \setminus \check e) \subset G \setminus \check e'$ then $f$ is not homotopic to the identity. 
\end{enumerate}
\end{exercise}

\begin{exercise}
Use Exercise~\ref{ExerciseMovedEdge} to obtain another proof of Corollary~\ref{CorollaryAutGInjection} (albeit one which gives no information regarding induced homology isomorphisms). 
\end{exercise}

The next exercise has a topological proof that we will give later, but it is already interesting to ponder, particularly the special case of rank~$2$.

\begin{exercise} 
\label{ExerciseOutCFaithful}
Prove that the action of $\Out(F_n)$ on $\C(F_n)$ is faithful, that is, if $\phi \in \Out(F_n)$ fixes every element of $\C(F_n)$ then $\phi$ is the identity outer automorphism.
\end{exercise}

\subsection{Statement of conjugacy versions of the Nielsen/Whitehead problems.} 

Consider a finite list of conjugacy classes $c_1,\ldots,c_K$ in $\C(F_n)$. We say that this list is \emph{represented by a (partial) free basis} if there are pairwise distinct representatives $g_1,\ldots,g_K \in F_n$ that form a (partial) free basis $\{g_1,\ldots,g_K\}$.
\begin{itemize}
\item Given $c_1,\ldots,c_K \in \C(F_n)$, how do you tell whether they are represented by a free basis? Or a partial free basis? Given a single conjugacy class $c \in \C(F_n)$, how do you tell whether $c$ is represented by a free basis element?
\end{itemize}
We will start with some simple negative and positive tests for the conjugacy versions of the Whitehead problems; but, just as in the original versions of these problems, there will be a large gap between these simple tests. 

To start, using the results of Section~\ref{SectionConjugacyClasses} one obtains a negative test:
 
\begin{exercise}\label{ExerciseJointPrimitiveTest}
Prove that the list $c_1,\ldots,c_K$ is represented by a partial free basis only if each is primitive, and for $i \ne j \in \{1,\ldots,K\}$ the conjugacy class $c_i$ is distinct from each of the conjugacy classes $c_j,c^\inv_j$; equivalently, for any marked graph $G$ with circuits $\gamma_1,\ldots,\gamma_K$ that represent $c_1,\ldots,c_K$ in $G$, each circuit $\gamma_i$ is primitive, and for $i \ne j \in \{1,\ldots,K\}$ the circuit $\gamma_i$ is distinct from each of the circuits $\gamma_j,\gamma^\inv_j$.
\end{exercise}

All of the negative tests described earlier for the original versions of Whitehead's problems work just as well for the conjugacy versions, because the function $F_n \mapsto \ab(F_n) \approx \Z^n$ is well-defined on conjugacy classes, and it takes each free basis of $F_n$ to a $\Z$-module basis for $\Z^n$. Thus, for example, in the group $F\<a,b\>$ neither $abab$ nor $a^2 b^2$ nor $a^5 b^{-4} a^{-2} b^{42} a b^{-36}$ is conjugate to a free basis element, because the image of each in $\Z^2$ is the vector $\<2,2\>$ which is not a basis element of $\Z^n$. Other simple negative tests also follow: given $c_1,\ldots,c_K$ as above, if $i \ne j$ and if $c_i$ is equal to either $c_j$ or~$c_j^\inv$ then no pairwise distinct representatives of $c_1,\ldots,c_K$ form a partial free basis of $F_n$.

The same idea behind the positive test for the original version of Whitehead's problems works here as well for the conjugacy versions:
\begin{exercise}
Given $c_1,\ldots,c_K$ as above, prove that this set has pairwise distinct representatives forming a free basis if and only if there exists $\phi \in \Out(F_n)$ such that $\phi[s_i] = c_i$ for each $i=1,\ldots,n$.
\end{exercise}

Here are some broader positive tests of a more topological nature.

Define a circuit $\sigma \from S^1 \mapsto G$ to be \emph{simple} if $\sigma$ is injective. More generally, $\sigma$ is \emph{edge simple} if it is injective over the interior of each edge of $G$, i.e.\ if $x \in G$ is in the interior of some edge of $G$ then $\sigma^\inv(x)$ is either empty or a single point. Equivalently, writing $\sigma = e_1 \cdots e_k$ as a cyclic concatenation of edges, $\sigma$ is edge simple if and only if for all $i \ne j \in \{1,\ldots,k\}$ we have $e_i \ne e_j^{\pm 1}$. 

\begin{exercise}\label{ExerciseEdgeSimpleCircuit}
Prove that if $G$ is a marked graph and $\sigma$ is an edge simple circuit then any element of $F_n$ whose conjugacy class is represented by $\sigma$ is a free basis element.
\end{exercise}

The next exercise is a ``partial free basis'' version of Exercise~\ref{ExerciseEdgeSimpleCircuit}. Given a compact oriented 1-manifold $C$ with components $C = C_1 \union\cdots\union C_K$, and a continuous map $\sigma \from C \to G$ with components $\sigma=\sigma_1 \union\cdots\union C_K$, $\sigma_k = \sigma \restrict C_k$, we say that $\sigma$ is a \emph{circuit family} in $G$ if $\sigma$ is an immersion and for all $i \ne j \in \{1,\ldots,K\}$ the components $\sigma_i$, $\sigma_j$ are not equivalent; it follows that the $\sigma_i$'s represent a pairwise distinct set of conjugacy classes $\{[\sigma_1],\ldots,[\sigma_K]\}$ of cardinality $K$. The individual $\sigma_i$'s are called the \emph{component circuits} of $\sigma$. We say that $\sigma$ is \emph{edge simple} if it is injective over the interior of each edge of $G$. Equivalently, writing each component circuit $\sigma_i = e_{i,1} \cdots e_{i,J_i}$ as a cyclic concatenation of oriented edges of $G$, the circuit family $\sigma$ is edge simple if and only if for all $1 \le i,i' \le K$, $1 \le j \le J_i$, $1 \le j' \le J_{i'}$, if $(i,j) \ne (i',j')$ then $e_{i,j} \ne e_{i',j'}^{\pm 1}$. 

\begin{exercise}\label{ExerciseEdgeSimpleCircuits}
Prove that if $G$ is a marked graph and if $\sigma = \sigma_1 \union\cdots\union \sigma_K$ is an edge simple circuit family in $G$ then there exists a partial free basis $g_1,\ldots,g_K$ of $F_n$ whose conjugacy classes are represented in $G$ by $\sigma$.
\end{exercise}

\subsection{Topological interpretation.} 

Following up on Exercise~\ref{ExerciseEdgeSimpleCircuits}, 
The following proposition translates the conjugacy version of Whitehead's problem into an equivalent topological statement expressed in the language of marked graphs. This is the version we will use in Section~\ref{SectionWhiteheadStartup} to start up the solution of Whitehead's problem.

\begin{proposition} 
\label{PropWhiteheadConjugacyTopological}
A finite list of conjugacy classes $c_1,\ldots,c_k \in \C(F_n)$ is represented by a partial free basis of $F_n$ if and only if there exists a marked graph $\rho \from R_n \mapsto G$ such that $c_1,\ldots,c_k$ are represented in $G$ by a pairwise disjoint circuit family $\sigma = \sigma_1 \union \ldots \union \sigma_k$. 
\end{proposition}

\begin{proof} We start from the evident fact that $c_1,\ldots,c_k$ are represented up to conjugacy by a partial free basis if and only if the following holds:
\begin{itemize}
\item[$(*)$] There exists a homotopy equivalence $f \from R_n \to R_n$ such that $c_1,\ldots,c_k$ are represented by the images under $f$ of $k$ distinct petals of the domain rose.
\end{itemize}
By blowing up the vertex of the rose $R_n$ we obtain a \emph{rank $n$ bola graph} $B_n$ depicted in Figure~\ref{FigureRankTwoCoreGraphs}, which can be marked by a homotopy equivalence $g \from B_n \mapsto R_n$ that collapses to a point the union of $n$ edges of $B_n$ that connect the valence $n$ vertex to the $n$ vertices of valence 3. The composed map $B_n \xrightarrow{g} R_n \xrightarrow{f} R_n$ is a homotopy equivalence. By marking $B_n$ with a homotopy inverse of this map, we get a marked graph in which $c_1,\ldots,c_k$ are represented by pairwise disjoint simple circuits, proving the ``only if'' direction of the proposition. 


For the ``if'' direction, suppose that $c_1,\ldots,c_k$ are represented in a marked graph $G$ by pairwise disjoint simple circuits $\sigma_1,\ldots,\sigma_k$. We first reduce to the case that each $\sigma_i$ has frontier in $G$ consisting of a single point. If not then for each $\sigma_i$ we choose a connected subgraph $\alpha_i \subset \sigma_i$ --- an arc or point --- that contains each point of the frontier of $\sigma_i$. The union $A = \union_i \alpha_i$ is a subforest of $G$, collapse of which produces a marked graph $G' = G/A$ in which $c_1,\ldots,c_k$ are represented by the pairwise disjoint simple circuits $\sigma'_i = \sigma_i / \alpha_i$ each of which has frontier consisting of a single point. By further collapsing a maximal tree in $G'$ which intersects each $\sigma'_i$ exactly at its one frontier point, we obtain a rose in which $c_1,\ldots,c_k$ are represented by $k$ distinct petals, and so $(*)$ holds.
\end{proof}

\section{Outer space and its spine}
\label{SectionOuterSpaceAndSpine}
\vspace{-.1in}
\hfill{\emph{\ldots the famousest of hobbits, and that's saying a lot.}}

\smallskip

\hfill  --- J.\ R.\ R.\ Tolkien, \emph{The Two Towers}

\bigskip

Our topological strategy for attacking the various problems of Nielsen and Whitehead will go like this. Consider Whitehead's problem which asks, given conjugacy classes $\{c_1,\ldots,c_K\}$ in $F_n$, whether they are represented by a partial free basis. We currently have rather weak necessary conditions (involving homology). We seek stronger necessary conditions. 

Consider for example the problem of whether 
$$w=abcba^{-1}bcb^{-1}acbac^{-1} \in F\<a,b,c\>=F_3
$$ 
is a free basis element. \emph{Assuming} that $w$ is indeed a free basis element, Proposition~\ref{PropWhiteheadConjugacyTopological} gives us a marked graph $\rho \from R_3 \to G$ such that the conjugacy class of $w$ is very nicely represented by a simple circuit in~$G$. Perhaps we can use the sheer existence of $G$ to extract more information about the word $w$ itself, particularly about the representation of $w$ as a non-simple circuit in the base rose $R_3$. In our attempt to do that, starting with a homotopy inverse $\bar \rho \from G \to R_3$ we shall move from $G$ to $R_n$ along a path of marked graphs, keeping track of as much information as we can.

But is there any actual meaning to ``moving from one marked graph to another''? Is this more than just a metaphor? Is there really some kind of mathematical path between two marked graphs? Is there some topological space in which marked graphs are represented as points or subsets or something, and in which a ``path of marked graphs'' is represented as a path in the ordinary topological sense? 

In their extraordinary paper \cite{CullerVogtmann:moduli}, Culler and Vogtmann proposed affirmative answers to these questions, by introducing what is now known as the \emph{outer space} of~$F_n$, denoted~$\X_n$. Their idea was to consider simple geometric structures on marked graphs, namely length structures in which each edge is assigned a positive length. Outer space is a topological space that is constructed so that its points represent length structures on marked graphs: one moves through outer space by letting those geometric structures vary in some continuous fashion \emph{and} by letting the underlying marked graph itself vary.

\subsection{Overview: Gluing ideal simplices to form outer space} 
\label{SectionOuterSpaceCells}

The Culler--Vogtmann outer space of $F_n$, denoted $\X_n$, is glued together out of cells, which we refer to as \emph{outer space cells}\index{cells!outer space}\index{outer space!cells}, one such cell for each equivalence class of marked graphs (see Section~\ref{SectionMarkedGraphEquiv} to review the definition of equivalence). The cell that corresponds to the equivalence class of a marked graph $G$ is a parameterization of a certain kind of geometry on $G$, namely a \emph{length structure} which assigns a real valued length to each natural edge of $G$.  It is convenient to require edge lengths to be normalized so that they sum up to~$1$, and in defining outer space we shall use only normalized length structures.
The tuple of edge lengths, one such tuple for each normalized length structure, serves as a parameter for the outer space cell associated to~$G$, which we shall denote $\Delta(G)$.

While the above discussion may suggest that $\Delta(G)$ is just a simplex, in which the tuple of barycentric coordinates of the simplex is identified with the edge length tuple, there is a question to consider:
\begin{itemize}
\item What is the geometric significance of an edge of length zero?
\end{itemize}
In order to move around in the outer space $\X_n$ in a useful manner, one must not only vary the edge lengths of a marked graph, one must also vary the topology of the marked graph itself --- that is, one must allow the equivalence class of the marked graph to vary. This is accomplished by allowing edge lengths to equal zero: assigning length~$0$ to a collection of edges of $G$ corresponds to changing the equivalence class of $G$ by collapsing each of those edges to a single point. But there is a danger to avoid: if one collapses too many edges at once, namely a union of edges that contains a circuit of $G$, collapsing those edges will reduce the rank of~$G$ and hence will change the homotopy type of $G$ (see Exercise~\ref{ExerciseCollapse}~below).

The outer space cell $\Delta(G)$ of a marked graph $G$ is therefore not a simplex in the ordinary sense: in order to prohibit collapsing circuits of $G$, one must strip away certain faces from the ordinary simplex, leaving what we call an \emph{ideal simplex}. Outer space itself is therefore not a simplicial complex, it is instead an  \emph{ideal simplicial complex}. 


\subsection{The ideal simplex of a marked graph}
\label{SectionIdealSimplexOfGraph}

We start with the abstract definition of an \emph{ideal simplex} and its faces. Then we define a natural way to associate to each marked graph $G$ an ideal simplex $\Delta(G)$ which parameterizes length structures on $G$.

\paragraph{Ideal simplices in the abstract.} Consider a finite set $E$. The \emph{orthant} on $E$ is the product space $[0,\infty)^E$ whose points are the functions $\ell \from E \to [0,\infty)$. The \emph{simplex on $E$} is defined to be the subspace
$$\Delta(E) = \{\ell \in [0,\infty)^E \,\suchthat\, \sum_{e \in E} \ell(e) = 1\}
$$
The individual projection functions $\pi_e \from \Delta(E) \mapsto [0,\infty)$, defined for each $e \in E$ by $\pi_e(\ell)=\ell(e)$, are called the \emph{barycentric coordinates} of $\Delta(E)$. In particular the \emph{barycenter} of $\Delta(E)$ is the point whose coordinates are constant, all equal to $\frac{1}{\abs{E}}$. 

Faces of a simplex $\Delta(E)$ are defined by allowing only certain barycentric coordinates to be nonzero. More precisely, associated to each nonempty subset $F \subset E$ there is a face 
$$\Delta(F \subset E) = \{\ell \in \Delta(E) \suchthat \ell(e)=0 \quad\text{for all $e \in E-F$}\},
$$
Note that $\Delta(F \subset E) \subset \Delta(E)$ is a proper face if and only if $F \subset E$ is a proper subset.  

A \emph{subcomplex} of $\Delta(E)$ is a union of faces of~$\Delta(E)$. 

An \emph{ideal simplex on $E$} is the complement of a proper subcomplex $\L \subsetneq \Delta(E)$, denoted
$$\Delta(E;\L) = \Delta(E) - \L
$$
Given an ideal simplex $\Delta(E;\L)$ on $E$ and a face $\Delta(F \subset E)$ of $\Delta(E)$, consider the following intersection:
$$\Delta(F \subset E;\L) = \Delta(F \subset E) \intersect \Delta(E;\L) = \Delta(F \subset E) - \L
$$
Note that $\Delta(F \subset E;\L) \ne \emptyset$ if and only if $\Delta(F \subset E) \not\subset \L$, and if this is so then we say that $\Delta(F \subset E;\L)$ is a \emph{face} of the ideal simplex $\Delta(E;\L)$. 

\paragraph{The ideal simplex, or outer space cell, of a marked graph.} Consider now a marked graph $G$ equipped with its natural cell structure, let $\EG$ denote its set of natural edges, and consider the simplex $\Delta(\EG)$. We think of each $\ell \in \Delta(\EG)$ as a \emph{length structure} on $G$, assigning a non-negative length to each edge so that the total length is normalized to equal~$1$.

For each natural subgraph $H \subset G$, denote \hbox{$\EH = \{e \in \EG \suchthat e \subset H\}$,} corresponding to which there is a face $\Delta(\EH \subset \EG)$ consisting of all $\ell \in \Delta(G)$ such that $\ell(e) = 0$ for each $e \in \EG - \EH$. If the complementary subgraph $K = G \setminus H$ contains no circuit, equivalently if $K$ is a forest, then we say that the natural subgraph $H$ and its corresponding face $\Delta\bigl(\EH \subset \EG\bigr)$ are \emph{concrete}\index{subgraph!concrete}\index{face!concrete}, otherwise they are \indexemph{nonconcrete}. The union of nonconcrete faces of $\Delta(G)$ forms a subcomplex $\LG \subset \Delta(G)$. The \emph{outer space cell}\index{cells!outer space}\index{outer space!cells} of $G$ is defined to be the ideal simplex
$$\Delta(G) = \Delta\bigl(\EG;\LG\bigr) = \Delta(\EG) - \LG
$$
The faces of the ideal simplex $\Delta(G)$ are indexed by the concrete natural subgraphs $H \subset G$, as follows:
$$\Delta(H \subset G) = \Delta(\EH \subset \EG;\LG) 
$$

For example, referring to Figure~\ref{FigureRankTwoCoreGraphs}, the ideal simplex of a rank~$n$ rose is obtained from an $n-1$ simplex by stripping away every face. The ideal simplex of a rank~$2$ theta graph or a rank~$3$ cat's eye graph is obtained from a simplex by stripping away all faces of codimension~$\ge 2$. 

\paragraph{Exercises for Section \ref{SectionIdealSimplexOfGraph}}

\begin{exercise}
\label{ExerciseIdealSimplexProps}
For each of the following properties of $\Delta(G)$, describe a topological or graph theoretic property which characterizes those marked graphs $G$ such that $\Delta(G)$ has the stated property:
\begin{enumerate}
\item\label{ExerciseIdealSimplexPropsNoncompact}
$\Delta(G)$ is compact.
\item $\Delta(G)$ is homeomorphic to an open ball, equivalently $\Delta(G)$ has no faces.
\item $\Delta(G)$ is obtained from $\Delta(\EG)$ by removing all faces of codimension~$\ge 2$.
\item $\Delta(G)$ has a unique face.
\end{enumerate}
\end{exercise}

\subsection{Face maps}
\label{SectionFaceMaps}

As in Section~\ref{SectionIdealSimplexOfGraph}, we again start with an abstract description of face maps, and then we apply that to describe the face maps that occur amongst the ideal simplices of marked graphs.

In what follows, an injection $f : X \to Z$ will often be written as $f : X \leftrightarrow Y \subset Z$ where $Y = \image(f)$, thus factoring $f$ into a bijection $X \leftrightarrow Y$ composed with an inclusion $Y \hookrightarrow Z$.

\paragraph{Face maps in the abstract.} For any injection of finite sets $s \from E' \leftrightarrow F \subset E$ there is an induced \emph{face map} from the simplex $\Delta(E')$ to the simplex $\Delta(E)$, an embedding denoted
$$(*) \qquad r \from \Delta(E') \leftrightarrow \Delta(F \subset E) \subset \Delta(E)
$$
that is defined for all $\ell \in \Delta(E')$ and $e \in E$ by the formula
$$r(\ell)(e) = \begin{cases}
\ell(s^\inv(e)) & \quad\text{if $e \in F$} \\
0                    &\quad\text{if $e \in E-F$}
\end{cases}
$$
In the special case that $s \from E' \to E$ is a bijection, the induced face map $r \from \Delta(E') \leftrightarrow \Delta(E)$ is a homeomorphism called a \emph{simplex isomorphism}. 

Given a face map denoted as above, for any proper subcomplex $\L \subset \Delta(E)$ that does not contain $\Delta(F \subset E)$ it follows that $\L' = r^\inv(\L) \subset \Delta(E')$ is a proper subcomplex, and by restriction of $r$ we obtain an induced face map of ideal simplices denoted 
$$r \from \Delta(E;\L') \leftrightarrow \Delta(F \subset E;\L) \subset \Delta(E;\L)
$$

\paragraph{Forest collapses between marked graphs.} To prepare for defining face maps between ideal simplices of marked graphs, we consider now the operation of a ``forest collapse''. Given two marked graphs $G,G'$, a map $q \from G \to G'$ is a \emph{forest collapse}\index{forest collapse} if the following conditions hold: 
\begin{enumerate}
\item $q$ is a homotopy equivalence;
\item $q$ preserves marking with respect to the (implicitly) given markings $\rho \from R_n \to G$ and $\rho' \from R_n \to G'$, meaning that $\rho'$ and $\rho \circ q \from R_n \to G'$ are homotopic;
\item there exists a concrete natural subgraph $H \subset G$ with complementary forest $K = G \setminus H$ such that $q$ is a quotient map which collapses to a point each component of the forest $K$, meaning that for each $x \in G'$ the pre-image $q^\inv(x)$ is either a single point of $G - K$ or a component of $K$. 
\end{enumerate}
We shall sometimes incorporate $H$ and/or $K$ into the notation for the collapse map $q \from G \to G'$ by writing $q_H \from G \to G'$, or $G \xrightarrow{[K]} G'$, or $q_H \from G \xrightarrow{[K]} G'$.

We define a relation amongst marked graphs, denoted $G \collapsesto G'$ and pronounced ``$G$~collapses to $G'$'', which is defined by the existence of a forest collapse $q \from G \to G'$. 

In order to know that various constructions are well-defined independent of the choice of a forest collapse --- most immediately, the face maps we will use to define outer space --- we will need the following result:


\begin{theorem}
\label{ThmCollapseUnique}
For any marked graphs $G,G'$, a forest collapse $q \from G \to G'$ is unique up to pre-composition by a homeomorphism of $G$ isotopic to the identity; also, the concrete natural subgraph $H \subset G$ and the forest $K=G \setminus H$ which witness that $q$ is a forest collapse are unique.
\end{theorem}

We put off the proof of this theorem while we continue our study of face maps.

\paragraph{The face map associated to a forest collapse.} Associated any two marked graphs $G$, $G'$ such that $G \collapsesto G'$ we shall define a \emph{face map of ideal simplices} from $\Delta(G') = \Delta(\EG';\LG')$ to $\Delta(G) = \Delta(\EG;\LG)$ denoted
$$r_{G',G} \from \Delta(G') \to \Delta(G)
$$
To define this map, first choose a forest collapse $q_H \from G \to G'$. Consider the induced bijection $\EH \leftrightarrow \EG'$, and note that this bijection is well-defined by Theorem~\ref{ThmCollapseUnique}, depending only on $G$ and $G'$, independent of the choice of $q_H$. The inverse bijection extends to an injection $s_{G',G} \from \EG' \leftrightarrow \EH \hookrightarrow \EG$ which induces a face map of simplices
$$r_{G',G} \from  \Delta(\EG') \leftrightarrow \Delta(\EH \subset \EG) \hookrightarrow \Delta(\EG)
$$
We claim that:
\begin{itemize}
\item $\LG' = r_{G',G}^\inv(\LG)$, equivalently, the nonconcrete faces of $\Delta(\EG')$ are precisely the pre-images under $r_{G',G}$ of the nonconcrete faces of $\Delta(\EG)$.
\end{itemize}
To verify this claim, consider the face $\Delta(\EH'_1 \subset \EG')$ of $\Delta(\EG')$ associated to a nontrivial natural subgraph $\EH'_1 \subset \EG'$.  First note that $H'_1$ is concrete if and only if $K'_1 = G' \setminus H'_1$ is a forest, if and only if $K_1 = q_H^\inv(K'_1 \union \Vertices G') \subset G$ is a forest, if and only if $H_1 = G \setminus K_1$ is concrete. Furthermore, the face map $r_{G',G}$ clearly takes $\Delta(\EH' \subset \EG')$ homeomorphically to $\Delta(\EH_1 \subset \EG)$. 

It follows immediately that if $\Delta(\EH'_1 \subset \EG')$ is a nonconcrete face then it is indeed the pre-image under $r_{G',G}$ of a nonconcrete face of $\Delta(\EG)$, namely $\Delta(\EH_1 \subset \EG)$. 

Conversely, suppose that $\Delta(\EH'_1 \subset \EG')$ is concrete, hence $\Delta(\EH_1 \subset \EG)$ is concrete, hence $K_1$ is a forest. Consider any face $\Delta(\EH_2 \subset \EG)$ whose $r_{G',G}$ pre-image equals $\Delta(\EH'_1 \subset \EG')$, and let $K_2 = G \setminus \EH_2$. Because $\Delta(\EH_1 \subset \EG)$ is the $r_{G',G}$ homeomorphic image of $\Delta(\EH' \subset \EG')$, it follows that $\Delta(\EH_1 \subset \EG) \subset \Delta(\EH_2 \subset \EG)$, hence $\EH_1 \subset \EH_2$, hence $K_2 \subset K_1$. The graph $K_2$ is therefore a forest, and so the face $\Delta(\EH_2 \subset \EG)$ is nonconcrete. This completes the proof of the claim.

It follows from the claim that $r_{G',G}$ restricts to a face map of ideal simplices
$$r_{G',G} \from \Delta(G') = \Delta(\EG';\LG') \leftrightarrow \Delta(\EH \subset \EG;\LG) \subset \Delta(\EG;\LG) = \Delta(G)
$$
Tracing back through the definition of this face map, we see that $r_{G',G} $ is well-defined, depending only on $G'$ and $G$, and independent of the choice of the forest collapse $q_H \from G \mapsto G'$, once we have completed the following proof:

%
%
%
%

\begin{proof}[Proof of Theorem \ref{ThmCollapseUnique}] After some preliminary setup, we will apply Exercise~\ref{ExerciseMovedEdge}~\pref{ItemMovesEdgeArcsNotIdentity} to complete the proof.

\newcommand\NV{N\!\Vertices}
\newcommand\VK{\Vertices\!K}
\newcommand\VG{\Vertices\!G}
\newcommand\NK{N\!K}
\newcommand\NVG{\NV\!G}

Suppose that for $i=1,2$ we have concrete, natural subgraphs $H_i \subset G$ with graph complements $K_i = G \setminus H_i$ and with collapse maps 
$$q_i = q_{H_i} \from G \xrightarrow{[K_i]} G'
$$
Since $q_1,q_2$ both preserve marking, it follows that $q_1,q_2$ are homotopic to each other. We shall produce a particular homotopy inverse $\bar q_i \from G' \to G$ of each $q_i$, thus obtaining by composition a map $\bar q_2 \circ q_1 \from G \to G$ which is homotopic to the identity; this is the map to which we will apply Exercise~\ref{ExerciseMovedEdge}~\pref{ItemMovesEdgeArcsNotIdentity}.

Let $e_1,\ldots,e_I$ be the edges of $G'$, and let $e_{1,j},\ldots,e_{I,j} \subset G$ be the edges of $H_j$ indexed so that $q_j$ maps $e_{i,j}$ to $e_i$. Choose an orientation of each $e_i$ and pull back by $q_j$ to get an orientation of $e_{i,j}$ ($i=1,\ldots,I$; $j=1,2$). To prove existence of a homeomorphism $h \from G \to G$ isotopic to the identity such that $q_2=q_1 \circ h$, it suffices to prove for each $i=1,\ldots,I$ that the edges $e_{i,1}$ and $e_{i,2}$ are equal and that their $q_1$ and $q_2$ pullback orientations are equal; this is the conclusion we will derive from applying Exercise~\ref{ExerciseMovedEdge}~\pref{ItemMovesEdgeArcsNotIdentity}.

Let $\NV \subset G'$ be a closed regular neighborhood of the natural vertex set $\VG'$, and hence $\NV$ is a forest. Let $\NK_j = q_j^\inv(\NV) \subset G$ which is a closed regular neighborhood of the forest $K_j$ $(j=1,2)$, and hence $\NK_j$ is a forest. Let $\check e_i = \closure(e_i - \NV)$, and let $\check e_{i,j} = \closure(e_{i,j} - \NK_j)$, all of which are arcs in the interiors of the corresponding edges. Note that $q_j$ restricts to a homeomorphism with inverse homeomorphism
$$(*) \qquad \check e_{i,j} \xrightarrow{q_{i,j}} \check e_i \xrightarrow{\bar q_{i,j}} \check e_{i,j}
$$ 

Let $W = \union_{i=1}^I \check e_i$ which equals the frontier of $\NV$, and let $W_j = \union_{i=1}^I \bdy \check e_{i,j} \subset G$ which equals to the frontier of $\NK_j$. Consider the restricted map of pairs $q_j \from (\NK_j,W_j) \to (\NV,W)$. This map restricts to a bijection $W_j \to W$, and the inverse bijection $W \mapsto W_j$ clearly extends to a continuous map of pairs $\bar q_j \from (\NV,W) \to (\NK_j,W_j)$. By applying Exercise~\ref{ExerciseMovedEdge}~\pref{ItemMovesEdgeArcsNotIdentity} to each component of $\NV$ it follows that these maps of pairs are homotopy inverses in the category of topological pairs:
$$(**) \qquad (\NK_j,W_j) \xrightarrow{q_j} (\NV,W) \xrightarrow{\bar q_j} (\NK_j,W_j)
$$
By extending $\bar q_j$ to each $\check e_i$ using the map $\bar q_{i,j}$, we obtain a homotopy inverse pair of homotopy equivalences 
$$G \xrightarrow{q_j} G' \xrightarrow{\bar q_j} G
$$
which have the restrictions $(*)$ and $(**)$ above. By composition we obtain a map $\bar q_2 \circ q_1 \from G \to G$ which is homotopic to the identity, which restricts to $\check e_{i,1} \mapsto \check e_{i,2}$ preserving orientations $(i=1,\ldots,I)$, and which restricts to a map $\NK_1 \mapsto \NK_2$. Applying Exercise~\ref{ExerciseMovedEdge}~\pref{ItemMovesEdgeArcsNotIdentity} it follows that $e_{i,1}=e_{i,2}$. Using that the restricted map of pairs $f \from (G,G \setminus e_{i,1}) \to (G,G \setminus e_{i,1})$ is homotopic to the identity: if $e_{i,1}$ separates then $f$ preserves each component; whereas if $e_{i,1}$ does not separate then $f$ preserves (up to homotopy) some oriented circle subgraph passing through $e_{i,1}$; in either case $f$ preserves orientation on $e_{i,1}$. \end{proof}

\paragraph{Exercises for Section \ref{SectionFaceMaps}}


\begin{exercise}
\label{ExerciseCollapse}
Consider any graph $G$ and subgraph $H \subset G$, and let $q_H \from G \to G'$ denote the quotient map obtained from $G$ obtained by collapsing each component of $K = G \setminus H$ to a point. Prove the following:
\begin{enumerate}
\item\label{ItemCollapseGraph}
 $G'$ is a graph with vertex set $\Vertices G' = q_H\bigl(\Vertices G \, \union \ (G \setminus H)\bigr)$, and $q_H$ induces a bijection between the edges of $H$ and the edges of $G'$. 
\item\label{ItemCollapseCoreGraph}
If $G$ is a core graph and $\Vertices G$ is its natural vertex set, and if $H$ is a nonempty natural subgraph of $G$, then $G'$ is a core graph and $\Vertices G'$ is its natural vertex set. Furthermore, $q_H \from G \to G'$ is a homotopy equivalence if and only if $H$ is a concrete subgraph of~$G$ (if and only if $K$ is a forest).
\end{enumerate}
\end{exercise}

\begin{exercise}
\label{ExerciseForestCollapseComp}
Prove that the relation $\collapsesto$ on marked graphs is transitive. More precisely, prove that for any marked graphs $G,G',G''$, if $q \from G \mapsto G'$ and $q' \from G' \mapsto G''$ are forest collapses then the composition $q' \circ q \from G \mapsto G''$ is also a forest collapse. 
\end{exercise}

\subsection{Outer space defined as an ideal simplicial complex} 
\label{SectionIdealSimplicial}

Ideal simplicial complexes can still be spotted in the mathematical countryside, although they are becoming hard to find. They occur naturally in the study of complete, noncompact hyperbolic manifolds $M$ of finite volume. In the case of dimension~2 one can always write $M = \overline M - P$ where $\overline M$ is a closed 2-manifold and $P \subset \overline M$ is a finite set of ``punctures'', and hence the set $P$ is in one-to-one correspondence with the cusps of~$M$. One can then choose a triangulation of $\overline M$ with vertex set $P$, and straighten each ``edge'' of this triangulation so that its interior is a locally geodesically embedded copy of the real line $\mathbb R$ in $M$. Lifting these ``edges'' to the universal covering space of $M$, which is identified with the hyperbolic plane $\hyp^2$, one obtains an ideal simplicial structure $\hyp^2$ which is invariant under the deck action by $\pi_1 M$. For a general discussion of the 2-dimensional case see \cite{BowditchEpstein:Triangulations}. In the case $n = 3$, similar constructions play a useful role in the theory of hyperbolic Dehn surgery of W.\ Thurston \cite{Thurston:GeomTop}. 

The famousest of the ideal simplicial complexes is the Farey complex, an ideal simplicial structure on the upper half plane model of 2-dimensional hyperbolic geometry $\hyp^2 = \{z=x+iy \suchthat y > 0\}$, which is invariant under the fractional linear action of the group $\SL(2,\Z)$. As we will see in Section~\ref{SectionRank2OuterSpace}, the Farey complex turns out to be (most of) the rank~2 outer space $\X_2$. The Farey complex has no ideal simplices of dimension~$0$, although the rational points $\mathbf Q \union \{\infty\}$ in the circle at infinity $S^1_\infty = \reals \union\{\infty\}$ are used as ``ideal vertices'' in defining the ideal 1-simplices and 2-simplices of the Farey complex. The ideal 1-simplices are those bi-infinite geodesics in $\hyp^2$ having a pair of ``ideal endpoints'' $\{\xi,\eta\} \subset \mathbf Q \union \{\infty\}$ of the form $\xi = \frac{a}{b}$, $\eta = \frac{c}{d}$ where $a,b,c,d \in \Z$ satisfy $ad-bc=\pm 1$. The ideal 2-simplices of the Farey complex are those ideal triangles in $\hyp^2$ having a triple of ``ideal vertices'' of the form $\{\xi,\eta,\zeta\} \in \mathbf Q \union \{\infty\}$ where $\{\xi = \frac{a}{b},\eta = \frac{c}{d}\}$ is an ideal endpoint pair of some ideal 1-simplex, and where the third ideal vertex $\zeta$ is obtained by ``Farey addition'', $\zeta = \frac{a+c}{b+d}$.

\paragraph{Ideal simplicial complexes in the abstract.} An \emph{ideal simplicial complex} is defined by gluing together a collection of ideal simplices, using ideal face maps as glue. More precisely, one is given a set $I$ with a partial order $i \le j$ called the \emph{face order}. For each $i \in I$ one is also given a finite set $E_i$ and a subcomplex $\L_i \subset \Delta(E_i)$ with corresponding ideal simplex $\Delta(E_i;\L_i)$. Also, for each $i < j \in I$ one is given a \emph{face ID function} consisting of a proper injection $s_{i,j} \from E_i \leftrightarrow E_{i,j} \hookrightarrow E_j$, associated to which is a face map $r_{i,j} \from \Delta(E_i) \leftrightarrow \Delta(E_{i,j} \subset E_j) \hookrightarrow \Delta(E_j)$. The following \emph{compatibility axioms} are required:
\begin{enumerate}
\item\label{ItemIdealFace}
For each $i \le j \in J$, \, $\Delta(E_{i,j} \subset E_j;\L_j)$ is a face of $\Delta(E_j;\L_j)$ (i.e.\ it is not empty), and $\L_i = (r_{i,j})^\inv(\L_j)$. By restriction we obtain an ideal face map also denoted
$$r_{i,j} \from \Delta(E_i;\L_i) \leftrightarrow \Delta(E_{i,j} \subset E_j;\L_j) \subset \Delta(E_j;\L_j)
$$
\item\label{ItemPosetIso}
For each $j \in I$, denoting $I_{<j} = \{i \in I \suchthat i < j\}$, the function $i \mapsto \Delta(E_{i,j} \subset E_j;\L_j)$ from the set $I_{<j}$ to the set of faces of $\Delta(E_j;\L_j)$ is a bijection. 
\item\label{ItemPosetCompo}
For each $i<j<k \in I$ we have a commutative diagram of face ID functions
$$\xymatrix{
E_i \ar[r]_{s_{i,j}} \ar@/^1pc/[rr]^{s_{i,k}} & E_j \ar[r]_{s_{j,k}} & E_k
}$$
and hence an induced commutative diagram of ideal face maps
$$\xymatrix{
\Delta(E_i;\L_i) \ar[r]_{r_{i,j}} \ar@/^1pc/[rr]^{r_{i,k}} & \Delta(E_j;\L_j) \ar[r]_{r_{j,k}} & \Delta(E_k;\L_k)
}$$
\end{enumerate}

Given gluing data satisfying the compatibility axioms as described above, the associated \emph{ideal simplicial complex} $\X$ is the quotient space of the disjoint union of ideal simplices $\coprod_{i \in I} \Delta(E_i;\L_i)$ using each face map $r_{i,j} \from \Delta(E_i;\L_i) \to \Delta(E_j;\L_j)$ $(i < j \in I)$ to identify each $x \in \Delta(E_i;\L_i)$ with $r_{i,j}(x) \in \Delta(E_j;\L_j)$; thus the equivalence relation on $\coprod_{i \in I} \Delta(E_i;\L_i)$ which defines $\X$ is generated by the relation $x \sim r_{i,j}(x)$ for all $i<j \in I$ and all $x \in \Delta(E_i;\L_i)$.

In the following proposition we list various properties that follow from the compatibility axioms and the quotient topology; details of verification will be left to the exercises. Within this proposition we also formulate terminology that will be used later in the context of outer space:


\begin{proposition} 
\label{PropIdealSimplComplProps}
\quad
\begin{enumerate}
\item\label{ItemSimplicesEmbedded}
The quotient map $\coprod_{i \in I} \Delta(E_i;\L_i) \mapsto \X$ restricts to an embedding of the ideal simplex $\Delta(E_i;\L_i) \hookrightarrow \X$ for each $i \in I$. We identify $\Delta(E_i;\L_i)$ with its image in $\X$ under this embedding, and we say that $\Delta(E_i;\L_i)$ is an \emph{ideal simplex} or just a \emph{cell} of~$\X$. Furthermore, $\Delta(E_i;\L_i)$ is a closed subset of~$\X$
\item\label{ItemFacesAreFaces}
For each $i \le j \in I$, the cell $\Delta(E_i;\L_i)$ is identified in $\X$ with the face $\Delta(E_{i,j} \subset E_i;\L_i)$ of $\Delta(E_j;\L_j)$. We therefore say that $\Delta(E_i;\L_i)$ is a \emph{face} of $\Delta(E_j;\L_j)$, and if $i < j$ then we say that it is a \emph{proper} face.
\item Every face of every cell of~$\X$ is a cell of $\X$.
\item\label{ItemCellIntersection}
The intersection of any two cells of $\X$, if not empty, is the largest common face of those two cells. As a special case, one cell is contained in another cell if and only if the first is a face of the second. \qed
\end{enumerate}
\end{proposition}

\paragraph{An indexed family of marked graphs and face maps.} To define outer space $\X_n$ as an ideal simplicial complex, we start by choosing an indexed collection of marked graphs $\{G_i \suchthat i \in I\}$ containing exactly one marked graph in every equivalence class. The ideal simplices of $\X_n$ are defined by $\Delta(G_i) = \Delta(\EG_i;\LG_i)$ for each $i \in I$ (see Section~\ref{SectionIdealSimplexOfGraph}), where $\LG_i$ is the subcomplex of nonconcrete faces of $\Delta(\EG_i)$. The face order on $I$ is defined so that $i < j$ if and only if there exists a forest collapse 
$$q_{i,j} = q_{H_{i,j}} \from G_j \xrightarrow{[K_{i,j}]} G_i
$$
Applying Theorem~\ref{ThmCollapseUnique}, the map $q_{i,j}$ is unique up to precomposition by isotopy of $G_j$, and the natural subgraphs $H_{i,j}$ and $K_{i,j}$ are unique. As shown in Section~\ref{SectionFaceMaps}, we thus obtain a well-defined, doubly indexed family of induced face ID functions and face maps 
$$s_{i,j} \from \EG_i \leftrightarrow \EH_{i,j} \subset \EG_j, \qquad r_{i,j}  \from \Delta(\EG_j) \to \Delta(\EG_i)
$$
It was also shown in Section~\ref{SectionFaceMaps} that $\L_i = r_{i,j}^\inv(\L_j)$ and we obtain a well-defined family of restricted ideal face maps
$$r_{i,j} \from \Delta(G_i) = \Delta(\EG_i;\L_i) \to \Delta(\EG_j;\L_j) = \Delta(G_j)
$$
which thus verifies Compatibility Axiom~\pref{ItemIdealFace}.
%

We must verify Compatibility Axiom~\pref{ItemPosetIso}, and so consider $j \in I$ and consider the map which associates to each $i \in I_{<j}$ a face of the ideal simplex $\Delta(G_j)$. Consider an arbitrary face of $\Delta(G_j)=\Delta(\EG_j;\L_j)$, having the form $\Delta(\EH \subset \EG_j;\L_j)$ for some concrete subgraph $H \subset G_j$ with complementary forest $K = G_j \setminus H$, and consider the quotient map $q \from G \xrightarrow{[K]} G'$ that collapses each component of $K$ to a point. 
Applying Exercise~\ref{ExerciseCollapse}, it follows that $G'$ is a core graph and $q$ is a homotopy equivalence. Applying Exercise~\ref{ExerciseInducedMarking}, $G'$ has a unique marking with respect to which $q$ preserves marking, which therefore gives $G'$ the structure of a marked graph and $q$ the structure of a forest collapse. There exists $i \in I$ such that $G'$ and $G_i$ are equivalent marked graphs, and so there exists a homeomorphism $h \from G' \to G_i$ that preserves marking. It follows that $h \circ q \from G_j \xrightarrow{[K]} G_i$ is a forest collapse, proving that $i \in I_{<j}$ and that $i$ is taken to the face $\Delta(\EH \subset \EG_j;\L_j)$ by the map defined in Compatibility Axiom~\pref{ItemPosetIso}. That map is hence surjective. To prove injectivity, given $j \in I$ and $i,i' \in I_{<j}$, if the faces $r_{i,j}(\Delta(G_i))=r_{i',j}(\Delta(G_{i'})$ of $\Delta(G_j)$ are equal to each other, each equal to $\Delta(H \subset G_j;\L_j)$ for some concrete subgraph $H \subset G_j$ with complementary forest $K = G_j \setminus H$, then there are collapse maps $q \from G_j \xrightarrow{[K]} G_i$ and $q' \from G_j \xrightarrow{[K]} G_{i'}$. Since these maps collapse the exact same subforest $K$, we obtain an induced homeomorphism $h \from G_i \to G_{i'}$ with $q' = h \circ q$. It follows that $h$ preserves marking, so $G_i$ and $G_{i'}$ are equivalent marked graphs, and so $i=i'$. 

Verification of Compatibility Axiom~\pref{ItemPosetCompo} is left to the reader in Exercise~\ref{ExerciseAxiomPosetCompo}.

\paragraph{The definition of outer space and its cells.} Outer space $\X_n$ may now be defined as the ideal simplicial complex obtained as the quotient of the disjoint union of ideal simplices $\Delta(G_i)$ for $i \in I$, using as glue the face maps $r_{i,j} \from \Delta(G_i) \to \Delta(G_j)$ for $i < j \in I$.

Earlier in Section~\ref{SectionIdealSimplexOfGraph} we defined outer space cells 
\index{cells!outer space}\index{outer space!cells} in the abstract, one such cell for each marked graph $G$, namely the ideal simple $\Delta(G)$. We can now embed each such cell naturally into $\X_n$, justifying the terminology ``outer space cell''. To do this, let $i \in I$ be the unique index such that $G$ is equivalent to $G_i$, and let $h \from G \to G_i$ be a homeomorphism that preserves marking. We may regard $h$ as a forest collapse with respect to the empty subforest of $G$, and hence we may apply Theorem~\ref{ThmCollapseUnique} to conclude that $h$ is unique up to isotopy. The entire discussion following Theorem~\ref{ThmCollapseUnique} thus applies, implying that $h$ induces a well-defined ideal face map $r \from \Delta(G) \to \Delta(G_i)$, depending only on~$G$, which in this situation is actually an ideal simplex isomorphism. By postcomposing $r$ with the natural embedding $\Delta(G_i) \hookrightarrow \X_n$ given by Proposition~\ref{PropIdealSimplComplProps}~\pref{ItemSimplicesEmbedded}, we therefore obtain a natural embedding $\Delta(G) \hookrightarrow \X_n$.

\paragraph{Exercises for Section \ref{SectionIdealSimplicial}}


\begin{exercise}\label{ExerciseIdealSimplicialFacts}
Prove Proposition~\ref{PropIdealSimplComplProps}, using the following hints. First show that if $i,j \in I$ have a lower bound under the face order then they have a unique greatest lower bound. Then show that if $x \in \Delta(E_i;\L_i)$, and if $y \in \Delta(E_j;\L_j)$, and if $x$ is identified with $y$ in $X$, then $i,j$ do indeed have a greatest lower bound $k \in I$ and there exists $z \in \Delta(E_k;\L_k)$ such that $r_{k,i}(z)=x$ and $r_{k,j}(z)=y$. For proving the last sentence of \pref{ItemSimplicesEmbedded}, mimic the proof for CW complexes.
\end{exercise}


%
%
%

\begin{exercise}
\label{ExerciseAxiomPosetCompo}
Use Theorem~\ref{ThmCollapseUnique} and Exercise~\ref{ExerciseForestCollapseComp} to verify Compatibility Axiom~\pref{ItemPosetCompo} in the construction of the outer space~$\X_n$.
\end{exercise}
%
%

\begin{exercise} In ranks $2$, $3$ and $4$, how many different homeomorphism types are there of marked graphs $G$ such that $\Delta(G)$ is maximal with respect to the face order?
\end{exercise}

\begin{exercise}\label{ExerciseMaximalFaces}
Given a marked graph $G$ with corresponding outer space cell $\Delta(G) \subset \X_n$, what topological or graph theoretic property of $G$ characterizes the property that $\Delta(G)$ is maximal with respect to face order?
\end{exercise}

\begin{exercise}
\label{ExerciseFinitelyManyExpansions}
Prove that for any outer space cell $\Delta(G)$ there are only finitely many outer space cells that contain $\Delta(G)$. More specifically, if $G$ has $K$ natural vertices $v_1,\ldots,v_K$ with respective valences $t_1 \le t_2 \le \cdots \le t_K$, then the number of outer space cells that contain $\Delta(G)$ (including $\Delta(G)$ itself) is equal to
$$\prod_{k=1}^K N(t_k)
$$
where $N(t)$ is the number of isomorphism classes of $t$-labelled trees (see Exercise~\ref{ExerciseTreeCount}). (Hint: To start, choose a pairwise disjoint set of regular neighborhoods $T(v_1),\ldots,T(v_K)$ and label the valence~$1$ vertices of the tree $T(v_k)$ with the integers $1,\ldots,t_k$.)
\end{exercise} 

\begin{exercise}
\label{ExerciseLocallyFinite}
Prove that the decomposition of $\X_n$ into its ideal simplices is locally finite: for each $x \in \X_n$ there exists a neighborhood $U \subset \X_n$ such that $U$ is disjoint from all but finitely many outer space cells. (Hint:  Exercise~\ref{ExerciseFinitelyManyExpansions}.)
\end{exercise}

\subsection{The action of $\Out(F_n)$ on outer space}
\label{SectionOuterSpaceAction}

In this section, we construct an \emph{ideal simplicial action} of $\Out(F_n)$ on the outer space~$\X_n$. We refer the reader to Section~\ref{SectionTopActions} for a quick review of the terminology of group actions.
%
%
We shall denote automorphisms of an ideal simplicial complex as acting from the \emph{right}, using postfix notation, which turns out to be the more natural direction from which $\Out(F_n)$ acts on outer space~$\X_n$. 

Consider an ideal simplicial complex $X$, with gluing data consisting of an indexed set of ideal simplices $\{\Delta(E_i;\L_i)\}_{i \in I}$, a face order on $I$, and for $i<j$ on $I$ an injection $E_i \leftrightarrow E_{i,j} \hookrightarrow E_j$ inducing an ideal face map $\Delta(E_i;\L_i) \leftrightarrow \Delta(E_{i,j} \subset E_j;\L_j) \hookrightarrow \Delta(E_j;\L_j)$,
such that the Compatibility Axioms~\pref{ItemIdealFace}--\pref{ItemPosetCompo} hold. 
%
An \emph{automorphism} of $X$ consists of a homeomorphism $f \from X \to X$ denoted $x \mapsto x \cdot f$ for which there exists a bijection $f \from I \to I$ denoted $i \mapsto i \cdot f$, and for each $i \in I$ there exist bijections $E_i \mapsto E_{i \cdot f}$ inducing a simplex isomorphism $f_i \from \Delta(E_i) \to \Delta(E_{i \cdot f})$ denoted $x \mapsto x \cdot f_i$, such that the following hold: $f \from I \to I$ preserves face order; for each $i \in I$ we have $\L_i \cdot f_i =\L_{i \cdot f}$; and the induced ideal simplex isomorphism $f_i \from \Delta(E_i;\L_i) \mapsto \Delta(E_{i \cdot f};\L_{i \cdot f})$ is equal to a restriction of the homeomorphism $f$. It is clear that the automorphisms of $X$ form a right group action under the operation of composition.

Using the gluing data for $\X_n$ described in Section~\ref{SectionIdealSimplicial}, we shall define for each $\phi \in \Out(F_n)$ an ideal simplicial automorphism of $\X_n$. The definition will proceed in steps. Step~1 defines the index bijection $i \mapsto i \cdot \phi$ on $I$, and for each $i \in I$ a bijection $\EG_i \mapsto \EG_{i \cdot \phi}$ inducing a simplex isomorphism $\phi_i \from \Delta(\EG_i) \mapsto \Delta(\EG_{i \cdot \phi})$. Later steps are concerned with verifying that the maps $\phi_i$ restrict to ideal simplicial isomorphisms $\Delta(G_i) \mapsto \Delta(G_{i\cdot\phi})$ that can be glued up into an ideal simplicial isomorphism of $\X_n$. Along the way we shall also prove the right action equation: 
$$(x \cdot \phi) \cdot \psi = x \cdot (\phi \psi) \quad\text{for all $x \in \X_n$ and all $\phi,\psi \in \Out(F_n)$}
$$

We will use the natural isomorphism $\Out(F_n) \approx \HMCG(R_n)$ from Exercise~\ref{ExerciseOutIsomorphicHMCG}, which associates to each $\phi \in \Out(F_n)$ the (well-defined) homotopy class of a certain homotopy equivalence $f_\phi \from R_n \to R_n$: one chooses $\Phi \in \Aut(F_n)$ representing $\phi$; and then one translates the formula for $\Phi$, expressed as reduced words in the standard free basis elements $s_1,\ldots,s_n \in F_n$ and their inverses, into a formula for $f_\phi$, expressed as concatenations of the standard oriented edges $e_1,\ldots,e_n \subset R_n$ and their inverses. 

\subparagraph{Step 1: Index bijection and simplex isomorphisms.} Consider the indexed set $\{G_i\}_{i \in I}$ of marked graphs used in Section~\ref{SectionIdealSimplicial} for constructing $\X_n$, one marked graph chosen from each equivalence class. For each $i \in I$ let $\rho_i \from R_n \to G_i$ denote the given marking of $G_i$, and so the ordered pair $(G_i,\rho_i)$ is a more formal notation for the marked graph we have been calling $G_i$. For each $\phi \in \Out(F_n)$ consider also the marked graph $(G_i, \rho_i \circ f_\phi)$, having the same underlying graph $G_i$, but with marking given by the composition
$$R_n \xrightarrow{f_\phi} R_n \xrightarrow{\rho_i} G_i
$$
There exists a unique element of $I$, which we define to be $i \cdot \phi \in I$, such that the marked graphs $(G_i,\rho_i \circ f_\phi)$ and $(G_{i \cdot \phi},\rho_{i \cdot \phi})$ are equivalent. Furthermore, knowing that $f_\phi$ is independent up to homotopy of the choice of $\Phi \in \Aut(F_n)$ representing $\phi$, it follows that $i \cdot \phi$ is also independent of this choice. 

Applying Exercise~\ref{ExerciseMarkingChangeUnique}, there exists a homeomorphism $$h_{i,\phi} \from G_i \to G_{i \cdot \phi}
$$
unique up to isotopy relative to the natural vertices, such that $h_{i,\phi}$ preserves the markings $\rho_i \circ f_\phi$ on $G_i$ and $\rho_{i \cdot \phi}$ on $G_{i \cdot \phi}$, meaning that the diagram of maps $D_{i,\phi}$ that is depicted in Figure~\ref{FigureD_phi_i} is homotopy commutative.

The map $h_{i,\phi}$ induces a bijection $\EG_i \mapsto \EG_{i \cdot \phi}$ which in turn induces a simplex isomorphism denoted $\phi_i \from \Delta(\EG_i) \mapsto \Delta(\EG_{i \cdot \phi})$. 
\begin{figure}
$$\xymatrix{
& R_n \ar[dl]^{\rho_i} 
      & R_n 
                 \ar[dr]_{\rho_{i \cdot \phi}} 
                 \ar[l]_{f_\phi} \\
G_i  \ar[rrr]_{h_{i,\phi}} 
      &&& G_{i \cdot \phi}
}$$
\caption{This diagram \, $D_{\phi,i}$ \, ($\phi \in \Out(F_n)$, $i \in I$), \, a homotopy commutative trapezoid of maps, shows that $h_{i,\phi}$ preserves marking, going from the marked graph $(G_i,\rho_i \circ f_\phi)$ to the marked graph $(G_{i \cdot \phi},\rho_{i \cdot \phi})$.}
\label{FigureD_phi_i}
\end{figure}
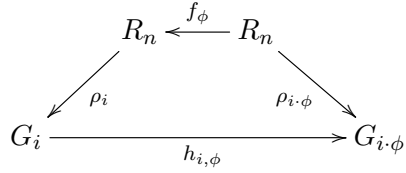

\subparagraph{Step 2: Preserving the concrete.} Using that $h_{i,\phi}$ is a homeomorphism, the simplex isomorphism $\phi_i$ takes concrete faces of  $\Delta(\EG_i)$ to concrete faces of $\Delta(\EG_{i \cdot \phi})$, because for each natural subgraph $H \subset G_{i}$ its graph complement $G_i \setminus H$ contains a circuit of~$G_i$ if and only if the graph complement $G_{i \cdot \phi} \setminus \phi_i(H)$ of its homeomorphic image $\phi_i(H) \subset G_{i \cdot \phi}$ contains a circuit of $G_{i \cdot \phi}$. It follows that $\L_i \cdot \phi_i = \L_{i \cdot \phi}$, and that the map $\phi_i$ restricts to an ideal simplex isomorphism:
$$\phi_i \from \Delta(G_i) = \Delta(\EG_i;\L_i) \to \Delta(\EG_{i \cdot \phi};\L_{i \cdot \phi}) = \Delta(G_{i \cdot \phi})
$$
By taking disjoint unions of the maps $\phi_i$, we therefore have an induced map
$$\phi^* \from \coprod_{i \in I} \Delta(G_i) \mapsto \coprod_{i \in I} \Delta(G_i)
$$
The right action equation $(i \cdot \phi) \cdot \psi = i \cdot (\phi\psi)$ is easily verified. Also, the right action equation $\psi_{i \cdot \phi} \circ \phi_i = (\psi\phi)_i$ depicted in the following commutative diagram of ideal simplex isomorphisms is easily verified:
$$\xymatrix{
\Delta(G_i) \ar[rr]_{\phi_i}\ar@/^2pc/[rrrrr]^{(\phi\psi)_i} && \Delta(G_{i \cdot \phi}) \ar[rr]_{\psi_{i \cdot \phi}} && \Delta(G_{(i \cdot \phi) \cdot \psi}) \ar@{=}[r] & \Delta(G_{i \cdot (\phi\psi)})
}
$$
The maps $\phi^*$ defined for each $\phi \in \Out(F_n)$ as above therefore define a right action of $\Out(F_n)$ on the disjoint union $\coprod_{i \in I} \Delta(G_i)$. 

\subparagraph{Step 3: Preserving the quotient.} We need to prove that for each $\phi \in \Out(F_n)$ the map $\phi^* \from \coprod_{i \in I} \Delta(G_i) \to \coprod_{i \in I} \Delta(G_i)$ is consistent with respect to the quotient map $\coprod_{i \in I} \Delta(G_i) \mapsto \X_n$, therefore inducing an isomorphism of the ideal simplicial complex~$\X_n$. 

To do this, for each $i < j \in I$ and each $\phi \in \Out(F_n)$ we must prove three things:
\begin{description}
\item[Face order is preserved:] $i \cdot \phi < j \cdot \phi$
\item[Face ID functions are preserved:] There is a commutative diagram of face ID functions as on the left of Figure~\ref{FigureFaceIDCommDiagram}.
\item[Ideal face maps are preserved:]
There is a commutative diagram of ideal face maps and ideal simplex isomorphisms as on the right of Figure~\ref{FigureFaceIDCommDiagram}.
\end{description}
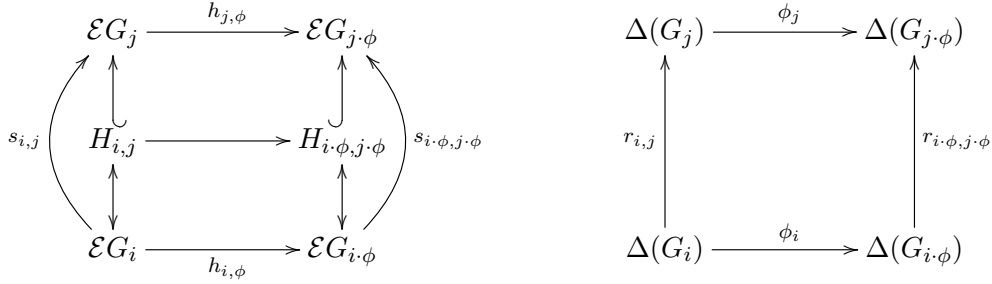
\begin{figure}
$$\xymatrix{
\EG_j  \ar[rr]^{h_{j,\phi}} 
	&& \EG_{j \cdot \phi} 
	&&& \Delta(G_j) \ar[rr]^{\phi_j} 
	&& \Delta(G_{j \cdot \phi}) 
	\\
 H_{i,j} \ar[rr] \ar@{_{(}->}[u] 
        && H_{i \cdot \phi, j \cdot \phi} \ar@{^{(}->}[u] \\
\EG_i \ar[rr]_{h_{i,\phi}}\ar@{<->}[u] \ar@/^2pc/[uu]^{s_{i,j}} 
       && \EG_{i \cdot \phi} \ar@{<->}[u] 
                                      \ar@/_2pc/[uu]_{s_{i \cdot \phi, j \cdot \phi}}
       &&& \Delta(G_i) \ar[rr]^{\phi_i} \ar[uu]^{r_{i,j}} 
       && \Delta(G_{i \cdot \phi}) \ar[uu]_{r_{i \cdot \phi,j \cdot \phi}}
}$$
\caption{Assuming that $i < j \in I$, $\phi \in \Out(F_n)$, and $i \cdot \phi < j \cdot \phi$, in Step 3 we produce two commutative diagrams: on the left, a diagram of bijections (horizontal arrows) and face ID functions (curved vertical arrows); on the right, a diagram of ideal simplex isomorphisms (horizontal arrows) and ideal face maps (vertical arrows).}
\label{FigureFaceIDCommDiagram}
\end{figure}

\begin{figure}
$$\xymatrix{
G_j \ar[dd]_{\hphantom{ = \, h_{i,\phi} \circ q_{i,j} \circ h_{j,\phi}^\inv} q_{i,j}}^{[K_{i,j}]} \ar[rrrrr]^{h_{j,\phi}} 
      &&&&& G_{j \cdot \phi} \ar[dd]_{[K_{i \cdot \phi,j \cdot \phi}]}^{q_{i \cdot \phi,j \cdot \phi} \, = \, h_{i,\phi} \circ q_{i,j} \circ h_{j,\phi}^\inv} \\
&& R_n \ar[ull]_{\rho_j} \ar[dll]^{\rho_i} 
      & R_n \ar[urr]^{\rho_{j \cdot \phi}} 
                 \ar[drr]_{\rho_{i \cdot \phi}} 
                 \ar[l]_{f_\phi} \\
G_i  \ar[rrrrr]_{h_{i,\phi}} 
      &&&&& G_{i \cdot \phi}
}$$
\caption{The implication $i < j \implies i \cdot \phi < j \cdot \phi$ for $i,j \in I$ and $\phi \in \Out(F_n)$, and the associated commutative diagram of collapse maps and ideal simplex isomorphisms.}
\label{Figure_i_less_j}
\end{figure}
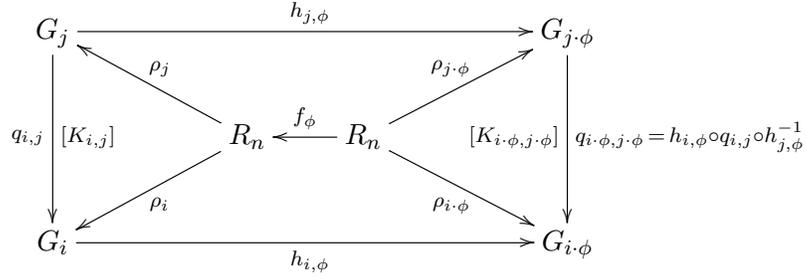

The proof that face order is preserved is summarized in a homotopy commutative diagram of collapse maps and graph homemorphisms depicted in Figure~\ref{Figure_i_less_j}. In that diagram, the lower trapezoid is the homotopy commutative diagram $D_{i,\phi}$ from Figure~\ref{FigureD_phi_i}, and the upper trapezoid is the homotopy commutative diagram $D_{j,\phi}$. Also, associated to the relation $i < j$ is the collapse map $q_{i,j} \from G_i \xrightarrow{[K_{i,j}]} G_j$ on the left side of the diagram, where the natural forest $K_{i,j} \subset G_j$ is the graph complement of the concrete subgraph $H_{i,j} \subset G_j$. Notice that the left triangle is the homotopy commutative diagram which witnesses that $q_{i,j}$ preserves marking, going from the marked graph $(G_j,\rho_j)$ to the marked graph $(G_i,\rho_i)$. The map on the right side of Figure~\ref{Figure_i_less_j} is defined by composing maps by going around the outside of the diagram: the inverse homeomorphism $h_{j,\phi}^\inv$ going backwards along the top arrow, followed by the collapse $q_{i,j}$ going down along the left arrow, followed by the homeomorphism $h_{i,\phi}$ going forward along the bottom arrow. It follows that the map on the right side collapses to a point each component of the forest $h_{j,\phi}(K_{i,j}) \subset G_{j \cdot \phi}$ which is the graph complement of the concrete subgraph $h_{j,\phi}(H_{i,j}) \subset G_{j \cdot \phi}$. Also, using homotopy commutativity of the two trapezoids and the left triangle, it follows that the right triangle is also homotopy commutative. Applying the definition of the face ordering on the set $I$, it follows that $i \cdot \phi < j \cdot \phi$. It also follows, by applying Theorem~\ref{ThmCollapseUnique}, that the map on the right side of Figure~\ref{Figure_i_less_j} is identified (up to isotopy rel natural vertices) with the collapse map $q_{i \cdot \phi, j \cdot \phi}$.

%
%
%
%

Recall from Section~\ref{SectionIdealSimplicial} that for each $i < j \in I$ the collapse map $q_{i,j}$ induces a face ID function $s_{i,j} \from \EG_i \leftrightarrow H_{i,j} \hookrightarrow \EG_j$, with corresponding ideal face map $r_{i,j} \from \Delta(G_i) \to \Delta(G_j)$. For each $\phi \in \Out(F_n)$ the homeomorphisms $h_{i,\phi} \from G_i \to G_{i \cdot \phi}$ and $h_{j,\phi} \from G_j \to G_{j \cdot \phi}$ induced bijections of natural edge sets $h_{i,\phi} \from \EG_i \to \EG_{i \cdot \phi}$ and $h_{j,\phi} \from \EG_j \to \EG_{j \cdot \phi}$. From commutativity of the outer square in Figure~\ref{Figure_i_less_j}, one obtains the commutative diagram of face ID functions and induced bijections depicted in the left half of Figure~\ref{FigureFaceIDCommDiagram}. From that we immediately deduce the desired commutative diagram of ideal face maps and ideal simplex isomorphisms depicted in the right half of  Figure~\ref{FigureFaceIDCommDiagram}.

\paragraph{The final step.} Finally, knowing the property ``Ideal Face Maps Are Preserved'', it follows that the right action of $\Out(F_n)$ on the disjoint union $\coprod_{i \in I} \Delta(G_i)$ descends, via the quotient map to $\X_n$, to a right action of $\Out(F_n)$ on $\X_n$ by ideal simplicial isomorphisms. 


\subsection{Properties of the action of $\Out(F_n)$ on $\X_n$.}

This section consists primarily of exercises which pull together various threads to establish properties of the action of $\Out(F_n)$ on $\X_n$; see also Section~\ref{SectionTopActions} for some basic definitions regarding actions.

\paragraph{An exercise on cells stabilizers.} 
\begin{exercise}\label{ExerciseFiniteCellStabilizers}
Prove that for any outer space cell $\Delta(G) \subset \X_n$, the subgroup
$$\Stab(\Delta(G)) = \{\phi \in \Out(F_n) \suchthat \Delta(G) \cdot \phi =\Delta(G)\}
$$
is isomorphic to the image of the natural injective homomorphism $\Aut(G) \inject \Out(F_n)$ described in Corollary~\ref{CorollaryAutGInjection}.
\end{exercise}

\paragraph{Exercises on cofiniteness and properness.}

\begin{exercise}\label{ExerciseCellsCofinite}
Prove that there are only finitely many orbits of outer space cells, under the action of $\Out(F_n)$ on $\X_n$. (Hint: Exercise~\ref{ExerciseFinManyCoreGraphs})
\end{exercise}

\begin{exercise}
\label{ExerciseProperActionOnOuterSpace}
Prove that the action of $\Out(F_n)$ on $\X_n$ is proper (Hint: Exercises~\ref{ExerciseLocallyFinite} and~\ref{ExerciseCellsCofinite})
\end{exercise}

\paragraph{Exercises on circuit length, systole, and non-cocompactness.}\index{circuit length}\index{systole}
This set of exercises builds up to Exercise~\ref{ExerciseNoncocompact} which uses circuit length functions and the systole function on $\X_n$ to prove that the action of $\Out(F_n)$ on $\X_n$ is not cocompact. 

Recall from Section~\ref{SectionConjugacyClasses} that the set of nontrivial conjugacy classes of the free group $F_n$ is denoted $\C(F_n)$, and that each $c \in \C(F_n)$ is represented by a circuit in $G$ which is itself represented by a cyclically reduced edge path in $G$, unique up to cyclic conjugacy. We use the notation $\gamma_c = e_1 \cdots e_K$, where $e_1,\ldots,e_K \in \EG$ is a sequence of oriented natural edges of $G$, to denote a choice of representative in this cyclic conjugacy class. 

Consider a point $\ell \in \X_n$. A choice of outer space cell $\Delta(G) \subset \X_n$ such that $\ell \in \Delta(G)$ determines a representation of $\ell$ as a normalized function $\ell \from \EG \to [0,\infty)$. Consider also $c \in \C(F_n)$ and $\gamma_c = e_1 \cdots e_K$. We define the \emph{length of $c$ with respect to $\ell$} to be the number
$$\Length(c;\ell) = \sum_{i=1}^K \ell(e_i) \in \reals
$$
Note that $\Length(c;\ell) > 0$, because the set of edges $\{e \in \EG \suchthat \ell(e)=0\}$ is a subforest of $G$ and hence contains no circuit. 

\begin{exercise}\label{ExerciseLengthFunctionWD} Prove that $\Length(c;\ell)$ is well-defined, depending only on $c \in \C(F_n)$ and $\ell \in \X_n$, independent of the choice of $\gamma_c$ and of outer space cell $\Delta(G)$ containing $\ell$.
\end{exercise}

Using Exercise~\ref{ExerciseLengthFunctionWD}, for each $c \in \C(F_n)$ we have a well-defined function 
$$\Length^c \from \X_n \to (0,\infty), \qquad \Length^c(\ell)=\Length(c;\ell)
$$

\begin{exercise}\label{ExerciseLengthContinuous}
Prove that $\Length^c \from \X_n \to (0,\infty)$ is continuous, for each $c \in \C(F_n)$.
\end{exercise}

Now fix $\ell \in \X_n$, and define the \emph{length spectrum} of $\ell$ to be the set
$$\Spec(\ell) = \{\Length(c;\ell) \suchthat c \in \C(F_n)\} \subset (0,\infty)
$$

\begin{exercise}\label{ExerciseSystole}
Prove that for each $\ell \in \X_n$, the set $\Spec(\ell)$ has a positive minimum. More specifically, prove that for any marked graph $G$ such that $\ell \in \Delta(G)$, the minimum is achieved by $\Length(c;\ell)$ for some $c \in \C(F_n)$ that is represented by an embedded circuit $\gamma_c \from S^1 \to G$.
\end{exercise}

The minimum value in Exercise~\ref{ExerciseSystole} is called the \emph{systole} of $\ell$, denoted $\sys(\ell)$.

\begin{exercise}
Prove that the function $\sys \from \X_n \to (0,\infty)$ is continuous. (Hint: Combine Exercises~\ref{ExerciseLengthContinuous} and Exercise~\ref{ExerciseSystole} to conclude that for each marked graph $G$, the restricted function $\sys \restrict \Delta(G)$ is the minimum of a finite set of continuous functions).
\end{exercise}

\begin{exercise}\label{ExerciseSysZeroLimit}
Prove that for any outer space cell $\Delta(G)$, the restricted function $\sys \from \Delta(G) \to (0,\infty)$ does not have a positive lower bound.
\end{exercise}

It follows from Exercise~\ref{ExerciseSysZeroLimit} that every outer space cell $\Delta(G)$ is noncompact, although this is already clear to those who have done Exercise~\ref{ExerciseIdealSimplexProps}.

\begin{exercise}\label{ExerciseNoncocompact}
Prove that the systole function $\sys \from \X_n \to (0,\infty)$ is invariant under the action of $\Out(F_n)$: for each $\ell \in \X_n$ and each $\phi \in \Out(F_n)$ we have $\sys(\ell) = \sys(\ell \cdot \phi)$. Conclude that the action of $\Out(F_n)$ on $\X_n$ is not cocompact.
\end{exercise}

\paragraph{Exercise: Another approach to non-cocompactness} 

\begin{exercise}
\label{ExerciseNoncocompactAgain}
Prove that the action of $\Out(F_n)$ on $\X_n$ has a noncompact fundamental domain (using Exercise~\ref{ExerciseCellsCofinite}), and conclude that the action is not cocompact (using Exercise~\ref{ExerciseCocompactAndNot}).
\end{exercise}

\paragraph{Exercise: A free action}

A group $\Gamma$ that possesses torsion elements, such as $\Out(F_n)$, cannot act freely and properly on a contractible CW-complex of finite dimension $k \in \N$, because that would imply that for every subgroup $H \subgroup \Gamma$, the cohomology groups $H^i(H;A)$ with arbitrary coefficients $A$ are trivial in dimensions $i > k$, which is false for any nontrivial finite cyclic group.

We nonetheless get an application to the finite index torsion free subgroup $\text{IA}_n(\Z/3) \subgroup \Out(F_n)$ given in Corollary~\ref{CorollaryIAFiniteIndex}:

\begin{exercise}
Prove that the restricted action of $\text{IA}_n(\Z/3)$ on $\X_n$ is free.
\end{exercise}

\subsection{The spine of outer space.} 
\label{SectionSpine}
Cocompactness fails for the action of $\Out(F_n)$ on $\X_n$, as seen in Exercises~\ref{ExerciseNoncocompact} and~\ref{ExerciseNoncocompactAgain} just above. In this section we prove the following theorem of Culler and Vogtmann which produces a spine of $\X_n$ on which $\Out(F_n)$ acts cocompactly.

\begin{theorem}[\cite{CullerVogtmann:moduli}] 
\label{TheoremOuterSpaceSpine}
The outer space $\X_n$ contains an \emph{equivariant spine}, meaning an $\Out(F_n)$ invariant simplicial complex $\K_n \subset \X_n$, intersecting each outer space cell in a finite subcomplex of $\K_n$, such that the restricted action of $\Out(F_n)$ on $\K_n$ is proper and cocompact, and there is an $\Out(F_n)$-invariant deformation retraction $\X_n \mapsto \K_n$. 
\end{theorem}

Most of the work of this theorem is contained in the following general version:

\begin{theorem} 
\label{TheoremGeneralSpine}
Every ideal simplicial complex $X$ contains an equivariant spine $K \subset X$, meaning an $\Aut(X)$ invariant simplicial complex, intersecting each ideal simplex of $X$ in a finite subcomplex of $K$, such that there is an $\Aut(X)$-invariant deformation retraction $X \mapsto K$.
\end{theorem}

\begin{proof}[Proof of Theorem~\ref{TheoremOuterSpaceSpine} assuming Theorem~\ref{TheoremGeneralSpine}] The two additional details to check are properness and cocompactness of the action of $\Out(F_n)$ on $\K_n$. 

We saw in Exercise~\ref{ExerciseProperActionOnOuterSpace} that the action of $\Out(F_n)$ on outer space $\X_n$ is proper, and the restriction of any proper action to any invariant subset is also proper, hence the action on $\K_n$ is proper.

From the solution to Exercise~\ref{ExerciseNoncocompactAgain} it follows that there is a finite collection of outer space cells $\Delta(G_1),\ldots,\Delta(G_J) \subset \X_n$ whose orbits cover $\X_n$: 
$$\coprod_{\phi \in \Out(F_n)} (\Delta(G_1) \union\cdots\union\Delta(G_j)) \cdot \phi = \X_n
$$
It follows from Theorem~\ref{TheoremGeneralSpine} that $\K_n \intersect \Delta(G_j)$ is a finite subcomplex of $\K_n$ for each $j=1,\ldots,J$, and furthermore that 
$$\coprod_{\phi \in \Out(F_n)} \bigl(\K_n \intersect (\Delta(G_1) \union\cdots\union\Delta(G_J)) \bigr) \cdot \phi = \K_n
$$
Since $\K_n \intersect (\Delta(G_1) \union\cdots\union\Delta(G_J))$ is a finite subcomplex of $\K_n$, this proves cocompactness.
\end{proof}

\subparagraph{Example.} Before turning to the proof of Theorem~\ref{TheoremGeneralSpine}, we revisit a familiar example, describing the spine of the Farey complex, the $\SL(2,\Z)$-invariant ideal simplicial structure on the hyperbolic plane $\mathbb H^2$ discussed in Section~\ref{SectionIdealSimplicial}. Its spine has one vertex at the barycenter of each ideal $2$-simplex, one vertex at the barycenter of each ideal $1$-simplex, and one $1$-simplex connecting the barycenter of each ideal $2$-simplex with the barycenter of each incident ideal $1$-simplex. This spine is a tree, and under the $\SL(2,\Z)$-action it has two vertex orbits and one edge orbit. As shown in Serre's book \cite{Serre:trees}, by using the theory of graphs of groups one can apply this spine to show that $\SL(2,\Z)$ is an amalgamated free product of the groups $\Z/4\Z$ and $\Z/6Z$, by amalgamating their $\Z/2\Z$ subgroups. 

\begin{proof}[Proof of Theorem \ref{TheoremGeneralSpine}] The construction of the spine of an ideal simplicial complex is adapted from the theory of simplicial complexes: given a simplicial complex $Y$ and subcomplex $L \subset Y$, the spine $K$ of $Y-L$ is the union of those simplices of the first barycentric subdivision of $Y$ that are disjoint from $L$; barycentric coordinates are used to define the deformation retraction $(Y-L) \times [0,1] \to Y-L$ from $Y-L$ to $K$. In fact this construction can be applied to outer space, and indeed to any ideal simplicial complex $X$, using a preliminary construction of a ``simplicial completion'' $Y$ of $X$ and a subcomplex $L \subset Y$ such that $Y-L$ and $X$ are isomorphic as ideal simplicial complexes. In lieu of discussing simplicial completions, we construct the spine of an ideal simplicial complex by hand.

The spine of an ideal simplicial complex is constructed in steps: first we construct the spine of an ideal simplex; then we show naturality of the spine with respect to ideal face maps; then for any ideal simplicial complex $X$ we use naturality to show that the union of the spines of the ideal simplices of $X$ is a spine of~$X$.

\smallskip\emph{The spine of an ideal simplex.} Consider the ideal simplex $\Delta(E;L) = \Delta(E)-L$ associated to a finite set $E$ and a subcomplex $L \subset \Delta(E)$. Let $L' \subset \Delta'(E)$ denote the first barycentric subdivisions. The \emph{spine of $\Delta(E;L)$} is the subcomplex $K(E;L) \subset \Delta'(E)$ consisting of the union of all simplices of $\Delta'(E)$ that are disjoint from~$L'$.

We construct a strong deformation retraction from $\Delta(E;L)$ to $K(E;L)$, meaning a homotopy $H \from \Delta(E;L) \times [0,1] \to \Delta(E;L)$ such that the following hold: the map $H_0(x)=H(x,0)$ is the identity on $\Delta(E;L)$; the map $H_1(x)=H(x,1)$ takes $\Delta(E;L)$ to $K(E;L)$; and the homotopy is stationary on $K(E;L)$ meaning that for each $t$ the map $H_t(x)=H(t,x)$ restricts to the identity on $K(E;L)$. 

Consider a simplex $\sigma$ of $\Delta'(E)$ that is not contained in $L$. Let $\sigma_L = \sigma \intersect L$, and so $\sigma \intersect \Delta(E;L) = \sigma-L = \sigma - \sigma_L$ is not empty. Let $\sigma_K = \sigma \intersect K(E;L)$, and note that $\sigma_K \ne \emptyset$. Note also that $\sigma_L$ and $\sigma_K$ are complementary faces, meaning that every vertex of $\sigma$ is contained in exactly one of $\sigma_L$ or $\sigma_K$. We shall define a deformation restriction $H^\sigma = H \restrict (\sigma - \sigma_L) \times [0,1]$ from $\sigma-\sigma_L$ to $\sigma_K$ as follows. The case that $\sigma_L = \emptyset$ is equivalent to $\sigma = \sigma_K \subset K(E;L)$, and in that case $H^\sigma$ is stationary on $\sigma$. In the case that $\sigma_L \ne \emptyset$, it follows that $\sigma$ is the join of its complementary faces $\sigma_L$ and $\sigma_K$, meaning that the function $f \from \sigma_L \times \sigma_K \times [0,1] \to \sigma$ defined by the formula $f(x,y,t) = (1-t)x + t y$ is a quotient map with the following properties:
\begin{enumerate}
\item\label{ItemLProj}
 the restriction $f \from \sigma_L \times \sigma_K \times \{0\} \mapsto \sigma_L$ is projection onto the first factor;
\item\label{ItemKProj}
the restriction $f \from \sigma_L \times \sigma_K \times \{1\} \to \sigma_K$ is projection, onto the second factor; 
\item\label{ItemNotLKHomeo}
the restriction $f \from \sigma_L \times \sigma_K \times (0,1) \to \sigma - (\sigma_L \union \sigma_K)$ is a homeomorphism.
\end{enumerate}
By restricting $f$ we obtain a quotient map $f^\sigma \from \sigma_L \times \sigma_K \times (0,1] \to \sigma - \sigma_L$ satisfying properties \pref{ItemKProj} and \pref{ItemNotLKHomeo} above. Define a homotopy 
$$\wh H_\sigma \from \bigl(\sigma_L \times \sigma_K \times (0,1]\bigr) \times [0,1] \to \sigma_L \times \sigma_K \times (0,1]
$$
by the formula
$$\wh H^\sigma(x,y,t,u) = (x,y,(1-u)t+u)
$$
and so $\wh H^\sigma$ is a strong deformation retraction from $\sigma_L \times \sigma_K \times (0,1]$ to $\sigma_L \times \sigma_K \times \{1\}$. Using universality properties of quotient maps, we obtain a unique continuous map $H^\sigma$ making the following diagram commute.
$$\xymatrix{
(\sigma_L \times \sigma_K \times (0,1]) \times [0,1] \ar[rrr]^{\wh H^\sigma} \ar[d]_{f_\sigma \times \text{Id}}
&&&  \sigma_L \times \sigma_K \times (0,1] \ar[d]^{f_\sigma}
\\
(\sigma - \sigma_L) \times [0,1] \ar[rrr]^{H^\sigma} 
&&& \sigma - \sigma_L
}$$
One easily checks that $H^\sigma$ is a strong deformation retraction from $\sigma - \sigma_L$ to $\sigma_K$. 

Having defined $H^\sigma$ for any simplex $\sigma$ of $\Delta'(E)$ not contained in~$L$, it follows easily that for any nested pair of simplices $\sigma \subset \tau$ of $\Delta'(E)$ not contained in~$L$, the restriction of $H^\tau$ to $(\sigma - \sigma_L) \times [0,1]$ is equal to $H^\sigma$. Thus, as $\sigma$ varies of the simplices of $\Delta'(E)$ not contained in~$L$, the maps $H^\sigma$ glue together to produce the desired strong deformation retraction 
$$H \from \Delta(E;L) \times [0,1] \to K(E;L)
$$

\smallskip\emph{Naturality of the spine with respect to face maps.} Consider next two ideal simplices $\Delta(E';L')$ and $\Delta(E;L)$ and an injection of finite sets $s \from E' \leftrightarrow F \subset E$ inducing a face map 
$$r \from \Delta(E') \leftrightarrow \Delta(F \subset E) \subset \Delta(E)
$$ 
such that $L' = r^\inv(L)$, and which therefore restricts to an ideal face map
$$r \from \Delta(E';L') = \Delta(E')-L' \to \Delta(E)-L=\Delta(E;L)
$$
Let $H^E$, $H^{E'}$ be the strong deformation retractions constructed above, $H^E$ from $\Delta(E;L)$ to $K(E;L)$, and let $H^{E'}$ from $\Delta(E';L')$ to $K(E';L')$. The map $r$ induces a bijection between simplices $\sigma'$ of the first barycentric subdivision of  $\Delta(E')$ not contained in $L'$, and simplices $\sigma$ of the first barycentric subdivision of $\Delta(E)$ that are contained in $\Delta(F \subset E)$ but not in $L$, mapping $\sigma' - L'$ to $\sigma - L$ by a homeomorphism that respects barycentric coordinates. Using this face, and tracing through the definitions of $H^E$ and $H^{E'}$, it is straightforward to derive the naturality condition
$$r(H^{E'}(x,t)) = H^E(r(x),t) \quad\text{for each $x \in \Delta(E';L')$ and $t \in [0,1]$.}
$$
We note that this holds as well in the special case that $s$ is a bijection and $r$ is an ideal simplex isomorphism; we will need this below in verifying that the action of an automorphism of an ideal simplicial complex preserves the spine.

\smallskip\emph{The spine of an ideal simplicial complex.} Consider now an ideal simplicial complex $X$ as described in Section~\ref{SectionIdealSimplicial}, expressed as a union of embedded ideal simplices $\Delta(E_i;\L_i)$ ($i \in I$) with a face order $i<j$ on the index set $I$ such that each inclusion $r_{i,j} \from \Delta(E_i;\L_i) \to \Delta(E_j;\L_j)$ is a face map, and such that $X$ is the quotient of the disjoint union $\coprod_{i \in I} \Delta(E_i;\L_i)$ with respect to these face maps. Let $K \subset X$ be the union $K=K(E_i;L_i)$ of the spines of the individual simplices $\Delta(E_i;\L_i) \subset X$. By the naturality condition described just above applied to each face map $r_{i,j}$, it follows that the collection deformation retractions from $\Delta(E_i;\L_i)$ to $K(E_i;L_i)$ for each $i \in I$ fit together to form a deformation retraction $X \mapsto K$. Furthermore, for any automorphism $f \from X \to X$ with corresponding bijection $f \from I \to I$ and ideal simplex isomorphisms $f_i \from \Delta(E_i;L_i) \to \Delta(E_{i \cdot f},L_{i \cdot f})$, the same naturality condition shows that $f_i$ takes $K(E_i;L_i)$ to $K(E_{i \cdot f},L_{i \cdot f})$. It follows that $f(K)=K$, proving that the spine $K$ is invariant under automorphisms of $K$.

This completes the proof of Theorem~\ref{TheoremGeneralSpine}.
\end{proof}

\subsection{$\Out(F_n)$ acts geometrically on $\K_n$ (assuming path connectivity).}
\label{SectionOutActsGeometrically}

One of the bedrock principles of geometric group theory is that the large scale geometry of a group $\Gamma$ can be studied using actions of $\Gamma$ on geometric objects, if those actions satisfy some useful basic properties. One of the most important lists of basic properties is collected in the following definition.

\begin{definition}[Geometric actions of groups]
Consider an action of a group~$\Gamma$ on a metric space $X$ (meaning that $\Gamma$ acts by isometries of~$X$). We say that this is a \emph{geometric action} if each of the following properties holds:
\begin{description}
\item[$X$ is a \emph{proper} metric space,] meaning that closed balls are compact. It follows that $X$ is locally compact.
\item[$X$ is a \emph{geodesic} metric space,] meaning that for any $x,y \in X$ there exists an isometric injection $\gamma \from [0,d(x,y)] \to X$ such that $d(\gamma(s),\gamma(t)) = \abs{s,t}$ for all $s,t \in [0,d(x,y)]$.
\item[The action is \emph{proper},] as defined in Section~\ref{SectionTopActions}.
\item[The action is \emph{cocompact},] as defined Section~\ref{SectionTopActions}.
\end{description}
\end{definition}
The formal statement of the ``bedrock principle'' referred to above is the following (the first sentence of which is an immediate consequence of Lemmas~\ref{LemmaFDFinGen} and~\ref{LemmaFDExists}):

\begin{lemma}[Milnor-Svarc Lemma \cite{}]
\label{LemmaMilnorSvarc}
For any geometric group action $\Gamma \act X$, the group $\Gamma$ is finitely generated. Furthermore, for any base point $p \in X$ the map $\Gamma \mapsto X$ defined by $\gamma \mapsto \gamma \cdot p$ is a quasi-isometry, with respect to the word metric $d_\Gamma$ on $\Gamma$ and the given metric $d_X$ on $X$, meaning that there exist constants $K \ge 1$, $C \ge 0$ such that for any $\gamma,\delta \in G$ we have
$$\frac{1}{K} d_\Gamma(\gamma,\delta) - C \le d_X(\gamma \cdot p, \delta \cdot p) \le K d_\Gamma(\gamma,\delta) + C
$$
\end{lemma}

As a consequence of this principle, one is free to choose a particular geometric action to study quasi-isometric properties of the word metric on $\Gamma$. Some geometric actions may be more useful than others, depending on the specifics of the group~$\Gamma$ and the space~$X$. For some purposes a Cayley graph of $\Gamma$ with respect to some finite generating set, or a Cayley 2-complex with respect to some finite presented, may be a perfectly adequate choice for~$X$. But it is often useful to choose $X$ in a manner which is more naturally suited to the group~$\Gamma$.

The outer space $\X_n$ and its spine $\K_n$ are very naturally suited to studying $\Out(F_n)$. For purposes of applying the Milnor-Svarc lemma and certain other tools of combinatorial and geometric group theory, we must reckon with noncocompactness of the $\X_n$ action. But the action on $\K_n$ is cocompact, and using that fact we can \emph{almost} that action is geometric, although at the moment there is a hole in our understanding which we will not be in a position to fill until Chapter 2:

\begin{theorem*}[see Theorem~\ref{TheoremPathConnected} in Section~\ref{SectionOuterSpaceConnectivity}]
$\X_n$ and $\K_n$ are path connected.
\end{theorem*}

In fact more is true: Culler and Vogtmann, in their original paper introducting outer space, proved that $\X_n$ and $\K_n$ are contractible (see Theorem~\ref{TheoremContractible} in Section~\ref{SectionOuterSpaceContractibility}).

For now, assuming path connectivity of $\K_n$ (Theorem~\ref{TheoremPathConnected}), we put together the rest of the pieces together to prove:

\begin{theorem}
\label{TheoremOutActionGeometric}
The action of $\Out(F_n)$ on $\K_n$ is geometric.
\end{theorem}

\begin{proof} For any two Euclidean $k$-simplices $\Delta \subset \reals^m$, $\Delta' \subset \reals^n$, if each of $\Delta$ and $\Delta'$ has edge lengths equal to $1$ then any barycentric coordinate preserving homeomorphism $\Delta \mapsto \Delta'$ is an isometry of the Euclidean metrics. For any connected simplicial complex~$X$, it follows that each $k$-simplex $\sigma \subset X$ has a unique geodesic metric $d_\sigma$ such that any barycentric coordinate preserving map from $\sigma$ to a regular Euclidean simplex with side lengths~$1$ is an isometry. Letting $\sigma$ vary over all simplices in $X$, there is a unique geodesic metric $d_X$ on $X$ such that for each simplex $\sigma$ the inclusion map $\sigma \hookrightarrow X$ is a localy isometry from $d_\sigma$ to $d_X$: for each $x,y \in X$ define $d_X(x,y)$ to be infimum of the lengths of piecewise simplicial simplicial paths from $x$ to $y$. If $X$ is a locally finite simplicial complex then the simplicial metric on $X$ is proper. Also, any simplicial isomorphism between simplicial complexes induces an isometry of simplicial metrics. 

On the spine of outer space $\K_n$, using path connectivity of $\K_n$ (Theorem~\ref{TheoremPathConnected}), the simplicial metric on $\K_n$ is defined. Knowing that $\K_n$ intersects each ideal simplex $\Delta(G) \subset \X_n$ in a finite subcomplex of $\K_n$ (Theorem~\ref{TheoremGeneralSpine}), and knowing that the ideal simplicial decomposition of $\X_n$ is locally finite (Exercise~\ref{ExerciseLocallyFinite}), it follows that $\K_n$ is locally finite, and hence its simplicial metric is proper.

Properness and cocompactness of the action were proved in Theorem~\ref{TheoremGeneralSpine}.
\end{proof}

\chapter{Fold paths in outer space
}
\label{ChapterFoldPaths}

In Chapter 1 we introduced the problems of Nielsen and Whitehead regarding a free group $F_n = F\<s_1,\ldots,s_n\>$, for example the problem of determining when a reduced word in the generating set $S=\{s_1,\ldots,s_n\}$ represents a free basis element of $F_n$. Motivated by a preliminary attempt to understand those problems, we introduced marked graphs, and the Culler-Vogtmann outer space $\X_n$, as a deformation space of length structures on marked graphs. The idea was to consider a given (cyclically) reduced word as a circuit in the base rose $R_n$, and then to study how that circuit evolves when one moves away from $R_n$ along paths in $\X_n$. 

In the opening sections of Chapter 2 we shall dig deeper into the problems of Nielsen and Whitehead, in order to motivate a specific and very useful class of paths in outer space, known nowadays as \emph{Stallings fold paths}.


\section{Fold sequences: An example}
\label{SectionFoldSequenceExample}
Stallings introduced fold sequences in his landmark paper ``Topology of finite graphs'' \cite{Stallings:folding}. In this section we set up an example along the lines of the Nielsen/Whitehead questions: a particular map on the rank~$2$ rose which we analyze to determine whether that map is a homotopy equivalence. Fold sequences arise quite naturally in the course of this analysis, and are easily re-interpreted as fold paths in outer space.

\subsection{A map to fold.} 
\label{SectionAMapToFold}
On the rose $R_2$ with oriented edges $a,b$ and with fundamental group $F_2 = \<a,b\>$, consider the four homotopy equivalences defined by the following list of positive transvections involving the generators and their inverses. In these formulas, the implicit assumption is that any generator which is not mentioned on the left hand side is fixed. So for example in formula (1) $a \mapsto ab$ it is implicit that $b \mapsto b$.
$$(1) \,\, a \mapsto ab \qquad (2)\,\,  a \mapsto ba \qquad (3)\,\, b \mapsto ba \qquad (4)\,\, b \mapsto ab
$$
Each of these transvections could be obtained from the first one, by pre and/or post composition with the transposition $a \leftrightarrow b$, but it is convenient to work with the whole list.

First we compute the following composition:
$$R_2 \xrightarrow{(1)} R_2 \xrightarrow{(3)} R_2 \xrightarrow{(2)} R_2 \xrightarrow{(4)} R_2 \xrightarrow{(4)} R_2
$$
$$
\begin{matrix} a \\b \end{matrix} \,\xrightarrow{(1)} \,
\begin{matrix} ab \\ b \end{matrix} \, \xrightarrow{(3)} \,
\begin{matrix} aba \\ ba \end{matrix} \, \xrightarrow{(2)} \,
\begin{matrix} babba \\ bba \end{matrix} \, \xrightarrow{(4)} \,
\begin{matrix} abaababa \\ ababa \end{matrix} \, \xrightarrow{(4)} \,
\begin{matrix} aabaaabaaba \\ aabaaba \end{matrix}
$$
So far this map must, of course, be a homotopy equivalence, since it is a composition of homotopy equivalences. Now we tweak the map: in the word $aabaaabaaba$, change the middle $a$ of the $aaa$ sub word to a $b$, and we get the map
$$
R_2 = G_0 \xrightarrow{f^0} G' = R_2
$$
$$
f^0 \, : \, \begin{matrix} a \\ b \end{matrix} \, \mapsto \, \begin{matrix} aabababaaba \\ aabaaba \end{matrix}
$$
(We'll have other graphs $G_1,G_2\ldots$ and maps $f^1,f^2\ldots$ in a moment.) This last tweak does not follow any recipe for a homotopy equivalence, and it seems there's a good chance that the map $f^0$ will not be a homotopy equivalence. We shall investigate this by factoring $f^0$ into folds and seeing what happens.

We depict the map $f^0$ by subdividing the $a$ and $b$ edges of the domain rose $G_0$ into little edgelets, and labeling each edgelet by its image in the range rose; see Figure~\ref{FigureFoldPath01}. The $a$ edge of the domain is subdivided into $11$ edgelets labelled $aabababaaba$, and the $b$ edge into 7 edgelets labelled $aabaaba$.

There is a lot we do not know about the map $f^0 \from G_0 \to G'$, and about the words $aabababaaba$ and $aabaaba$ used to define the map:
\begin{itemize}
\item Is $f^0$ $\pi_1$-injection? In other words, is the homomorphism $F_2 \mapsto F_2$ defined by $a \mapsto aabababaaba$ and $b \mapsto aabaaba$ injective? (For the answer, see the discussion following the statement of Corollary~\ref{CorollaryLocallyInjectiveProperties}).
\item Is $f^0$ a $\pi_1$-surjection? If not, what is a free basis for the $\pi_1$-image? (For the answer, see Section~\ref{SectionPiOneSurjectivity}.
\item Is $aabababaaba$ a basis element of $F_2$?
\end{itemize}

\begin{figure}
\centerline{
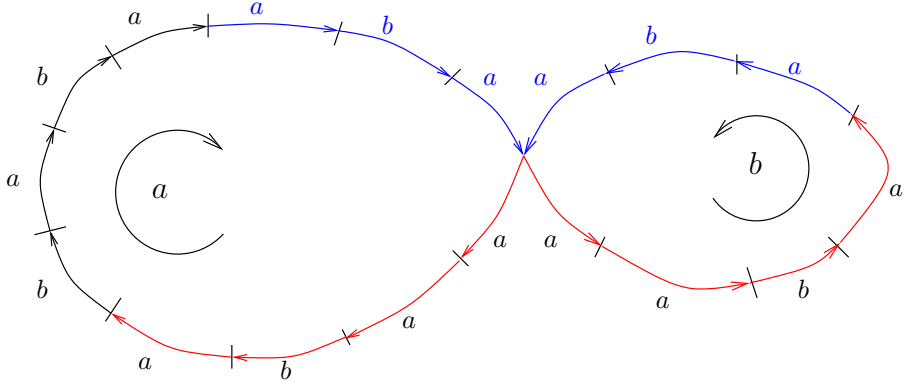 
}
\caption{The graph $G_0$ is the domain of a map of roses $f^0 \from G_0 \to G'$ associated to the endomorphism of $F_2=F\<a,b\>$ defined by $a \mapsto aabababaaba$, $b \mapsto aabaaba$. The natural edge loops $a,b$ of $G_0$ are subdivided into edgelets to depict the map $f^0$. The two red segments of $G_0$, initial segments of the $a$ and $b$ loops respectively, have the same initial endpoint and the same image path in $G'$, namely $aaba$. Similarly, the two blue segments, terminal segments of the $a$ and $b$ loops, have the same terminal endpoint and the same image path $aba$.}
\label{FigureFoldPath01}
\end{figure}

\subsection{The first fold.}
\label{SectionFirstFold}
Using Figure~\ref{FigureFoldPath01}, or just using the formula from which it is derived, one can see that in the graph $G_0$, the initial subpaths of the $a$ and $b$ edges, consisting of the first four edgelets of each, are mapped to the identical path in the range graph $G'$, namely the path $aaba$. 

Pondering this fact, one might feel a subliminal urge to identify those two subpaths. Let us follow our urge, by folding together their initial $a$ edgelets, and then folding together the following $a$ edgelets, and then the following $b$ edgelets, and then the final $a$ edgelets. But before getting carried away, let us then stop and ponder what we have done.

The outcome is that we have factored the map $f^0 \from G_0 \to G'$ by folding these two initial subpaths together as follows:
$$\xymatrix{
G_0 \ar[r]_{g_1} \ar@/^2pc/[rr]^{f^0} & G_1 \ar[r]_{f^1} & G'
}$$
The graph $G_1$ is the quotient space that is obtained from $G_0$ by folding two segments into one edge, namely the two red segment depicted in Figure~\ref{FigureFoldPath01}: the initial $aaba$ edgelet segment of the $a$ edge of $G_0$; and the initial $aaba$ edgelet segment of the $b$ edge of $G_0$. Those two segments are folded together and identified to a single edge of $G_1$. The resulting quotient map $g_1 \from G_0 \to G_1$ is called a ``fold map''. The \emph{key fact} to observe is that $g_1$ is a homotopy equivalence, because the two subpaths being identified by this fold are embedded in $G_0$, they share their initial endpoints, and they are otherwise disjoint. 

The quotient graph $G_1$ is a theta graph, depicted in Figure~\ref{FigureFoldPath02}, and the formula for the quotient $g_1$ is given below. Since $g_1(x)=g_1(y)$ only if $f^0(x)=f^0(y)$, it follows that the map $f^0$ factors as $f^0 = f^1 \composed g_1$ for some map $f^1 \from G_1 \to G'$; this is the ``quotient factorization theorem'', see e.g.\ \cite{Munkres:topology} Theorem~11.1. We compute a formula for the map $f^1$, using an edgelet subdivision and labeling of the graph $G_1$ that is shown in Figure~\ref{FigureFoldPath02} and that it inherits from the graph $G_0$ via the fold map $g_1$.\footnote{Neither $g_1$ nor $f^1$ should be interpreted as any kind of ``self-map'' of any object, neither a geometric object such as a graph nor an algebraic object such as a group. The domains and ranges of $g_1$ are not even isomorphic graphs, and similarly for the map~$f^1$. Even if the domain or the range \emph{were} isomorphic, we do not necessarily \emph{want} to pick an isomorphism, nor to interpret the domain and range as being the ``same graph''.} 
The factor maps $g_1$ and $f^1$ are defined by the following formulas: 
$$g_1  \, : \, \begin{matrix} G_0 \\ a \\ b \end{matrix} \mapsto \begin{matrix} G_1 \\ ca' \\ cb'  \end{matrix} \qquad\qquad f^1 \, : \, \begin{matrix} G_1 \\ a' \\ b' \\ c  \end{matrix} \mapsto \begin{matrix} G' \\ babaaba \\ aba \\ aaba \end{matrix}
$$

\begin{figure}
\centerline{
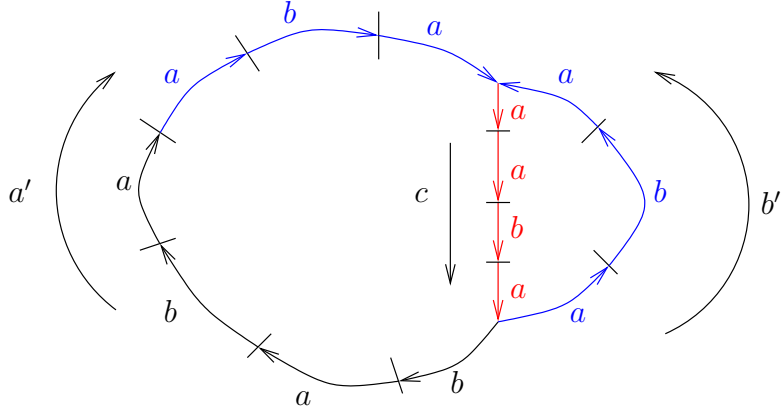 
}
\caption{The graph $G_1$. The edgelet subdivisions of its natural edges $a',b',c$, and the labels and coloring, depict the maps $G_0 \xrightarrow{g_1} G_1 \xrightarrow{f^1} G'$. The two red $aaba$ paths of $G_0$ are both taken by $g_1$ to the edge $c$ in $G_1$. Both of the two blue edgelet paths of $G_1$ are taken by $f^1$ to the same path in $G'$.}
\label{FigureFoldPath02}
\end{figure}

Before continuing, note that we had a choice in folding the graph $G_0$. We chose to fold the two red segments, whereas we could have chosen to instead fold the two blue segments. Because of such choices, fold paths in outer space are not uniquely determined. This stands in stark contrast to, let's say, a geodesic segment in a hyperbolic space, or a Euclidean space, or any manifold equipped with a Riemannian metric, in which a geodesic segment of a given length is uniquely determined by its initial point and its initial tangent direction. In this regard fold paths are more like geodesics in the taxicab or $L^\infinity$ metric on $\reals^n$. This lack of uniqueness can be useful in some contexts, giving flexibility that can be exploited to prove things. But in other contexts the difficulty of wading through choices can be a burden: one might compare the technical difficulties of the original Culler-Vogtmann proof of contractibility of outer space \cite{CullerVogtmann:moduli} to the cleanly slick proof of Skora \cite{Skora:deformations} (to be presented later in this work) which is based on ``canonical'' fold paths.

\subsection{Subsequent folds.} 
\label{SectionSubsequentFolds}
%
%

We now repeat the folding process. In the graph $G_1$, the two terminal blue subpaths of the edges $a'$ and $b'$ (images of the two blue paths in $G_0$) both consist of three edgelets labeled $aba$, they share terminal endpoints, and both are mapped to the identical path in~$G'$; in the case of $b'$ this terminal segment is in fact all of $b'$. We may factor $f^1 \from G_1 \to G'$ by folding these two subpaths together, extending our earlier commutative diagram as follows:
$$\xymatrix{
G_0 \ar[r]_{g_1} \ar@/^2.5pc/[rrr]^<<<<<<<<<{f^0} & G_1 \ar@/^1pc/[rr]^<<<<<<{f^1} \ar[r]_{g_2} &  G_2 \ar[r]_{f^2} & G'
}$$
The graph $G_2$ is depicted in Figure~\ref{FigureFoldPath03} with labelled edges and edgelets depicting appropriate maps. The map $f^2 \from G_2 \to G'$ then factors as well as into a product of a fold $g_3 \from G_2 \to G_3$ and a map $f^3 \from G_3 \to G'$, as follows:
$$\xymatrix{
G_0 \ar[r]_{g_1} \ar@/^5pc/[rrrr]^<<<<<<<<<<<<<<<<<<<<{f^0} & G_1 \ar[r]_{g_2} \ar@/^3.5pc/[rrr]^<<<<<<<<<<<<<<{f^1} & G_2 \ar[r]_{g_3} \ar@/^2pc/[rr]^<<<<<<<<{f^2} & G_3 \ar[r]^{f^3} & G'
}$$
The graph $G_3$ is also depicted in Figure~\ref{FigureFoldPath03}, with labelled edges and edgelets depicting appropriate maps.

\begin{figure}
\centerline{
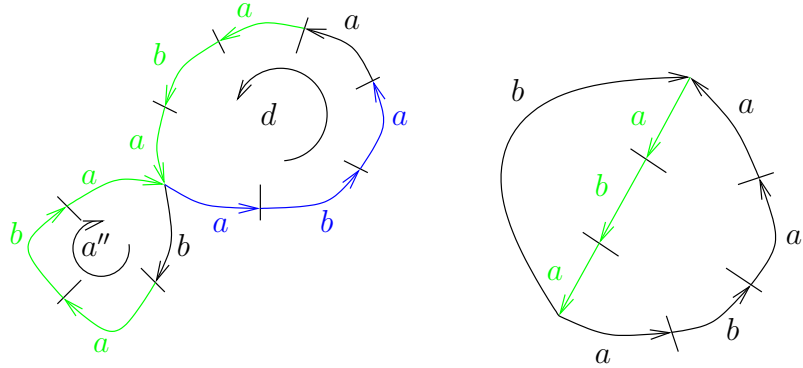 
}
\caption{The graphs $G_2$ and $G_3$, labelled and colored to depict the maps $G_1 \xrightarrow{g_2} G_2 \xrightarrow{f^2} G'$ and the maps $G_2 \xrightarrow{g_3} G_3 \xrightarrow{f^3} G'$.}
\label{FigureFoldPath03}
\end{figure}

At this stage, having produced the map $f^3 \from G_3 \to G'$, we must stop the process, because $f^3$ is locally injective: there are no more folds to do because no two directions at any vertex of the domain graph $G_3$ are mapped to the same direction in the image graph~$G'$. This is in contrast to the maps $f^0, f^1, f^2$ for each of which there exists a domain vertex at which there are two directions with the same image. Failure/success of local injectivity, expressed as the existence/nonexistence (respectively) of directions at some domain vertex having the same image, is \emph{exactly} what determines existence/nonexistence of a factorization of each of the maps $f^i \from G_i \to G'$ ($i=0,1,2,3$) as a fold followed by another map.


The sequence of maps $G_0 \mapsto G_1 \mapsto G_2 \mapsto G_3$ is an example of a \emph{fold sequence} or \emph{fold~path}. 

\subsection{The outcome of folding.} Consider now the final map $f^3 \from G_3 \to G'$, where $G'$ is the standard rank~$2$ rose, and $G_3$ is depicted in Figure~\ref{FigureFoldPath03} with its edgelet subdivision depicting $f^3$. As noted above, $f^3$ is a local injection. But it is not a local homeomorphism because each of its two valence~3 vertices maps to the valence~$4$ vertex of $G'$. So $f^3$ is certainly not a covering map. 

It turns out, though, that we can extend $f^3$ to a covering map by a simple process:


\begin{proposition}
\label{PropEndOfStallings}
For any locally injective map $f \from G \to H$ of connected graphs there exists a graph $\wh G$, an embedding $G \inject \wh G$, and an extension $\hat f \from \wh G \to H$ of $f$, such that $\wh G$ deformation retracts to $G$ and $\hat f$ is a covering map. 
\end{proposition}

\begin{corollary} 
\label{CorollaryLocallyInjectiveProperties}
Under the same hypotheses as the previous proposition, for any $v \in G$ the induced homomorphism $f_* \from \pi_1(G,v) \to \pi_1(H,f(v))$ is injective. Furthermore, if we assume that $G,H$ are finite core graphs then one of the following holds: 
\begin{enumerate}
\item\label{ItemGraphNotCoveringMap}
$f$ is not a covering map, in which case $\image(f_*)$ has infinite index
\item\label{ItemGraphCoveringMap}
$f$ is a covering map, in which case $\image(f_*)$ has finite index equal to the degree of $f$.
\item\label{ItemGraphHomeoHE}
$f$ is a homeomorphism, which happens $\iff$ $f$ is a homotopy equivalence $\iff$ $f_*$ is an isomorphism. 
\end{enumerate}
Note that \pref{ItemGraphCoveringMap} and \pref{ItemGraphHomeoHE} are not mutually exclusive.
\end{corollary}

Before proving these results, we apply them to conclude the example of Section~\ref{SectionFoldSequenceExample}. The map $f^3 \from G_3 \to G'$ is a local injection which is not a covering map, and so $\image(f^3_*)$ has infinite index and $f^3$ is not a homotopy equivalence. Using the fold factorization $f^0 = f^3 \composed g_3 \composed g_2 \composed g_1 \from G_0 \to G'$ and the evident fact that each of the folds $g_1,g_2,g_3$ is a homotopy equivalence, it follows that $f^0$ is not a homotopy equivalence.

%
%

\begin{proof}[Proof of Proposition~\ref{PropEndOfStallings}] At each vertex $v \in G$, if $f$ is not locally surjective at $v$ then we shall attach to $v$ a piece of the universal covering tree $\wt H$ of $H$ and extend $f$ over that piece using the universal covering map $q \from \wt H \mapsto H$. To so this, pick a vertex $u \in \wt H$ so that $q(u)=f(v)$. We have maps
$$T_v G \xrightarrow{D_v f} T_{f(v)} H \xleftarrow{D_{u} q} T_{u} \wt H
$$
the first an injection and the second a bijection. By composition we get an injection $T_v G \inject  T_{u} \wt H$. Let $\wt H(v)$ be the subtree of $\wt H$ which is the union of $u$ with each component of $\wt H - u$ whose direction at $u$ is \emph{not} in the image of the injection $T_v G \inject T_{u} \wt H$. Attach a disjoint copy of $\wt H(v)$ to $G$ by identifying the copy of $u$ to $v$, and extend $f$ by using the copy of the map $q \restrict \wt H(v)$. Doing this for each vertex $v \in G$ we have proved the proposition.
\end{proof}

\begin{proof}[Proof of Corollary \ref{CorollaryLocallyInjectiveProperties}]
Applying Proposition~\ref{PropEndOfStallings} the map $f_*$ factors as $\pi_1(G,v) \approx \pi_1(\wh G,v) \xrightarrow{\hat f_*} \pi_1(H,f(v))$, where the first map is an isomorphism because it is induced by inclusion of a deformation retraction, and the second map is an injection by covering space theory.

Assume now that $f,g$ are finite core graphs. Item~\pref{ItemGraphCoveringMap} also follows by covering space theory, and \pref{ItemGraphHomeoHE} is an immediate consequence. 

To prove \pref{ItemGraphNotCoveringMap}, first note that a locally injective map of connected graphs, taking vertices to vertices and edges to edges, is a covering map if and only if it is locally surjective (this is extremely far from true for topological spaces in general). Assuming that $f$ is not a covering map, it follows that the inclusion $G \inject \wh G$ of Proposition~\ref{PropEndOfStallings} is not surjective. Let $S \subset \wh G$ be the closure of some component of $\wh G - G$, so $S$ is a tree with at least one edge. But $S$ cannot be a finite tree, because if it were then there would be a vertex $w \in S$ which has valence~$1$ in $\wh G$, and so the covering map $\wh G \mapsto H$ would take $w$ to a vertex of valence~$1$ in $H$, contradicting that $H$ is a core graph. Therefore $S$ is an infinite tree, the graph $\wh G$ is an infinite graph, and the covering map $\wh G \mapsto H$ has infinite degree, proving that $\image(f_*)$ has infinite index.
\end{proof}

\section{Fold sequences: general theory}

\hfill \emph{``Fold and live to fold again''}

\hfill --- Stu Ungar

\bigskip

The example of the previous section was carefully set up in order to be able to immediately start folding. In general one should ask how that setup be generalized, which we shall answer with the concept of a \emph{foldable map}.

Describing the theory of fold sequences will break into several tasks: constructing foldable maps (Section~\ref{SectionFoldableConstruction}); factoring foldable maps into fold sequences (Sections~\ref{SectionFirstFoldFactorization}--\ref{SectionFoldSeqConstr}); showing how and why fold sequences stop (Proposition~\ref{ThmStallingsFolds}); and drawing conclusions from the manner in which fold sequences stop (Proposition~\ref{PropEndOfStallings} and Corollary~\ref{CorollaryLocallyInjectiveProperties}, and applications).

\subsection{Directions, tight maps, and gates.} 
\label{SectionDirections}
Consider a graph $G$ and a point $p \in G$. A \emph{direction}\index{direction} of $G$ at $p$ is defined to be the germ of a locally injective path with initial endpoint $p$: two locally injective paths $\gamma_1 \from [0,m_1] \to G$, $\gamma_2 \from [0,m_2] \to G$ with $\gamma_1(0)=\gamma_2(0)=p$ have the same germ if there exist $\epsilon_1 \in (0,m_1]$, $\epsilon_2 \in (0,m_2]$ and a homeomorphism $h \from [0,\epsilon_1] \to [0,\epsilon_2]$ such that $\gamma_2 \composed h = \gamma_1 \restrict [0,\epsilon_1]$. We denote the \emph{direction set} of $G$ at $p$ as $T_p G$, and its cardinality is equal to the valence of~$p$. The set $T_p G$ is a kind of ``tangent space'' to $G$ at $p$. Assuming that $p$ is a vertex of $G$ --- which may be arranged by subdividing $G$ at $p$ --- there is a natural bijection between $T_p G$ and the set of oriented edges $E \subset G$ with initial vertex $p$, such that $E$ corresponds to the germ of any orientation preserving parameterization $\gamma \from [0,1] \to E$.

Let $TG = \union_p T_p G$, a kind of ``tangent bundle''. Elements of $TG$ are often denoted with the symbol~``$d$'' for ``direction''. 

Given a continuous map of graphs $f \from G \to H$, we say $f$ is \emph{tight}\index{tight map} if it takes vertices to vertices and its restriction to each edge is either a constant or a tight edge path.

\begin{proposition}\label{PropTightening}
Any continuous map of graphs $f \from G \to H$ can be tightened, that is, it can be homotoped to a tight map.
\end{proposition}

\begin{proof}
The restriction of $f$ to the vertex set of $G$ may easily be homotoped so that its image is in the vertex set of $H$. Applying the homotopy extension lemma, we may assume that $f$ itself takes vertices to vertices. Then the restriction of $f$ to each edge of $G$ may be homotoped relative to its endpoints to be either a constant path or a tight edge path. 
\end{proof}

Proposition~\ref{PropTightening} will be applied very often without comment.

\medskip

Assuming $f \from G \to H$ is tight, we say that $f$ is \emph{nondegenerate} if it is nonconstant on each edge. More generally, given a point $p \in G$ we say that $f$ is \emph{nondegenerate at~$p$} if $f$ is nonconstant on each edge that intersects~$p$, in which case the \emph{derivative} $D_p f \from T_p G \to T_{f(p)} H$ is defined by requiring $D_p f(d)$ to be the germ at $f(p)$ of $f \composed \gamma$ where $\gamma \from [0,1] \to G$ is chosen to have germ~$d$. It follows that if $f$ is nondegenerate then the derivative map $Df \from TG \to TH$ is defined everywhere.

Given a nondegenerate tight map $f \from G \to H$, for each $p \in G$ we define an equivalence relation on $T_p G$ where $d \sim d'$ if and only if $D_p f(d) = D_p f(d')$; the equivalence classes are called the \emph{gates of $f$ at $p$}.\index{gate}

\subsection{Foldable maps and their construction.} 
\label{SectionFoldableConstruction}
Consider a finite core graph $G$, a graph $H$, and a tight map $f \from G \to H$. We say that $f$ is a \emph{foldable map}\index{foldable map} if $f$ has at least two gates at every point $p \in G$. To be more precise, $f$ is foldable if and only if the following two properties hold:
\begin{enumerate}
\item $f$ is nondegenerate.
\item For each $p \in G$ the map $D_p f \from T_p G \to T_{f(p)} H$ is nonconstant. 
\end{enumerate}
For example all the maps occurring in Section~\ref{SectionFoldSequenceExample} are foldable.

We describe a procedure which converts a tight map $f \from G \to H$ into a foldable map, except in an extreme case. The goal is to factor $f$ up to homotopy as $G \xrightarrow{q} G_0 \xrightarrow{f^0} H$ where $q$ is a collapse map and $f^0$ is foldable. The severity of the collapse map $q \from G \mapsto G_0$ is not controllable a priori, for example the subgraph $K \subset G$ which is collapsed by $q$ need not be a subforest. The extreme case occurs when $K=G$, in which case $G_0$ is a point, which does not qualify as a core graph; although this case is rare in applications, it must be accounted for in several statements, for example in Proposition~\ref{PropFoldableMapConstruction}~\pref{ItemFoldableIfPiOneInj} below. 


\smallskip
\textbf{Step 0:} We start by passing to the \emph{edgelet subdivisions}\index{edgelet subdivision} of $G$ and $H$, meaning that we first subdivide $H$ at the points of $f(\Vertices(G))$, and we then subdivide $G$ at the points of $(f')^\inv(\Vertices(H))$. Once this is done, the $1$-cells of $G$ and of $H$ are called \emph{edgelets},\index{edgelet} and $f$ is a cellular map taking each vertex to a vertex and taking each edgelet either to a vertex or to an edgelet by a map which restricts to a homeomorphism of edgelet interiors. 

Let $K_0 \subset G$ be the union of  edgelets on which $f$ is constant. Collapse each component of $K_0$ to a point, obtaining a quotient graph $G'$ and an induced map to $H$, thereby factoring $f$ as 
$$\xymatrix{
G \ar[r]_{[K_0]} \ar@/^1pc/[rr]^{f} & G' \ar[r]_{f'} & H
}$$
Foldability has been partially established, in that $f' \from G' \to H$ is nonconstant on the interior of each edge. We must now deal with vertices at which $f'$ has one gate.

\textbf{Step 1:} If $f'$ has at least two gates at each vertex of $G'$ then $f'$ is foldable and we stop. Otherwise, consider a vertex $v \in G'$ at which $f'$ has one gate. Let $e_1,\ldots,e_I$ be the oriented edgelets of $G'$ with initial vertex $v$, all mapping to the same oriented edge $\eta = f'(e_i) \subset H$. Denote the terminal endpoint of $\eta$ as $u$. Let $K_1 = e_1\union\cdots\union e_I$. We alter the map $f' \from G' \to H$ in two steps. First, do a homotopy relative to the complementary subgraph $G' \setminus K_1$, homotoping $f'$ to a map $f'_1 \from G' \to H$ that is constant on $K_1$, taking $K_1$ to the point $u$. Next, factor the map $f'_1$ as $G' \xrightarrow{[K_1]} G'' \xrightarrow{f''} H$ where the first factor $G' \mapsto G''$ collapses $K_1$ to a point. The second map $f'' \from G'' \to H$ is tight, as is easily seen.

\textbf{Induction:} Now repeat Step~1 inductively on the map $f'' \from G'' \to H$. The induction must stop because $G''$ has strictly fewer edgelets than $G'$. When the induction stops, the resulting map is foldable.

\medskip

We may summarize the construction of foldable maps in the following proposition, much of which is already evident from the induction just described.


\begin{proposition}\label{PropFoldableMapConstruction}
For any rank~$n$ marked graph $G$, any graph $H$, and any tight map $f \from G \to H$, after passing to edgelet subdivisions, there exists an edgelet subgraph $K \subset G$ and a homotopy commutative diagram of maps 
$$\xymatrix{
G \ar[r]_{[K]}^{q\vphantom{f^0}} \ar@/^2pc/[rr]^{f} & G_0 \ar[r]^{f^0} & H
}$$
with the following properties:
\begin{enumerate}
\item \label{ItemFirstMapIsCollapse}
The map $q \from G \to G_0$ is a quotient that collapses to a point each component of $K$. 
\item \label{ItemFoldableIfPiOneInj} 
If $K$ is proper in $G$ then $G_0$ is a core graph, the map $f^0 \from G_0 \to H$ is foldable, and $f^0$ is a cellular map with respect to edgelet subdivisions.
\item \label{ItemCommutativeOffK}
The diagram is commutative when restricted to the subgraph $G \setminus K$:
$$f \restrict (G \setminus K) = f^0 \circ q \restrict (G \setminus K)
$$
(i.e.\ the homotopy that commutes the diagram is stationary on $G \setminus K$).
\item\label{ItemNotPiOneInjective}
If $f$ is $\pi_1$-injective then $K$ is a forest, and hence $q$ is a homotopy equivalence. It follows that if $f$ is a homotopy equivalence then $f^0$ is a homotopy equivalence.
\item\label{ItemSubgraphFoldable}
If $C \subset G$ is a core subgraph on which the restriction $f \restrict C$ is already foldable then $C \subset G \setminus K$. Hence, applying~\pref{ItemCommutativeOffK}, the diagram is commutative when restricted to~$C$:
$$f \restrict C = f^0 \circ q \restrict C
$$
%
\end{enumerate}
Furthermore, the diagram above is algorithmically constructible given $f \from G \to H$.
\end{proposition} 

\noindent
\textbf{Remarks.} Item~\pref{ItemSubgraphFoldable} will be applied in Proposition~\ref{PropWhiteheadConjugacyTopological}, the topological translation of Whitehead's problem on conjugacy classes, and it will play a role in the solution of that problem.

%

\begin{proof} There are a few observations to make in order to verify some of the finer points of this proposition. Consider the sequence of maps up through the $j^{\text{th}}$ step of the induction:
$$\xymatrix{
G \ar[r]_{[K]} \ar@/^5pc/[rrrrrr]^{f}& G' \ar[r]_{[K']} \ar@/^3.5pc/[rrrrr]^{f'} & G'' \ar[r]_{[K'']} \ar@/^2pc/[rrrr]^{f''} & \cdots \ar[r]_{[K^{(j-1)}]} & G^{(j)} \ar[rr]^{f^{(j)}} & \quad & H
}$$
By induction, $f^{(j)}$ is a tight map which does not collapse any edge.
Also, each horizontal arrow (except the last one $f^{(j)}$) is a collapse map, a quotient map obtained by collapsing to a point each component of the graph denoted in brackets $[\cdots]$. If the tight map $f^{(j)}$ has at least 2 gates at each vertex of $G^{(j)}$ then it is foldable and the induction stops. Otherwise, the next step of the induction can be summarized as follows. The map $f^{(j)}$ has 1 gate at some vertex $v^j$. One takes $K^{(j)} \subset G^{(j)}$ to be the subgraph consisting of the oriented edgelets whose common initial vertex is $v^j$, and one homotopes $f^{(j)}$ to a map $f^{(j)}_1$ by a homotopy which is stationary on $G^{(j)} \setminus K^{(j)}$, so that $f^{(j)}_1$ takes $K^{(j)}$ to a point. By collapsing each component of~$K^{(j)}$ to a point one obtains the core graph $G^{(j+1)}$. The map $f^{(j+1)} \from G^{(j+1)} \to H$, which is induced by $f^{(j)}$ and the collapse map $G^{(j)} \mapsto G^{(j+1)}$, is a tight map which does not collapse any edge.

Conclusion~\pref{ItemFirstMapIsCollapse} is evident from the observation that a composition of subgraph collapses is a subgraph collapse (see Exercise~\ref{ExerciseForestCollapseComp}). We have already observed that Conclusion~\pref{ItemFoldableIfPiOneInj} holds once the induction stops. Conclusion~\pref{ItemCommutativeOffK} follows by induction: for each map $f^{(j)} \from G^{(j)} \to H$, if an edge $e \subset G^{(j)}$ is not in the tree $K^{(j)}$ then the homotopy from $f^{(j)}$ to $f^{(j)}_1$ is stationary on $e$ and the subsequent collapse of $K^{(j)}$ leaves~$e$ unscathed, taking it to an edge in $G^{(j+1)}$. Conclusion~\pref{ItemNotPiOneInjective} follows from the fact that $q$ is not $\pi_1$-injective if $K$ is not a forest, and the fact that every forest collapse is a homotopy equivalence.

Conclusion~\pref{ItemSubgraphFoldable} also requires an inductive proof. The induction hypothesis is that $C$ contains no edgelet that is collapsed by the collapse map $G \mapsto G^{(j)}$, the image $C^{(j)} \subset G^{(j)}$ of $C$ under this collapse map is a core subgraph, and the restricted map $f^{(j)} \restrict C^{(j)}$ is foldable. For each vertex $v \in C^{(j)}$, the map $f^{(j)}$ has at least two gates at~$v$, because the restricted map $f^{(j)} \restrict C^{(j)}$ has at least two gates at $v$. Since each edgelet in $K^{(j+1)}$ has an endpoint with one gate, and since $C^{(j)}$ is a core graph, it follows that no edgelet in $C^{(j)}$ is contained in $K^{(j+1)}$. Thus no edgelet of $C^{(j)}$ is collapsed by the map $G^{(j)} \mapsto G^{(j+1)}$, and so no edgelet of $C$ is collapsed by the map $G \mapsto G^{(j+1)}$. The restriction of this collapse map to $C^{(j)}$ is a quotient map that does no more than to identify some vertices, hence the image $C^{(j+1)} \subset G^{(j+1)}$ is still a core graph. The map $f^{(j+1)} \restrict C^{(j+1)}$ that is induced by the foldable map $f^{(j)} \restrict C^{(j)}$ is clearly foldable too.
%
%
%
%
\end{proof}

\begin{exercise}
\label{ExerciseFoldableComposition}
Prove that if $G \mapsto G' \mapsto G''$ are maps of graphs, if $G,G'$ are both core graphs, and if both of these maps are foldable, then the composition $G \mapsto G''$ is foldable.
\end{exercise}

\begin{exercise}
\label{ExercisePiOneTrivial}
In the statement of Proposition~\ref{PropFoldableMapConstruction}, prove that $G=K_0$ if and only if the induced homomorphism $f_* \from \pi_1 G \to \pi_1 H$ is trivial. Using this, describe an algorithm to decide whether a tight map $f \from G \to H$ is $\pi_1$-trivial.
\end{exercise}


\subsection{The first fold factorization.}
\label{SectionFirstFoldFactorization}
Once we apply Proposition~\ref{PropFoldableMapConstruction} to a tight map, factoring it as a collapse map followed by a foldable map, we next want to analyze that foldable map by ``folding''~it. Given a foldable map $f \from G \to H$, the long term goal of folding is to simplify $f$ step-by-step, factoring it as the composition of a sequence of ``fold maps''. The first step of this factorization of $f$ is guided entirely by the local behavior of~$f$, which falls into two major cases:
\begin{description}
\item[$f$ is locally injective:] For each point $p \in G$, the following equivalent conditions hold: $p$ has a neighborhood on which $f$ is injective $\iff$ $D_p f$ is injective $\iff$ each gate of $f$ at $p$ is trivial, consisting of a single direction.
\item[$f$ is not locally injective:] There exists a point $p \in G$ at which local injectivity fails: $p$ has no neighborhood on which $f$ is injective $\iff$ $D_p f$ is not injective $\iff$ some gate of $f$ at $p$ is nontrivial, consisting of two or more directions. It follows from foldability of $f$ that $p$ is a vertex of valence~$\ge 3$.
\end{description}
If $f$ is locally injective then it is already as simple as possible, and may be analyzed using Proposition~\ref{PropEndOfStallings} and Corollary~\ref{CorollaryLocallyInjectiveProperties}. 

If $f$ is not locally injective, the first step of simplifying it is to factor into two factors, the first factor being a fold map, as we now describe. 

The first fold of a fold factorization of $f$ is determined by choosing a vertex $v \in G$ and two directions $d \ne d' \in T_v G$ which are in the same gate of $f$; this choice is not unique, and hence fold sequences are not unique, as discussed at the end of Section~\ref{SectionFirstFold}. Let $E,E' \subset G$ be the oriented edges representing $d,d'$. It follows that there exist initial segments $\eta \subset E$, $\eta' \subset E'$ that are \emph{folded by $f$}, meaning that there exists an orientation preserving homeomorphism $h \from \eta \to \eta'$ such that $(f \restrict \eta') \composed h = f \restrict \eta$. We note that if $\eta,\eta'$ are folded by $f$ then the interiors of $\eta$ and $\eta'$ are disjoint, for otherwise it would follow that $E' = \overline E$, and that $\eta,\eta'$ have initial subsegments $\eta^\vp_0,\eta'_0$ that are folded by $f$ and that have a common terminal point $q \in \interior(E)$; but then it would follow that $f$ is not locally injective at $q$, contradicting that $f$ is foldable. We say that $\eta,\eta'$ are the \emph{maximal} initial segments folded by $f$ if there do not exist strictly longer initial segments $\eta \subsetneq \eta_1 \subset E$, $\eta' \subsetneq \eta'_1 \subset E'$ such that $\eta_1,\eta'_1$ are folded by~$f$.


Choose $v,d,d'$ as above, and choose $\eta,\eta'$ to be initial segments representing $d,d'$ that are folded by $f$. Define $G'$ to be the quotient graph obtained from $G$ by first subdividing $e,e'$ at the terminal points of $\eta,\eta'$ if necessary, thereby making $\eta,\eta'$ into edges, and then identifying $\eta$ to $\eta'$ bijectively, using the homeomorphism $h \from \eta \to \eta'$ described above. Letting $g \from G \to G'$ be the quotient map, we have factored $f$ as follows: 
$$\xymatrix{
G \ar[r]_{g^\vp} \ar@/^1pc/[rr]^{f} & G' \ar[r]_{f'} & H
}$$
We shall call this the \emph{first fold factorization} of $f$ determined by $d,d',\eta,\eta'$. More specifically this is called the \emph{maximal first fold factorization} of $f$ determined by $d,d'$ if, in addition, $\eta,\eta'$ are the maximal initial segments representing $d,d'$. The two factors of a first fold factorization will be considered separately: the ``fold factor'' $g \from G \to G'$; and the ``foldable quotient'' $f' \from G' \to H$ (we have not yet verified that $f'$ is actually a foldable map; see Proposition~\ref{ThmStallingsFolds}).


\begin{exercise} 
\label{ExerciseMaximalFoldConsequences}
Consider the point $q \in G'$ which is equal to the common image of the terminal points of $\eta,\eta'$. What properties of $q$ and of the gates at $q$ of $f' \from G' \to H$ are always true assuming maximality of $\eta,\eta'$, but sometimes fail when maximality is dropped?
\end{exercise}

\subsection{Folds.} Abstracting the above discussion, consider a rank~$n$ core graph $G$, and another core graph $G'$ (in this generality we do not assume anything about the rank of $G'$). Consider a map $g \from G \to G'$ which is foldable and surjective (hence $g$ is a quotient map). We say that $g$ is a \emph{fold} if there exist oriented edges $e \ne e' \subset G$ with the same initial vertex~$v$, there exist initial segments $\eta \subset e$ and $\eta' \subset e'$ with disjoint interiors, and there exist orientation preserving parameterizations $h \from [0,1] \to \eta$ and $h' \from [0,1] \to \eta'$, such that if $X \subset G$ is a subset with at least two points and if $g$ is constant on $X$ then there exists $t \in (0,1]$ such that $X = \{h(t),h'(t)\}$. 

Fold maps $g \from G \to G'$ are classed into several types. The first type is somewhat exceptional:
\begin{description}
\item[$g$ is a bigon fold:] This means that $\eta$, $\eta'$ have the same terminal point $w$. It follows that $\eta=e$ and $\eta'=e'$ and that they have disjoint interiors. We say that $e,e'$ forms a \emph{bigon} which is folded by~$g$. 
\end{description}
The exceptional nature of a bigon fold is explained in Exercise~\ref{ExerciseBigonFoldNotPiOneInjective}: bigon folds are not $\pi_1$-injective. To prove the ``It follows\ldots'' statement in the definition of a bigon fold, were it not so then $e,e'$ would be opposite orientations on the same edge, their common terminal point $w$ would be an interior point of that edge, and $\eta,\eta'$ would be opposite halves of that edge oriented towards~$w$; but this implies that $g$ is not locally injective at $w$, contradicting that $g$ is foldable. 

The non-bigon folds are classified as follows:
\begin{description}
\item[$g$ is a partial fold,] meaning that both of $\eta \subset e$ and $\eta' \subset e'$ are proper initial segments. We say that \emph{$g$ partially folds $e$ and $e'$}.
\item[$g$ is a full fold,] meaning that at least one of the two inclusions $\eta \subset e$ or $\eta' \subset e'$ is not a proper initial segment, instead is the whole edge. There are two subcases, depending on whether one or both of $\eta,\eta'$ is the whole edge.
\begin{description}
\item[$g$ is an improper full fold:] This means that $\eta = e$ and $\eta' = e'$, and we say that $g$ \emph{improperly folds $e$ and $e'$}. 
\item[$g$ is a proper full fold:] This means that exactly one of the inclusions is the whole edge. By switching notation if necessary we may assume that $\eta = e$ and that $\eta' \subsetneq e'$, and we say that \emph{$g$ properly folds $e'$ over $e$}.
\end{description}
\end{description}
The reader who looks through Figures~\ref{FigureFoldPath01},~\ref{FigureFoldPath02} and~\ref{FigureFoldPath03} will see examples of some (but not all) classes of folds.

Next we summarize some properties of non-bigon folds, relating them to cells in outer space.

\begin{lemma}\label{LemmaNonbigonFolds}
Each non-bigon fold $g \from G \to G'$ is a homotopy equivalence; it follows that to each marking $\rho \from R_n \to G$ there corresponds a marking $G \xrightarrow{g \composed \rho} G'$, and this correspondence induces a bijection between homotopy classes of markings of $G$ and of~$G'$. Furthermore when related markings have been specified in this manner, thereby giving $G,G'$ the structure of marked graphs, then there exists a marked graph $G''$ for which there are cell inclusions $\Delta(G) \subset \Delta(G'') \supset \Delta(G')$.
\end{lemma}

\begin{proof} During this proof, once it is evident that $g$ is a homotopy equivalence (perhaps by describing a homotopy inverse of $G$, or by describing $g$ as a product of homotopy equivalences), we shall assume that marked graph structures on $G,G'$ have been specified as in the statement of the lemma, and we will explicitly describe the relation between the cells $\Delta(G)$ and~$\Delta(G')$.

When $g$ is a partial fold it has a homotopy inverse $G' \mapsto G$ defined by collapsing the segment $g(\eta)=g(\eta') \subset G'$ to a point, and so $\Delta(G) \subset \Delta(G')$ (with codimension~$0$ or~$1$ depending on whether the valence of $v$ is $3$ or $\ge 4$). 

When $g$ is an improper full fold, the edges $e$ and $e'$ end at distinct vertices $w,w'$ (because $g$ is not a bigon fold), and we break the proof into various cases. Consider first the case that $v$ has valence~$3$, and that $w,w'$ are both distinct from $v$, hence $e \union e'$ is a tree, in fact an embedded arc; in this case $g$ is homotopic to a homotopy equivalence that collapses $e \union e'$ to a point, and so $\Delta(G) \supset \Delta(G')$ (with codimension $0$, $1$ or $2$ depending on how many of $w,w'$ have valence~$2$). In the remaining cases either $v$ has valence~$\ge 4$ or one $w,w'$ is equal to $v$ (hence one of $e,e'$ is a loop edge). In these cases the map $G \xrightarrow{g} G'$ factors as $G \mapsto G'' \mapsto G'$ where the first factor $G \mapsto G''$ is a partial fold and the second factor $G'' \mapsto G'$ is an improper fold fold at a vertex of valence~$3$ that is distinct from the endpoints of the folded edges, and hence $\Delta(G) \subset \Delta(G'') \supset \Delta(G')$. 

Consider finally the case that $g$ is a proper full fold of an edge $e'$ over an edge $e$. If $e$ is not a loop edge and $v$ has valence~$3$ then $g$ is homotopic to a map that collapses $e$ and so $\Delta(G) \supset \Delta(G')$ (with codimension~$0$ or~$1$ depending on whether the terminal vertex of~$e$ has valence~$2$). Otherwise, if $e$ is a loop edge, or if $v$ has valence~$\ge 4$, then $g$ factors as $G \mapsto G'' \mapsto G'$ where the first factor is a partial fold and the second factor is a proper full fold at a valence~$3$ vertex over a non-loop edge, and so $\Delta(G) \subset \Delta(G'') \supset \Delta(G)$.
\end{proof}

\begin{exercise}\label{ExerciseBigonFoldNotPiOneInjective}
Suppose that $g \from G \to G'$ is a bigon fold as described above. Prove that the induced homomorphism $g_* \from \pi_1(G,x) \to \pi_1(G',x')$ is surjective but not injective (with arbitrary $x \in G$, and $x'=g(x) \in G'$), and that $\rank(G') = \rank(G)-1$. Describe a specific circuit $c$ in $G$ such that the kernel of $g_*$ is normally generated by any based loop in $(G,x)$ that is freely homotopic to $c$.
\end{exercise}

\subsection{Stallings Fold Theorem}
\label{SectionFoldSeqConstr}
In this section we state and prove Stallings Fold Theorem in the context of a foldable map defined on a finite rank core graph, building on the construction of the first fold factorization that is carried out in Section~\ref{SectionFirstFoldFactorization}. The idea is to apply that construction inductively until no further fold is possible, in which case the final  map is locally injective. 

Here is the full statement of the theorem in the form that we shall need, which is obtained from the more general version in Stallings paper \cite{Stallings:folding} by specializing to the setting of finite core graphs.


\begin{theorem}[Stallings Fold Theorem]
\label{ThmStallingsFolds}
Given a finite rank core graph $G$, a graph~$H$, and a foldable map $f \from G \to H$, there exists a commutative diagram of foldable maps of core graphs
$$\xymatrix{
G =G_0 \ar[r]_{g_1}\ar@/^6pc/[rrrrrr]^<<<<<<<<<<<<<<<<<<<<<<<<<{f=f^0} 
& G_1 \ar[r]_{g_2}\ar@/^4.5pc/[rrrrr]^<<<<<<<<<<<<<<<<<<<<{f^1} 
& G_2 \ar[r]_{g_3}\ar@/^3pc/[rrrr]^<<<<<<<<<<<<<<<{f^2} & \cdots \ar[r]_{g_{K-1}} 
& G_{K-1} \ar[r]_{g_K}\ar@/^1pc/[rr]^<<<<<{f^{K-1}} & G_K \ar[r]_{f^K} & H
}$$
such that the following hold:
\begin{enumerate}
\item\label{ItemFoldableSequence}
Each of the compositions $f^i_j = f_j \composed \cdots \composed f_{i+1} \from G_i \to G_j$ is foldable.
\item\label{ItemFirstFold}
Each factorization 
$$f^i \from G_i \xrightarrow{g_{i+1}} G_{i+1} \xrightarrow{f^{i+1}} H
$$
is a maximal first fold factorization determined by some pair of directions in $G_i$ that are in the same gate of the map $f^i$.
\item\label{ItemContinuingFolds}
For each $i=1,\ldots,K$ exactly one of the following occurs:
\begin{enumerate}
\item $\rank(G_i) = \rank(G_{i-1})$, which occurs if and only if $g_i$ is a homotopy equivalence, equivalently $g_i$ is not a bigon fold.
\item $\rank(G_i) = \rank(G_{i-1})-1$, which occurs if and only if $g_i$ is not $\pi_1$-injective, equivalently $g_i$ is a bigon fold.
\end{enumerate}
\item \label{ItemStoppingStallingsFolds} For each $k=1,\ldots,K$, the map $G_k \xrightarrow{f^k} H$ is a local injection if and only if $k=K$.
\item \label{ItemStallingsFoldPath}
If $f$ is a homotopy equivalence then none of the folds $g_i$ is a bigon fold, and the map $f^K$ is an injective homotopy equivalence. If furthermore $H$ is a core graph then $f^K$ is a homeomorphism.
\end{enumerate}
Furthermore there is an algorithm which, given as input $G$, $H$, and $f \from G \to H$, produces as output a diagram of graphs and maps satisfying the conclusions as above.
\end{theorem}
 
The sequence of maps $G=G_0 \mapsto\cdots\mapsto G_K = H$ produced by Proposition~\ref{ThmStallingsFolds} is called a \emph{maximal fold factorization} of the map $f \from G \to H$. 

\begin{proof} Conclusion~\pref{ItemStallingsFoldPath} will be deduced at the very end of the proof. The construction of the commutative diagram, and the deduction of conclusions~\pref{ItemFoldableSequence}--\pref{ItemStoppingStallingsFolds}, is carried out by induction, using the same ideas as in the examples of maximal first fold factorizations described in Section~\ref{SectionFirstFoldFactorization}. 

To prepare for induction, we first prove: 

\begin{lemma}[Left-cancellation for foldable maps.] 
\label{LemmaLeftCancelFoldable}
For any finite rank core graph $G$, any graphs $G'$, $H$, any foldable map $f \from G \to H$ and any factorization $G \xrightarrow{g} G' \xrightarrow{f'} H$ of~$f$, if $g$ is surjective and foldable then $G'$ is a finite rank core graph and $f'$ is foldable. It follows that if $g$ is a fold map then $f'$ is foldable. Furthermore, if  $f$ is a homotopy equivalence then $g$ is not a bigon fold and $f'$ is a homotopy equivalence.
\end{lemma}

\begin{proof} Note first that $G'$ is the continuous image of the connected space $G$, hence $G'$ is connected. Also $G'$ cannot be a point, or else $g$ is constant hence $f$ is constant, violating foldability of $f$. 

Consider a point $p \in G'$ of valence $K \ge 1$.
Choose a regular neighborhood expressed as a star graph $N_p = \eta_1 \union\cdots\union \eta_K$ based at $p$, meaning a union of oriented segments $\eta_k$ called ``rays'' that intersect pairwise only at their common initial point $p$. If $N_p$ is chosen sufficiently small then for each $q \in g^\inv(p)$, the component $N_q$ of $g^\inv(N_p)$ that contains $q$ is a star graph based at $q$, each of whose rays maps homeomorphically to $\eta_k$ for some $k=1,\ldots,K$. For each ray $\eta_k$ of $N_p$, since $g$ is surjective there exists $q \in g^\inv(p)$ and a ray $\eta'$ of $N_q$ which $g$ maps homeomorphically to $\eta_k$; also, since $f$ is foldable it is locally injective on $\eta'$; it follows that $f'$ is locally injective on $\eta_k$. Another consequence of foldability of $f$ is the following fact: 
\begin{itemize}
\item For any $q \in g^\inv(p)$ there are two rays $\eta_{q,1},\eta_{q,2}$ in $N_q$ having distinct image under~$f$, hence $g(\eta_{q,i})$ for $i=1,2$ are two rays in $N_p$ having distinct image under $f'$. 
\end{itemize}
Applying this fact when $K=1$ leads to a contradiction, hence $K \ge 2$; since $G$ is connected and compact it follows that $G'$ is a finite rank core graph. Applying this fact again when $K=2$, it follows that $f'$ is a tight map that it is locally injective on each edge of~$G'$. Applying this fact once more for all $K \ge 2$, it follows that $f'$ is foldable.

To prove the ``Furthermore'' clause, suppose that $f$ is a homotopy equivalence. If $g$ were a bigon fold then the induced homomorphism $g_* \from \pi_1(G) \to \pi_1(G')$ would not be injective (by Exercise~\ref{ExerciseBigonFoldNotPiOneInjective}) hence $f_* = f'_* \circ g_* \from \pi_1(G,x) \to \pi_1(H,f(x))$ would not be injective, contradicting that $f$ is a homotopy equivalence. Since $f$ and $g$ are both homotopy equivalences, it follows that $f'$ is a homotopy equivalence.
\end{proof}

Proceeding now by induction, assuming it is known that $f_i \from G_i \to H$ is foldable and that a maximal first fold factorization $f_i \from G_i \xrightarrow{g_i} G_{i+1} \xrightarrow{f^{i+1}} H$ has been chosen, it follows by Lemma~\ref{LemmaLeftCancelFoldable} that $f^{i+1}$ is foldable and the induction continues. Furthermore, by Exercise~\ref{ExerciseFoldableComposition} each of the composed maps $f^i_j \from G_i \to G_j$ is foldable. This proves items~\pref{ItemFoldableSequence} and~\pref{ItemFirstFold}. 

Next we show that the induction stops; this is where we use maximality of the first fold factorizations. Consider the ``edgelet subdivisions'' of the $G_i$'s, defined by first subdividing $H$ along the set $f^0(\Vertices(G_0))$, then subdividing $G_0$ along $f_0^\inv(\Vertices(H))$, and then inductively subdividing each subsequent $G_{i}$ along $g_{i}(\Vertices(G_{i-1}))$. Each foldable map $f^i \from G_i \to H$ and each fold map $g_i \from G_{i-1} \to G_i$ is a nondegenerate edgelet map, meaning that it maps each vertex to a vertex and each edgelet to an edgelet. This is proved by induction. Suppose that $f^{i-1} \from G_{i-1} \to H$ is a nondegenerate edgelet map. Let $E,E' \subset G_{i-1}$ be the oriented edges with initial segments $\eta \subset E$, $\eta' \subset E'$ folded by $g_i$. Since $\eta,\eta'$ are the \emph{maximal} initial segments of $E,E'$ that are folded by $f^i$, the common point $f^{i-1}(q)=f^{i-1}(q') \in H$ to which the terminal points $q \in \eta$, $q \in \eta'$ are mapped by~$f^{i-1}$ is a vertex of $H$, indeed it is a vertex of valence~$\ge 3$ (see Exercise~\ref{ExerciseMaximalFoldConsequences}). It follows that $\eta,\eta'$ are subcomplexes of the edgelet subdivision of $G_{i-1}$, from which it follows further that $g_i$ and $f^{i+1}$ are nondegenerate edgelet maps.

Clearly the number of edgelets of $G_i$ is a positive integer. Each fold $g_i \from G_{i-1} \to G_i$ is a nondegenerate edgelet map, identifying the edgelets of $\eta$ to the edgelets of $\eta'$ in pairs. Since $g_i$ is a surjective, it follows that $G_{i}$ has strictly fewer edgelets than~$G_{i-1}$. The induction must therefore stop. If $f^k$ is not a local injection then some vertex of $G_k$ has a gate of cardinality $\ge 2$, hence we live to fold again at $G_k$ and the induction has not yet stopped, proving one direction of item~\pref{ItemStoppingStallingsFolds}. But if $f^k$ is a local injection, then no fold at $G_k$ is possible, proving the other direction of~\pref{ItemStoppingStallingsFolds}. Item~\pref{ItemContinuingFolds} has already been noted in the definition of~folds.

We turn to the proof of Conclusion~\pref{ItemStallingsFoldPath}. We are assuming that $f=f^0$ is a homotopy equivalence. Inductively applying the ``Furthermore'' clause of Lemma~\ref{LemmaLeftCancelFoldable}, it follows that each $g_i$ is not a bigon fold and that each $f^i$ is a homotopy equivalence. We know from~\pref{ItemStoppingStallingsFolds} that $f^K$ is a local injection. Knowing that $G_K$ and $H$ are core graphs and that $f^K$ is a locally injective homotopy equivalence, by applying Corollary~\ref{CorollaryLocallyInjectiveProperties} it follows that $f^K$ is a homeomorphism.
\end{proof}

\section{Determining $\pi_1$-injectivity and surjectivity.}

We now have enough tools to demonstrate the solvability of some of the Nielsen/Whitehead problems.

\subsection{$\pi_1$-injectivity.}

\begin{corollary} Given a tight map of finite graphs $f \from G \to H$, the following problems are algorithmically solvable:
\begin{enumerate}
\item\label{ItemInjectiveSolvable}
Determine whether $f$ is $\pi_1$-injective.
\item\label{ItemHESolvable}
Determine whether $f$ is a homotopy equivalence.
\end{enumerate}
The following problem is also algorithmically solvable:
\begin{enumeratecontinue}
\item\label{ItemFreeBasisSolvable}
Determine whether an $n$-tuple of elements of $F_n$ is a free basis.
\end{enumeratecontinue}
\end{corollary}

\begin{proof} To solve problems \pref{ItemInjectiveSolvable} and~\pref{ItemHESolvable}, first apply Proposition~\ref{PropFoldableMapConstruction} to construct a homotopy factorization of $f$ as a collapse map $G \xrightarrow{[K]} G_0$ followed by second map $G_0 \mapsto H$. If $K$ is not a forest, $f$ is not $\pi_1$-injective. Otherwise $K$ is a forest and the second map is foldable, and we apply Proposition~\ref{ThmStallingsFolds} to factor it as a fold a fold sequence. If the fold sequence contains a bigon fold, then $f$ is not $\pi_1$-injective. Otherwise all of the maps $G \mapsto G_0 \mapsto\cdots\mapsto G_K$ are $\pi_1$-isomorphisms, and the remaining map $f^K \from G_K \to H$ is locally injective and therefore $\pi_1$-injective by Corollary~\ref{CorollaryLocallyInjectiveProperties}, so it follows that $f^K$ is $\pi_1$-injective, and therefore $f$ is $\pi_1$-injective. Furthermore, by Corollary~\ref{CorollaryLocallyInjectiveProperties} $f^K$ is a $\pi_1$-isomorphism if and only if $f^K$ is an injection and its image is a deformation retraction of $H$, which holds if and only if $f$ is a $\pi_1$-isomorphism, equivalently $f$ is a homotopy equivalence.

Item~\pref{ItemFreeBasisSolvable} is settled once we note that a tuple $w_1,\ldots,w_n \in F_n$ is a free basis if and only if the self-map of the rose $R_n$ defined by mapping each edge $s_i$ to the edge path $w_i$ is a homotopy equivalence.
\end{proof}

\begin{exercise}
Go back and look at Exercise~\ref{ExerciseNonFreeBasis} again.
\end{exercise}

\subsection{$\pi_1$-surjectivity and Stallings graphs.}

\label{SectionPiOneSurjectivity}
One of the most well known applications of Stallings fold sequences is the solution of the following problems:
\begin{description}
\item[Algebraic version:] Given elements $w_1,\ldots,w_k$ of a free group $F_n$, determine whether $w_1,\ldots,w_k$ generate $F_n$, and more generally determine a free basis for the subgroup of $F_n$ generated by $w_1,\ldots,w_k$.
\end{description}
Representing $w_1,\ldots,w_k$ by reduced words, and using those words to construct a tight map of rose graphs $R_k \to R_n$, the above problem becomes a special case of the following
\begin{description}
\item[Topological intepretation:] Given a core graph $G$, a graph $H$, and a tight map of graphs $f \from G \to H$, determine whether $f$ is $\pi_1$-surjective, and determine a free basis for the image of $f_* \from \pi_1 G \to \pi_1 H$ (with respect to appropriate choices of base points).
\end{description}
To solve this problem, consider a tight map of graphs $f \from G \to H$. Apply Proposition~\ref{PropFoldableMapConstruction} to factor $f \from G \xrightarrow{[K]} G_0 \xrightarrow{f_0} H$ as a collapse map followed by a foldable map (in the extreme case where $G_0$ is a point, the map $f$ is $\pi_1$-trivial and we are done). Note that $\image(f_*) = \image((f_0)_*)$ with respect to any base point of $G$ and its images in $G_0$ and $H$. Thus we are reduced to the case that $f$ is foldable.

Assuming $f$ is foldable, apply Proposition~\ref{ThmStallingsFolds} to obtain a fold factorization 
$$(*)\qquad\qquad f \from G=G_0 \mapsto G_1 \mapsto \cdots \mapsto G_{K-1} \mapsto G_K \mapsto H
$$
Since every fold map is a $\pi_1$-surjection, all of the maps $G_{k-1} \mapsto G_k$, $k=1,\ldots,K$, are $\pi_1$-surjections, and so the final map $G_K \to H$ has the same $\pi_1$-image as the original map $f \from G \to H$. From Proposition~\ref{ThmStallingsFolds} it follows that the map $G_K \to H$ is locally injective, and so from Proposition~\ref{PropEndOfStallings} it is $\pi_1$-injective. Furthermore the map $G_K \mapsto H$ is $\pi_1$-surjective if and only if it is a homeomorphism, thereby determining $\pi_1$-surjectivity. More generally, suppose that we choose a base point $p_0 \in G_0$, and we let $p_i \in G_i$ and $q \in H$ be its images along the fold sequences. A free basis for the image of $\pi_1(G_K,p_K) \mapsto \pi_1(H,q)$ is determined by first writing out a free basis for $\pi_1(G_K,p_K)$ in the usual fashion: choose a maximal tree $T \subset G_K$, enumerate and orient the edges of $G_K \setminus T$ as $E_1,\ldots,E_M$, and for each $m$ let $\gamma_m$ be the loop in $G_K$ which goes from $p_K$ through $T$ to the initial endpoint of $E_m$, then across $E_m$, then from the terminal endpoint back through $T$ to~$p_K$. Mapping those loops $\gamma_1,\ldots,\gamma_M$ over to $H$ we obtain a free basis for the image subgroup in~$\pi_1(H,q)$.

The final immersion $G_K \mapsto H$ is often denoted by subdividing the edges of $G_K$ into edgelets that are oriented and labelled so as to indicate their images in~$H$, much as was done in our examples back in Section~\ref{SectionFoldSequenceExample}. When this is done, the graph $G_K$ is sometimes called the \emph{Stallings graph} of the subgroup. Reduced words representing generators of the image subgroup in $\pi_1(H,q)$ can then be read off from the edge labels around the lops $\gamma_1,\ldots,\gamma_M$ described above.
For example, Figure~\ref{FigureFoldPath03} depicts the Stallings graph of the subgroup of $F_2 = \<a,b\>$ generated by the words $aabababaaba$ and $aabaaba$.

\bigskip

The methods of this section can also be used to prove that the group $F_n$ is \emph{co-Hopfian}, meaning that every surjective homomorphism $F_n \mapsto F_n$ is injective:

\begin{exercise} Prove that $F_n$ is co-Hopfian.
\end{exercise}

%

\section{Exploring conjugacy classes using fold paths}
\label{SectionWhiteheadStartup}

\hfill{\emph{What do the simple folk do?}}

\smallskip

\hfill  --- from the musical \emph{Camelot}, by Lerner and Lowe

\bigskip

In this section we ponder Whitehead's problem on conjugacy classes. Consider a set of conjugacy classes $\{c_1,\ldots,c_M\} \subset \C(F_n)$. We wish to determine whether these represent a partial free basis of $F_n$. One necessary condition --- derived using that their homology classes in the free abelian group $H_1(F_n;\Z)$ form a partial basis over $\Z$ --- is that each $c_i$ is a root-free circuit, and if $i \ne j$ then $c_i$ is equal to neither $c_j$ nor $c_j^\inv$.

Our strategy for attacking Whitehead's problem is to explore the possibility of a stronger necessary condition. To start, we make the
\begin{description}
\item[Assumption:] The set of conjugacy classes $\{c_1,\ldots,c_M\}$ \emph{DOES} represent a partial free basis of $F_n$.
\end{description}
Guided by our topological investigations up to this point, we shall investigate a simple natural construction of a fold path, looking for simple patterns which might reveal general conditions that the set $\{c_1,\ldots,c_M\}$ satisfies under the above assumption. In the sections to follow we will turn this exploration into a theory, formalizing various concepts which arise naturally in our exploration.

Using the above assumption, we may apply the topological criterion of Proposition~\ref{PropWhiteheadConjugacyTopological} to obtain a marked graph $G$ with marking $\rho \from R_n \to G$ such that $c_1,\ldots,c_M$ are represented in $G$ by a circuit family $\sigma=\sigma_1 \union\cdots\union \sigma_M$ in $G$ where $\sigma_1,\ldots,\sigma_M$ are pairwise disjoint embedded circles. We may always take $G$ to be a bola graph, as shown in the proof of Proposition~\ref{PropWhiteheadConjugacyTopological}. Choose a homotopy inverse $f \from G \to R_n$ of the marking map~$\rho$. Marking $R_n$ by the identity map, the map $f$ preserves marking. We may homotope $f$ so that it is a tight map that restricts to a local injection on circle family $\Sigma = \sigma_1 \union \cdots \union \sigma_M$: first homotope the restriction $f \restrict \Sigma$ to be a local injection that takes vertices to vertices; then apply the homotopy extension theorem; then tighten $f$ on each edge of $G \setminus \Sigma$. 


Using Proposition~\ref{PropFoldableMapConstruction} and~\ref{ThmStallingsFolds} together, we factor $f$ as a collapse followed by a Stallings fold sequence, as shown in the following diagram:
$$\xymatrix{
G \ar[r]^{[U]}_{q}\ar@/^7.5pc/[rrrrrrr]^<<<<<<<<<<<<<<<<<<<<<<<<<<{f} 
& G_0 \ar[r]_{g_1}\ar@/^6pc/[rrrrrr]^<<<<<<<<<<<<<<<<<<<<<{f^0} 
& G_1 \ar[r]_{g_2}\ar@/^4.5pc/[rrrrr]^<<<<<<<<<<<<<<<<<<<{f^1} 
& G_2 \ar[r]_{g_3}\ar@/^3pc/[rrrr]^<<<<<<<<<<<<<<<{f^2} & \cdots \ar[r]_{g_{L-1}} 
& G_{L-1} \ar[r]_{g_L}\ar@/^1pc/[rr]^<<<<{f^{L-1}} 
& G_L \ar@{=}[r]_{f^L} & R_n
}$$
From the conclusions of Proposition~\ref{PropFoldableMapConstruction} and~\ref{ThmStallingsFolds}, we obtain various properties of this diagram. First, every map $g_i$ is a fold and every composition $g^j_i = g_i \composed \cdots g_{j+1} \from G_j \to G_i$ is foldable. Next,  since $f$ is a homotopy equivalence, it follows that the collapse graph $U \subset G \setminus \Sigma$ is a forest, no of the fold map $g_l \from G_{l-1} \to G_l$ is a bigon fold, and the final map $f^L \from G_L \to R_n$ is a homeomorphism. Also, the entire diagram is commutative except for the leftmost triangle, which \emph{is} commutative on $G \setminus U$ but only homotopy commutative on $U$ itself. And finally, since the map $f$ is locally injective on the disjoint union of circles $\Sigma \subset G$, by conclusion~\pref{ItemCommutativeOffK} of Proposition~\ref{PropFoldableMapConstruction} it follows that the homotopy by which $f$ is altered is stationary on the circle family $\Sigma$. 


Since $G$ is a marked graph and $f \from G \to R_n$ preserves marking, and since all the maps along the fold sequence are homotopy equivalences, we may push the marking forward all along that sequence, hence we may assume that each $G_l$ is a marked graph and each fold map $g_l \from G_{l-1} \to G_l$ preserves marking. 

We shall walk through the terms of this fold sequence to explore the representations of the conjugacy classes $c_1,\ldots,c_M$ by circuit families in each~$G_m$. These conjugacy classes are all represented by a circuit family $\gamma_l \from \Sigma \to G_l$, and this sequence of circuit families is preserved by the fold maps $g_l$, in the following sense. Since $f$ restricts to an immersion on $\Sigma$, and since the collapse map $q$ restricts to an immersion on $\Sigma$ it follows that the composition of the inclusion $i_\Sigma \from \Sigma \hookrightarrow G$ with the collapse $q \from G \to G_0$ is an immersion $\gamma_0 \from \Sigma \to G_0$; next, the composition of $\gamma_0$ with the fold map $g_1 \from G_0 \to G_1$ is the immersion $\gamma_1 \from \Sigma \to G_1$; and so on inductively. In general we have a sequence of immersions
$$\gamma_l = g_l \circ \cdots \circ g_1 \circ i_\Sigma \from \Sigma \to G_l, \quad 1 \le l \le L
$$
satisfying the property $g_l \circ \gamma_{l-1} = \gamma_l$.

Starting with the inclusion $\Sigma\hookrightarrow G$, we want to consider the inductive effect of the collapse map $q$ and the subsequent folds $g_1,\ldots,g_L$ on the circuit families $\gamma_l \from \Sigma \to G_l$. But we will not be concerned with global effects. Instead we shall be simple folk, and shall consider only local effects: instead of examining how each component of $\Sigma$ winds around as a whole in $G_l$, we focus solely on the infinitesmal behavior of $\gamma_l$ near each vertex of $G_l$, namely we focus on how $\gamma_l$ crosses through a regular neighborhood of each vertex.

Why should we expect that such a simple idea will work, that it will give us a good, strong necessary condition for $\{c_1,\ldots,c_M\}$ to represent a partial basis? Perhaps one might be motivated by the experience of differential topology and differential geometry, where strong global conditions often follow from simple infinitesmal assumptions. In fact we have already seen this approach work in our construction and application of fold maps: fold factorizations themselves are constructed by infinitesmal information, namely by observing where a foldable map fails to be locally injective. 

\textbf{Step 0: The collapse map $q \from G \xrightarrow{[U]} G_0$.} To describe the effect of $q$ on the embedding $\Sigma \hookrightarrow G$, we first note that no edge of $\Sigma$ can be collapsed by the map $q \from G \xrightarrow{[U]} G_0$, because $q$ restricts to an immersion on $\Sigma$. Nonetheless, an arc in $U$ connecting two points of $\Sigma$ might be collapsed by $q$, hence $q$ takes those two points to a single vertex of $G_0$. More precisely, fix a component $\tau$ of $U$. Since $\Sigma \subset G \setminus U$, the intersection $\tau \intersect \Sigma$ is a finite set of vertices of $G$; suppose there are $A$ such vertices (other than this finite-to-one behavior over the $g$-images of the components $\tau$ of $U$, the immersion $\gamma_0$ is one-to-one). The immersion $\gamma_0$ is therefore $A$-to-$1$ over the image vertex $u_0 = q(\tau) \in G_0$. Furthermore, at the vertex $u_0$ there is a collection of $A$ different ``turns'' each taken by different strands of the immersion $\gamma_0$; see Figure~\ref{FigureCollapsingSigma}. We shall shortly formalize the concept of a ``turn''; for now we write these turns as
$$\{E_1,E'_1\},\{E_2,E'_2\},\ldots,\{E_A,E'_A\}
$$
where $E_1,E'_1,E_2,E'_2,\ldots,E_A,E'_A$ is a set of $2A$ distinct elements of the direction set $T_{u_0} G_0$, represented by $2A$ distinct oriented edges with initial vertex~$u_0$.

\begin{figure}
\centerline{
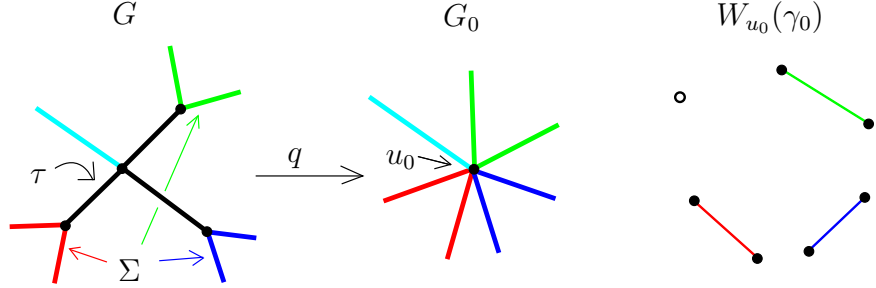 
}
\caption{\textbf{The collapse map $q \from G \xrightarrow{[U]} G_0$.} The figures shows a regular neighborhood of a component $\tau$ of $U$. This regular neighborhood intersects the circle family $\Sigma \subset G$ in three strands --- red, dark blue, and green --- and $\Sigma$ intersects $\tau$ in three points. Those three points are identified to the single point $u_0=q(\tau)$ in $G_0$. The immersion $\gamma_0 \from \Sigma \to G_0$ is 3--to--1 over the point~$u_0$. The manner in which $\gamma_0(\Sigma)$ crosses $u_0$ is recorded in the Whitehead graph $W_{u_0}(\gamma_0)$, which is a disjoint union of three edges.  The hollow dot represents a vertex of $T_{u_0}(G_0)$ that is not a vertex of $W_{u_0}(\gamma_0)$, because no strand of $\Sigma$ crosses the cyan colored edge of $G$.}
\label{FigureCollapsingSigma}
\end{figure}

Figure~\ref{FigureCollapsingSigma} shows the case where $A=3$, together with an abstract representation of the collection of turns $\{E^\vp_1,E'_1\},\{E^\vp_2,E'_2\},\{E^\vp_3,E'_3\}$ as the edges of the \emph{Whitehead graph} of the immersion $\gamma_0$ at the vertex $u_0$, a finite graph that we denote $W_{u_0}(\gamma_0)$. One can think of the edge $\turn{E^\vp_a E'_a}$ corresponding to the turn $\{E^\vp_q,E'_a\}$ as an abstract representation of those 2-word subpaths of the immersion $\gamma_0$ having one of the two forms $\overline E^\vp_a E'_a$ or $\overline E'_a E^\vp_a$. Formally this Whitehead graph $W_v$ is a subgraph of a complete graph, namely the complete graph with vertex set equal to the direction set $T_{u_0} G_0$, a graph which we shall denote $\turngraph_{u_0} G_0$. For later emphasis we note that this particular Whitehead graph $W_{u_0}$ is a pairwise union of pairwise disjoint edges of the graph of turns $\Lambda_{u_0} G_0$.

\medskip
\textbf{Step 1: The first fold $g_1$.} Continuing now with the example depicted in Figure~\ref{FigureCollapsingSigma}, we shall examine the effects of the first fold map $g_1 \from G_0 \to G_1$ on the immersion \hbox{$\gamma_0 \from \Sigma \to G_0$,} specifically the manner in which $g_1$ alters the Whitehead graphs of $\gamma_0$ at the vertices of $G_0$ to produce the Whitehead graphs of $\gamma_1$ at the vertices of~$G_1$. 

Pick a vertex $u_0 \in G_0$ and denote its image $u_1 = g_1(u_0) \in G_1$. 

\textbf{Case 1:} The simplest case is that $g_1^\inv(u_1) = \{u_0\}$ and $u_0 \ne v$, in which case $W_{u_0}(\gamma_0)$ and $W_{u_1}(\gamma_1)$ are isomorphic: since $g_1$ maps a regular neighborhood of $u_0$ homeomorphically to a regular neighborhood of $u_1$, and since $\gamma_1 = g_1 \circ \gamma_0$, the derivative map $D_{u_0} g_1 \from T_{u_0} G_0 \to T_{u_1} G_1$ is a bijection that induces a graph isomorphism from $W_{u_0}(\gamma_0)$ to $W_{u_1}(\gamma_1)$.

\textbf{Case 2:} The next case to consider is that $g_1^\inv(u_1)=\{u_0\}$ and $u_0=v$. Denote the oriented edges with initial direction $v$ that  $g_1$ folds as $E,E'$, and consider their initial directions $d=d_vE$, $d'=d_vE'$; we say that these two directions form the \emph{illegal turn} of the fold map $g_1$.

\textbf{Subcase 2a:} If one or both of $d,d'$ are not vertices of the Whitehead graph $W_v(\gamma_0)$, the derivative map $D_v g_1 \from T_v G_0 \to T_w G_1$ induces a graph isomorphism $W_{u_0}(\gamma_0) =W_v(\gamma_0) \approx W_{u_1}(\gamma_1)$. This would happen in Figure~\ref{FigureCollapsingSigma}, for example, if one of $d$ or $d'$ was the hollow vertex of $T_v G_0$. 

\textbf{Subcase 2b:} The more interesting case to consider is when both of $d,d'$ are vertices of $W_v(\gamma_0)$. A key feature to notice is that $W_v(\gamma_0)$ does \emph{not} have an edge $\turn{d d'}$ --- equivalently, neither $\overline E E'$ nor $\overline E' E$ is a subpath of $q \restrict \Sigma$ --- because $\gamma_1 = g_1 \circ \gamma_0 \from \Sigma \to G_1$ is an immersion. For example, in Figure~\ref{FigureCollapsingSigma} the vertices $d,d'$ cannot be the endpoint pair of the green edge, nor of the blue edge, nor of the red edge. An example of this behavior is shown in Figure~\ref{FigureFirstFoldWGraph}, in which we continue the example of Figure~\ref{FigureCollapsingSigma} by a fold that identifies a green and a blue direction. The effect of this fold on Whitehead graphs is to alter $W_v(\gamma_0)$ by identifying an endpoint of the green edge to an endpoint of the blue edge, thus producing the Whitehead graph $W_{u_1}(\gamma_1)$. 

\begin{figure}
\centerline{
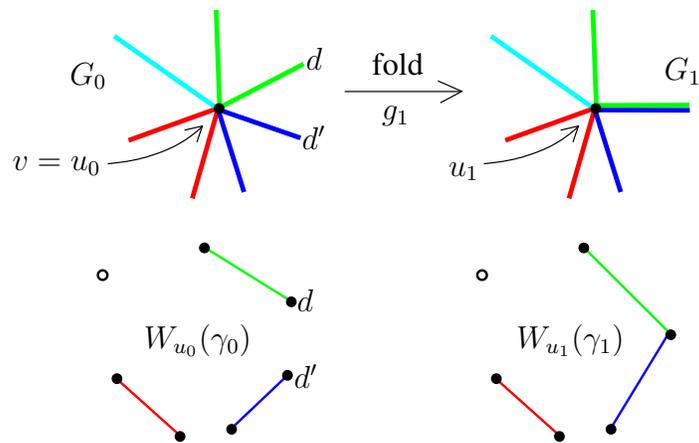 
}
\caption{\textbf{The first fold $g_1$: Case 2, Subcase 2b.} The fold map $g_1 \from G_0 \to G_1$ is depicted in a neighborhood of the vertex $v$ (with $G_0$ and $v$ as in Figure~\ref{FigureCollapsingSigma}). In this case $g_1^\inv(u_1)=\{v\}=\{u_0\}$. The two directions $d,d'$ at $u_0$ that form the illegal turn of the fold map $g_1$ --- meaning the two directions that are identified by $D_{u_0} g_1$ --- are each vertices of the Whitehead graph $W_{u_0} \gamma_0$, forming a turn at $u_0$. But that turn is not taken by $\gamma_0$ hence $d,d'$ are not the endpoints of an edge of $W_{u_0}(\gamma_0)$. The fold induces a quotient map of Whitehead graphs $W_{u_0}(\gamma_0) \mapsto W_{u_1}(\gamma_1)$ under which the two illegal turn directions $d,d'$ at $u_0$ are identified to a single direction at $u_1$.}
\label{FigureFirstFoldWGraph}
\end{figure}

\textbf{Case 3.} Thus far we have considered all of the cases where $g_1^\inv(u^\vp_1)$ is a single point --- in words, $g_1$ is one-to-one over $u_1$. Since $g_1$ is a fold map, it only remains to consider the cases where $g_1$ is two-to-one over $u_1$. We may assume that $u_1$ has valence~$\ge 3$ in $G_1$, for if $u_1$ has valence~$2$ then the Whitehead graph $W_{u_1}(\gamma_1)$ is easy to determine: it is either empty (when $u_1 \not\in \gamma_1(\Sigma)$) or a graph with one edge connecting two vertices (when $u_1 \in \gamma_1(\Sigma)$).

We denote $g_1^\inv(u_1)=\{u^\vp_0,u'_0\}$. \emph{Both} of the Whitehead graphs $W_{u_0}(\gamma_0)$ and $W_{u'_0}(\gamma_0)$ are needed in order to describe the Whitehead graph $W_{u_1}(\gamma_1)$. Both of the derivative maps 
$$D_{u_0} \from T_{u_0} G_0 \to T_{u_1} G_1 \qquad D_{u'_0} \from T_{u'_0} G_0 \to T_{u_1} G_1
$$
induce maps 
$$W_{u_0}(\gamma_0) \to \Lambda_{u_1}(\gamma_1) \qquad W_{u'_0}(\gamma_0) \mapsto \Lambda_{u_1}(\gamma_1)
$$
The Whitehead graph $W_{u_1}(\gamma_1)$ is the union in $\Lambda_{u_1}(\gamma_1)$ of the images of these two induced maps. Those two images can be computed using the methods of Cases 1 and 2: for whichever of $u_0,u'_0$ is distinct from $v$, the corresponding induced map of Whitehead graphs is an embedding; for whichever is equal to $v$, the corresponding induced map is as described in Case 2 above, either an embedding or an identification of a pair of vertices.

In Figure~\ref{FigureFirstFoldCutPoint} we show a key special case, where $u_0=v$ is the valence~$6$ vertex of $G_0$ taken from the middle diagram of Figure~\ref{FigureCollapsingSigma}, and $u'_0$ is some point in the interior of an edge, hence the fold map is a proper full fold. The key feature of this example is that the Whitehead graph $W_{u_1}(\gamma_1)$ is formed as the quotient of the disjoint union of $W_{u_0}(\gamma_0)$ and $W_{u'_0}(\gamma_0)$ by identifying one vertex of $W_{u_0}(\gamma_0)$ with one vertex of $W_{u'_0}(\gamma_0)$. 

\begin{figure}
\centerline{
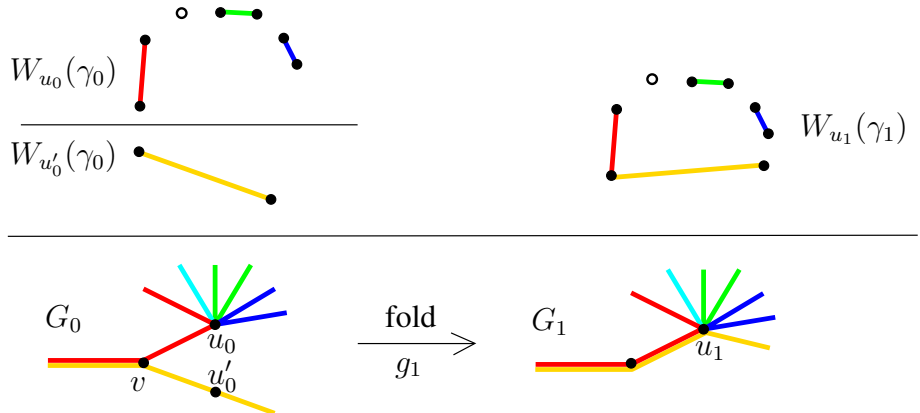 
}
\caption{\textbf{The first fold $g_1$: A key subcase of Case 3.} Under the fold map $g_1 \from G_0 \to G_1$ we have $g_1^\inv(u_1)=\{u_0,u'_0\}$ where $u_0$ is the valence~6 vertex of $G_0$ taken from Figure~\ref{FigureCollapsingSigma} but redrawn, and $u'_0$ is a point in the interior of some edge. In this example, in addition to the red, blue, and green strands of $\gamma_0,\gamma_1$, there are also yellow strands; cyan colored edges are not crossed by $\gamma_0$ nor by $\gamma_1$. The Whitehead graph $W_{u_1}(\gamma_1)$ is formed from the disjoint union of the two Whitehead graphs $W_{u_0}(\gamma_0)$ and $W_{u'_0}(\gamma_0)$ by identifying a vertex in one with a vertex of the other.
}
\label{FigureFirstFoldCutPoint}
\end{figure}

\smallskip
At this stage we can begin to appreciate the first part of Whitehead's insight. First, when studying a circuit family, it is useful to focus attention on the ``infinitesmal behavior'' of that circuit family, namely a description of how the circuit family crosses a vertex, encoded into what we now call the \emph{Whitehead graph} of the circuit family. Second, given a fold map $G \xrightarrow{g} G'$ and a collection of immersed circuits in $G$ which stay immersed under the fold $g$, one can compute the Whitehead graphs in $G'$ solely from the description of the fold $g$ and the Whitehead graphs in $G$.  

Further insights come by examining the results of many such computations, or by simply intuiting key general properties of Whitehead graphs. Perhaps one can apply inductive computations to derive such properties. 

\begin{exercise}
\label{ExerciseComputeWhitehead}
Using the fold sequence $G_0 \mapsto G_1 \mapsto G_2 \mapsto G_3$ depicted in Sections~\ref{SectionFirstFold}--\ref{SectionSubsequentFolds}, letting $\Sigma$ be a circle, and letting $\gamma_0 \from \Sigma \to G_0$ be an embedding with image equal to the loop labelled $a$ in $G_0$, compute the Whitehead graphs of $\gamma_i \from \Sigma \to G_i$ at all vertices of $G_i$, for all $i=1,2,3$. 
\end{exercise} 

\begin{exercise}
\label{ExerciseMakeYourOwn}
Make your own example: start with a marked graph $G_0$ and an embedded collection of circuits in $G_0$; write down a long complicated fold sequence $G_0 \mapsto G_1 \mapsto \cdots \mapsto G_n$; try to choose the fold sequence so that all maps $G_0 \mapsto G_i$ restrict to an immersion on $\Sigma_0$, forming a system of circuits $\Sigma_i$ in $G_i$; examine the $\Sigma_i$'s for any extraordinary behavior.
\end{exercise}

\section{Whitehead's Algorithm} 
\label{SectionWhiteheadAlgorithm}

The reader of Section~\ref{SectionWhiteheadStartup} who has carried out Exercises~\ref{ExerciseComputeWhitehead} and~\ref{ExerciseMakeYourOwn} may have observed the following pattern: except for the simplest cases, given a marked graph $G$ and a circuit family $\gamma$ that represents a partial free basis, somewhere amongst all the Whitehead graphs of $\gamma$ at all the various point of $G$, there tends to be a component with a cut vertex. In Proposition~\ref{PropCutPointTest} of Section~\ref{SectionCutPointCondition} we formalize this observation as a necessary (but not sufficient) condition called the \emph{cut vertex test}. The cut vertex test is algorithmically decidable, and forms the first of two major subroutines for Whitehead's algorithm. Since the cut vertex test is not a sufficient condition, we are forced to further analyze the general situation where a circuit family passes the cut vertex test. In Section~\ref{SectionSplitOperation} we shall show that for any system of circuits that represents a partial free basis, if that system does pass the cut vertex test, and if that system is not yet ``visibly'' a partial free basis, then there is a ``split operation'' which simplifies the situation.

The precise statement of Whitehead's algorithm is found in Section~\ref{SectionWhiteheadAlgStated}; here is a brief outline. One inputs a marked graph and a set of primitive circuits, no two of which are equal to each other or to each others' inverse. Now one starts a loop: if those circuits visibly represent a partial free basis then the algorithm halts; otherwise one checks the cut vertex test; if that test fails then those circuits do not represent a partial free basis and the algorithm halts; otherwise one carries out the split operation and then repeats the loop. The algorithm must halt, because the split operation simplifies the circuits.

\subsection{Whitehead graphs of circuit families and other things.}
\label{SectionWhiteheadGraphs}
We begin by formalizing the notions of ``Whitehead graphs'' that have been informally introduced in examples.

Consider a graph $G$. We have earlier defined the direction set $T_v G$ of $G$ at any vertex $v \in G$, namely the set of germs of oriented tight edge paths with initial vertex~$v$. Each germ is represented by a unique oriented edge $E$ with initial vertex~$v$, and the germ of $E$ at $v$ is an element of $T_v G$ denoted $d_v E$ or just $dE$, called the \emph{initial direction} of $E$. Also, if $w$ is the terminal vertex of $E$ then the initial direction of the orientation reversed edge $\overline E$ is an element of $T_w G$ called the \emph{terminal direction of~$E$} and denoted $d_w \overline E$ or just $d \overline E$. We will often abuse notation by dropping the $d$, letting $E$ denote its own initial direction $d E$ and $\overline E$ its own terminal direction. 

For each vertex $v \in G$ define a \emph{nondegenerate turn} at $v$ to be a 2-element subset $\{d,d'\}$ of the direction set $T_v G$, also denoted $\turn{dd'}=\turn{d'd}$. In later chapters we will also be concerned with \emph{degenerate turns} at $v$, which are simply 1-element subsets; but for now we consider only nondegenerate turns, hence we shall abuse terminology by dropping the adjective ``nondegenerate''. The \emph{graph of turns} of $G$ at $v$, denoted $\Lambda_v G$, is formally equal to the complete graph with vertex set $T_v G$, having one unoriented edge connecting any two directions $d \ne d' \in T_v G$; we re-use the notation $\turn{dd'}=\turn{d'd}$ to represent that edge. To avoid confusion, we will try to stick with the terminology of ``directions'' and ``turns'' for vertices and edges of the graph $\Lambda_v G$ and its subgraphs.

Consider now a 1-manifold $M$, the components of which may be a mix of circles, arcs, rays, or lines, and there may be any number of components. A \emph{proper immersion} $\gamma \from M \to G$ is a continuous, locally injective map such that for each component of $M$, and for any orientation of that component, the restriction of $\gamma$ to that component is an edge path. If each component of $M$ is compact then it suffices to require that $\gamma$ is continuous and locally injective and we will refer to $\gamma$ as an \emph{immersion}. For the application started in Section~\ref{SectionWhiteheadStartup} and to continue in Section~\ref{SectionWhiteheadAlgorithm}, $M=\Sigma$ is a finite union of circles and the map $\gamma \from \Sigma \to G$ is called a \emph{circuit family}.

\paragraph{Definition of Whitehead graphs.} Given a proper immersion $\gamma \from M \to G$ and a vertex $v \in G$, the \emph{Whitehead graph} of $\gamma$ at $v$ is the subgraph of $\turngraph_v G$ denoted $W_v(\gamma)$ which encodes how $\gamma$ crosses~$v$, as follows:
\begin{itemize}
\item For any oriented edge $E$ with initial vertex $v$, the direction $d_v E$ is a vertex of $W_v(\gamma)$ if and only if there is an orientation of $M$ such that $E$ is a subpath of~$\gamma$ (which holds if and only if the image of $\gamma$ contains $E$).
\item For any two oriented edges $E \ne E'$ with initial vertex $v$ and initial directions $d=d_vE$, $d'=d_vE'$, the turn $\turn{dd'}$ is an edge of $W_v(\gamma)$ if and only if there is an orientation of $M$ such that $\overline E E'$ is a subpath of~$\gamma$; equivalently, there is an orientation such that $\overline E' E$ is a subpath of $\gamma$.
\end{itemize}
We will also say that the turns in $W_v(\gamma)$ are the turns that are \emph{taken} by~$\gamma$.

We note that any complete graph, such as $\turngraph_v G$, is a \emph{simplicial graph}\index{simplicial graph}, meaning that it is a simplicial 1-complex: there are no loop edges, and no two distinct edges have the same endpoint pair. It follows that any Whitehead graph $W_v(\gamma)$ is also a simplicial graph.

\medskip

One may visualize $W_v(\gamma)$ as follows. Fix a regular neighborhood $N_v \subset G$ of $v$ with frontier denoted $\bdy N_v$. We may identify the vertices of $\turngraph_v G$ with the set $\bdy N_v$, and we may identify the set of edges of $\turngraph_v$ with the set of arcs in $N_v$ having endpoints on $\bdy N_v$. The vertex set of $W_v(\gamma)$ is then identified with $\image(\gamma) \intersect \bdy N_v$. The edges of $W_v(\gamma)$, meaning the turns taken by $\gamma$, are identified with the images under $\gamma$ of the components of $\gamma^\inv(N_v)$. 

See Figures~\ref{FigureCollapsingSigma}, \ref{FigureFirstFoldWGraph} and~\ref{FigureFirstFoldCutPoint} for very rather simple examples of Whitehead graphs, and Figure~\ref{FigureTypicalWGraph} for a more complicated example.

\begin{figure}
\centerline{
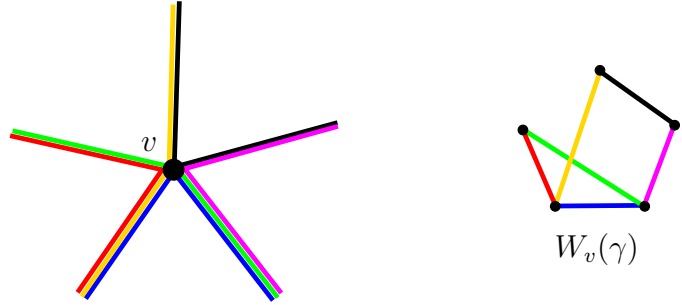 
}
\caption{A somewhat typical Whitehead graph $W_v(\gamma)$, depicting six different turns taken by $\gamma$ at the vertex $v$.}
\label{FigureTypicalWGraph}
\end{figure}

\paragraph{Exercises on atypical Whitehead graphs.} Exercises \ref{ExerciseVisible}--\ref{ExerciseCollapseVisible} below explain consequences of some very simple and rather atypical behavior of Whitehead graphs. 

For these exercises we fix some notation: 
\begin{itemize}
\item $G$ is a marked graph; 
\item $\gamma \from C \to G$ is a circuit family with components $\gamma_i$, $i=1,\ldots,K$, representing conjugacy classes $[c_1],\ldots,[c_K]$ of $F_n$.
\end{itemize} 
We say that $\gamma$ is \emph{jointly primitive} if for each $i \ne j \in \{1,\ldots,K\}$ the circuit $c_i$ is primitive and is distinct from the circuits $c_j$ and $c_j^\inv$. We note that $\gamma$ is jointly primitive if and only if each conjugacy class $[c_i]$ is primitive and if $i \ne j$ then the conjugacy class $[c_i]$ is distinct from the conjugacy classes $[c_j]$ and $[c_j^\inv]$ (see Exercise~\ref{ExerciseJointPrimitiveTest} for a closely related discussion).

We say that $\gamma$ satisfies the \emph{visibility condition} if for each $v \in G$ the Whitehead graph $W_v\gamma$ is either empty or is a single edge of $\turngraph_v G$. More generally, $\gamma$ satisfies the \emph{near visibility condition} if for each $v \in G$ each component of $W_v\gamma$ consists of a single edge of~$\turngraph_v G$.

\begin{exercise}\label{ExerciseVisible}
Prove that if $\gamma$ satisfies the visibility condition then the following are equivalent:
\begin{enumerate}
\item $\gamma$ is jointly primitive.
\item $\gamma$ is injective.
\item $\{[c_1],\ldots,[c_K]\}$ is a partial free basis of $F_n$ of cardinality $K$.
\end{enumerate}
\end{exercise}

\begin{exercise}\label{ExerciseAlmostVisible}
Prove that if $\gamma$ satisfies the near visibility condition then the following are equivalent:
\begin{enumerate}
\item $\gamma$ is jointly primitive.
\item $\gamma$ is injective when restricted to the complement of some finite subset of $C$. 
\item $\{[c_1],\ldots,[c_K]\}$ is a partial free basis of $F_n$ of cardinality~$K$.
\end{enumerate}
\end{exercise}

\smallskip

The visibility and near visibility conditions are related as follows (see Step 0 in Section~\ref{SectionWhiteheadStartup}; and see Figure~\ref{FigureCollapsingSigma}):

\begin{exercise}\label{ExerciseCollapseVisible}
Prove that if $\gamma$ satisfies the visibility condition, if $L \subset G$ is a subforest which contains no edge in the image of $\gamma$, and if $q \from G \xrightarrow{[L]} H$ is the collapse map which collapses to a point each component of $L$, then $q \circ \gamma \from C \to H$ satisfies the near visibility condition.
\end{exercise}

\subsection{Induced maps of Whitehead graphs.} 
\label{SectionInducedWhitehead}
Consider now two graphs $G,G'$ and a nondegenerate tight map $g \from G \to G'$. For each $v \in G$ with image $v' \in G'$ the derivative $D_v g \from T_v G \to T_{v'} G'$ is defined. Given $v \in G$ and a turn $\turn{dd'}$ at $v$, if $D_v g(d) \ne D_v g(d')$ then we say that $\{d,d'\}$ is a \emph{legal turn for $g$} at~$v$, otherwise it is an \emph{illegal turn}. Since the graph of turns $\turngraph_v G$ is just the complete graph on the vertex set $T_v G$, and similarly for $\turngraph_{v'} G'$, the map $D_v g$ extends uniquely to a simplicial map $D_v g \from \turngraph_v G \to \turngraph_v G'$ which we shall call the \emph{induced turn map}. For each turn $\turn{dd'}$ at $v$, if $\turn{dd'}$ is legal then its image is the turn $D_v g(\turn{dd'}) = \{D_v g(d),D_v g(d')\}$, whereas if $\turn{dd'}$ is illegal then its image is the direction $D_v g(d)=D_v g(d')$. 

Let $\gamma \from M \to G$ be a proper immersion in $G$, and consider the map $\gamma' = g \composed \gamma \from M \to G'$. For each $p \in G$ with image $p' = g(p) \in G'$ we may restrict the simplicial map $D_v g \from \turngraph_p G \to \turngraph_{p'} G'$ to the Whitehead graph $W_p \gamma$.  In the following simple proposition, item~\pref{ItemWhiteheadImmersionCriterion} enumerates some ways to detect when $g \composed \gamma \from M \to G'$ is an immersion, and when it is item~\pref{ItemWhiteheadImage} gives useful information regarding Whitehead graphs. The proofs follow directly from the definitions, and we leave it to the reader to check the details.

\begin{proposition}
\label{PropWhiteheadMapUnion}
Given a nondegenerate tight map of graphs $g \from G \to G'$, and given a proper immersion of a 1-manifold $\gamma \from M \to G$, the following hold:
\begin{enumerate}
\item \label{ItemWhiteheadImmersionCriterion}
$g \composed \gamma \from M \to G'$ is a proper immersion if and only if each taken turn is a legal turn, if and only if for each $p \in G$ each turn in the Whitehead graph $W_p(\gamma)$ is a legal turn of the map~$g$.
\item \label{ItemWhiteheadImage}
If $\gamma' = g \composed \gamma \from M \to G'$ is a proper immersion then 
\begin{enumerate}
\item for each $p \in G$ with image $p'=f(g) \in G'$, the map $D_p g \from \Lambda_p G \to \Lambda_{p'} G'$ restricts to a nondegenerate simplicial map $D_p g \from W_p \gamma \to W_{p'} \gamma'$
\item for each $p' \in G'$ we have $\ds W_{p'} \gamma' = \bigcup_{p \in g^\inv(p')} D_p g(W_p \gamma)$
\end{enumerate}
\end{enumerate}
\qed\end{proposition}

\subsection{Whitehead's cut vertex test}
\label{SectionCutPointCondition}

Let $W$ be a finite, connected graph, and assume that $W$ is simplicial, which holds for example if $W$ is a Whitehead graph. A vertex $w \in W$ is a \emph{cut vertex}\index{cut vertex} if any of the following equivalent statements is true:
\begin{itemize}
\item the topological space $W - \{w\}$ is disconnected;
\item there exist vertices $u,v \in W-\{w\}$ such that every edge path in $W$ from $u$ to $v$ contains~$w$;
\item there exist subgraphs $W_1,W_2 \subset W$ such that $W = W_1 \union W_2$, $W_1 \intersect W_2 = \{w\}$, and each $W_i$ has an edge incident to $w$.
\end{itemize}
The following proposition was first proved by Whitehead \cite[]{} in the special case of a rose graph:

\begin{proposition}[The cut vertex test]
\label{PropCutPointTest}
If $H$ is a marked graph, and if $\gamma$ is a collection of circuits in $H$ representing a partial free basis $C$ of $F_n$, then $\gamma$ satisfies one of two conditions (the first of which we repeat from just before Exercise~\ref{ExerciseVisible}):
\begin{description}
\item[Near Visibility Condition:] For each $v \in H$, each component of $W_v \gamma$ is a single edge.
\item[Cut Vertex Condition:] For some $v \in H$, some component of $W_v \gamma$ has a cut~vertex.
\end{description}
\end{proposition}

\paragraph{Remarks.} In each of Figures~\ref{FigureFirstFoldWGraph} and~\ref{FigureFirstFoldCutPoint}, after folding one can observe the appearance of a new cut vertex in a component of the Whitehead graph denoted $W_{u_1}(\gamma_1)$. For example in Figure~\ref{FigureFirstFoldWGraph} the green and blue edges of $W_{u_1}(\gamma_1)$ intersect in a new cut vertex; however this is not particularly significant for the proof of Proposition~\ref{PropCutPointTest}, because in other examples of the subcase depicted in Figure~\ref{FigureFirstFoldWGraph} the green and blue edges of $W_{u_0}\gamma_0$ might already have been in the same component of $W_{u_0}\gamma_0$, and no new cut vertex would have appeared in $W_{u_1}\gamma_1$. On the other hand, the new cut vertex that appears in Figure~\ref{FigureFirstFoldCutPoint} is highly significant for the proof: under the subcase depicted in Figure~\ref{FigureFirstFoldCutPoint} that new cut vertex will always appear.

\begin{proof}
Choose a marked graph $G$ in which the partial free basis $C$ is represented by a subgraph $\Delta \subset G$ consisting of pairwise disjoint circles. Letting $\delta \from \Delta \inject G$ denote the inclusion map, it follows that $\delta$ satisfies the visibility condition (Exercise~\ref{ExerciseVisible}). 

Choose a tight homotopy equivalence $f \from G \to H$ which preserves marking and which restricts to a local embedding on $\Delta$. Since $f$ preserves marking, since $C$ is represented in $G$ by the inclusion $\delta \from \Delta \inject G$, and since $C$ is represented in $H$ by the given collection of circuits $\gamma$, it follows that the composition $f \circ \delta \from \Delta \xrightarrow{\delta} G \xrightarrow{f} H$ is a collection of circuits in $H$ that is equivalent to $\gamma$.  Apply Proposition~\ref{PropFoldableMapConstruction} to factor $f$ up to homotopy rel $\delta$ as a forest collapse map followed by a foldable map
$$G \xrightarrow{q} G_0 \xrightarrow{f^0} H
$$
Since $q$ only collapses edges of $G \setminus \delta$ (Proposition~\ref{PropFoldableMapConstruction}~\pref{ItemSubgraphFoldable}), it follows that the composition $\delta_0 = q \composed \delta \from \Delta \to G_0$ defines a collection of circuits in $G_0$ that represent $C$ in $G_0$ and that satisfy the near visibility condition (Exercise~\ref{ExerciseCollapseVisible}).

Next apply Proposition~\ref{ThmStallingsFolds} to obtain a maximal fold factorization of $f^0$:
$$G_0 \xrightarrow{g_1} G_1 \xrightarrow{g_2} \cdots \xrightarrow{g_K} G_K \approx H
$$
If one of the folds $g_k \from G_{k-1} \to G_k$ happens to be a full fold over a loop edge --- meaning an improper full fold where one of the two folded edges is a loop edge, or a proper full fold of some edge over a loop edge --- then we can factor $g_k$ into a partial fold followed by a full fold. By doing this wherever needed, we may assume that $g_k$ is never a full fold over a loop edge (this will simplify our later analysis by removing that case from consideration). Also, no $g_k$ is a bigon fold. By induction let $\delta_k = g_k \circ \delta_{k-1} \from \Delta \to G_k$. Since $f^0$ is a homotopy equivalence taking $\delta_0$ to $\delta_K$ without cancellation, it follows that so $\delta_k$ is a collection of circuits in $G_k$ representing $C$; also, $\delta_K = \gamma$.  

We shall prove by induction on $k$ that $\delta_k$ satisfies either the near visibility condition or the cut vertex condition, starting with the base case $k=0$ where we have already noted that $\delta_0$ satisfies the near visibility condition. 

For the induction step, assuming that $\delta = \delta_{k-1}$ satisfies either the near visibility condition or the cut vertex condition, we must prove that $\delta' = \delta_k$ also satisfies one of those conditions. To summarize what we know so far: the fold map $g_k \from G_{k-1} \to G_k$ is not a bigon fold and it is not a full fold over a loop edge; and the map $g_k$ takes the circuits $\delta$ to the circuits $\delta'$ without cancellation, hence each turn of $G_{k-1}$ taken by the circuits $\delta$ is a legal turn with respect to the fold map $g_k$.

Let $e_0,e_1$ be the oriented edges of $G_{k-1}$ that are folded by $g$, and let $\eta_0 \subset e_0$, $\eta_1 \subset e_1$ be the maximal initial segments that are identified by $g$. In $G_{k-1}$ we know that the unique illegal turn $\turn{e_0e_1}=\turn{\eta_0 \eta_1}$ is not taken by the collection of circuits $\delta$. Let $v$ be the common initial vertex of $\eta_0,\eta_1$, and let $w_0$, $w_1$ be their respective terminal points. Since $g_k$ is not a bigon fold (Proposition~\ref{ThmStallingsFolds}~\pref{ItemContinuingFolds}) we have $w_0 \ne w_1$. Since $g_k$ is not a full fold over a loop edge, we have $v \ne w_0,w_1$. In $G_k$ let $v' = g_k(v)$, $\eta' = g_k(\eta_0)=g_k(\eta_1)$, and let $w' = g_k(w_0)=g_k(w_1)$. 

First we focus on the two points $w_0,w_1 \in G_{k-1}$, each mapping to $w' \in G_k$, and the two induced turn maps
$$D_{w_0} g \from \turngraph_{w_0} G_{k-1} \to \turngraph_{w'} G_k, \qquad\qquad D_{w_1} g \from \turngraph_{w_1} G_{k-1} \to \turngraph_{w'} G_k
$$
Letting $\bar\eta_i \in \turngraph_{w_i} G_{k-1}$ denote the terminal direction of the path $\eta_i$, and similarly letting $\bar \eta' \in \turngraph_{w'} G_k$ denote the terminal direction of $\eta'$, we have $D_{w_0} g(\bar\eta_0)=D_{w_1} g(\bar\eta_1) = \bar\eta'$. Furthermore, the two turn maps $D_{w_0} g$, $D_{w_1} g$ are embeddings (since $w_0 \ne v$ and $w_1 \ne v$), and the intersection of their images equals $\bar\eta'$. By applying Proposition~\ref{PropWhiteheadMapUnion}, each of $D_{w_i} g$ restricts to an isomorphism from the Whitehead graph $W_{w_i}\delta$ onto its image in the turn graph $\turngraph_{w'}G_k$,
$$W_{w_i} \delta \xrightarrow{D_{w_i} g} D_{w_i}g(W_{w_i}\delta) \subset \turngraph_{w'} G_k
$$
and furthermore 
\begin{align*}
(\#) \qquad\qquad D_{w_0}g (W_{w_0}\delta) \union D_{w_1}g (W_{w_1}(\delta)) &= W_{w'}(\delta')  \\
D_{w_0}g (W_{w_0}\delta) \intersect D_{w_1}g (W_{w_1}(\delta)) &= \{\bar\eta'\} \,\,\text{or}\,\, \emptyset
\end{align*}
Note that the intersection in the second line is empty if and only if $\bar\eta_0 \not\in W_{w_0}\delta$ and $\bar\eta_1 \not\in W_{w_1}(\delta)$. 

\medskip

\noindent\textbf{Case 1:} Suppose that for each $i \in \{0,1\}$ we have $\bar\eta_i \in W_{w_i}\delta$, in particular $W_{w_i}\delta$ is not empty. Let $X_i$ be the component of $W_{w_i}\delta$ containing $\bar\eta_i$. It follows from $(\#)$ that $D_{w_0}(X_0) \union D_{w_1}(X_1)$ is a component of $W_{w'} \delta'$, and that $\bar\eta'$ is a cut vertex of that component, hence the Cut Vertex Condition is proved for $\delta'$.

\medskip

\noindent\textbf{Case 2:} Suppose there exists $i \in \{0,1\}$ such that $\bar\eta_i \not\in W_{w_i} \delta$. It follows from $(\#)$ that:
\begin{description}
\item[$(*)$] The two maps $D_{w_0} g$ and $D_{w_1} g$ take the disjoint union of the two Whitehead graphs $W_{w_0}\delta$, $W_{w_1}\delta$ isomorphically onto the Whitehead graph $W_{w'}\delta'$. 
\end{description}
Next we shall show:
\begin{description}
\item[$(**)$] For any vertex $u' \in G_k$ such that $u' \ne w'$, letting $u \in G_{k-1}$ be its unique pre-image, the map $D_{u} g$ takes $W_u\delta$ isomorphically onto $W_{u'}\delta'$
\end{description}
Once this has been shown then, by combining $(*)$ and $(**)$, it follows that as $u'$ varies over all vertices of $G_k$ and as $u$ varies over the points of $g^\inv(u)$ in $G_{k-1}$ the map $Dg$ takes the disjoint union $\coprod W_u\delta$ to the disjoint union $\coprod W_{u'}\delta'$ by a graph isomorphism. Since the Near Visibility Condition and the Cut Vertex Condition are both invariants of graph isomorphism, whichever of those conditions is satisfied by $\delta$ it is also satisfied by~$\delta'$. Thus it remains to prove $(**)$, after which we will be done. Since $u' \ne w'$ it follows that $u \not\in \{w_0,w_1\}$, and we break into two cases depending on whether $u=v$.

To prove $(**)$ when $u \ne v$, we have $u \not\in \{v,w_0,w_1\}$, and so $(**)$ follows from Proposition~\ref{PropWhiteheadMapUnion} after noting that the map $D_u g \from \turngraph_u G_{k-1} \to \turngraph_{u'} G_k$ is a graph isomorphism. 

To prove $(**)$ when $u=v$, first note that $u'=v'$. From the Case 2 assumption, by symmetry of notation we may assume that $\bar\eta_1 \not\in W_{w_1} \delta$. It follows that $\delta$ does not cross the edge $\eta_1$. Moving from the terminal endpoint $w_1$ of $\eta_1$ to its initial endpoint $v$, it follows that its initial direction is not contained in $W_v\delta$. The entire Whitehead graph $W_v\delta$ is therefore contained in the subgraph of $\turngraph_v G_{k-1}$ obtained by removing the initial direction of $\eta_1$ and all its incident turns. The restriction of $D_v g$ to that subgraph is an injection, and so by applying Proposition~\ref{PropWhiteheadMapUnion}, it follows that $D_v g$ restricts to an isomorphism between $W_v \delta$ and $W_{v'}\delta'$.

\end{proof}

%

\subsection{Splitting a marked graph}
\label{SectionSplitOperation}

The intuition of splitting is that a cut vertex of a Whitehead graph $W_v \gamma$ of a circuit family $\gamma$ in a marked graph $G$ suggests a way to split $G$ so as to simplify $\gamma$, as the example in Figure~\ref{FigureWordToSplit} suggests. In this section we give the formal definitions needed to make this intuition rigorous, culminating in a description of the second main subroutine of Whitehead's algorithm, the \emph{split operation}. Our main result is Proposition~\ref{PropSplit} which says that if a circuit family in a marked graph satisfies the Cut Vertex Condition of Proposition~\ref{PropCutPointTest}, then one can carry out a split operation which simplifies the circuit family.

\begin{figure}
\centerline{
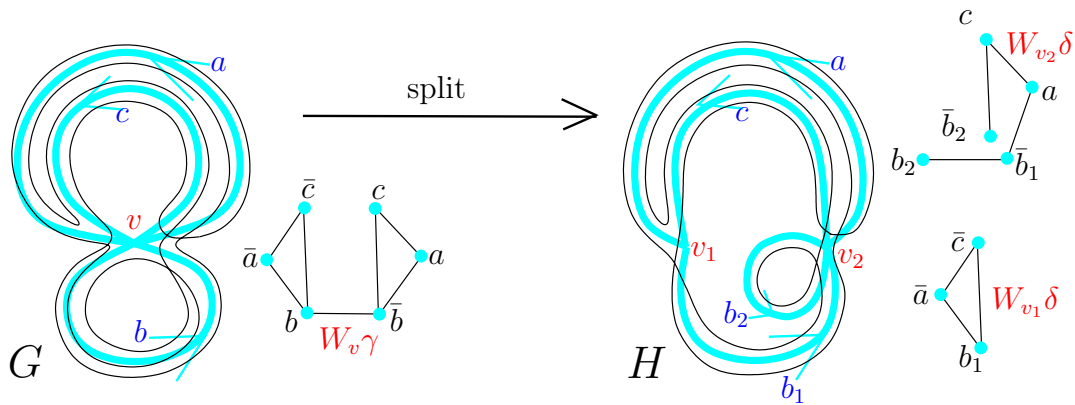 
}
\caption{The circuit $\gamma = a b^2 c \bar a \bar  b \bar  c$ in the marked graph $G$ has a Whitehead graph $W_v\gamma$ with two cut vertices $b$ and $b^\inv$. The graph $G$ can be split along the direction $b$, pulling the directions $\{\bar a, \bar c\}$ to one side and pulling the directions $\{a,c,\bar b\}$ to the other side, resulting in a marked graph $H$ with two vertices, in which the conjugacy class represented by $\gamma$ has a simpler representative $\delta = a b_1 b_2 c \bar a \bar b_1 \bar c$: while $\gamma$ crosses the $b$ edge of $G$ three times, the circuit $\delta$ crosses no edge of $H$ more than two times. Of the two vertices of $H$ one of them, namely $v_2$, has a Whitehead graph $W_{v_2}$ with cut vertices, allowing further splitting.}
\label{FigureWordToSplit}
\end{figure}

The example depicted in Figure~\ref{FigureWordToSplit}, which shows a marked graph $G$, a circuit $\gamma = a b^2 c \bar a \bar b \bar c$, and a splitting of $G$ along its $b$ direction, resulting in a marked graph $H$ and a circuit $\delta$ representing the same conjugacy class as $\gamma$. Because the circuit $\gamma$ crosses the $b$ edge of $G$ three times, whereas the circuit $\delta$ does not cross any edge of $H$ more than two times, clearly $\delta$ is clearly simpler than $\gamma$. In order to formalize what it means to ``simplify'' a circuit family, we will use a certain ordinal valued \emph{weight sequence} defined on circuit families in marked graphs. And in order to formalize ``splitting'', we must first specify exactly what is being split apart from what, which we do using the concept of a \emph{cut} of a Whitehead graph.

Consider a marked graph $G$, a circuit family $\gamma$ in $G$, and a vertex $v \in G$ with corresponding Whitehead graph $W_v\gamma \subset \turngraph_v G$. The graph $W_v\gamma$ need not contain every element of the direction set $T_v G$, so first we augment $W_v\gamma$ by throwing in each of those directions:
$$\wh W_v\gamma = W_v\gamma \union T_vG
$$
A \emph{cut} of the augmented Whitehead graph $\wh W_v\gamma$ is a pair of subgraphs $\wh W_1, \wh W_2 \subset \wh W_v\gamma$ having the following properties:
\begin{enumerate}
\item $\wh W_v\gamma = \wh W_1 \union \wh W_2$
\item $\wh W_1 \intersect \wh W_2 = \{d\}$ for some direction $d \in W_v$.
\item\label{ItemCutHasCutVertex} Each of $\wh W_1,\wh W_2$ contains at least one edge incident to~$d$.
\end{enumerate}
We also say that $\{\wh W_1,\wh W_2\}$ is a cut \emph{along the direction $d$}. Notice that item~\pref{ItemCutHasCutVertex} implies that $v$ has valence~$\ge 3$, so cuts of $\wh W_v\gamma$ only exist when $v$ is a natural vertex. 

Using the various equivalent definitions of cut vertices given in Section~\ref{SectionCutPointCondition}, one may see that a cut of $\wh W_v\gamma$ along the direction $d$ exists if and only if $d$ is a cut vertex of some component of $W_v\gamma$. Furthermore, one might have observed that in Case 1 of the proof of the Cut Vertex Test (Proposition~\ref{PropCutPointTest}), the manner in which the Cut Vertex Condition was verified was by using a fold map to exhibit a particular cut. By reading this observation in reverse one arrives at the concept of a split, which informally can be thought of as the inverse of a fold. We first define splits in a more general context where circuits and turns are ignored; afterwards we specialize splits to a Whitehead graph context. 

A split of a marked graph $G$ is determined once the following data has been specified: a natural vertex $v \in G$; a direction $d \in T_v G$; and a \emph{direction cut of $T_v G$ along~$d$} which by definition means a pair of subsets $T_1,T_2 \subset T_v G$ such that $T_v G = T_1 \union T_2$, and $T_1 \intersect T_2 = \{d\}$, and $T_1,T_2$ each contain at least one direction other than~$d$. The split of $G$ using a direction cut $\{T_1,T_2\}$ along $d \in T_v G$ is a marked graph $H$ equipped with a fold map $f \from H \to G$. Informally, grab the set $T_1 - \{d\}$ with one hand, and grab the set $T_2 - \{d\}$ with the other hand, and pulllllllll them apart, splitting in two the edge that contains $d$. To formalize this construction, let $E$ be the oriented edge of $G$ with terminal direction $d\overline E = d$ and terminal vertex $v$, and let $u$ be the initial vertex of $E$. The definition proceeds in two cases, depicted in Figures~\ref{FigureSplittingOne} and~\ref{FigureSplittingTwo}. 

\textbf{Case 1: $u \ne v$.}  (See Figure~\ref{FigureSplittingOne}). In this case $H$ is defined by altering $G$ as follows. First, detach from the vertex~$v$ all of the directions of the set $T_v G$. Next, remove the interior of $E$ and its terminal vertex $v$, but keep the initial vertex $u$. Next, add two new vertices $v_1,v_2$ in place of $v$. Next, attach two new oriented edges $E_1,E_2$ in place of $E$, having respective terminal vertices $v_1,v_2$, and having common initial vertex $u$. Finally, reattach the directions of the set $T_1-\{d\}$ to the vertex $v_1$, and reattach the directions of $T_2-\{d\}$ to $v_2$.

\begin{figure}
\centerline{
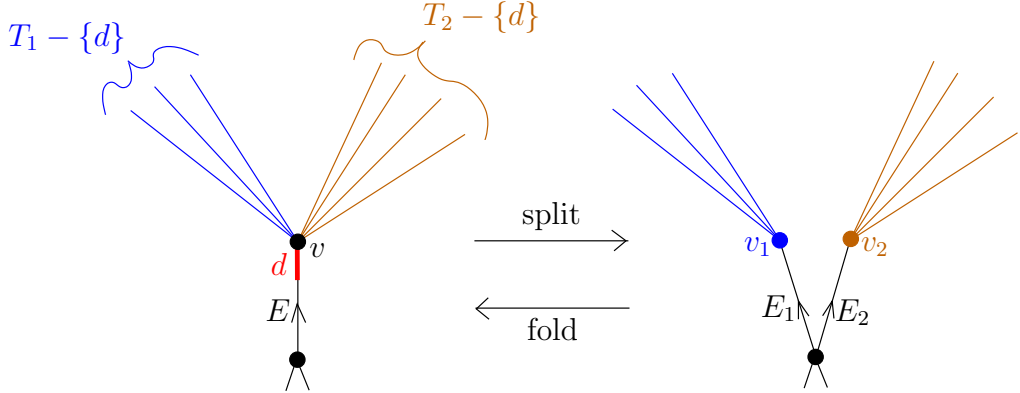 
}
\caption{Case 1: Splitting a marked graph $G$ along $d=d\overline E \in T_v G$, using a direction cut $T_v G = T_1 \union T_2$. The initial vertex of $E$ is distinct from $v$ in this case.}
\label{FigureSplittingOne}
\end{figure}

\textbf{Case 2: $u=v$.} (See Figure~\ref{FigureSplittingTwo}, and Figure~\ref{FigureWordToSplit} for an explicit example). We first note that initial direction $dE$ and the terminal direction $d=d\overline E$ of the edge $E$ are both based at the vertex $v$, and those directions are distinct elements of $T_v G$, hence $dE \in T_v G - \{d\}$. It follows that $dE$ is contained in one of the two sets $T_1 - \{d\}$ or $T_2 - \{d\}$; by transposing indices if necessary we may assume that $dE \in T_1 - \{d\}$. We now define $H$ by altering $G$ as follows. First detach from $v$ all directions of $T_v G$. Next, remove the interior of $E$ and the vertex $v$. Next, add two new vertices $v_1,v_2$ in place of $v$. Next, attach two new oriented edges $E_1,E_2$ in place of $E$, with respective terminal vertices $v_1,v_2$ and with common initial vertex $v_1$. Finally, reattach the directions of $T_1-\{d,dE\}$ to $v_1$ and reattach the directions of $T_2-\{d\}$ to~$v_2$.

\begin{figure}
\centerline{
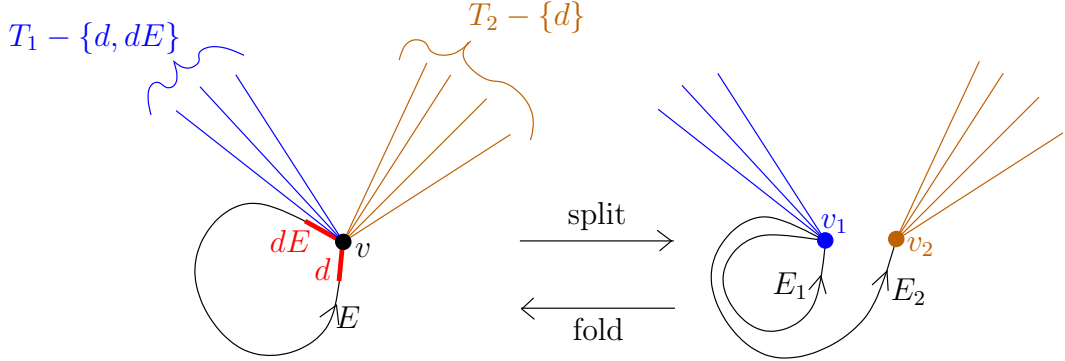 
}
\caption{Case 2: Splitting a marked graph $G$ along $d=d\overline E \in T_v G$, using a direction cut $T_v G = T_1 \union T_2$. The initial vertex of $E$ is equal to $v$ in this case, and the indexing on $T_1,T_2$ is chosen so that the initial direction $dE$ is in $T_1$.}
\label{FigureSplittingTwo}
\end{figure}

In both Cases 1 and 2 the graph $H$ is a core graph, and the construction produces a natural simplicial quotient map $f \from H \to G$ such that $f$ makes the identifications $f(v_1)=f(v_2)=v$ and $f(E_1)=f(E_2)=E$, and such that $f$ makes no other identifications of vertices nor of edges. It follows that the cell structure on $H$ is the unique \emph{pullback cell structure} on $H$, the vertex set of which is equal to the inverse image under $f$ of the vertex set of $G$. The bijection induced by $f$ between the edges of $G \setminus E$ and the pullback edges of $H \setminus (E_1 \union E_2)$ will be written as
$$H \setminus (E_1 \union E_2) = \bigcup_{j \in J} E_{H,j}, \quad G \setminus E = \bigcup_{j \in J} E_{G,j}, \quad f(E_{H,j}) = E_{G,j}
$$
From these descriptions it is clear that the map $f \from H \to G$ is a nonbigon fold map and hence is a homotopy equivalence, using which we may mark $H$ by composing the marking $R_n \to G$ with a homotopy inverse of $f$. We note that when the construction is carried out using the natural cell structure on $G$, every vertex of $H$ except $v_1$ and $v_2$ is natural. Regarding the two vertices $v_1,v_2$, it may happen that none, one, or both of $v_1,v_2$ is natural, depending on the cardinalities of the cut sets $T_1,T_2$ and on whether the splitting follows Case 1 or Case 2.

We record the following immediate consequence of the construction of $f \from H \to G$:

\begin{lemma}\label{LemmaDirectionSplitProps}
For each $i=1,2$ we have $\ds D_{v_i} f(T_{v_i} H) = T_i$. Furthermore the direction $d\overline E_i \in T_{v_i} H$ is the unique element of $T_{v_i} H$ whose image under $D_{v_i}f$ equals $d$.
\qed
\end{lemma}

For certain purposes it is convenient to break the Case 2 splitting into two successive Case 1 splittings, as follows. As a prelimary step, subdivide the edge $E$ by inserting a valence~2 vertex $w$ in its interior, thus decomposing $E$ as a concatenation of two edges $E = E' E''$. Now split $G$ along $d\overline E'' = d \overline E$ using the direction cut $\{T_1,T_2\}$, obtaining a marked graph $H'$ with two edges $E''_1,E''_2$ and a fold map $H' \to G$ that identifies $E''_1,E''_2$ to~$E''$. Next split $H'$ along $d \overline E'$ using the direction cut $\{T'_1,T'_2\}$ where $T'_1 = \{dE''_1, d\overline E'\}$, $T'_2 = \{dE''_2,d\overline E'\}$; the result is a marked graph $H$ with edges $E'_1,E'_2$ and a fold map $H \mapsto H'$ that identifies $E'_1,E'_2$ to~$E'$. The composed fold map $H \mapsto H' \mapsto G$ identifies $E_1 = E'_1E''_1$ and $E_2 = E'_2E''_2$ to $E$, and after deleting from the vertex set the two pre-images of $w$ under the fold map $H \mapsto G$ the result is identical the the description given in Case~2.

\medskip

The definition of splitting in the context of Whitehead graphs is as follows. Given a marked graph $G$, a circuit family $\gamma$ in $G$, a vertex $v \in G$, and a cut $\{\wh W_1, \wh W_2\}$ of $\wh W_v\gamma$, the splitting of $G$ using that cut is defined to be the splitting of $G$ using the direction cut $\{T_1,T_2\}$ where $T_i = \wh W_i \intersect T_v G$. 

Our last task before stating Proposition~\ref{PropSplit} is to define the weight sequence of a circuit family $\gamma \from C \to G$ in a marked graph $G$. Let $\E$ denote the edge set of $G$. For each $e \in \E$ let $k(e,\gamma)$ denote the number of times that $\gamma$ crosses~$e$, which equals the cardinality 
$$k(e,\gamma) = \abs{\gamma^\inv(x)} \quad\text{where $x \in \interior(e)$ is arbitrary.}
$$ 
The \indexemph{weight sequence} of $\gamma$ is the natural number sequence $w(\gamma)=(w_0(\gamma),w_1(\gamma),w_2(\gamma),\ldots)$ where $w_k(\gamma)$ is the number of edges of $G$ that are crossed $k$ times by $\gamma$,
$$w_k(\gamma) = \abs{\{e \in \E \suchthat k=k(e,\gamma)\}}
$$ 
Note that the sum $\sum_k w_k(\gamma)$ is equal to the cardinality of $\E$, and in particular $w_k(\gamma)$ is nonzero for only a finite set of $k$. Note also that $w(\gamma)$ is \emph{not} independent of how $G$ is subdivided into vertices and edges, although as a special case we will focus on the \emph{natural weight sequence} which is defined by taking $\E$ to be the natural edge set of $G$. The set of sequences of natural numbers with all but finitely many terms equal to zero is well-ordered by the dictionary ordering, where $(w_0,w_1,w_2,\ldots) < (w'_0,w'_1,w'_2,\ldots)$ if there exists $k \in \{0,1,2,\ldots\}$ such that $w_i=w'_i$ for $i > k$ and $w_k < w'_k$. As a special case, if $w(\gamma)$ is the natural weight sequence, and if $w'(\gamma)$ is the weight sequence defined using an arbitrary cell structure on $G$, then $w(\gamma) \le w'(\gamma)$: this inequality follows using that any cell structure on $G$ is a subdivision of the natural cell structure, together with the fact that for any natural edge $E$ the number of times that $\gamma$ crosses $E$ is equal to the number of times that $\gamma$ crosses any of the edges in the subdivision of $E$. 

\begin{proposition}[The Split Operation]\label{PropSplit}
Consider a marked graph $G$ equipped with its natural cell structure. Consider also a circuit family $\gamma \from C \to G$, a vertex $v \in G$, and a cut $\{\wh W_1,\wh W_2\}$ of $\wh W_v\gamma$ along some direction $d \in W_v\gamma$. Let $H$ be the marked graph obtained by splitting $G$ along $d$ using the cut $\{\wh W_1,\wh W_2\}$. Let $f \from H \to G$ be the corresponding fold map. Under these conditions there exists a unique circuit family $\delta \from C \to H$ with the following properties:
\begin{enumerate}
\item\label{ItemNoIllegalTurn} $\delta$ does not take the illegal turn of $f$ 
\item\label{ItemDeltaGammaEquiv} The circuit family $f \circ \delta \from C \to G$ is equivalent to $\gamma$. 
\end{enumerate}
Furthermore we have strict inequality of natural weight sequences:
\begin{enumeratecontinue}
\item\label{ItemWeightInequality}
 $w(\delta) < w(\gamma)$.
\end{enumeratecontinue}
The constructions of $H$, of $f \from H \to G$, and of $\delta$ are algorithmic given the input $G$, $\gamma$, and $\{\wh W_1,\wh W_2\}$.
\end{proposition} 
In this proposition, the implicit assertion in conclusion~\pref{ItemDeltaGammaEquiv} that $f \circ \delta \from C \to G$ is indeed a circuit family (i.e.\ that the function $f \circ \delta$ is an immersion) is a consequence of conclusion~\pref{ItemNoIllegalTurn} combined with Proposition~\ref{PropWhiteheadMapUnion}~\pref{ItemWhiteheadImmersionCriterion}.

\begin{proof} We adopt all the notation in the definition of splitting, where $E$ is the oriented natural edge of $G$ with terminal vertex $v$, terminal direction $d\overline E = d$, and initial vertex~$u$, where $H$ is endowed with the pullback cell structure of the natural cell structure on~$G$, and where $E_1,E_2 \subset H$ are the two oriented edges which are folded together by $f$ to produce $E \subset G$.

We start by assuming existence of $\delta \from C \to H$ satisfying conclusions~\pref{ItemNoIllegalTurn} and~\pref{ItemDeltaGammaEquiv}, using which we shall prove uniqueness of $\delta$ and conclusion~\pref{ItemWeightInequality}. Afterwards we shall take up the proof of existence.

To prove uniqueness of $\delta$, since $f \circ \delta$ is equivalent to $\gamma$ (by conclusion~\pref{ItemDeltaGammaEquiv}) it follows that $f \circ \delta$ and $\gamma$ represent the same set of conjugacy classes. But $f \circ \delta$ and $\delta$ also represent the same set of conjugacy classes, because $f$ is a homotopy equivalence that preserves marking. Uniqueness of $\delta$ follows, because the circuit family in the marked graph $H$ representing a given set of conjugacy classes (in this case, the set represented by $\gamma$) is unique.

Next we prove \pref{ItemWeightInequality}, the inequality of natural weight sequences, $w(\delta) < w(\gamma)$. Let $w'(\delta)$ be the weight sequence of $\delta$ with respect to the pullback cell structure on $H$. Since $w(\delta)$ is the natural weight sequence, we have $w(\delta) \le w'(\delta)$ as noted earlier in the definition of weight sequences. It therefore suffices to prove that $w'(\delta) < w(\gamma)$. Since $f \circ \delta$ has no cancellations and is equivalent to $\gamma$, and since $E_{H,j}$ is the unique edge of $H$ mapped by $f$ to $E_{G,j}$, it is clear that $k'(E_{H,j},\delta) = k(E_{G,j},\gamma)$ for each $j \in J$. It is similarly clear that $k'(E_1,\delta) + k'(E_2,\delta) = k(E,\gamma)$. It therefore suffices to check that each of the two terms $k'(E_1,\delta)$ and $k'(E_2,\delta)$ is positive for then each of $k'(E_1,\delta)$ and $k'(E_2,\delta)$ is strictly smaller than $k(E,\gamma)$, for it then follows that $w'(\delta)$ is obtained from $w(\gamma)$ by decrementing the term $w_{k(E,\gamma)}$ and incrementing two strictly lower terms, proving that $w'(\delta) < w(\gamma)$. 

Fix $i \in \{1,2\}$. Positivity of $k'(E_i,\delta)$ is equivalent to saying that $\delta$ crosses $E_i$ at least once, which we now prove. By Proposition~\ref{PropWhiteheadMapUnion}~\pref{ItemWhiteheadImage}, the map 
$$Df_{v_i} \from W_{v_i} \delta \to W_{v} \gamma
$$
is a nondegenerate simplicial map, and we have
$$(*) \qquad W_1 \union W_2 = W_v \gamma = Df_{v_1} (W_{v_1}\delta) \union Df_{v_2}(W_{v_2}\delta)
$$ 
We know that the direction set of $W_{v_i}\delta$ is contained in $T_{v_i} H$, and we also know from Lemma~\ref{LemmaDirectionSplitProps} that $Df_{v_i}(T_{v_i} H) = T_i$ which contains the direction set of $W_i$; hence we have $Df_{v_i}(W_{v_i} \delta) \subset W_i$ for $i=1,2$. Furthermore we know that $Df_{v_1}(T_{v_1}) \intersect D_{v_2}(T_{v_2}) = T_1 \intersect T_2 = W_1 \intersect W_2 = \{d\}$, and so combined with $(*)$ it follows that $W_i = Df_{v_i}(W_{v_i} \delta)$ for $i=1,2$. From the definition of a cut, the graph $W_i$ contains a turn incident to the direction $d$ having the form $\turn{d d_i}$ for some direction $d_i \in T_i - \{d\}$, and so the graph $W_{v_i}\delta$ contains a turn whose image under $Df_{v_i}$ equals $\turn{d d_i}$. By Lemma~\ref{LemmaDirectionSplitProps}, the unique pre-image of the direction $d \in W_v\gamma$ under the map $D_{v_i} f$ is $d \overline E_i \in W_{v_i}\delta$, and so $W_{v_i}\delta$ has a turn of the form $\turn{d \overline E_i \, d'_i}$ whose image under $Df_{v_i}$ equals $\turn{d d_i}$. This proves that $\delta$ crosses the edge $E_i$ at least once.

\medskip

We turn to the construction of $\delta \from C \to H$. 

We may choose default orientations on each edge of $H$ and $G$ so that $f \from H \to G$ preserves orientation, requiring that the chosen orientations on $E_1,E_2 \subset H$ and on $E \subset G$ satisfy $d_i = d \overline E_i$ and $d=d\overline E$. Recall also the notations $H \setminus (E_1 \union E_j) = \union_{j\in J} E_{H,j}$ and $G \setminus E = \union_{j\in J} E_{G,j}$ with $f(E_{H,j}) = E_{G,j}$.

Subdivide $C$ as a graph so that $\gamma$ takes vertices of $C$ to natural vertices of $G$ and edges of $C$ to natural edges of $G$. Given an edge $\eta \subset C$ its image $\delta(\eta)$ will be an edge of $H$ as we now define. Using the orientation on $\eta$ which is mapped by $\gamma$ to the default orientation on $\gamma(\eta)$, and letting $\eta'$ denote the oriented edge of $C$ just after $\eta$ so that $\eta\eta'$ is a subpath of $C$, we define
$$\delta(\eta) = \begin{cases}
E_{H,j} &\quad\text{if $\gamma(\eta) = E_{G,j} \subset G \setminus E$} \\
E_i &\quad\text{if $\gamma(\eta) = E$ and $d\gamma(\eta') \in T_i$} 
\end{cases}
$$
Consider a vertex $v \in C$, which is the initial point of two distinct oriented edges $\eta,\eta'$ in~$C$, we must check three things:
\begin{description}
\item[(i) Well-definedness:] The oriented edges $\delta(\eta)$, $\delta(\eta') \subset H$ have the same initial vertex in $H$, which we may take to be $\delta(v)$; 
\item[(ii) Immersion:] The two directions $d\delta(\eta)$, $d\delta(\eta') \in T_{\delta(v)} H$ are distinct;
\item[(iii) No illegal turn:] The direction pair $\{d\delta(\eta),d\delta(\eta')\}$ is not the illegal turn of $f$.
\end{description}
Once (i), (ii) and (iii) are verified for all $v$ we will be done, because: from items~(i) and (ii) it follows that $\delta \from C \to H$ is a well-defined continuous immersion; conclusion~\pref{ItemNoIllegalTurn} follows from item~(iii); and from the definition of $\delta$ it follows that $f\circ\delta(\eta)=g(\eta)$ for each edge $\eta \subset C$, hence conclusion~\pref{ItemDeltaGammaEquiv} is satisfied.

Both (ii) and (iii) follow from the fact that $\gamma$ is an immersion, for it then follows that $d\gamma(\eta)$, $d\gamma(\eta')$ are distinct directions in $G$, hence $d\delta(\eta)$, $d\delta(\eta')$ are distinct directions in $H$ and are not both in the set $\{dE_1,dE_2\}$ whose image under $g$ is $\{dE\}$.

To prove (i), letting $p$, $p' \in H$ be the initial vertices of the two edges $\delta(\eta)$, $\delta(\eta')$ respectively, we must prove $p=p'$. Using continuity of $\gamma$ and the fact that $f \circ \delta(\eta)=\gamma(\eta)$ and $f \circ \delta(\eta')=\gamma(\eta')$, it follows that $f(p)=f(p')$ which we denote $w$. If $w \ne v$ then $f$ is one-to-one over $w$ hence $p=p'$. We may therefore assume $w=v$. The direction pair $\{d\gamma(\eta), d\gamma(\eta')\}$ forms a turn in the Whitehead graph $W_v\gamma$, and using the cut $W_v\gamma = W_1 \union W_2$ it follows that this turn is contained in $W_j$ for a unique $j \in \{1,2\}$; by tracing through the definitions of $H$ and of the map $\delta$ it follows that $p=p'=v_j$. 
\end{proof}

\subsection{Statement of Whitehead's Algorithm}
\label{SectionWhiteheadAlgStated}

Informally, the algorithm starts with a marked graph $G$ and a circuit family $\gamma$ in $G$, and then repeats the following loop: as long as there is a vertex in $G$ at which the Whitehead graph of $\gamma$ has a cut point, choose a cut at that point, split $G$ guided by that cut, replace $G$ with the result of that split, and repeat. The algorithm must stop. When it does stop, $G$ has no vertex at which the Whitehead graph has a cut. By inspection, one can now see whether $\gamma$ represents a partial free basis.

Here is the formal statement of the Whitehead's Algorithm, modernized to use the language of marked graphs, folds, and splits.
\begin{description}
\item[Step 1:] Check whether $\gamma$ is jointly primitive. If not, stop, $\gamma$ does not represent a partial free basis.
\item[Step 2:] Loop through the following sequence of computations:
\begin{description}
\item[Step 2a:] Compute the Whitehead graph $W_v\gamma$ at each vertex $v \in G$.
\item[Step 2b:] Check whether the near visibility condition holds, i.e.\ whether for all $v$ the graph $W_v\gamma$ is a disjoint union of edges. If so, stop, $\gamma$ represents a partial free basis.
\item[Step 2c:] Check whether the cut vertex condition holds, i.e.\ whether there exists $v$ such that some component of $W_v\gamma$ has a cut vertex. If not, step, $\gamma$ does not represent a partial free basis.
\item[Step 2d:] Having reached this step, there does exist a vertex $v \in V$ and a component of $W_v\gamma$ having a cut vertex. Choose a cut $\wh W_v\gamma = \wh W_1 \union \wh W_2$. Following the algorithm in Proposition~\ref{PropSplit}, construct the marked graph $H$, the fold $f \from H \to G$, and the circuit family $\delta$ in $H$. Replace $G$ and $\gamma$ by $H$ and $\delta$, and go back to the beginning of Step 2.
\end{description}
\end{description}
By Proposition~\ref{PropSplit}, the weight sequence of $\gamma$ strictly decreases under each iteration of~Step 2, hence the algorithm must stop.

\section{Connectivity of outer space and applications (draft)}
\label{SectionOuterSpaceConnectivity}


\begin{description}
\item[Disclaimer and Critique from the Author:] In the current version of the preceding sections, we have focussed on Stallings fold \emph{sequences}, which are sequences of marked graphs and maps between them. A top priority in the next revision is to related these sequences more clearly to \emph{paths} in outer space. For example, the proofs in this section, and the statement of Lemma~\ref{LemmaNonbigonFolds}, should be rewritten to express them in terms of paths in the spine of outer space rather than sequences of outer space cells. This will be helpful for the proof of \emph{connectivity} given in this section, although not strictly necessary. But this is of absolute necessity before Skora's ``fold path'' proof of \emph{contractibility} can be presented in Section~\ref{SectionOuterSpaceContractibility}.
\end{description}

\bigskip

The main theorem of this section, proved in Section~\ref{SectionProofOfConnectivity}, is:
\begin{theorem}\label{TheoremPathConnected}
Outer space $\X_n$ and its spine $\K_n$ are path connected.
\end{theorem}

Recall from Section~\ref{SectionOutActsGeometrically} that path connectivity of $\K_n$ was the last missing piece in the proof of Theorem~\ref{TheoremOutActionGeometric}, which says that the action of $\Out(F_n)$ on $\K_n$ is geometric. We may therefore apply the first sentence of the Milnor-Svarc Lemma (Lemma~\ref{LemmaMilnorSvarc}) to conclude:

\begin{corollary}
$\Out(F_n)$ is finitely generated.\qed
\end{corollary}

In Section~\ref{SectionRank2OuterSpace} we shall apply path connectivity in rank~$2$ to derive the complete geometric structure of $\X_2$ and $\K_2$, and to describe the complete algebraic structure of $\Out(F_2)$.

In Section~\ref{SectionWhiteheadNielsenGenerators} we shall apply path connectivithy to derive specific generating sets for $\Out(F_n)$, namely Nielsen's generators (see Section~\ref{SectionPositiveTests}) as well as Whitehead's generators. As it turns out, while Whitehead constructed his generating set some years after Nielsen, a good logical progression is to first derive Whitehead's generators, and then use those to derive Nielsen's generators, and this is what we do in Section~\ref{SectionWhiteheadNielsenGenerators}.

\subsection{Proof of path connectivity (draft)}
\label{SectionProofOfConnectivity}

Given points $x,y$ contained in two outer space cells $\Delta(G)$, $\Delta(H)$ represented by marked graphs $G,H$, we shall construct a sequence of marked graphs
$$G=G_0,G_1,\ldots,G_K=H
$$
such that for any $k=1,\ldots,K$, either $\Delta(G_{k-1})$ is a face of $\Delta(G_k)$, or $\Delta(G_k)$ is a face of $\Delta(G_{k-1})$; possibly the inclusion is not proper in which case $\Delta(G_{k-1})=\Delta(G_k)$. A path from $x$ to $y$ is then easily produced: choose a sequence of points $x=x_0,x_1,\ldots,x_K=y$ so that $x_i \in \Delta(G_k)$; and for each $k=1,\ldots,K$ choose a path $\sigma_k$ connecting $x_{k-1}$ to $x_k$ in whichever of the two cells $\Delta(G_{k-1})$ or $\Delta(G_k)$ contains the other. The concatenation $\sigma_1 * \ldots * \sigma_K$ thus connects $x$ to~$y$.

Let $\rho_G$, $\rho_H$ denote the given markings on $G$ and $H$. Choose a homotopy equivalence from $G$ to $H$ that preserves marking, and apply Proposition~\ref{PropTightening} to homotope it to a tight homotopy equivalence map $f \from G=G_0 \to H$ that preserves marking. Apply Proposition~\ref{PropFoldableMapConstruction} to factor $f$ up to homotopy as
$$\xymatrix{
G = G'_0 \ar[r]_{f_1} \ar@/^1pc/[rr]^{f} & G'_1 \ar[r]_{h^1} & H
}
$$
so that $f_1$ is a collapse map and $h^1$ is a foldable map. Applying conclusion~\pref{ItemNotPiOneInjective} of Proposition~\ref{PropFoldableMapConstruction}, both of $f_1$ and $h^1$ are homotopy equivalences, in fact the subgraph of $G'_0$ that is collaped by $q$ is a subforest, and so $G'_1$ is a rank~$n$ core graph. 

Next apply Stallings Fold Theorem~\ref{ThmStallingsFolds} to factor~$h^1$, thus obtaining a further factorization of $f$ up to homotopy:
$$\xymatrix{
G = G'_0 \ar[r] _{f_1} \ar@/^4pc/[rrrrr]^{f} 
& G'_1 \ar[r]_{f_2}  \ar@/^2pc/[rrrr]^{h^1}
& G'_2 \ar[r]_{f_3} & \cdots \ar[r]_{f_{J}} & G'_J \ar[r]_{h^J} & H
}$$
From the conclusions of Theorem~\ref{ThmStallingsFolds}, each of $f_2,\ldots,f_J$ is a fold map, each of $G'_2,\ldots,G'_J$ is a core graph, and $f^J$ is locally injective. Furthermore, by applying conclusions~\pref{ItemContinuingFolds} and~\pref{ItemStallingsFoldPath} of Theorem~\ref{ThmStallingsFolds}, and using that $H$ is a core graph and the map $h^1$ is a homotopy equivalence, it follows that each of $f_2,\ldots,f_J$ is not a bigon fold hence is a homotopy equivalence, each of $G'_2,\ldots,G'_J$ is a rank~$n$ core graph, and $h^J$ is a homeomorphism.

Applying Exercise~\ref{ExerciseDiagramMarkingOne}, there exist unique markings on $G'_1,\ldots,G'_J,H$ such that each of the maps $f_1,\ldots,f_J,h^J$ preserves marking; let $\rho'_H$ denote that marking on $H$. Since $f$ is homotopic to $h^J \circ f_J \circ\ldots\circ f_1$, the two markings $\rho_H,\rho'_H$ agree up to homotopy (this paragraph can be slightly shortened by applying a good solution to Exercise~\ref{ExerciseDiagramMarkingTwo} instead of Exercise~\ref{ExerciseDiagramMarkingOne}).

Using all the markings just obtained, consider now the sequence of outer space cells 
$$\Delta(G)=\Delta(G'_0), \, \Delta(G'_1), \, \ldots  \, , \, \Delta(G'_J)=\Delta(H)
$$
We already know that $f_1 \from G'_0 \to G'_1$ is a map that collapses a subforest of $G'_0$, hence $\Delta(G'_1)$ is a face of $\Delta(G'_0)$. Applying Lemma~\ref{LemmaNonbigonFolds}, there exist marked graphs $G''_2,\ldots,G''_J$ such that for $j=2,\ldots,J$ the cells $\Delta(G'_{j-1}),\Delta(G'_j)$ are faces of $\Delta(G''_j)$, and the proof is complete.

\subsection{Outer space in rank~$2$  (stub)}
\label{SectionRank2OuterSpace}

\subsection{Whitehead's generators and Nielsen's generators (stub)}
\label{SectionWhiteheadNielsenGenerators}
%
%
%
%
%

\section{Contractibility of outer space and applications (stub)}
\label{SectionOuterSpaceContractibility}

\begin{theorem}\label{TheoremContractible}
Outer space $\X_n$ and its spine $\K_n$ are contractible.
\end{theorem}


\part{Conjugacy growth and relative train track maps (stub)}

\chapter{Conjugacy growth in $\Out(F_n)$: Concepts and examples (stub)}

\section{Conjugacy growth of outer automorphisms: general concepts (stub)}

\section{Example: Conjugacy growth in $\GL_n(\Z)$ (stub)}

\section{Growth in $\Out(F_n)$: Examples, questions, theorems. (stub)}

\chapter{Relative train track maps (stub)}

\part{Periodic points of topological representatives and attracting laminations (stub)}

\chapter{Laminations (stub)}

\bibliographystyle{amsalpha} 
\bibliography{mosher} 

\printindex

\end{document}